\documentclass[a4paper,11pt]{article}
\usepackage{amsfonts}
\usepackage{mathrsfs}
\usepackage{amsmath,amssymb,amsthm}
\usepackage{mathtools}
\usepackage{latexsym}
\usepackage{indentfirst}
\usepackage[top=1in, bottom=1in, left=1in, right=1in]{geometry}
\usepackage{color}
\usepackage{footmisc}
\usepackage{endnotes}
\usepackage{algpseudocode}
\usepackage{algorithm}
\usepackage{float}
\usepackage{graphicx}
\usepackage{caption}
\usepackage{subcaption}
\usepackage{epstopdf}
\usepackage{enumitem}
\usepackage{comment}
\usepackage{longtable}
\usepackage{multirow}
\usepackage[font={footnotesize}]{caption}
\usepackage{authblk}
\usepackage{bookmark}
\usepackage{natbib}
\usepackage[toc,page,titletoc]{appendix}
\usepackage[capitalise,nameinlink]{cleveref}
\usepackage{arydshln}

\let\footnote=\endnote

\newcommand{\GG}[1]{}

\def\T{{ \mathrm{\scriptscriptstyle T} }}

\DeclareMathOperator*{\pr}{Pr}
\DeclareMathOperator*{\Cov}{Cov} \DeclareMathOperator*{\Var}{Var}

\DeclareMathOperator*{\diag}{diag}

\DeclareMathOperator*{\dett}{det}
\DeclareMathOperator*{\gp}{\Gcal}
\DeclareMathOperator*{\kl}{KL}
\DeclareMathOperator*{\op}{op}
\DeclareMathOperator*{\tap}{tap}

\DeclareMathOperator*{\tr}{tr}

\newcommand{\ud}{\mathrm{d}}

\newcommand{\HW}{\mathrm{HW}}

\newtheorem{theorem}{Theorem}
\newtheorem{lemma}{Lemma}

\newtheorem{proposition}{Proposition}
\newtheorem{assumption}{Assumption}
\newtheorem{example}{Example}
\allowdisplaybreaks

\def\pof{\text{Proof of}}

\DeclareMathOperator{\Acal}{\mathcal{A}}
\DeclareMathOperator{\Bcal}{\mathcal{B}}

\DeclareMathOperator{\Dcal}{\mathcal{D}}
\DeclareMathOperator{\Ecal}{\mathcal{E}}

\DeclareMathOperator{\Gcal}{\mathcal{G}}

\DeclareMathOperator{\Ical}{\mathcal{I}}

\DeclareMathOperator{\Kcal}{\mathcal{K}}
\DeclareMathOperator{\Lcal}{\mathcal{L}}

\DeclareMathOperator{\Ncal}{\mathcal{N}}

\DeclareMathOperator{\Scal}{\mathcal{S}}

\DeclareMathOperator{\Ucal}{\mathcal{U}}

\DeclareMathOperator{\Wcal}{\mathcal{W}}

\def\RR{\mathbb{R}}
\def\NN{\mathbb{N}}
\def\ZZ{\mathbb{Z}}

\def\EE{\mathbb{E}}
\def\PP{\mathbb{P}}

\def\bb{\mathrm{b}}
\def\ee{\mathrm{e}}
\def\ff{\mathrm{f}}

\def\mm{\mathrm{m}}

\def\vv{\mathrm{v}}

\def\bfb{\mathbf{b}}
\def\bfe{\mathbf{e}}

\def\bfi{\mathbf{i}}
\def\bfj{\mathbf{j}}
\def\bfk{\mathbf{k}}

\def\bfs{\mathbf{s}}
\def\bft{\mathbf{t}}
\def\bfx{\mathbf{x}}

\def\sD{\mathsf{D}}
\def\sK{\mathsf{K}}
\def\sa{\mathsf{a}}
\def\sj{\mathsf{j}}

\allowdisplaybreaks

\begin{document}
\title{\bf Fixed-domain Posterior Contraction Rates for Spatial Gaussian Process Model with Nugget}

\author[1]{Cheng Li \thanks{stalic@nus.edu.sg}}
\author[1]{Saifei Sun \thanks{sunsf1015@163.com}}
\affil[1]{Department of Statistics and Data Science, National University of Singapore}
\author[2]{Yichen Zhu \thanks{yichen.zhu@duke.edu}}
\affil[2]{Department of Statistical Science, Duke University}

\date{}
\maketitle

\begin{abstract}
Spatial Gaussian process regression models typically contain finite dimensional covariance parameters that need to be estimated from the data. We study the Bayesian estimation of covariance parameters including the nugget parameter in a general class of stationary covariance functions under fixed-domain asymptotics, which is theoretically challenging due to the increasingly strong dependence among spatial observations. We propose a novel adaptation of the Schwartz's consistency theorem for showing posterior contraction rates of the covariance parameters including the nugget. We derive a new polynomial evidence lower bound, and propose consistent higher-order quadratic variation estimators that satisfy concentration inequalities with exponentially small tails. Our Bayesian fixed-domain asymptotics theory leads to explicit posterior contraction rates for the microergodic and nugget parameters in the isotropic Mat\'ern covariance function under a general stratified sampling design. We verify our theory and the Bayesian predictive performance in simulation studies and an application to sea surface temperature data.
\end{abstract}

\section{Introduction} \label{sec:intro}
Gaussian process has been widely used in spatial statistics for modeling spatial and spatiotemporal correlations. This paper studies the parameter estimation in the following spatial Gaussian process regression model:
\begin{align} \label{eq:obs.model}
& Y(\bfs) = \ff(\bfs)^\T  \beta + X(\bfs) + \varepsilon(\bfs), \text{ for } \bfs\in [0,1]^d,
\end{align}
where $d$ is the dimension of domain, $\ff(\cdot)=(\ff_1(\cdot),\ldots,\ff_p(\cdot))^\T  \in \RR^p$ is a $p$-dimensional vector of deterministic known functions on $[0,1]^d$, and $\beta\in \RR^p$ is the regression coefficient vector. In applications, $\ff_1,\ldots,\ff_p$ can include the constant function $1$, and hence $\beta$ can include an intercept term. The term $X(\cdot)$ is a Gaussian process $X=\left\{X(\bfs): \bfs\in [0,1]^d \right\}$, and $\{\varepsilon(\bfs): \bfs\in [0,1]^d\}$ is a Gaussian white noise process independent of $X(\cdot)$ that satisfies $\varepsilon(\bfs)\sim \Ncal(0,\tau)$ for all $\bfs\in [0,1]^d$. In practice, we observe the process $Y(\cdot)$ and the functions $\ff(\cdot)$ on a set of distinct sampling points $S_n=\{\bfs_1,\ldots,\bfs_n\}\subseteq [0,1]^d$. The observed data from the model \eqref{eq:obs.model} are $Y_n=(Y(\bfs_1),\ldots,Y(\bfs_n))^\T $ and $ F_n  = (\ff(\bfs_1)^\T ,\ldots,\ff(\bfs_n)^\T ) \in \RR^{n\times p}$. All statistical inference is based on the data $(Y_n, F_n)$.

Model \eqref{eq:obs.model} is one of the most important models in spatial statistics and can be used as building blocks for other more sophisticated models. Besides the regression term $\ff(\bfs)^\T  \beta$, the Gaussian process term $X(\cdot)$ captures the spatial association and explains the spatial random effects from unmeasured or unobserved covariates with spatial pattern (\citealt{Banetal08}). The noise $\varepsilon(\cdot)$ captures the measurement error, whose variance parameter $\tau$ is also known as the nugget parameter.

It remains to specify the Gaussian process for $X(\cdot)$. We assume that the mean function of $X(\cdot)$ is zero, and the covariance function of $X(\cdot)$ can be written as $\Cov(X(\bfs),X(\bft)) = \theta K_{\alpha,\nu}(\bfs-\bft)$ for $\bfs,\bft\in [0,1]^d$, where $\theta>0$ is the microergodic parameter that will be defined later, $K_{\alpha,\nu}(\cdot):\RR^d\to\RR$ is a positive definite function, $\alpha$ is the spatial range parameter, and $\nu$ is the smoothness parameter. We will give several examples of covariance functions with this form in Section \ref{sec:general.theorem}. For example, one popular choice in spatial statistics is the Mat\'ern covariance function (\citealt{Stein99a})
\begin{align} \label{eq:MaternCov}
& \Cov(X(\bfs),X(\bft)) = \sigma^2\frac{2^{1-\nu}}{\Gamma(\nu)} \left(\alpha\|\bfs-\bft\|\right)^{\nu} \Kcal_{\nu}\left(\alpha\|\bfs-\bft\|\right),
\end{align}
where $\nu>0$ is the smoothness parameter, $\sigma^2>0$ is the variance (partial sill) parameter, $\alpha>0$ is the inverse range (or length-scale) parameter, $\Gamma(\cdot)$ is the gamma function, $\Kcal_{\nu}(\cdot)$ is the modified Bessel function of the second kind, and $\|\cdot\|$ is the Euclidean norm. If we let $\theta=\sigma^2\alpha^{2\nu}$  and $K(x)=\{2^{1-\nu}/\Gamma(\nu)\} \left(\|x\|/\alpha\right)^{\nu} \Kcal_{\nu}\left(\alpha\|x\|\right)$ for $x\in \RR^d$, then \eqref{eq:MaternCov} can be written as $\theta K_{\alpha,\nu}(\bfs-\bft)$ for $\bfs,\bft\in [0,1]^d$. For the Gaussian processes $X$ and $Y$ in Model \eqref{eq:obs.model}, we write $X\sim \gp(0,\theta K_{\alpha,\nu})$ and $Y\sim \gp( \beta^\T  \ff, \theta K_{\alpha,\nu} + \tau \delta_0)$, respectively, where $\delta_0$ denotes the Dirac delta function at zero.

We study the Bayesian inference of Model \eqref{eq:obs.model} (\citealt{HanSte93}, \citealt{DeO97}), where the common practice is to assign prior distributions on the regression coefficients $\beta$, the nugget parameter $\tau$, and the parameters in the covariance function of $X(\cdot)$ such as $\theta$ and $\alpha$. Bayesian posterior inference, including the prediction of $Y(s^*)$ at a new location $s^*\in [0,1]^d$, is fully based on the posterior distribution of these model parameters. In application, these parameters are randomly drawn from their posterior using sampling algorithms such as Markov chain Monte Carlo. Therefore, the behavior of their posterior distributions will heavily affect the Bayesian posterior predictive performance. However, despite the routine practice of using them in spatial statistics, there is a severe lack of theoretical understanding of Bayesian parameter estimation in the general spatial Gaussian process regression model with nugget \eqref{eq:obs.model}, especially on the asymptotic behavior of the posterior distribution of model parameters when the number of observations becomes large. Such theory becomes increasingly important given the advance in high-resolution remote sensing technology, and the research on Gaussian process with massive spatial datasets is prevalent in Bayesian spatial statistics, such as \citet{Banetal08}, \citet{SanHua12}, \citet{Datetal16}, \citet{Heatonetal18}, \citet{KatGui21}, \citet{Peretal21}, \citet{Guhetal22}, etc.

We consider the Bayesian fixed-domain asymptotics (or infill asymptotics) framework (\citealt{Stein99a}, \citealt{Zhang04}). For the first time in the literature, we derive the fixed-domain Bayesian posterior contraction rates of model parameters in the general spatial Gaussian process regression model \eqref{eq:obs.model} with both the regression term $\ff(\bfs)^\T \beta$ and the measurement error (with the nugget $\tau$). In the fixed-domain asymptotics regime, the spatial domain $[0,1]^d$, as we have assumed, always remains fixed and bounded. This implies that as $n$ increases, the sampling locations in $S_n$ become increasingly dense in the domain $[0,1]^d$, leading to increasingly stronger dependence between adjacent observations in the data $Y_n$. The fixed-domain asymptotics has several advantages over other regimes such as increasing-domain asymptotics (\citealt{MarMar84}). First, a fixed domain matches up with the reality in many spatial applications. For instance, the advances in remote sensing technology enable routine collection of spatial data in larger volume and higher resolution from a given region (\citealt{Sunetal18}). Second, the stationarity assumption on the Gaussian process $X(\cdot)$ is more likely to hold on a fixed domain rather than an expanding domain. Therefore, the fixed-domain asymptotics regime is more suitable for interpolation of spatial processes (Section 3.3 of \citealt{Stein99a}). Third, fixed-domain asymptotics has better parameter estimation performance than the increasing-domain asymptotics (\citealt{ZhaZim05}).

On the other hand, fixed-domain asymptotics also poses several significant theoretical challenges. We take the isotropic Mat\'ern covariance function for example. The first challenge comes from the lack of identification for the covariance parameters $(\sigma^2,\alpha)$ in \eqref{eq:MaternCov} when the domain dimension $d=1,2,3$, due to the increasingly stronger dependence in the data $Y_n$. We briefly review some basic results in equivalence of Gaussian processes. If two Gaussian processes are defined on a fixed domain, then by the result in \citet{IbrRoz78} and Chapter 4 of \citet{Stein99a}, they must be either equivalent or orthogonal to each other, where the equivalence means that their Gaussian measures are mutually absolutely continuous. In particular, \citet{Zhang04} has shown the well-known result that when the spatial domain has the dimension $d=1,2,3$, for two isotropic Mat\'ern with parameters $(\sigma^2_1,\alpha_1)$ and $(\sigma^2_2,\alpha_2)$ and the same smoothness $\nu$, their Gaussian measures are equivalent if $\sigma_1^2\alpha_1^{2\nu}=\sigma_2^2\alpha_2^{2\nu}$, and they are orthogonal otherwise. This equivalence result for $d=1,2,3$ is a consequence of the integral test for the spectral densities of two covariance functions; see for example, Theorem A.1 of \citet{Stein04}. On the other hand, when $d\geq 5$, \citet{And10} has proposed consistent moment estimators for both $\sigma^2$ and $\alpha$ with a given $\nu$ under fixed-domain asymptotics. Recently \citet{BolKir21} have shown that for any $d\geq 4$, two Gaussian measures are equivalent if and only if they have the same parameters $(\sigma^2,\alpha,\nu)$, though finding the consistent estimators of $(\sigma^2,\alpha)$ for $d=4$ under fixed-domain asymptotics requires further study.

A direct consequence of the theory on equivalence of Gaussian processes is that for the domain dimension $d=1,2,3$, which is of primary interest to spatial statistics, only the microergodic parameter $\theta=\sigma^2\alpha^{2\nu}$ and continuous functions of $\theta$ can be consistently estimated from the data under fixed-domain asymptotics. However, the individual variance parameter $\sigma^2$ and range parameter $\alpha$ have no consistent estimators when $d=1,2,3$, and therefore they cannot have Bayesian posterior consistency. In general, the microergodic parameter is defined to be the parameter that uniquely determines the equivalence class of Gaussian processes; see Section 6.2 of \citet{Stein99a} for a detailed explanation on the microergodic parameter. Besides Mat\'ern, the exact form of microergodic parameter $\theta$ for $d=1,2,3$ has been identified in several other families, such as generalized Wendland (\citealt{Bevetal19}) and confluent hypergeometric covariance functions (\citealt{MaBha22}); see Examples \ref{ex:GW} and \ref{ex:CH} in Section \ref{sec:general.theorem}. The microergodic parameter has been further studied for covariance functions on non-Euclidean domains, such as spheres (\citealt{Araetal18}) and Riemannian manifolds (\citealt{Lietal21}), where \citet{Lietal21} have discovered a similar phenomenon for the Mat\'ern covariogram when $d=1,2,3$. Consequently, statistical inference for fixed-domain asymptotics when $d=1,2,3$ differs completely from that of regular parametric models with independent or weakly dependent data. For example, instead of increasing with the sample size, the Fisher information for the non-microergodic parameters $(\sigma^2,\alpha)$ has a finite limit in the case $d=1,2,3$ under fixed-domain asymptotics (\citealt{ZhaZim05}, \citealt{Loh05}).

The second main challenge comes from the impact of the regression terms $\ff(\bfs)^\T \beta$ and the nugget parameter $\tau$ on the asymptotics of other parameters, in particular the microergodic parameter $\theta$ when $d=1,2,3$. In real spatial applications, typically one assumes that $Y_n$ from Model \eqref{eq:obs.model} contains some regression terms and measurement error, though most existing frequentist fixed-domain asymptotics works have only considered estimation in a Gaussian process without regression terms and measurement error (\citealt{Ying91}, \citealt{Zhang04}, \citealt{Duetal09}, \citealt{WangLoh11}, \citealt{KauSha13}, \citealt{Bevetal19}, \citealt{MaBha22}, etc.) In these works, the maximum likelihood estimator (MLE) of the microergodic parameter $\theta$ is asymptotically normal with a parametric $n^{-1/2}$ convergence rate. The cross-validation estimator (\citealt{Bacetal17}) and the composite likelihood estimator (\citealt{Bacetal19}) have the same parametric rate for Mat\'ern with $\nu=1/2$ and $d=1$. This is no longer true for the microergodic parameter $\theta$ in Model \eqref{eq:obs.model}. Although the nugget $\tau$ is also identifiable in Model \eqref{eq:obs.model} (\citealt{Stein99a}, \citealt{Tangetal21}), it will dramatically change the convergence rate for the estimators of $\theta$. For example, \citet{Chenetal00} have studied a special case of Model \eqref{eq:obs.model} without regression terms, with $d=1$, $\nu=1/2$, and $S_n$ being the equispaced grid in $[0,1]$. The convergence rate for the MLE of $\theta$ slows down from $n^{-1/2}$ in the model without $\tau$ to $n^{-1/4}$ in the model with $\tau$. Under the assumptions of equispaced grid and exact decaying rate of eigenvalues of the Mat\'ern covariance matrix, \citet{Tangetal21} have shown that the asymptotic normality for the MLE of $\theta$, whose convergence rate is $n^{-1/(2+4\nu/d)}$. This deterioration is mainly caused by the convolution of $X(\cdot)$ with the Gaussian noise $\varepsilon(\cdot)$ in Model \eqref{eq:obs.model}. As a consequence, the Bayesian fixed-domain asymptotic theory for the model with nugget \eqref{eq:obs.model} becomes much more challenging than the Bayesian theory for the model without nugget (\citealt{Li20}).

The main contribution of this work is to devise a general theoretical framework for studying the Bayesian fixed-domain posterior contraction rates of the microergodic parameter $\theta$ and the nugget parameter $\tau$ in Model \eqref{eq:obs.model} where $X(\cdot)$ has the covariance function $\theta K_{\alpha,\nu}(\bfs-\bft)$, which covers a wide range of stationary covariance functions in spatial statistics; see our Examples \ref{ex:Matern}-\ref{ex:CH} in Section \ref{sec:general.theorem}. In Bayesian asymptotic theory, the posterior contraction rate describes how fast the posterior distribution of certain parameters contracts to their true values. General Bayesian theory of posterior consistency and contraction rates has been extensively studied; see Chapters 6--9 in \citet{GhoVan17} for a thorough treatment. The main idea in establishing such Bayesian asymptotic results is to adapt the posterior consistency theorem from \citet{Sch65}. That is, for models with independent data, one only needs to verify two conditions: (a) the prior distribution contains a Kullback-Leibler support of the target parameter or nonparametric function, and (b) the existence of an exponentially consistent test. Our theoretical framework resembles the Schwartz's consistency framework, but with two key adaptations to the spatial Gaussian process regression model \eqref{eq:obs.model} under fixed-domain asymptotics.

First, we establish a new polynomial lower bound for the denominator of the posterior density given the strongly dependent data $Y_n$ under fixed-domain asymptotics, based on the spectral analysis of covariance function. This is vastly different from the lower bound of denominator for independent data where the lower bound is established by the law of large numbers in the original Schwartz's theorem. Based on this new evidence lower bound, we establish a general posterior contraction rate theorem for the microergodic parameter and the nugget parameter in Model \eqref{eq:obs.model} under fixed-domain asymptotics.

Second, we propose new exponentially consistent tests for the microergodic parameter $\theta$ and the nugget parameter $\tau$ using higher-order quadratic variation estimators. We illustrate this using the isotropic Mat\'ern covariance function in Example \ref{ex:Matern}. Higher-order quadratic variation has been successfully applied to parameter estimation in frequentist fixed-domain asymptotics for Gaussian process models without nugget (\citealt{Loh15}, \citealt{Lohetal21}). We further generalize this powerful tool to Model \eqref{eq:obs.model} with nugget, similar to the recent work \citet{LohSun23}.

Our new techniques allow us to bypass the difficult problem of expanding the likelihood function for $(\theta,\tau)$ in the presence of nonidentifiable parameters such as the range parameter $\alpha$. For a fixed value of range parameter $\alpha$, the MLE of $(\theta,\tau)$ in Model \eqref{eq:obs.model} without the regression terms $\ff(\bfs)^\T \beta$ satisfies the frequentist asymptotic normality (Theorem 5 of \citealt{Tangetal21}). However, these MLEs do not have any closed-form expression or even first-order approximation except for a few special cases (such as $\nu=1/2$, $d=1$, $S_n$ being an equispaced grid in \citealt{Chenetal00}), which makes it technically highly challenging to construct exponentially consistent tests based on these MLEs. Our higher-order quadratic variation estimators are constructive, explicit, and completely circumvent the difficulty caused by the complicated likelihood function. They can be directly applied to Model \eqref{eq:obs.model} with the regression terms $\ff(\bfs)^\T \beta$. More importantly, they satisfy concentration inequalities with exponentially small tails which work uniformly well for the range parameter $\alpha$ in a wide expanding interval and for a large class of stratified sampling designs without requiring the sampling points $S_n$ to be on an equispaced grid. Details will be discussed in Section \ref{sec:estimator}.

The rest of the paper is organized as follows. Section \ref{sec:general.theorem} presents a new general framework for deriving posterior contraction rates for model parameters under fixed-domain asymptotics. Section \ref{sec:estimator} presents the details of our high-order quadratic variation estimators and the explicit posterior contraction rates for the isotropic Mat\'ern covariance function. Section \ref{sec:numerical} contains simulation studies and a real data example that verify our theory. Section \ref{sec:discussion} includes some discussion. The technical proofs of all theorems and propositions as well as additional simulation results can be found in the Supplementary Material.

\section{General Theory for Fixed-Domain Posterior \\ Contraction Rates} \label{sec:general.theorem}
We first provide some examples of the assumed covariance function $\theta K_{\alpha,\nu}(\bfs-\bft)$ for $\bfs,\bft\in [0,1]^d$. This format of covariance function is very general and includes several classes of covariance functions used in spatial statistics.
\begin{example} \label{ex:Matern}
The isotropic Mat\'ern covariance function defined in \eqref{eq:MaternCov} takes the form $\theta K_{\alpha,\nu}$ with $\theta=\sigma^2\alpha^{2\nu}$ and $K_{\alpha,\nu}(x)=\{2^{1-\nu}/\Gamma(\nu)\} \left(\|x\|/\alpha\right)^{\nu} \Kcal_{\nu}\left(\alpha\|x\|\right)$ for $x\in \RR^d$.
\end{example}
\begin{example} \label{ex:taper}
The tapered isotropic Mat\'ern covariance function (\citealt{Kauetal08}) is $\theta \tilde K_{\alpha,\nu}(\bfs-\bft)= \theta K_{\alpha,\nu}(\bfs-\bft) \times K_{\tap}(\bfs-\bft)$ for any $\bfs,\bft\in [0,1]^d$, where $\theta K_{\alpha,\nu}(\bfs-\bft)$ is the isotropic Mat\'ern covariance function in Example \ref{ex:Matern}, and $K_{\tap}: \RR^d \to \RR$ is a positive definite function, whose spectral density $f_{\tap}(w) = (2\pi)^{-d}\int_{\RR^d} \exp\left(-\imath w^\T  x \right) K_{\tap}(x) \ud x$ satisfies $0< f_{\tap}(w)\leq C_{\tap} (1+\|w\|^2)^{-(\nu+d/2+\eta_{\tap})}$ for some constants $C_{\tap}>0,\eta_{\tap}>0$ and all $w\in \RR^d$, where $\imath^2=-1$.
\end{example}
\begin{example} \label{ex:GW}
The isotropic generalized Wendland covariance function (\citealt{Bevetal19}) is $\theta K_{\alpha,\nu}(\bfs-\bft)$ for any $\bfs,\bft\in [0,1]^d$, where
$\theta=\sigma^2 \alpha^{2\nu}$,
\begin{align} \label{eq:GWCov}
& K_{\alpha,\nu}(x) = \frac{\Ical(0\leq \|x\|<1/\alpha)}{\alpha^{2\nu} B(2\nu,\mu)}\int_{\alpha\|x\|}^1 \big(t^2-\alpha^2\|x\|^2 \big)^{\nu-1/2} (1-t)^{\mu -1} \ud t, \quad \text{ for } x\in \RR^d,
\end{align}
and $\sigma^2>0$, $\alpha>0$, $\nu \geq 1/2$, $\mu > \nu + d$, $B(\cdot,\cdot)$ is the Beta function, $\Ical(\cdot)$ is the indicator function. We assume that $\mu$ is known and suppress the dependence of $K_{\alpha,\nu}(\cdot)$ on $\mu$.
\end{example}
\begin{example} \label{ex:CH}
The isotropic confluent hypergeometric covariance function (\citealt{MaBha22}) is $\theta K_{\alpha,\nu}(\bfs-\bft)$ for any $\bfs,\bft\in [0,1]^d$, where $\theta=\sigma^2 \alpha^{2\nu} \Gamma(\nu+\mu)/\Gamma(\mu)$,
\begin{align} \label{eq:CHCov}
& K_{\alpha,\nu}(x) = \frac{1}{\Gamma(\nu)} \int_0^{\infty} t^{\nu-1} \big(\alpha^2t+1\big)^{-(\nu+\mu)} \exp\left(-\nu \|x\|^2/t \right) \ud t , \quad \text{ for } x\in \RR^d,
\end{align}
where $\sigma^2>0$, $\alpha>0$, $\nu>0$ and $\mu>0$. We assume that $\mu$ is known and suppress the dependence of $K_{\alpha,\nu}(\cdot)$ on $\mu$, though $\mu$ can also be estimated as in \citet{MaBha22}.
\end{example}
Among these examples, the Mat\'ern covariance function is of central importance since the other three can be viewed as its generalizations. The tapered Mat\'ern and the generalized Wendland covariance functions have compact supports and hence computationally more efficient, while the confluent hypergeometric is suitable for modeling polynomially decaying spatial dependence.

Throughout the paper, we assume that the smoothness parameter $\nu>0$ is fixed and known. Estimation of $\nu$ is an important and technically challenging problem for spatial statistics, with some recent advances in the frequentist fixed-domain asymptotics (\citealt{Loh15}, \citealt{Lohetal21}, \citealt{LohSun23}); see our detailed discussion in Section \ref{sec:discussion}. Our main goal is to perform Bayesian inference on the parameters $(\theta,\alpha,\tau,\beta)$ in Model \eqref{eq:obs.model} based on the observed data $(Y_n, F_n)$. We assume that the true parameters are $(\theta_0,\alpha_0,\tau_0,\beta_0)$. We use $\PP_{(\theta,\alpha,\tau,\beta)}$ to denote the probability measure of $\gp(\beta^\T  \ff,\sigma_0^2K_{\alpha_0,\nu}+\tau \delta_{0})$, and hence $\PP_{(\theta_0,\alpha_0,\tau_0,\beta_0)}$ is the true probability measure of the observed process $Y(\cdot)$.

The log-likelihood function based on the data $(Y_n, F_n )$ is
\begin{align} \label{eq:loglik}
\Lcal_n(\theta,\alpha,\tau,\beta) &= - \frac{1}{2} (Y_n-  F_n  \beta)^\T  \left\{\theta K_{\alpha,\nu}(S_n) + \tau I_n \right\}^{-1} (Y_n-  F_n  \beta)  \nonumber \\
&\quad  - \frac{1}{2}\log
\dett\left\{\theta K_{\alpha,\nu}(S_n) + \tau I_n \right\},
\end{align}
where $K_{\alpha,\nu}(S_n)$ is the $n\times n$ covariance matrix whose $(i,j)$-entry is $K_{\alpha,\nu}(\bfs_i-\bfs_j)$, $I_n$ is the $n\times n$ identity matrix, $\dett(A)$ is the determinant of a matrix $A$, and $cA$ denotes the matrix of $A$ with all entries multiplied by the number $c$.

For Bayesian inference, we impose a prior distribution with the density $\pi(\theta,\alpha,\tau,\beta)$ on the parameters. Let $\RR_+=(0,+\infty)$ and $\ZZ_+$ be the set of all positive integers. Then the posterior density of $(\theta,\alpha,\tau,\beta)$ is
\begin{align}\label{jointpost1}
\pi(\theta,\alpha,\tau,\beta~|~ Y_n, F_n ) &= \frac{\exp\left\{ \Lcal_n(\theta,\alpha,\tau,\beta) \right\} \pi(\theta,\alpha,\tau,\beta)}{\int_{\RR_+^{3}\times \RR^p} \exp\left\{\Lcal_n(\theta',\alpha',\tau',\beta') \right\} \pi(\theta',\alpha',\tau',\beta')\ud \theta' \ud \tau' \ud \alpha' \ud \beta' } .
\end{align}
We use $\Pi(\cdot)$ and $\Pi(\cdot \mid Y_n, F_n )$ to denote the prior and posterior measures, respectively.

Our general framework requires a few assumptions on the model setup.
\begin{assumption} \label{cond:model}
The data $(Y_n, F_n)$ are observed on the distinct sampling points $S_n=\{\bfs_i:i=1,\ldots,n\}$ in a fixed domain $[0,1]^d$ following the model \eqref{eq:obs.model}. The sequence of $S_n$ is getting dense in $[0,1]^d$ as $n\to\infty$, in the sense that $\sup_{\bfs^*\in [0,1]^d} \min_{1\leq i\leq n} \|\bfs^*-\bfs_i\| \to 0$ as $n\to\infty$. There exists a constant $C_{\ff}>0$ such that $|\ff_l(\bfs)|\leq C_{\ff}$ for all $l=1,\ldots,p$ and all $\bfs\in [0,1]^d$.
\end{assumption}
\begin{assumption} \label{cond:spectral}
Let $f_{\theta,\alpha,\nu}(w)= (2\pi)^{-d}\int_{\RR^d} \exp(-\imath w^\T  x) \theta K_{\alpha,\nu}(x) \ud x$ be the spectral density of the covariance function $\theta K_{\alpha,\nu}$, where $\imath^2 = -1$ and $\alpha \in \RR_+ $. Then \\
\noindent (i) There exist constants $L>0,r_0\in (0,1/2),\kappa>0$ that may depend on $\theta_0,\alpha_0,\nu,d$, such that for all $\alpha$ that satisfies $|\alpha/\alpha_0 -1|\leq r_0$,
\begin{align*}
\sup_{w\in \RR^d} \left|f_{\theta_0,\alpha,\nu}(w)/f_{\theta_0,\alpha_0,\nu}(w)-1\right| \leq L |\alpha/\alpha_0-1|^{\kappa}.
\end{align*}
\noindent (ii) For any given $\theta>0$ and $w\in \RR^d$, $f_{\theta,\alpha,\nu}(w)$ is a non-increasing function in $\alpha$.
\end{assumption}
\begin{proposition} \label{prop:4cov}
All the four covariance functions in Examples \ref{ex:Matern}--\ref{ex:CH} satisfy Assumption \ref{cond:spectral}.
\end{proposition}
Assumption \ref{cond:spectral} imposes mild and general conditions on the spectral density $f_{\theta,\alpha,\nu}$. For example, the spectral density of the isotropic Mat\'ern in Example \ref{ex:Matern} is $f_{\theta,\nu}(w)=\Gamma(\nu+d/2)/\Gamma(\nu) \times \pi^{-d/2}\theta (\alpha^2+\|w\|^2)^{\nu+d/2}$, which is both Lipschitz continuous ($\kappa=1$) and decreasing in $\alpha$. The same can be proved for Examples \ref{ex:taper}--\ref{ex:CH}. Assumption \ref{cond:spectral} makes our Bayesian fixed-domain asymptotic theory much more general and works far beyond the popular Mat\'ern covariance function. The proof of Proposition \ref{prop:4cov} is in the Supplementary Material.
\begin{assumption} \label{cond:prior1}
The prior probability measure $\Pi(\cdot)$ does not depend on $n$ and has a proper density $\pi(\theta,\alpha,\tau,\beta)$ that is continuous in a neighborhood of $(\theta_0,\alpha_0,\tau_0,\beta_0)$, and satisfies $\pi(\theta_0,\alpha_0,\tau_0,\beta_0)>0$. The marginal prior of $\tau$ satisfies $\int_0^{\infty} \tau^{-n/2}\pi(\tau) \ud \tau <\infty$ for any $n\in \ZZ_+$.
\end{assumption}
Assumption \ref{cond:prior1} requires that the prior distribution has some positive probability mass around the true parameter $(\theta_0,\alpha_0,\tau_0,\beta_0)$. This is a mild condition and necessary for showing posterior consistency. The finite integral condition on $\pi(\tau)$ is to ensure the propriety of the posterior density defined in \eqref{jointpost1}, and is satisfied by the commonly used inverse gamma prior on $\tau$; see Proposition \ref{prop:prior} in Section \ref{subsec:matern.rate} below.

To formulate the posterior contraction, for any $\epsilon_1>0,\epsilon_2>0$, we define the neighborhood
$$\Bcal_0(\epsilon_1,\epsilon_2)=\{(\theta,\alpha,\tau,\beta): |\theta/\theta_0-1| < \epsilon_1 , |\tau/\tau_0-1|< \epsilon_2, \alpha \in \RR_+, \beta \in \RR^p \} .$$
This is an open neighborhood for the microergodic parameter $\theta$ and the nugget parameter $\tau$, but with the range parameter $\alpha$ and the regression coefficient $\beta$ unconstrained. We will focus on showing that the posterior probability of this neighborhood converges to one and characterize the dependence of $\epsilon_1,\epsilon_2$ on the sample size $n$ under the fixed-domain asymptotics regime.

Our posterior contraction results will exclude the range parameter $\alpha$ and the regression coefficient $\beta$. This is because they are in general inconsistent under fixed-domain asymptotics following the general frequentist theory of \citet{Zhang04} when the domain dimension $d=1,2,3$. This will be proved rigorously in Section S3 of the Supplementary Material. In general, their inconsistency is not likely to affect the asymptotic prediction performance of Model \eqref{eq:obs.model} as long as $(\theta,\tau)$ are consistent (\citealt{Stein90b}).

In the original Schwartz's consistency theorem, one important step is to show that the denominator of the posterior density has a lower bound. In the regular parametric models with independent observations, this is achieved by applying concentration inequalities to the log-likelihood ratio together with Jensen's inequality and Fubini's theorem; see for example, the proof of Theorem 6.17 and Lemma 8.10 in \citet{GhoVan17}. This lower bound of denominator is also known as the \textit{evidence lower bound}, which plays a key role in determining the posterior contraction rates.

Our first result is the following theorem that characterizes the evidence lower bound based on the strongly dependent data $Y_n$ observed from Model \eqref{eq:obs.model} under fixed-domain asymptotics.
\begin{theorem}[Fixed-domain Evidence Lower Bound] \label{thm:conv.rate.denom}
Suppose that Assumptions \ref{cond:model}, \ref{cond:spectral} and \ref{cond:prior1} hold. There exists a constant $D>0$ and a large integer $N_1$ that only depend on $\nu,d,p$, $\theta_0,\tau_0,\alpha_0$, $C_{\ff},L,r_0,\kappa$ and $\pi(\theta_0,\tau_0,\alpha_0,\beta_0)$, such that for all $n>N_1$, with probability at least $1-3\exp(-\log^2 n)$,
\begin{align}
& \int \exp\left\{\Lcal_n(\theta,\alpha,\tau,\beta) - \Lcal_n(\theta_0,\alpha_0,\tau_0,\beta_0) \right\} \cdot \pi(\theta,\alpha,\tau,\beta)~ \ud \theta \ud \tau \ud \alpha \ud \beta \geq  Dn^{-(3p+2+2/\kappa)}.  \nonumber
\end{align}
\end{theorem}
Theorem \ref{thm:conv.rate.denom} shows a polynomially decaying evidence lower bound for the posterior distribution, which can be used for showing posterior contraction together with some exponentially consistent tests. The major difference between the proof of Theorem \ref{thm:conv.rate.denom} and that of regular parametric models with independent data is that our data $Y_n$ are strongly dependent and our likelihood ratio depends heavily on the covariance matrix $\theta K_{\alpha,\nu}(S_n)+\tau I_n$, where $\theta K_{\alpha,\nu}(S_n)$ becomes increasingly singular as $S_n$ becomes denser in the domain $[0,1]^d$. Thus the standard techniques for showing evidence lower bound with independent data such as Lemma 8.10 of \citet{GhoVan17} cannot be directly applied here. Fortunately, the structure of this covariance matrix can be related to the spectral density $f_{\theta,\alpha,\nu}$, such that under Assumption \ref{cond:spectral}, the likelihood ratio can be lower bounded when $(\theta,\alpha,\tau,\beta)$ is in a shrinking neighborhood of $(\theta_0,\alpha_0,\tau_0,\beta_0)$.

The next two assumptions are on the existence of consistent estimators for $(\theta,\tau)$ and the prior distribution. We define the function $\varphi(x) = \min(x^2,x)$ for $x>0$, and $\Ecal^c$ to be the complement of a generic set $\Ecal$.
\begin{assumption} \label{cond:exptest1}
There exist estimators $\widehat\theta_n$ and $\widehat\tau_n$, positive constants $b_1,b_2,c_1,c_2$, and a parameter set $\Ecal_n\subseteq \RR_+^{3} \times \RR^p$, such that for any $\epsilon_1,\epsilon_2 \in (0,1/2)$, for all $n>n_0$ where $n_0$ depends on $\epsilon_1,\epsilon_2$,
\begin{align}
\PP_{(\theta_0,\alpha_0,\tau_0,\beta_0)} \left(|\widehat\theta_n/\theta_0-1|\geq \epsilon_1/2 \right) & \leq \exp\left\{- c_1 \varphi\left( n^{b_1} \epsilon_1 \right)\right\}, \label{eq:test.theta1} \\
\PP_{(\theta_0,\alpha_0,\tau_0,\beta_0)} \left( |\widehat\tau_n/\tau_0-1|\geq \epsilon_2/2 \right) & \leq \exp\left\{ - c_2 \varphi\left( n^{b_2} \epsilon_2 \right) \right\}, \label{eq:test.tau1} \\
\sup_{\Bcal_0(\epsilon_1,\epsilon_2)^c \cap \Ecal_n} \PP_{(\theta,\alpha,\tau,\beta)} \left( |\widehat\theta_n/\theta_0-1|\leq \epsilon_1/2 \right) & \leq \exp\left\{ - c_1 \varphi\left(n^{b_1} \epsilon_1\right)\right\}, \label{eq:test.theta2} \\
\sup_{\Bcal_0(\epsilon_1,\epsilon_2)^c \cap \Ecal_n} \PP_{(\theta,\alpha,\tau,\beta)} \left(  |\widehat\tau_n/\tau_0-1|\leq\epsilon_2/2 \right) & \leq \exp\left\{- c_2 \varphi\left( n^{b_2} \epsilon_2\right)\right\}, \label{eq:test.tau2}
\end{align}
where the supremum is taken over $(\theta,\alpha,\tau,\beta)$.
\end{assumption}
\begin{assumption} \label{cond:prior2}
For the parameter set $\Ecal_n$ in Assumption \ref{cond:exptest1}, the prior satisfies $\Pi(\Ecal_n^c ) \leq n^{-(3p+4+2/\kappa)}$ for all sufficiently large $n$, where $\kappa$ is given in Assumption \ref{cond:spectral}.
\end{assumption}
Assumption \ref{cond:exptest1} assumes the existence of consistent estimators $\widehat\theta_n$ and $\widehat\tau_n$ with exponential tail probabilities for large values of $\epsilon_1,\epsilon_2$ and sub-Gaussian tails for small values of $\epsilon_1,\epsilon_2$, which are reasonable given the normality assumption on both $X(\cdot)$ and $\varepsilon(\cdot)$ and can usually be derived from the Hanson-Wright inequality (\citealt{RudVer13}). The four inequalities are used for constructing uniformly consistent tests in the Schwartz's theorem; see for example, Proposition 6.22 in \citet{GhoVan17}. In particular, the set $\Ecal_n$ is a sieve parameter space, which typically expands as $n$ increases and eventually covers the whole parameter space. Furthermore, we require in Assumption \ref{cond:prior2} that the prior probability of the sieve space $\Ecal_n$ converges to one faster than the evidence lower bound in Theorem \ref{thm:conv.rate.denom}. This puts a mild restriction on the tail behavior of the prior distribution, which is generally weaker than assuming a prior tail probability exponentially small in $n$ as used in the Bayesian nonparametrics literature; see for example, Condition (iii) in Theorems 8.9 and 8.11 in \citet{GhoVan17}. In Section \ref{sec:estimator}, we will show that the higher-order quadratic variation estimators of $\theta$ and $\tau$ satisfy Assumption \ref{cond:exptest1}, and will discuss the concrete prior tail conditions that satisfy Assumption \ref{cond:prior2}.

The following theorem provides the general posterior contraction rates for $\theta$ and $\tau$.
\begin{theorem} [General Fixed-domain Posterior Contraction Rates] \label{thm:convergence}
Suppose that Assumptions \ref{cond:model}, \ref{cond:spectral}, \ref{cond:prior1}, \ref{cond:exptest1}, and \ref{cond:prior2} hold. Then the posterior distribution $\Pi(\cdot \mid Y_n, F_n )$ with the density given in \eqref{jointpost1} exists and contracts to the true parameters $(\theta_0,\tau_0)$ at the rate $(n^{-b_1} \log n, n^{-b_2}\log n)$ with $b_1$ and $b_2$ given in Assumption \ref{cond:exptest1}, in the sense that for any positive sequence $M_n\to \infty$ as $n\to\infty$,
\begin{align}
\Pi \left( |\theta/\theta_0 - 1| < M_n n^{-b_1} \log n \text{ and } |\tau/\tau_0 - 1|< M_n n^{-b_2}\log n \mid Y_n, F_n  \right) \to 1,  \nonumber
\end{align}
almost surely $\PP_{(\theta_0,\alpha_0,\tau_0,\beta_0)}$ as $n\to\infty$.
\end{theorem}
Theorem \ref{thm:convergence} is an adapted version of the Schwartz's theorem and is established based on our new evidence lower bound in Theorem \ref{thm:conv.rate.denom}, as well as the uniformly consistent tests constructed from the estimators $\widehat\theta_n$ and $\widehat\tau_n$ in Assumption \ref{cond:exptest1}. Theorem \ref{thm:convergence} essentially translates the complicated problem of finding Bayesian posterior contraction rates under fixed-domain asymptotics to the more direct problem of finding the frequentist consistent estimators $\widehat\theta_n$ and $\widehat\tau_n$ that satisfy Assumption \ref{cond:exptest1}. The provable rates $n^{-b_1} \log n$ for the microergodic parameter $\theta$ and $n^{-b_2} \log n$ for the nugget parameter $\tau$ are completely determined by how efficient the frequentist estimators of $\theta$ and $\tau$ are inside the four inequalities of Assumption \ref{cond:exptest1}. We find $b_1$ and $b_2$ explicitly in Section \ref{sec:estimator}.

\section{Higher-Order Quadratic Variation Estimators for \\ Isotropic Mat\'ern} \label{sec:estimator}

For the isotropic Mat\'ern covariance function in Example \ref{ex:Matern}, we verify Assumptions \ref{cond:exptest1} and \ref{cond:prior2} by proposing the higher-order quadratic variation estimators of the microergodic parameter $\theta$ and the nugget parameter $\tau$. We explicitly find the rates $n^{-b_1} \log n$ and $n^{-b_2} \log n$ in Theorem \ref{thm:convergence}, and thus provide the posterior contraction rates for $\theta$ and $\tau$ in Model \eqref{eq:obs.model} under fixed-domain asymptotics. To focus on the main idea and also due to the space limit, we will only show the rates for isotropic Mat\'ern which can be compared directly with the existing frequentist fixed-domain asymptotic results. Our techniques can also be extended to the three covariance functions in Examples \ref{ex:taper}, \ref{ex:GW} and \ref{ex:CH} with additional technical adjustment.

Higher-order quadratic variation has been adopted in recent theoretical works of spatial Gaussian processes (\citealt{Loh15}, \citealt{Lohetal21}, \citealt{LohSun23}). The basic idea is to use a carefully designed series of constants ($c_{\bfi,d,\ell}^{(k_1,\ldots,k_d)}$ and $c_{d,\ell}^{(k_1,\ldots,k_d)}$ in Lemma \ref{lem:Cor1.Loh21}), such that finite differencing the observations $Y_n$ weighted by these constants can approximately solve for both the microergodic and nugget parameters. This method has at least two major technical advantages. First, as a frequentist estimator, it is simple to implement in practice and does not involve any numerical optimization such as the maximum likelihood estimator. Second, it does not require the observations $Y_n$ to be made strictly on an equispaced grid and hence is applicable to a much wider range of sampling designs. The previous works \citet{Loh15} and \citet{Lohetal21} have considered such estimators for the model without measurement error and nugget parameter $\tau$, while our version of higher-order quadratic variation estimators are similar to those in \citet{LohSun23} which work for the general spatial model \eqref{eq:obs.model} with measurement error and the nugget $\tau$.

\subsection{Construction of Estimators} \label{subsec:construction.estimator}
We consider the stratified sampling design, in which the domain $[0,1]^d$ is divided into cells with the same size and observations are made inside each cell, but not necessarily on the grid points. Without loss of generality, we let $n=m^d$ for an integer $m$ since we only study the asymptotics in $n$. We rewrite the set $S_n=\{\bfs_1,\ldots,\bfs_n\}$ using the ordering in each coordinate:
\begin{align} \label{eq:Sn}
S_n=\left\{\bfs(\bfi)=(s_1(\bfi),\ldots,s_d(\bfi))^\T :~ \bfi=(i_1,\ldots,i_d), 1\leq i_1,\ldots,i_d\leq m\right\},
\end{align}
and for $k=1,\ldots,d$, $s_k(\bfi)$ satisfies
\begin{align} \label{eq:si}
s_k(\bfi) &= \frac{i_k-1}{m} + \frac{\delta_{\bfi;k}}{m} ,
\end{align}
where $0\leq \delta_{\bfi;k}<1$ is arbitrary for all $k=1,\ldots,d$. We make the following assumption on the stratified sampling design.
\begin{assumption} \label{cond:sampling}
Suppose that the sampling points $S_n$ satisfies \eqref{eq:Sn} and \eqref{eq:si}, where all $\delta_{\bfi;k}$'s can take arbitrary values in $[0,1)$. Then for any $1\leq i_1,\ldots,i_d\leq m$,
\begin{align*}
& \bfs(\bfi) \in \left[\frac{i_1-1}{m}, \frac{i_1}{m} \right) \times \ldots \times \left[\frac{i_d-1}{m}, \frac{i_d}{m} \right) .
\end{align*}
\end{assumption}
This stratified sampling design is much more relaxed than assuming sampling locations strictly on equispaced grid points. We allow the perturbations $\delta_{\bfi;k}$ to be arbitrary in $[0,1)$, so there is little restriction other than that the sampling locations need to be roughly evenly distributed across the domain.
\begin{assumption} \label{cond:deriv}
For the regression functions $\ff_1,\ldots,\ff_p$, there exists a constant $C_{\ff}'>0$ such that their partial derivatives satisfy $\left|\sD^{\sj}~\ff_l(\bfs)\right| \leq C_{\ff}'$ for all $l=1,\ldots,p$, all $\bfs\in [0,1]^d$, and all index vector $\sj=(j_1,\ldots,j_d)\in \NN^d$ that satisfies $j_1+\ldots +j_d\leq \lceil \nu+d/2\rceil$, where $\lceil x\rceil$ is the smallest integer greater than or equal to $x$, and $\sD^{\sj}$ is the partial differentiation operator of order $\sj$.
\end{assumption}
Assumption \ref{cond:deriv} requires the regression functions $\ff_1,\ldots,\ff_p$ to have bounded derivatives up to the order of $\lceil \nu+d/2\rceil$. Because the sample path $X(\cdot)$ of the Mat\'ern covariance function \eqref{eq:MaternCov} are only mean square differentiable up to the order of $\nu$, we essentially assume that the functions $\ff_1,\ldots,\ff_p$ are smoother than the sample path $X(\cdot)$. Furthermore, the order $\lceil \nu+d/2\rceil$ implies that the functions $\ff_1,\ldots,\ff_p$ all lie in the reproducing kernel Hilbert space of $\theta K_{\alpha,\nu}$, which is known to be norm equivalent to the Sobolev space of order $\nu+d/2$ (Corollary 10.48 of \citealt{Wen05}), and therefore satisfy Assumption \ref{cond:deriv}.

Let $\lfloor x\rfloor$ be the largest integer less than or equal to $x$. Define $\omega_m=\lfloor m^{\gamma} \rfloor$ for a positive constant $\max\big\{1-d/(4\nu),0\big\} < \gamma < 1$ and assume that $\omega_m$ is an even integer without loss of generality. The higher-order quadratic variation method relies on the following sequence of $m$-dependent (and so $n$-dependent) constants $c_{\bfi,d,\ell}^{(k_1,\ldots,k_d)}$ and $c_{d,\ell}^{(k_1,\ldots,k_d)}$ described in \citet{Lohetal21}.
\begin{lemma}\label{lem:Cor1.Loh21}
(Corollary 1 and Lemma 2 of \citealt{Lohetal21}) Let $d,\ell\in \ZZ_+$. Let $\bfi=(i_1,\ldots,i_d)^\T $ where $1\leq i_1,\ldots,i_d\leq m - \ell \omega_m$. Then there exists a sequence of constants
$$\left\{c_{\bfi, d, \ell}^{(k_1,\ldots,k_d)}: 0\leq k_1,\ldots,k_d\leq \ell, \text{and } \bfi= (i_1,\ldots,i_d), 1\leq i_1,\ldots,i_d\leq m \right\}, $$
such that for any $\bfs(\bfi) = (s_1(\bfi),\ldots,s_d(\bfi))^\T  \in [0,1]^d$, for all integers $l_1,\ldots,l_d$ satisfying $0\leq l_1,\ldots,l_{d-1}\leq \ell$, $0\leq l_d\leq \ell-1$, and $0\leq l_1+\ldots + l_d\leq \ell$,
\begin{align*}
& \textstyle \sum_{0\leq k_1,\ldots,k_d\leq \ell} c_{\bfi,d,\ell}^{(k_1,\ldots,k_d)} \cdot s_d(i_1+k_1\omega_m,\ldots,i_d+k_d\omega_m)^{\ell} = \ell! \left(\omega_m/m\right)^{\ell}, \\
\text{and } & \textstyle \sum_{0\leq k_1,\ldots,k_d\leq \ell} c_{\bfi,d,\ell}^{(k_1,\ldots,k_d)} \cdot \textstyle \prod_{j=1}^d s_j(i_1+k_1\omega_m,\ldots,i_d+k_d\omega_m)^{l_j} = 0 .
\end{align*}
Furthermore, there exists a sequence of $m$-independent constants $c_{d,\ell}^{(k_1,\ldots,k_d)}$, such that
\begin{align*}
& c_{\bfi,d,\ell}^{(k_1,\ldots,k_d)} = c_{d,\ell}^{(k_1,\ldots,k_d)} + O(\omega_m^{-1}), \quad \text{for all } 0\leq k_1,\ldots,k_d\leq \ell,
\end{align*}
as $m\to\infty$, where the term $O(\omega_m^{-1})$ is uniform over all $1\leq i_1,\ldots,i_d\leq m- \ell\omega_m$ and any $\delta_{\bfi;k}\in [0,1)$ for $k=1,\ldots,d$.
\end{lemma}

For $d, \ell \in \ZZ_+$ and $\bfi=(i_1,\ldots, i_d)^\T$ such that $1\leq i_1, \ldots, i_d\leq m-\ell \omega_m$, we
define the following differencing operator
\begin{align*}
& \nabla_{ d, \ell} Y \left(\bfs(\bfi) \right)
=  \textstyle \sum_{0\leq k_1, \ldots, k_d\leq \ell}  c_{\bfi, d, \ell}^{(k_1, \ldots, k_d)}  Y \left(\bfs (i_1+k_1 \omega_m , \ldots, i_d+ k_d \omega_m ) \right),
\end{align*}
and we define the $\ell$th-order quadratic variation $V_{u,d, \ell} = V_{u,d, \ell}(m)$ by
\begin{align} \label{eq:3.1}
& V_{u, d, \ell} = \textstyle \sum_{\bfi\in \Xi_{u,m}} \left\{ \nabla_{d, \ell} Y \left(\bfs(\bfi)\right) \right\}
\left \{ \nabla_{d, \ell} Y \left(\bfs(\bfi + u\bfe_1) \right) \right\},
\end{align}
where $\bfe_1 = (1,0,\ldots,0)^\T  \in \RR^d$, $u\in\{0,1\}$ and
$\Xi_{u,m}=\left\{\bfi:~ 1\leq i_1+ u,i_1, \ldots, i_d \leq m- 2 \ell \omega_m \right\}$. The cardinality of the set $\Xi_{u,m}$, denoted by $|\Xi_{u,m}|$, has the same order as $m^d$.

We set $\ell_{\star}=\lceil \nu+d/2 \rceil$. Our higher-order quadratic variation estimators for $\tau$ and $\theta$ are defined as follows:
\begin{align}
\widehat{\tau}_n &= V_{0,d,\ell_{\star}} / C_{V,0},  \qquad \widehat{\theta}_n = V_{1, d,\ell_{\star}} / g_{\ell_{\star},\nu} ,\label{eq:tau.theta.es} \\
\text{where }
C_{V,0} &= \textstyle\sum_{\bfi \in \Xi_{0,m}} \textstyle\sum_{0\leq k_1, \ldots, k_{d}\leq \ell_{\star}} \left\{c_{\bfi, d, \ell_{\star}}^{(k_1, \ldots, k_d)} \right\}^2, \quad g_{\ell_{\star},\nu} = \left(\omega_m/m\right)^{2 \nu} |\Xi_{1, m}| \xi^*_{\nu} H_{\ell_{\star},\nu},  \nonumber \\
H_{\ell_{\star},\nu} &= \textstyle\sum_{0\leq k_1,\ldots, k_{2d}\leq \ell_{\star}} c_{d,\ell_{\star}}^{(k_1,\ldots, k_d)} c_{d,\ell_{\star}}^{(k_{d+1},\ldots, k_{2d})}
	G_\nu \left( \big\|  (k_1,\ldots, k_d)^\T  - (k_{d+1},\ldots, k_{2d} )^\T  \big\| \right), \nonumber \\
\text{and } G_\nu (t)& = t^{2 \nu}, ~~ \xi^*_{\nu} = -\pi/\{2^{2\nu}\Gamma(\nu+1)\Gamma(\nu)\sin(\nu\pi) \}  , \text{ if } \nu\notin \ZZ, \nonumber \\
G_\nu (t)& = t^{2 \nu} \log(t), ~~ \xi^*_{\nu} = (-1)^{\nu+1}/\{2^{2\nu-1}\nu!(\nu-1)!\} , \text{ if } \nu\in \ZZ_+ . \label{eq:CgH}
\end{align}

We also define the set $\Ecal_{n}$ in Assumption \ref{cond:exptest1} as
\begin{align} \label{eq:E1n}
\Ecal_{n}=\left\{ (\theta,\alpha,\tau,\beta)\in \RR_+^{3}\times \RR^p: \frac{\|\beta\|^2}{\theta}  \leq  n^{\rho_1}, \frac{\tau}{\theta} \in [n^{-\rho_{21}},n^{\rho_{22}}], \alpha\in [n^{-\rho_{31}},n^{\rho_{32}}] \right\} ,
\end{align}
where the constants $\rho_1,\rho_{21},\rho_{22},\rho_{31},\rho_{32}$ satisfy the following relations:
\begin{align} \label{eq:rho.relation1}
& 0< \rho_1 < 1-\gamma,  \quad
0<\rho_{21} < 2\nu(1-\gamma)/d, \nonumber \\
& 0<\rho_{22} < 1/2 - 2(1-\gamma)\nu/d, \quad \rho_{31}>0, \quad 0<\rho_{32} < (1-\gamma)/d.
\end{align}
In the next theorem, we show that the estimators of $\widehat{\tau}_n$ and $\widehat\theta_n$ in \eqref{eq:tau.theta.es} together with the set $\Ecal_{n}$ in \eqref{eq:E1n} satisfy Assumption \ref{cond:exptest1}.

\begin{theorem}\label{thm:E1n.rate}
Suppose that $X(\cdot)$ in Model \eqref{eq:obs.model} has the isotropic Mat\'ern covariance function in \eqref{eq:MaternCov}. Suppose that Assumptions \ref{cond:model}, \ref{cond:sampling} and \ref{cond:deriv} hold. Let $\omega_m=\lfloor m^{\gamma} \rfloor$ be an even number for a constant $\max\big\{1-d/(4\nu),0\big\} < \gamma < 1$. Let $\Ecal_n$ in Assumption \ref{cond:exptest1} be defined in \eqref{eq:E1n} with $\rho_1,\rho_{21},\rho_{22},\rho_{31},\rho_{32}$ satisfying \eqref{eq:rho.relation1}. Then uniform over all $\delta_{\bfi;k}\in [0,1)$ with $\bfi=(i_1,\ldots,i_d)^\T $ for all $1\leq i_1,\ldots,i_d\leq m$ and $1\leq k \leq d$, for any $d\in \ZZ_+$,
\begin{itemize}[leftmargin=7mm]
\item[(i)] $\widehat{\theta}_n$ in \eqref{eq:tau.theta.es} satisfies \eqref{eq:test.theta1} and \eqref{eq:test.theta2} in Assumption \ref{cond:exptest1}, where $b_1$ is
\begin{align} \label{eq:E1n.b1}
b_1& = \min\Bigg\{\frac{1}{2}  - \frac{2(1-\gamma)\nu}{d}- \rho_{22} , ~ \frac{1-\gamma}{2} - \varsigma , ~\frac{1}{4} + \frac{(1-\gamma)(\ell_{\star}-2\nu)}{d}- \frac{\rho_{1} +\rho_{22}}{2} , \nonumber \\
& \qquad \qquad \frac{1-\gamma}{4} + \frac{(1-\gamma)(\ell_{\star}-\nu)}{d} - \frac{\rho_{1}}{2} - \varsigma  \Bigg\}  ,
\end{align}
for an arbitrarily small constant $\varsigma>0$.
\item[(ii)] $\widehat{\tau}_n$ in \eqref{eq:tau.theta.es} satisfies \eqref{eq:test.tau1} and \eqref{eq:test.tau2} in Assumption \ref{cond:exptest1}, where $b_2$ is
\begin{align} \label{eq:E1n.b2}
b_2& =  \min\Bigg\{\frac{1}{2} , ~ \frac{(1-\gamma)(4\nu+d)}{2d} - \rho_{21} - \varsigma , ~ \frac{1}{4} +\frac{ (1-\gamma)\ell_{\star}}{d} - \frac{\rho_{1} + \rho_{21}}{2}, ~ \nonumber \\
& \qquad \qquad \frac{(1-\gamma)(4\nu+d +4\ell_{\star})}{4d} - \frac{\rho_{1} + 2\rho_{21} }{2}  - \varsigma  \Bigg\} ,
\end{align}
for an arbitrarily small constant $\varsigma>0$.
\end{itemize}
\end{theorem}

Theorem \ref{thm:E1n.rate} shows that the higher-order quadratic variation estimators $\widehat{\theta}_n$ and $\widehat{\tau}_n$ satisfy Assumption \ref{cond:exptest1}. We explicitly give the values of $b_1$ and $b_2$ in the four exponential tail inequalities in Assumption \ref{cond:exptest1}, which directly determine the posterior contraction rates for $\theta$ and $\tau$ in Theorem \ref{thm:convergence}. Further simplified values of $b_1$ and $b_2$ will be given in Section \ref{subsec:matern.rate}.

Theorem \ref{thm:E1n.rate} demonstrates several advantages of the higher-order quadratic variation estimators $\widehat{\theta}_n$ and $\widehat{\tau}_n$ for constructing the exponentially consistent tests needed for our Bayesian theory. First, they satisfy Assumption \ref{cond:exptest1} with the sieve $\Ecal_n$ defined in \eqref{eq:E1n}, where the range parameter $\alpha$ lies in an expanding interval $[n^{-\rho_{31}},n^{\rho_{32}}]$ which eventually covers the entire $\RR_+$ as $n\to\infty$. Though varying $\alpha$ is technically challenging in studying frequentist estimators (\citealt{Chenetal00} and \citealt{Tangetal21}), we do not need to fix $\alpha$ at a given value and we can assign a general prior on $\alpha$. Second, the inequalities in Assumption \ref{cond:exptest1} hold uniformly over all possible stratified sampling designs in Assumption \ref{cond:sampling}, for arbitrary values of $\delta_{\bfi;k}\in [0,1)$. This generality significantly broadens the applicability of our results to real-world spatial data, as we do not need the sampling locations $S_n$ to be exactly on equispaced grids as in many frequentist fixed-domain asymptotics works. Third, Theorem \ref{thm:E1n.rate} holds for any domain dimension $d\in \ZZ_+$, which includes both the case of $d\in \{1,2,3\}$ when the range parameter $\alpha$ cannot be consistently estimated (\citealt{Zhang04}), and the case of $d\geq 5$ when $\alpha$ can be consistently estimated (\citealt{And10}).

\subsection{Explicit Posterior Contraction Rates for Isotropic Mat\'ern} \label{subsec:matern.rate}
We recall that if Assumptions \ref{cond:model}-\ref{cond:prior2} hold, then Theorem \ref{thm:convergence} shows that the posterior contraction rates for $\theta$ and $\tau$ are $n^{-b_1}\log n$ and $n^{-b_2}\log n$, respectively. Theorem \ref{thm:E1n.rate} provides the explicit values for $b_1$ and $b_2$ based on the higher-order quadratic variation estimators in \eqref{eq:tau.theta.es}. In the following, we simplify their general expressions in \eqref{eq:E1n.b1} and \eqref{eq:E1n.b2}. The positive constants $\rho_1,\rho_{21},\rho_{22},\rho_{31},\rho_{32}$ need to satisfy \eqref{eq:rho.relation1}, but can be all taken as arbitrarily small and close to zero. The condition on $\gamma$ in Theorem \ref{thm:E1n.rate} is equivalent to $1-\gamma\in \big(0,\min\{d/(4\nu),1\}\big)$, and $\gamma$ is related to the cell size $\omega_m=\lfloor m^{\gamma} \rfloor$ of the higher-order quadratic variation estimators. We also emphasize that the value of $\gamma$ can be chosen differently for the two estimators $\widehat\theta_n$ and $\widehat\tau_n$, since in Assumption \ref{cond:exptest1}, the inequalities \eqref{eq:test.theta1} and \eqref{eq:test.theta2} for $\widehat\theta_n$ and \eqref{eq:test.tau1} and \eqref{eq:test.tau2} for $\widehat\tau_n$ are fully separate and do not affect one another.

First we consider $b_1$ defined in \eqref{eq:E1n.b1}, which is related to the posterior contraction rate of the microergodic parameter $\theta$. When the constants $\rho_1,\rho_{21},\rho_{22},\rho_{31},\rho_{32},\varsigma$ are very close to zero, we can show that the first two terms in the minimum are smaller than the last two terms, i.e., $b_1\approx \min\big\{0.5 - 2(1-\gamma)\nu/d, (1-\gamma)/2 \big\}$. We can choose $\gamma$ to balance the two terms by setting $1-\gamma=1/(4\nu/d+1)$, such that $b_1\approx 1/\{2(4\nu/d+1)\}$. From our general Theorem \ref{thm:convergence}, the posterior contraction rate of $\theta$ then becomes $n^{-b_1}\log n\approx n^{-1/\{2(4\nu/d+1)\}} \log n$. This rate seems to be slightly slower than the rate $n^{-1/\{2(2\nu/d+1)\}}$ for the MLE in Theorem 5 of \citet{Tangetal21}, where the observations $Y_n$ are assumed to be made on an equispaced grid. Our stratified sampling design in Assumption \ref{cond:sampling} does not require the equispaced grid. Meanwhile, it is completely unknown whether the faster rate $n^{-1/\{2(2\nu/d+1)\}}$ for $\theta$ can still hold under our stratified sampling design.  We emphasize that the implied posterior contraction rate for $\theta$ (Theorem \ref{thm:post.rate2} below) is the first of its kind in the Bayesian literature, and the optimal convergence rate for such stratified sampling design in Assumption \ref{cond:sampling} remains an open problem.

Next, we consider $b_2$ defined in \eqref{eq:E1n.b2}, which is related to the posterior contraction rate of the nugget parameter $\tau$. Again, we set the constants $\rho_1,\rho_{21},\rho_{22},\rho_{31},\rho_{32},\varsigma$ to be very close to zero. For $\gamma$, we choose it to be very close to its lower bound $\max\{0,1-d/(4\nu)\}$, such that the minimum in \eqref{eq:E1n.b2} is equal to $1/2$. As a result, from our general Theorem \ref{thm:convergence}, the largest possible posterior contraction rate for the nugget parameter $\tau$ implied by our Theorem \ref{thm:E1n.rate} is $n^{-1/2}\log n$. This is almost the same rate (up to the logarithm factor) as the parametric rate $n^{-1/2}$ for the MLE of $\tau$ as shown for the special case of $d=1,\nu=1/2$ in \citet{Chenetal00} and for the general cases in Theorem 5 of \citet{Tangetal21}.

We summarize the analysis above for the posterior contraction rates in the following theorem, which works for the general stratified sampling design in Assumption \ref{cond:sampling} and for any domain dimension $d\in \ZZ_+$.

\begin{theorem}[Explicit Posterior Contraction Rates for Isotropic Mat\'ern] \label{thm:post.rate2}
Suppose that $X(\cdot)$ in Model \eqref{eq:obs.model} has the isotropic Mat\'ern covariance function in Example \ref{ex:Matern}. Suppose that Assumptions \ref{cond:model}, \ref{cond:prior1}, \ref{cond:prior2}, \ref{cond:sampling} and \ref{cond:deriv} hold with $\kappa=1$, $\Ecal_n$ defined in \eqref{eq:E1n} and $\rho_1,\rho_{21},\rho_{22},\rho_{31},\rho_{32}$ satisfying
\begin{align} \label{eq:rho.relation2}
& 0< \rho_1 < \frac{d}{4\nu+d},   ~~ 0< \rho_{21} <
\min\left(\frac{2\nu}{4\nu+d}, \frac{d}{8\nu} \right) ,
~~ 0<\rho_{22} < \frac{d}{2(4\nu+d)},   \nonumber \\
&  \rho_{31}>0, ~~ 0<\rho_{32} < \frac{1}{4\nu+d},~~ \rho_{1} + \rho_{21} < \min\left(\frac{1}{2}+\frac{2\nu}{d}, \frac{d}{4\nu}\right) .
\end{align}
Then uniform over all $\delta_{\bfi;k}\in [0,1)$ with $\bfi=(i_1,\ldots,i_d)^\T $ for all $1\leq i_1,\ldots,i_d\leq m$ and $1\leq k \leq d$, for any positive sequence $M_n\to \infty$ as $n\to\infty$,
\[\Pi \left(  |\theta/\theta_0 - 1| < M_n n^{-\frac{1}{2(4\nu/d+1)}+\varrho} \log n, \text{ and } |\tau/\tau_0 - 1|< M_n n^{-\frac{1}{2}}\log n  \mid  Y_n, F_n  \right) \to 1,  \]
almost surely $\PP_{(\theta_0,\alpha_0,\tau_0,\beta_0)}$ as $n\to\infty$, where
$\varrho  = \max(\rho_1/2, \rho_{22})$ and $\rho_1,\rho_{22}$ are as defined in \eqref{eq:E1n} and \eqref{eq:rho.relation2}.
\end{theorem}
We remark on the prior Assumptions \ref{cond:prior1} and \ref{cond:prior2}. They can be verified by many commonly used priors on $(\theta,\alpha,\tau,\beta)$, such as the priors in the following proposition. Let $\text{IG}(a,b)$ be the inverse gamma distribution with shape parameter $a>0$ and rate parameter $b>0$. Let $\text{IGauss}(\mu,\lambda)$ be the inverse Gaussian distribution with mean parameter $\mu>0$ and shape parameter $\lambda>0$. Then we have the following proposition.

\begin{proposition} \label{prop:prior}
Suppose that $X(\cdot)$ in Model \eqref{eq:obs.model} has the isotropic Mat\'ern covariance function in Example \ref{ex:Matern}. Suppose that the independent priors are assigned on $\theta,\alpha,\tau,\beta$, where $\beta\sim \Ncal(0,a_0I_p)$, $\theta\sim \textup{IG}(a_1,b_1)$, $\tau\sim \textup{IG}(a_2,b_2)$, $\alpha\in \textup{IGauss}(\mu,\lambda)$, for some positive constant hyperparameters $a_0,a_1,b_1,a_2,b_2,\mu,\lambda$ that satisfy $a_1 > 2(3p+6)/\rho_{21}$ and $a_2 > 2(3p+6)/\rho_{22}$. Then Assumption \ref{cond:prior1} is satisfied, and Assumption \ref{cond:prior2} is satisfied by the set $\Ecal_{n}$ defined in \eqref{eq:E1n} with $\kappa=1$.
\end{proposition}

\section{Numerical Experiments} \label{sec:numerical}

We investigate the posterior contraction behavior for the posterior distribution of parameters in Model \eqref{eq:obs.model}. In particular, we focus on how fast the marginal posteriors of the microergodic parameter $\theta$ and the nugget parameter $\tau$ contracts towards their true values as we increase the sample size $n$. Simulation studies for the frequentist properties of the proposed higher-order quadratic variation estimators $\widehat\theta_n$ and $\widehat\tau_n$ in Section \ref{sec:estimator} can be found in the PhD thesis \citet{Sun21}. The simulations below will focus exclusively on the Bayesian posterior contraction for $\theta$ and $\tau$ as well as the Bayesian posterior predictive performance.

\subsection{Simulations} \label{subsec:simulations}
We present the results for Model \eqref{eq:obs.model} with domain dimension $d=2$ with the isotropic Mat\'ern covariance function in \eqref{eq:MaternCov}. We also have additional simulation results for $d=1$ in Section S8 of the Supplementary Material. We consider two values of the smoothness parameter $\nu=1/2$ and $\nu=1/4$, characterizing different smoothness of the Gaussian process sample paths. The true covariance parameters are set to be $\theta_0=5,\alpha_0=1,\tau_0=0.5$ for both $\nu=1/2$ and $\nu=1/4$. For $d=2$ and the domain $[0,1]^2$, we choose the sampling points $S_n$ to be the regular grid $\left((2i-1)/(2m),(2j-1)/(2m)\right)$ for $i,j=1,\ldots,m$, where we choose $m=20, 22, 25, 28, 31, 35, 39, 44, 49, 55$ such that the sample size roughly follows the geometric sequence $n= m^2\approx 400\times 1.25^{k-1}$ for $k=1,\ldots,10$. For the regression functions, we let $\ff(\bfs)=\left(1,s_1,s_2,s_1^2,s_1s_2,s_2^2\right)^\T$ for $\bfs=(s_1,s_2)\in [0,1]^2$ and $\beta_0=(1,-1.5,-1.5,2,1,2)^\T$.

For Bayesian inference, we assign the prior specified in Proposition \ref{prop:prior} on $(\theta,\alpha,\tau,\beta)$, with hyperparameters $a_0=10^6$, $a_1=b_1=a_2=b_2=0.1$, and $\mu=\lambda=1$. Such values of $a_1,a_2$ are small and do not satisfy the sufficient conditions in Proposition \ref{prop:prior}, but we show that this does not affect the convergence results. We integrate out $\beta$ given the conjugate normal prior and then use the random walk Metropolis algorithm to draw $2000$ samples of $(\theta,\alpha,\tau)$ after $1000$ burnins from the posterior density $\pi(\theta,\alpha,\tau|Y_n)$. We simulate Model \eqref{eq:obs.model} for 40 independent copies of the dataset $(Y_n,F_n)$ and find their posterior distributions. The results are summarized in Figure \ref{fig:para.0.5} for $\nu=1/2$ and Figure \ref{fig:para.0.25} for $\nu=1/4$. The boxplots are the marginal posterior distributions of $\theta$ and $\tau$, obtained by averaging over the 40 macro replications of posterior distributions using the Wasserstein-2 barycenter (\citealt{Lietal17}). Clearly in both $\nu=1/2$ and $\nu=1/4$ cases, the posterior distribution contracts to the true parameter $\theta_0=5,\tau_0=0.5$ as $n$ increases. The right panels of Figures \ref{fig:para.0.5} and Figure \ref{fig:para.0.25} are the means of absolute differences from all posterior draws of $(\theta,\tau)$ to the true parameters $(\theta_0,\tau_0)$ versus the sample size on the logarithm scale. We can see that they approximately decrease in straight lines on the logarithm scale after $n$ becomes larger. These linear trends indicate that the posterior contractions for both $\theta$ and $\tau$ happen at polynomial rates and hence corroborate our theory.

\begin{figure}[ht]
\center
\includegraphics[width=0.8\textwidth]{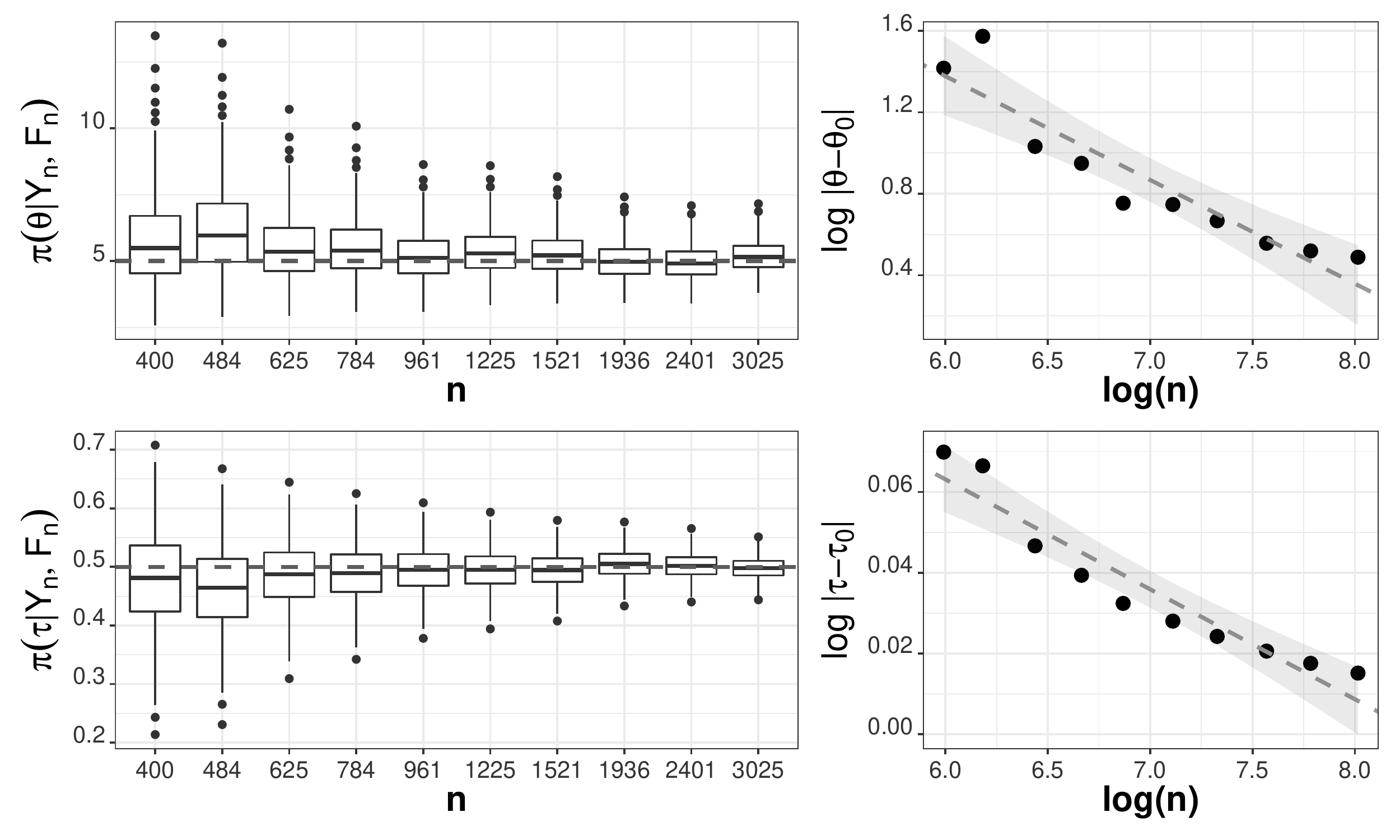}
\caption{\small Posterior contraction for $\nu=1/2$. Left column: Boxplots for the marginal posterior densities of $\theta$ and $\tau$ versus the increasing sample size $n$. The grey dashed lines are the true parameters $\theta_0=5$ and $\tau_0=0.5$. Right column: Posterior means of $|\theta-\theta_0|$ and $|\tau-\tau_0|$ versus the increasing sample size $n$, on the logarithm scale. The dashed lines are the linear regression fits, and the grey shaded areas are the 95\% confidence bands. All posterior summaries are averaged over 40 macro Monte Carlo replications.}
\label{fig:para.0.5}
\end{figure}
\begin{figure}[ht]
\center
\includegraphics[width=0.8\textwidth]{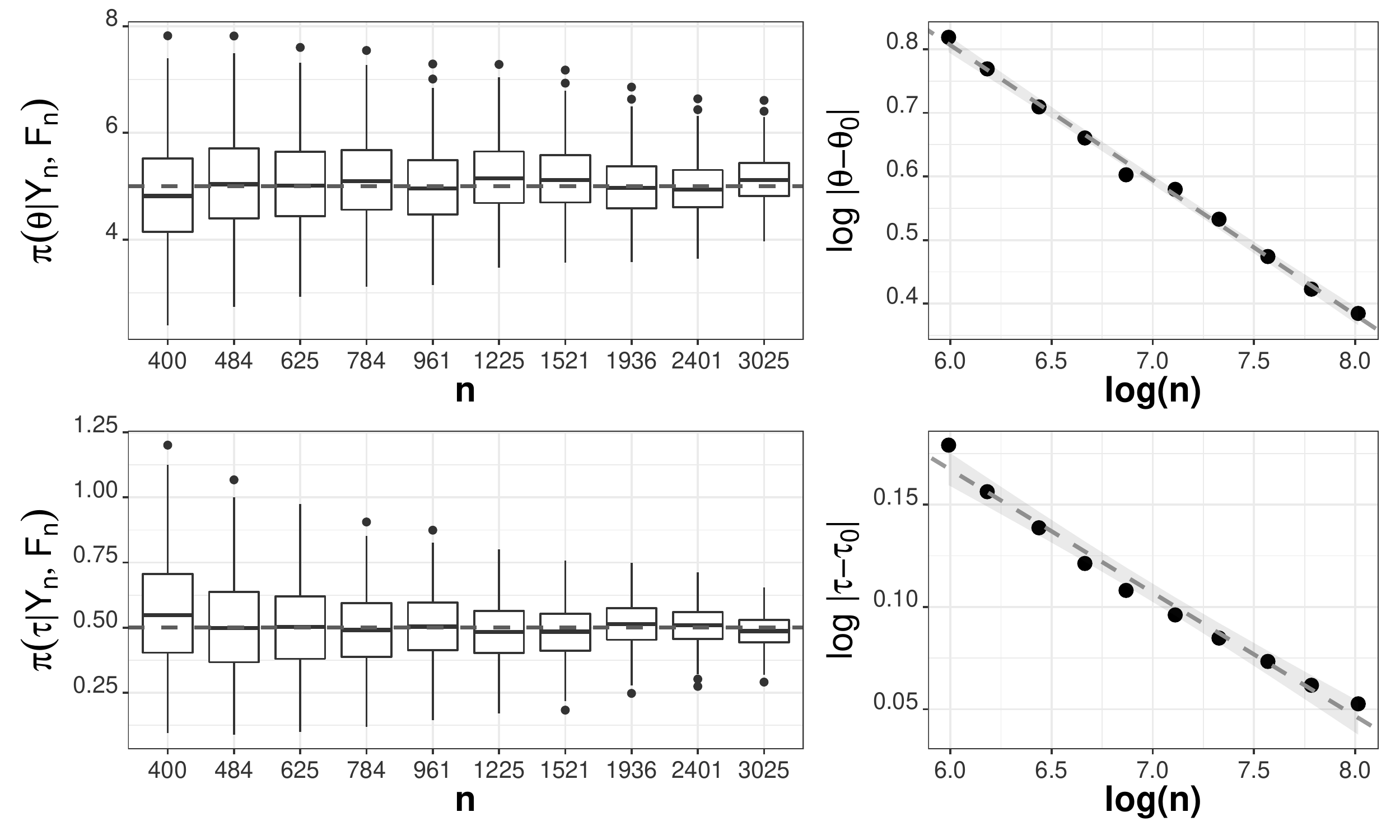}
\caption{\small Posterior contraction for $\nu=1/4$. Left column: Boxplots for the marginal posterior densities of $\theta$ and $\tau$ versus the increasing sample size $n$. The grey dashed lines are the true parameters $\theta_0=5$ and $\tau_0=0.5$. Right column: Posterior means of $|\theta-\theta_0|$ and $|\tau-\tau_0|$ versus the increasing sample size $n$, on the logarithm scale. The dashed lines are the linear regression fits, and the grey shaded areas are the 95\% confidence bands. All posterior summaries are averaged over 40 macro Monte Carlo replications.}
\label{fig:para.0.25}
\end{figure}

We further investigate the Bayesian posterior prediction performance and compare with the best possible prediction. We draw another $N=2500$ points $\bfs_1^*,\ldots,\bfs_N^*$ uniformly from the domain $[0,1]^2$ as the testing locations, and compute the prediction mean squared error
$$\EE_{Y_n}  \Bigg\{N^{-1}\sum_{l=1}^N M_{\text{post}}(\bfs_l^*)\Bigg\} =\EE_{Y_n} \Bigg\{ N^{-1}\sum_{l=1}^N \EE_{\theta,\alpha,\tau,\beta\mid Y_n,F_n} \big\{(\tilde Y(\bfs_l^*)-Y_0(\bfs_l^*)\big\}^2\Bigg\}, $$
where $\tilde Y(\bfs_l^*)$ is a draw from the Bayesian predictive posterior distribution of $Y(\cdot)$ at $\bfs_l^*$ given a random draw $(\theta,\alpha,\tau,\beta)$ from the posterior distribution $\Pi(\cdot\mid Y_n,F_n)$, $Y_0(\bfs^*) = \beta_0^\T \ff(\bfs^*) + X(\bfs^*)$ is the true mean function at $\bfs^*$ without measurement error, and the two layers of expectations are with respect to both the posterior distribution and the distribution of $Y_n$ given the true parameters $(\theta_0,\alpha_0,\tau_0,\beta_0)$. The oracle prediction mean squared error for Model \eqref{eq:obs.model} is calculated as $\EE_{Y_n}  \big\{N^{-1}\sum_{l=1}^N M_{0}(\bfs_l^*)\big\} = \EE_{Y_n} \big[N^{-1}\sum_{l=1}^N \big\{(\tilde Y_0(\bfs_l^*)-Y_0(\bfs_l^*)\big\}^2\big]$ where $\tilde Y_0(\cdot)$ is the best linear unbiased predictor of $Y(\cdot)$ given the true parameters $(\theta_0,\alpha_0,\tau_0,\beta_0)$ (\citealt{Stein99a}). The exact formulas to calculate $\tilde Y(\bfs_l^*),M_{\text{post}}(\bfs_l^*),\tilde Y_0(\bfs_l^*),M_{0}(\bfs_l^*)$ for Model \eqref{eq:obs.model} can be found in Section S8 of the Supplementary Material.

In the left panel of Figure \ref{fig:mse.2}, we plot the prediction mean squared errors for both $\nu=1/2$ and $\nu=1/4$ together with those from $(\theta_0,\alpha_0,\tau_0,\beta_0)$ on the logarithm scale. Clearly the Bayesian predictive posterior has almost the same mean square errors as the oracle prediction. We further take the ratio of the two prediction mean squared errors and plot $\EE_{Y_n} \big\{N^{-1}\sum_{l=1}^N M_{\text{post}}(\bfs_l^*)/M_{0}(\bfs_l^*)\big\}$ in the right panel of Figure \ref{fig:mse.2}, which characterizes the relative efficiency of Bayesian prediction. \citet{Stein88, Stein90a, Stein90b, Stein93} have developed the frequentist theory of posterior asymptotic efficiency for Gaussian processes without regression terms and nugget. The decreasing trend of the ratio towards 1 in the right panel of Figure \ref{fig:mse.2} indicates a possibly similar phenomenon that the Bayesian posterior prediction is asymptotically efficient even for the more general spatial model \eqref{eq:obs.model} under the fixed-domain asymptotics framework.
\begin{figure}[ht]
\center
\includegraphics[width=0.9\textwidth]{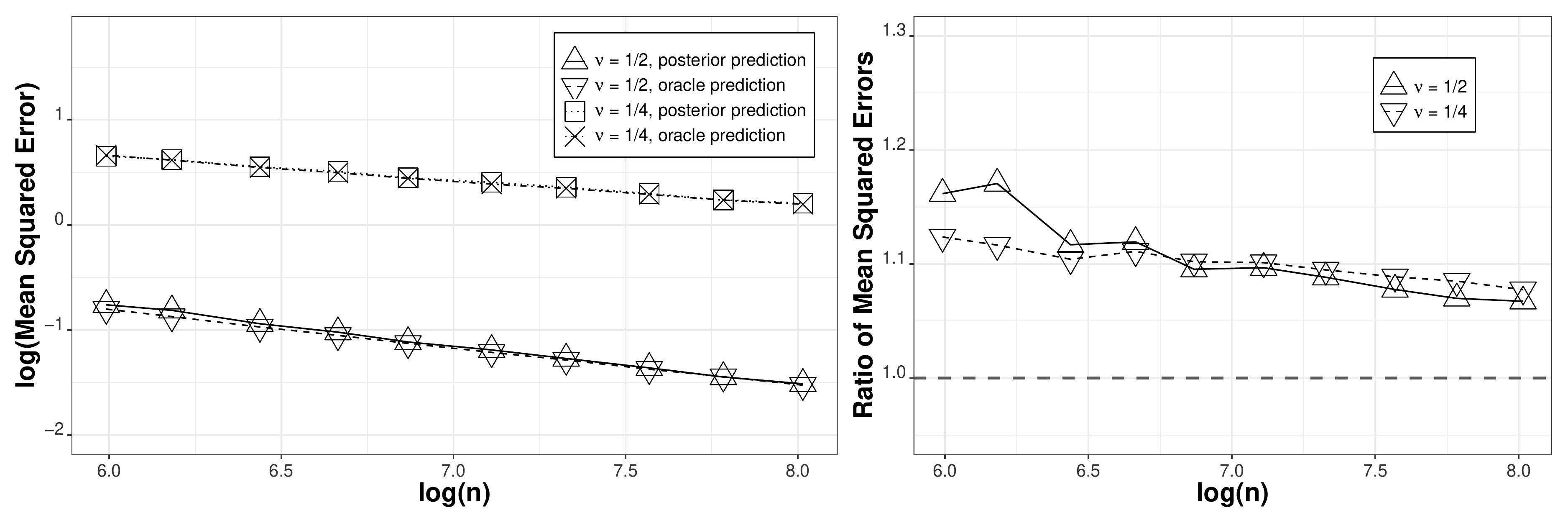}
\caption{\small Prediction mean squared errors for $\nu=1/2$ and $\nu=1/4$. Left panel: The prediction mean squared errors under both the Bayesian posterior prediction and the oracle prediction based on the true parameters. Right panel: Ratios of the Bayesian prediction mean squared error and the oracle prediction mean squared error. All posterior summaries are averaged over 2500 testing locations in $[0,1]^2$ and 40 macro Monte Carlo replications.}
\label{fig:mse.2}
\end{figure}

\subsection{Sea Surface Temperature Data} \label{subsec:SST.data}

For a real data analysis, we fit Model \eqref{eq:obs.model} and the isotropic Mat\'ern covariance function \eqref{eq:MaternCov} with $\nu=1/2$ to the sea surface temperature data on the Pacific Ocean between $45^\circ$--$48^\circ$ north latitudes and $149^\circ$--$152^\circ$ west longitudes on August 16, 2016. The data were collected from remote sensing satellites with a high-resolution on a $0.025^\circ\times 0.025^\circ$ grid, and can be obtained from National Oceanographic Data Centres (NODC) World Ocean Database (\url{https://www.ncei.noaa.gov/products/world-ocean-database}). For Bayesian estimation, we use a total of 3600 observations on the $0.05^\circ\times 0.05^\circ$ grid, and set the sample size $n = 700, 1050, 1600, 2400, 3600$. For each $n$ smaller than 3600, we randomly choose 10 subsets of data $Y_n$. We set the regressors $\ff(\bfs)=(1,s_1,s_2)^\T$ for the latitude $s_1$ and the longitude $s_2$. We assign the same prior as in Proposition \ref{prop:prior} with hyperparameters $a_0=10^6$, $a_1=b_1=a_2=b_2=0.1$, and $\mu=\lambda=1$. We draw $4000$ posterior samples of $(\theta,\alpha,\tau)$ after $1000$ burnins for each of the 10 subsets, and then average the 10 marginal posterior distributions into one summary posterior distribution using the Wasserstein-2 barycenter (\citealt{Lietal17}). Figure \ref{fig:sst.set3} presents the summary marginal posterior densities of $\theta$ and $\tau$ for $n\in \{700, 1050, 1600, 2400, 3600\}$. As $n$ increases, we can see the clear trend of posterior contraction for both $\theta$ and $\tau$, and the posterior of $\tau$ seems to contract faster than that of $\theta$.

\begin{figure}[ht]
\center
\includegraphics[width=\textwidth]{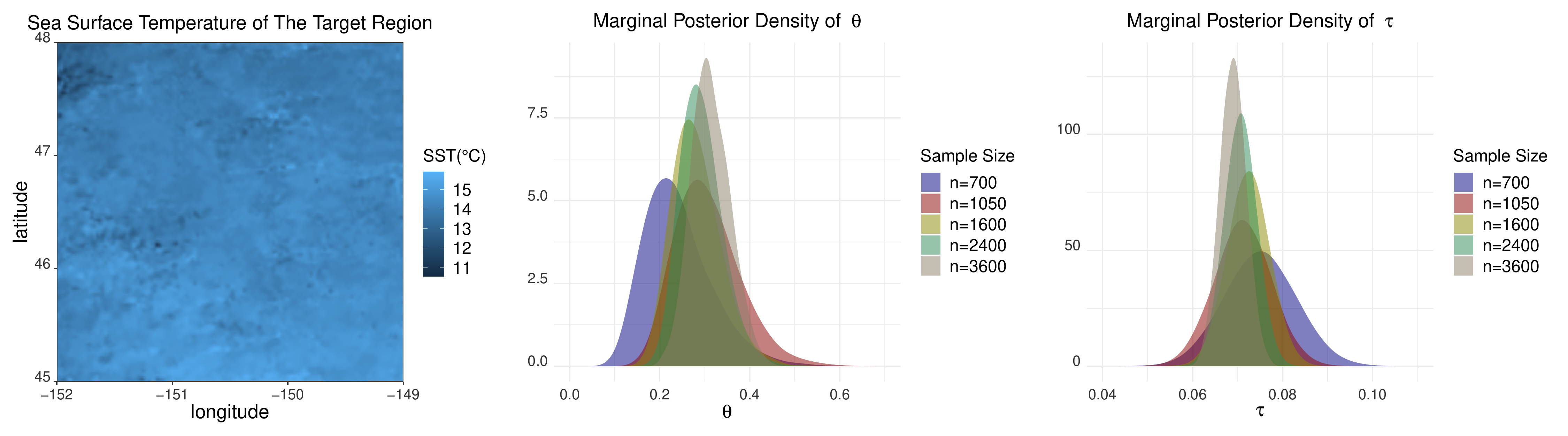}
\caption{\small Sea surface temperature data. Left panel: Temperature in the target region. Middle panel: Marginal posterior densities of $\theta$ with various sample size $n$. Right panel: Marginal posterior densities of $\tau$ with various sample size $n$.}
\label{fig:sst.set3}
\end{figure}

\section{Discussion} \label{sec:discussion}
The general fixed-domain posterior contraction theory developed in this paper can be potentially applied to many stationary covariance functions in spatial statistics. With the new evidence lower bound, we have effectively transformed the problem of finding posterior contraction rates to the problem of finding efficient frequentist estimators that satisfy the concentration inequalities with exponentially small tails as in Assumption \ref{cond:exptest1}.

There are many potential directions to extend the current work. First, for isotropic Mat\'ern, our higher-order quadratic variation estimators deliver the explicit posterior contraction rates, though it remains unclear what is the optimal rate for the microergodic parameter $\theta$ under our flexible stratified sampling design. It would be of further interest to find the limiting Bayesian posterior distribution of $(\theta,\alpha,\tau,\beta)$ and establish the posterior asymptotic normality for both $\theta$ and the nugget $\tau$, similar to the Bayesian fixed-domain asymptotic theory for the model without nugget in \citet{Li20}.

Second, we have assumed that the smoothness parameter $\nu$ is fixed and known in our theory. Because the smoothness parameter $\nu$ determines the degree of mean square differentiability of the random field, estimation of $\nu$ has been an important and meanwhile challenging problem in the spatial literature. Recently, using the new higher-order quadratic variation method, \citet{Loh15}, \citet{Lohetal21}, and \citet{LohSun23} have proposed consistent estimators of $\nu$ for irregularly spaced spatial data under fixed-domain asymptotics. In fact, our estimator of the microergodic parameter $\theta$ in Section \ref{subsec:construction.estimator} is the same as the estimator in \citet{Lohetal21} and \citet{LohSun23} when $\nu$ is assumed to be known. It requires further study whether we can include $\nu$ as part of the unknown parameters in the  Bayesian framework, and establish the posterior contraction for $\nu$ based on the new estimators.

Third, the isotropic covariance function considered in this paper leads to the simplest spatial Gaussian process model for real data. To extend our method to anisotropic covariance functions or even nonstationary spatial processes, we need to both establish the evidence lower bound similar to Theorem \ref{thm:conv.rate.denom} and find some consistent frequentist estimators for the microergodic parameters in these more general covariance functions under fixed-domain asymptotics. From the technical perspective, it seems that finding consistent estimators will be a more challenging task than deriving the evidence lower bound, though both of them may require a case-by-case analysis for different spatial covariance functions. Meanwhile, it is possible to adapt the current proof techniques to a general nonparametric class of covariance functions based the principal irregular terms in their Taylor expansions (see Section 2.7 of \citealt{Stein99a}), in a similar spirit to the recent work \citet{BacLag20}.

Fourth, given the frequentist asymptotic efficiency in \citet{Stein90a,Stein90b} for Gaussian processes without nugget as well as our promising simulation results, it would be of theoretical interest to show that the asymptotic efficiency can still be preserved even in the Bayesian posterior prediction for the true mean function $Y_0(\bfs) = \beta_0^\T \ff(\bfs) + X(\bfs)$ in the more general model \eqref{eq:obs.model} with both regression terms and nugget.

These potential developments will together provide strong theoretical justification for the existing Bayesian spatial inference based on the Gaussian process regression model, including parameter estimation, uncertainty quantification, and prediction. We leave these directions for future research.

\section*{Acknowledgements}
The authors thank Professor Wei-Liem Loh for helpful discussion. This work was supported by Singapore Ministry of Education Academic Research Funds Tier 1 Grant A-0004822-00-00.

\section*{Supplementary Material}
\noindent {\bf Supplementary Material for ``Fixed-domain Posterior Contraction Rates for Spatial Gaussian Process Model with Nugget'':} Technical proofs of all theorems, propositions, and additional simulation results.
\vspace{2mm}

\newpage

\setcounter{section}{0}
\setcounter{equation}{0}
\setcounter{lemma}{0}
\setcounter{proposition}{0}
\renewcommand{\theequation}{S.\arabic{equation}}
\renewcommand{\theproposition}{S.\arabic{proposition}}
\renewcommand{\thelemma}{S.\arabic{lemma}}
\renewcommand\thesection{S\arabic{section}}

\makeatletter
\renewcommand{\thefigure}{S\@arabic\c@figure}
\makeatother

\begin{center}
{\bf \Large Supplementary Material for ``Fixed-domain Posterior Contraction Rates for Spatial Gaussian Process Model with Nugget"}
\end{center}

This supplementary material contains the technical proofs of the theorems and propositions in the main paper as well as additional simulation results. Section \ref{sec:proof.main} contains the proof of Theorem \ref{thm:convergence}. Section \ref{sec:proof.denom} contains the proof of Theorem \ref{thm:conv.rate.denom} and auxiliary technical results on spectral analysis. Section \ref{sec:inconsistency} proves a proposition on the posterior inconsistency of the range parameter $\alpha$ and the regression coefficient vector $\beta$. Sections \ref{sec:proof.prop1}, \ref{sec:proof.post.rate2}, and \ref{sec:proof.prop:prior} include the proofs of Proposition \ref{prop:4cov}, Theorem \ref{thm:post.rate2}, and Proposition \ref{prop:prior}, respectively. Section \ref{sec:proof.E1n.rate} contains the lengthy proof of Theorem \ref{thm:E1n.rate}. Section \ref{sec:add.simu} includes the formulas for calculating the Bayesian and oracle prediction mean squared errors for Model \eqref{eq:obs.model} in the main text and additional simulation results for the case of domain dimension $d=1$.
\vspace{2mm}

We define some universal notation that will be used throughout the proof. Let $\RR_+=(0,+\infty)$, $\ZZ_+$ be the set of all positive integers, and $\NN=\ZZ_+\cup \{0\}$. For any $x\in \RR$, $\lceil x\rceil$ and $\lfloor x\rfloor$ denote the smallest integer $\geq x$ and the largest integer $\leq x$. For any $x=(x_1,\ldots,x_d)^\T  \in \RR^d$, we let $\|x\|=\big(\sum_{i=1}^d x_i^2\big)^{1/2}$, $\|x\|_1=\sum_{i=1}^d |x_i|$, and $\|x\|_{\infty}=\max(x_1,\ldots,x_d)$. For a generic set $\Ecal$, we use $\Ecal^c$ to denote its complement. If $\Ecal$ contains finitely many elements, then $|\Ecal|$ denotes its cardinality.

For two positive sequences $a_n$ and $b_n$, we use $a_n\prec b_n$ and $b_n\succ a_n$ to denote the relation $\lim_{n\to\infty} a_n/b_n=0$. $a_n=O(b_n)$, $a_n\preceq b_n$ and $b_n\succeq a_n$ denote the relation $\limsup_{n\to\infty} a_n/b_n<+\infty$. $a_n\asymp b_n$ denotes the relation $a_n\preceq b_n$ and $a_n\succeq b_n$.

For any integers $k,m$, we let $I_k$ be the $k\times k$ identity matrix, $0_k$ be the $k$-dimensional column vectors of all zeros, $0_{k\times m}$ be the $k\times m$ zero matrix, and $\diag\{c_1,\ldots,c_k\}$ be the diagonal matrix with diagonal entries $c_1,\ldots,c_k\in \RR$. For any generic matrix $A$, $cA$ denotes the matrix of $A$ with all entries multiplied by the number $c$, and $\dett(A)$ denotes the determinant of $A$. For a square matrix $A$, $\tr(A)$ denotes the trace of $A$. If $A$ is symmetric positive semidefinite, then $\mathsf{s}_{\min}(A)$ and $\mathsf{s}_{\max}(A)$ denote the smallest and largest eigenvalues of $A$, and $A^{1/2}$ denotes a symmetric positive semidefinite square root of $A$. For two symmetric matrices $A$ and $B$, we use $A\preceq B$ and $B\succeq A$ to denote the relation that $B-A$ is symmetric positive semidefinite, and use $A\prec B$ and $B\succ A$ to denote the relation that $B-A$ is symmetric positive definite. For any matrix $A$, $\|A\|_{\op}=\big\{\mathsf{s}_{\max}(A^\T  A)\big\}^{1/2} = \big\{\mathsf{s}_{\max}(AA^\T )\big\}^{1/2}$ denotes the operator norm of $A$, and $\|A\|_F= \big\{\tr\left(A^\T  A\right)\big\}^{1/2}$ denotes the Frobenius norm of $A$. We use $\pr(\cdot)$ to denote the probability under the probability measure $\PP_{(\theta_0,\alpha_0,\tau_0,\beta_0)}$.

\section{Proof of Theorem \ref{thm:convergence}} \label{sec:proof.main}
\begin{proof}

We first show the propriety of the posterior distribution defined in Equation \eqref{jointpost1} of the main text, which is equivalent to showing that the integral in the denominator on the right-hand side of \eqref{jointpost1} is finite. Since for two positive definite matrices $A$ and $B$, $\dett(A+B)\geq \dett(A)$, we have that for any given $n\in \ZZ_+$,
\begin{align*}
&\quad~ \int_{\RR_+^{3}\times \RR^p} \exp\left\{\Lcal_n(\theta,\alpha,\tau,\beta) \right\} \pi(\theta,\alpha,\tau,\beta) \ud \theta \ud \tau \ud \alpha \ud \beta  \\
&= \int_{\RR_+^{3}\times \RR^p} \frac{1}{\left[\dett\left\{\theta K_{\alpha,\nu}(S_n)+\tau I_n\right\}\right]^{1/2}} \exp\left\{ - \frac{1}{2} (Y_n-  F_n  \beta)^\T  \left\{\theta K_{\alpha,\nu}(S_n) + \tau I_n \right\}^{-1} (Y_n-  F_n  \beta) \right\} \\
&\qquad \times \pi(\theta,\alpha,\tau,\beta) \ud \theta \ud \tau \ud \alpha \ud \beta  \\
&\leq \int_{\RR_+^{3}\times \RR^p} \frac{1}{\left\{\dett(\tau I_n)\right\}^{1/2}} \cdot 1 \cdot \pi(\theta,\alpha,\tau,\beta) \ud \theta \ud \tau \ud \alpha \ud \beta  \\
&= \int_0^{\infty} \tau^{-n/2} \pi(\tau) \ud \tau < \infty,
\end{align*}
where the last integral is finite following Assumption \ref{cond:prior1}. This proves that the joint posterior density $\pi(\theta,\alpha,\tau,\beta~|~ Y_n, F_n)$ in Equation \eqref{jointpost1} is well defined.

Next, we prove the posterior contraction rates for $\theta$ and $\tau$ in Theorem \ref{thm:convergence}. We follow the classical proof of the Schwartz's theorem for posterior consistency (\citealt{Sch65}); see for example, Theorem 6.17 in \citet{GhoVan17}. Let $\varepsilon_{1n}= M_n n^{-b_1} \log n$ and $\varepsilon_{2n}= M_n n^{-b_2} \log n$. Define the testing function
\begin{align} \label{eq:test}
\phi_n & = \Ical \left( |\widehat\theta_n/\theta_0-1|\geq \varepsilon_{1n}/2, \text{ or } |\widehat\tau_n/\tau_0-1|\geq \varepsilon_{2n}/2 \right),
\end{align}
where $\widehat\theta_n$ and $\widehat\tau_n$ are from Assumption \ref{cond:exptest1}.

We start with the decomposition
\begin{align}\label{eq:cons1}
& \Pi \left(\Bcal_0(\varepsilon_{1n},\varepsilon_{2n})^c ~|~ Y_n, F_n  \right)
= \frac{\int_{\Bcal_0(\varepsilon_{1n},\varepsilon_{2n})^c}\exp\left\{ \Lcal_n(\theta,\alpha,\tau,\beta) \right\} \pi(\theta,\alpha,\tau,\beta) \ud \theta \ud \tau \ud \alpha \ud \beta }
{\int_{\RR_+^{3} \times \RR^p} \exp\left\{\Lcal_n(\theta',\alpha',\tau',\beta') \right\} \pi(\theta',\alpha',\tau',\beta')\ud \theta' \ud \tau' \ud \alpha' \ud \beta' } \nonumber \\
&\leq \phi_n + \frac{(1-\phi_n)\int_{\Bcal_0(\varepsilon_{1n},\varepsilon_{2n})^c}\exp\left\{ \Lcal_n(\theta,\alpha,\tau,\beta)\right\} \pi(\theta,\alpha,\tau,\beta) \ud \theta \ud \tau \ud \alpha \ud \beta}
{\int_{\RR_+^{3} \times \RR^p} \exp\left\{\Lcal_n(\theta',\alpha',\tau',\beta') \right\} \pi(\theta',\alpha',\tau',\beta')\ud \theta' \ud \tau' \ud \alpha' \ud \beta' } .
\end{align}
By Assumption \ref{cond:exptest1}, we have that for some large positive constant $N_0$, for all $n>N_0$,
\begin{align*}
{\EE}_{(\theta_0,\alpha_0,\tau_0,\beta_0)} (\phi_n) & \leq \exp\left\{- c_1 \varphi\left(n^{b_1} \varepsilon_{1n} \right)\right\} + \exp\left\{- c_2 \varphi\left(n^{b_2} \varepsilon_{2n} \right)\right\} \\
& \leq \exp\left(- c_1 M_n \log n \right) + \exp\left(- c_2 M_n \log n  \right)  \to 0,
\end{align*}
where the last step follows from $M_n \to +\infty$ as $n\to\infty$. Notice that for any constant $c>0$ and an arbitrary $c'>1$, $\exp\left(-c M_n \log n\right)\leq \exp(-c' \log n) = n^{-c'}$ for all sufficiently large $n$ and hence the sequence $\exp\left(-c M_n \log n\right)$ is summable over $n$. Therefore, by the Markov's inequality and the Borel-Cantelli lemma, $\phi_n\to 0$ almost surely $\PP_{(\theta_0,\alpha_0,\tau_0)}$ as $n\to\infty$.

We now turn to the second term in \eqref{eq:cons1}. It can be further decomposed into two terms:
\begin{align}\label{eq:cons2}
& \frac{(1-\phi_n)\int_{\Bcal_0(\varepsilon_{1n},\varepsilon_{2n})^c} \exp\left\{ \Lcal_n(\theta,\alpha,\tau,\beta)\right\} \pi(\theta,\alpha,\tau,\beta) \ud \theta \ud \tau \ud \alpha \ud \beta}
{\int_{\RR_+^{3} \times \RR^p} \exp\left\{\Lcal_n(\theta',\alpha',\tau',\beta') \right\} \pi(\theta',\alpha',\tau',\beta')\ud \theta' \ud \tau' \ud \alpha' \ud \beta'}  = T_1+T_2, \quad \text{ where }\nonumber \\
& T_1 = \frac{(1-\phi_n)\int_{\Bcal_0(\varepsilon_{1n},\varepsilon_{2n})^c \cap \Ecal_n} \exp\left\{ \Lcal_n(\theta,\alpha,\tau,\beta)-\Lcal_n(\theta_0,\alpha_0,\tau_0,\beta_0) \right\} \pi(\theta,\alpha,\tau,\beta) \ud \theta \ud \tau \ud \alpha \ud \beta }
{\int_{\RR_+^{3} \times \RR^p} \exp\left\{\Lcal_n(\theta',\alpha',\tau',\beta') -\Lcal_n(\theta_0,\alpha_0,\tau_0,\beta_0)\right\} \pi(\theta',\alpha',\tau',\beta')\ud \theta' \ud \tau' \ud \alpha' \ud \beta' } ,  \nonumber \\
&T_2 =  \frac{(1-\phi_n)\int_{\Bcal_0(\varepsilon_{1n},\varepsilon_{2n})^c \cap \Ecal_n^c} \exp\left\{ \Lcal_n(\theta,\alpha,\tau,\beta)-\Lcal_n(\theta_0,\alpha_0,\tau_0,\beta_0) \right\} \pi(\theta,\alpha,\tau,\beta) \ud \theta \ud \tau \ud \alpha \ud \beta  }
{\int_{\RR_+^{3}\times \RR^p} \exp\left\{\Lcal_n(\theta',\alpha',\tau',\beta')-\Lcal_n(\theta_0,\alpha_0,\tau_0,\beta_0) \right\} \pi(\theta',\alpha',\tau',\beta')\ud \theta' \ud \tau' \ud \alpha' \ud \beta'} .
\end{align}
For the first term $T_1$ in \eqref{eq:cons2}, by Assumption \ref{cond:exptest1} and the Fubini's theorem, the numerator has expectation
\begin{align}\label{eq:cons2.1}
& {\EE}_{(\theta_0,\alpha_0,\tau_0,\beta_0)} (1-\phi_n)\int_{\Bcal_0(\varepsilon_{1n},\varepsilon_{2n})^c \cap \Ecal_n} \exp\Big\{ \Lcal_n(\theta,\alpha,\tau,\beta) \nonumber \\
&~~ -\Lcal_n(\theta_0,\alpha_0,\tau_0,\beta_0)  \Big\} \pi(\theta,\alpha,\tau,\beta) \ud \theta \ud \tau \ud \alpha \ud \beta \nonumber \\
={} & \int_{\Bcal_0(\varepsilon_{1n},\varepsilon_{2n})^c \cap \Ecal_n} {\EE}_{(\theta,\alpha,\tau,\beta)} (1-\phi_n) \pi(\theta,\alpha,\tau,\beta) \ud \theta \ud \tau \ud \alpha \ud \beta \nonumber \\
\leq{}& \sup_{\Bcal_0(\varepsilon_{1n},\varepsilon_{2n})^c \cap \Ecal_n} {\EE}_{(\theta,\alpha,\tau,\beta)} \left\{(1-\phi_n) \cdot \int_{\Bcal_0(\epsilon_1,\epsilon_2)^c \cap \Ecal_n} \pi(\theta,\alpha,\tau,\beta) \ud \theta \ud \tau \ud \alpha \ud \beta \right\} \nonumber \\
\leq{}& \sup_{\Bcal_0(\varepsilon_{1n},\varepsilon_{2n})^c \cap \Ecal_n} {\EE}_{(\theta,\alpha,\tau,\beta)} (1-\phi_n) \nonumber \\
\leq{}& \sup_{\Bcal_0(\varepsilon_{1n},\varepsilon_{2n})^c \cap \Ecal_n} \PP_{(\theta,\alpha,\tau,\beta)} \left( |\widehat\theta_n/\theta_0-1|\leq \varepsilon_{1n}/2 \right)
+ \sup_{\Bcal_0(\varepsilon_{1n},\varepsilon_{2n})^c \cap \Ecal_n} \PP_{(\theta,\alpha,\tau,\beta)} \left( |\widehat\tau_n/\tau_0-1|\leq \varepsilon_{2n}/2 \right) \nonumber \\
\leq{}& \exp\left\{- c_1 \varphi\left(n^{b_1}\varepsilon_{1n} \right)\right\} + \exp\left\{- c_2 \varphi\left(n^{b_2}\varepsilon_{2n} \right)\right\} \nonumber \\
={}& \exp\left(-c_1 M_n \log n\right) + \exp\left(-c_2 M_n \log n\right) .
\end{align}
Therefore, by applying the Markov's inequality and the Borel-Cantelli Lemma to \eqref{eq:cons2.1}, the numerator of $T_1$ in \eqref{eq:cons2} is smaller than $\exp\left\{-\min(c_1,c_2) M_n \log n / 2\right\}$ almost surely $\PP_{(\theta_0,\alpha_0,\tau_0,\beta_0)}$ as $n\to\infty$. On the other hand, Theorem \ref{thm:conv.rate.denom} shows that almost surely $\PP_{(\theta_0,\alpha_0,\tau_0,\beta_0)}$ as $n\to\infty$, for all $d\in \ZZ_+$, the denominator in $T_1$ in \eqref{eq:cons2} is lower bounded by
\begin{align}\label{eq:cons2.2}
& \int_{\RR_+^{3}\times \RR^p} \exp\left\{\Lcal_n(\theta',\alpha',\tau',\beta')-\Lcal_n(\theta_0,\alpha_0,\tau_0,\beta_0)  \right\} \nonumber \\
& ~~ \times  \pi(\theta',\alpha',\tau',\beta')\ud \theta' \ud \tau' \ud \alpha' \ud \beta' \geq D n^{-(3p+2+2/\kappa)},
\end{align}
for some constant $D>0$. Hence for all sufficiently large $n$, the first term $T_1$ in \eqref{eq:cons2} is upper bounded by
\begin{align}\label{eq:cons2.3}
T_1 & \leq D^{-1} n^{(3p+2+2/\kappa)} \exp\left\{-\min(c_1,c_2) M_n \log n / 2\right\}\leq \exp(-c' \log n)  \to 0,
\end{align}
for an arbitrary $c'>1$, almost surely $\PP_{(\theta_0,\alpha_0,\tau_0,\beta_0)}$.

For the second term $T_2$ in \eqref{eq:cons2}, by Assumption \ref{cond:exptest1} and the Fubini's theorem, the numerator has expectation
\begin{align}\label{eq:cons2.4}
& {\EE}_{(\theta_0,\alpha_0,\tau_0,\beta_0)} (1-\phi_n)\int_{\Bcal_0(\varepsilon_{1n},\varepsilon_{2n})^c \cap \Ecal_n^c} \exp\Big\{ \Lcal_n(\theta,\alpha,\tau,\beta) \nonumber \\
&\quad -\Lcal_n(\theta_0,\alpha_0,\tau_0,\beta_0) \Big\} \pi(\theta,\alpha,\tau,\beta) \ud \theta \ud \tau \ud \alpha \ud \beta \nonumber \\
={} & \int_{\Bcal_0(\varepsilon_{1n},\varepsilon_{2n})^c \cap \Ecal_n^c} {\EE}_{(\theta_0,\alpha_0,\tau_0,\beta_0)} (1-\phi_n)\exp\Big\{ \Lcal_n(\theta,\alpha,\tau,\beta) \nonumber \\
&\quad -\Lcal_n(\theta_0,\alpha_0,\tau_0,\beta_0) \Big\} \pi(\theta,\alpha,\tau,\beta) \ud \theta \ud \tau \ud \alpha \ud \beta  \nonumber \\
={} & \int_{\Bcal_0(\varepsilon_{1n},\varepsilon_{2n})^c \cap \Ecal_n^c} {\EE}_{(\theta,\alpha,\tau,\beta)} (1-\phi_n) \pi(\theta,\alpha,\tau,\beta) \ud \theta \ud \tau \ud \alpha \ud \beta \nonumber \\
\leq{}& \int_{\Ecal_n^c} 1\cdot \pi(\theta,\alpha,\tau,\beta) \ud \theta \ud \tau \ud \alpha \ud \beta
= \Pi(\Ecal_n^c) \leq n^{-(3p+4+2/\kappa)} .
\end{align}
Therefore, by applying the Markov's inequality to \eqref{eq:cons2.4}, the numerator of $T_2$ in \eqref{eq:cons2} is smaller than $n^{-\left(3p+\frac{5}{2}+\frac{2}{\kappa}\right)}$ with probability at least $1-n^{-3/2}$. Since the sequence $\{n^{-3/2}:n\geq 1\}$ is summable over $n$, by the Borel-Cantelli Lemma, we have that the numerator of $T_2$ in \eqref{eq:cons2} is smaller than $n^{-\left(3p+\frac{5}{2}+\frac{2}{\kappa}\right)}$ almost surely $\PP_{(\theta_0,\alpha_0,\tau_0,\beta_0)}$ as $n\to\infty$.
We combine this with the lower bound in \eqref{eq:cons2.2} to conclude that almost surely $\PP_{(\theta_0,\alpha_0,\tau_0,\beta_0)}$ as $n\to\infty$, the second term $T_2$ in \eqref{eq:cons2} is upper bounded by
\begin{align}\label{eq:cons2.5}
T_2 &\leq  D^{-1} n^{3p+2+2/\kappa} \cdot n^{-\left(3p+\frac{5}{2}+\frac{2}{\kappa}\right)} = D^{-1} n^{-1/2} \to 0.
\end{align}

Finally, we combine \eqref{eq:cons1}, \eqref{eq:cons2}, \eqref{eq:cons2.3}, and \eqref{eq:cons2.5} to conclude that the right-hand side of \eqref{eq:cons1} converges to zero almost surely $\PP_{(\theta_0,\alpha_0,\tau_0,\beta_0)}$ as $n\to\infty$, which completes the proof of Theorem \ref{thm:convergence}.
\end{proof}

\section{Proof of Theorems \ref{thm:conv.rate.denom} and Auxiliary Technical Lemmas} \label{sec:proof.denom}
\subsection{Proof of Theorem \ref{thm:conv.rate.denom}} \label{sec:proof.thm1}
\begin{proof}
Define the sets $\Acal_{n}$ and $\Dcal_{n}$ as
\begin{align}
& \Acal_{n} = \left\{(\theta,\alpha,\tau,\beta)\in \RR_+^{3} \times \RR^p:~ |\alpha/\alpha_0-1|\leq n^{-2/\kappa} \right\} , \nonumber \\
& \Dcal_{n} = \left\{(\theta,\alpha,\tau,\beta)\in \RR_+^{3} \times \RR^p:~ \|\beta-\beta_0\|\leq n^{-3} \right\}. \nonumber
\end{align}
By the continuity of the prior density in Assumption \ref{cond:prior1}, we have that for all sufficiently large $n$, $\pi(\theta,\tau,\alpha,\beta)>\pi(\theta_0,\tau_0,\alpha_0,\beta_0)/2>0$ for all $(\theta,\tau,\alpha,\beta)\in \Bcal_0(n^{-1}, n^{-1}) \cap \Acal_{n} \cap \Dcal_{n}$. Therefore,
\begin{align} \label{eq:prior.lower1}
\Pi \left(\Bcal_0(n^{-1}, n^{-1}) \cap \Acal_{n} \cap \Dcal_{n}\right)
&\geq \frac{\pi(\theta_0,\tau_0,\alpha_0,\beta_0)}{2} \cdot (2\theta_0n^{-1}) \cdot (2\tau_0n^{-1}) \cdot (2\alpha_0 n^{-2/\kappa}) \cdot v_0 n^{-3p}  \nonumber \\
& = 4 v_0 \theta_0\tau_0\alpha_0 \pi(\theta_0,\tau_0,\alpha_0,\beta_0) n^{-(3p+2+2/\kappa)},
\end{align}
where $v_0 = \pi^{p/2}/\Gamma(p/2+1) $. Using the definition of log-likelihood in \eqref{eq:loglik}, the exponent in Theorem \ref{thm:conv.rate.denom} can be written as
\begin{align}
& \Lcal_n(\theta,\alpha,\tau,\beta) - \Lcal_n(\theta_0,\alpha_0,\tau_0,\beta_0)  \nonumber \\
= {}& -\frac{1}{2} \left\{\tr\left[\left\{\theta K_{\alpha,\nu}(S_n) + \tau I_n\right\}^{-1} \widetilde Y_n \widetilde Y_n^\T  \right] - \tr\left[\left\{\theta_0 K_{\alpha_0,\nu}(S_n) + \tau_0 I_n\right\}^{-1} \widetilde Y_n \widetilde Y_n^\T  \right] \right\} \nonumber \\
& - \frac{1}{2} \log \frac{\dett\left\{\theta K_{\alpha,\nu}(S_n) +\tau I_n\right\}}{\dett\left\{\theta_0 K_{\alpha_0,\nu}(S_n) +\tau_0 I_n\right\}} \nonumber \\
& + \widetilde Y_n^\T  \left[\left\{\theta K_{\alpha,\nu}(S_n) +\tau I_n\right\}^{-1}-\left\{\theta_0 K_{\alpha_0,\nu}(S_n) +\tau_0 I_n\right\}^{-1}\right] F_n (\beta-\beta_0) \nonumber \\
& - \frac{1}{2} (\beta-\beta_0)^\T   F_n ^\T  \left[\left\{\theta K_{\alpha,\nu}(S_n) +\tau I_n\right\}^{-1}-\left\{\theta_0 K_{\alpha_0,\nu}(S_n) +\tau_0 I_n\right\}^{-1} \right] F_n (\beta-\beta_0). \nonumber
\end{align}
Using the lower bounds in Part (i) of Lemma \ref{lem:denom.term12} and Lemma \ref{lem:denom.term34} in Section \ref{sec:spectral} below, we obtain that there exists a large integer $N_1$ that depends on $\nu,d,C_{\ff},\alpha_0,\theta_0$, $\tau_0,L,r_0,\kappa$, such that for any $\epsilon_1,\epsilon_2\in (0,1/2)$, for all $n>N_1$, with probability at least $1-3\exp(-\log^2 n)$,
\begin{align} \label{eq:denom.lower1}
& \int_{\Bcal_0(\epsilon_1,\epsilon_2)\cap \Acal_{n}\cap\Dcal_{n}} \exp\left\{\Lcal_n(\theta,\alpha,\tau,\beta) - \Lcal_n(\theta_0,\alpha_0,\tau_0,\beta_0) \right\} \cdot \pi(\theta,\alpha,\tau,\beta)~ \ud \theta \ud \tau \ud \alpha \ud \beta  \nonumber \\
& \geq  \Pi(\Bcal_0(\epsilon_1,\epsilon_2) \cap \Acal_{n} \cap \Dcal_{n}) \cdot \exp\Big\{-  \left(\frac{5}{2}n \epsilon_{1\vee 2} + 5Ln^{-1} \right)  \nonumber \\
& \quad - \left[8C_{\ff}p^{1/2} \left\{\theta_0 K_{\alpha_0,\nu}(0) + \tau_0 \right\}^{1/2} \tau_0^{-1} n^{-1} + 2C_{\ff}^2 p \tau_0^{-1} n^{-4} \right] \Big\} ,
\end{align}
where $\epsilon_{1\vee 2} = \max(\epsilon_1,\epsilon_2)$. We can take $\epsilon_1=\epsilon_2=n^{-1}$ in \eqref{eq:denom.lower1}, so $\epsilon_{1\vee 2}=n^{-1}$. We then combine \eqref{eq:prior.lower1} and \eqref{eq:denom.lower1} to obtain that for all sufficiently large $n>N_1$, with probability at least $1-3\exp(-\log^2 n)$,
\begin{align} \label{eq:post.lower1}
& \int \exp\left\{\Lcal_n(\theta,\alpha,\tau,\beta) - \Lcal_n(\theta_0,\alpha_0,\tau_0,\beta_0) \right\} \cdot \pi(\theta,\alpha,\tau,\beta)~ \ud \theta \ud \tau \ud \alpha \ud \beta  \nonumber \\
\geq{}& \int_{\Bcal_0(n^{-1},n^{-1})\cap \Acal_{n} \cap \Dcal_{n}} \exp\left\{\Lcal_n(\theta,\alpha,\tau,\beta) - \Lcal_n(\theta_0,\alpha_0,\tau_0,\beta_0) \right\} \cdot \pi(\theta,\alpha,\tau,\beta)~ \ud \theta \ud \tau \ud \alpha \ud \beta  \nonumber \\
\geq{}& \Pi\left(\Bcal_0\big(n^{-1},n^{-1}\big) \cap \Acal_{n}\cap \Dcal_{n}\right) \cdot \exp\Big\{- \Big(\frac{5}{2}n \cdot n^{-1} + 5Ln^{-1} \Big)  \nonumber \\
& \quad - \left[8C_{\ff}p^{1/2} \left\{\theta_0 K_{\alpha_0,\nu}(0) + \tau_0 \right\}^{1/2} \tau_0^{-1} n^{-1} + 2C_{\ff}^2 p \tau_0^{-1} n^{-4} \right] \Big\} \nonumber \\
\geq{}& \Pi\left(\Bcal_0\big(n^{-1},n^{-1}\big) \cap \Acal_{n}\cap \Dcal_{n}\right)   \exp\left[- 2.5 - 5L - 8C_{\ff}p^{1/2} \left\{\theta_0 K_{\alpha_0,\nu}(0) + \tau_0 \right\}^{1/2} \tau_0^{-1} - 2C_{\ff}^2 p \tau_0^{-1}  \right] \nonumber \\
\geq{}& Dn^{-(3p+2+2/\kappa)},
\end{align}
where
$$D=4v_0\theta_0\tau_0\alpha_0 \pi(\theta_0,\tau_0,\alpha_0,\beta_0)\cdot \exp\left[-2.5-5L -8C_{\ff}p^{1/2} \left\{\theta_0 K_{\alpha_0,\nu}(0) + \tau_0 \right\}^{1/2} \tau_0^{-1} -2C_{\ff}^2 p \tau_0^{-1}  \right]$$
is a positive constant. This completes the proof of Theorem \ref{thm:conv.rate.denom}.
\end{proof}

\subsection{Spectral Analysis of Covariance Functions} \label{sec:spectral}

We present a series of results for the spectral analysis of the covariance function $\theta K_{\alpha,\nu}$ for the Gaussian process $X(\cdot)$, which is used in the proof of Theorem \ref{thm:conv.rate.denom}.

For $w \in \RR^d$, let
\begin{align}\label{f.specden}
f_{\theta,\alpha,\nu}(w) &= \frac{1}{(2\pi)^d}\int_{\RR^d} \exp\left(-\imath w^\T  x\right) \theta K_{\alpha,\nu}(x) \ud x
\end{align}
be the spectral density of the covariance function $\theta K_{\alpha,\nu}$.  For any given pair $(\theta,\alpha)$, let $\|h\|_{f_{\theta,\alpha,\nu}}^2 = \langle h,h \rangle_{f_{\theta,\alpha,\nu}} =\int_{\RR^d} |h(w)|^2 f_{\theta,\alpha,\nu}(w)\ud w$ be the norm of a generic function $h$ in the Hilbert space $L_2(f_{\theta,\alpha,\nu})$, with inner product $\langle h_1,h_2\rangle_{f_{\theta,\alpha,\nu}} =\int_{\RR^d} h_1(w) \overline{h_2(w)} f_{\theta,\alpha,\nu}(w)\ud w$ for any functions $h_1,h_2\in L_2(f_{\theta,\alpha,\nu})$.

The following lemma is important for our spectral analysis.
\begin{lemma}\label{lem:URU}
Let $K_{\alpha,\nu}(S_n)$ be the $n\times n$ covariance matrix whose $(i,j)$-entry is $K_{\alpha,\nu}(\bfs_i-\bfs_j)$. Then for any $\alpha \in \RR_+$, there exists an $n\times n$ invertible matrix $U_{\alpha}$ that depends on $\alpha,\alpha_0,\theta_0,\nu,S_n$, such that
\begin{align}\label{diagonalize}
& \theta_0  U_{\alpha}^\T  K_{\alpha_0,\nu}(S_n) U_{\alpha}= I_n, \qquad \theta_0 U_{\alpha}^\T  K_{\alpha,\nu}(S_n) U_{\alpha}= \Lambda_n(\alpha) = \diag\{\lambda_{k,n}(\alpha):k=1,\ldots,n\},
\end{align}
where $\{\lambda_{k,n}(\alpha),k=1,\ldots,n\}$ are the positive diagonal entries of the diagonal matrix $\Lambda_n(\alpha)$.

Furthermore, there exist orthonormal basis functions $\psi_1,\ldots,\psi_n \in L_2(f_{\theta_0,\alpha_0,\nu})$, such that for any $j,k\in \{1,\ldots,n\}$,
\begin{align}\label{wangloh.eq13}
& \langle \psi_j,\psi_k \rangle_{f_{\theta_0,\alpha_0,\nu}} =  \Ical(j=k), \qquad \langle \psi_j,\psi_k \rangle_{f_{\theta,\alpha,\nu}} =  \lambda_{j,n}(\alpha) \Ical(j=k),
\end{align}
where $\Ical(\cdot)$ is the indicator function.
\end{lemma}

\begin{proof}[$\pof$ Lemma \ref{lem:URU}]
The existence of such an invertible $U_{\alpha}$ is guaranteed by Theorem 7.6.4 and Corollary 7.6.5 on page 465--466 of \citet{HorJoh85}. For completeness, we directly prove the existence of such an invertible matrix in the following general claim.
\vspace{2mm}

\noindent \underline{\sc Claim:} Suppose that $A$ and $B$ are two generic $n\times n$ symmetric positive definite matrices. Then there always exists an invertible matrix $U$, such that
\begin{align} \label{eq:UAU}
& U^\T  A U = I_n,\qquad U^\T  B U = \Lambda,
\end{align}
where $I_n$ is the $n\times n$ identity matrix and $\Lambda$ is an $n\times n$ diagonal matrix whose diagonal entries are all positive.
\vspace{2mm}

\noindent \underline{Proof of the Claim:}
Since $B$ is symmetric positive definite, let $B=LL^\T $ be the Cholesky decomposition of $B$, where $L$ is an $n\times n$ lower triangular matrix with all positive diagonal entries and $L$ is invertible. Let $G=L^{-1} A L^{-\T}$. Then obviously $G$ is also a symmetric positive definite matrix with $G^\T =G$. Suppose that $G$ has the spectral decomposition $G=PDP^{-1}$ where $P$ is an $n\times n$ orthogonal matrix ($P^{-1}=P^\T $) and $D$ is a $n\times n$ diagonal matrix whose diagonal entries are all eigenvalues of $G$ and they are all positive. Then $P^\T  G P = D$. We let $U=L^{-\T} P D^{-1/2}$. It follows that
\begin{align*}
U^\T  A U &= D^{-1/2} P^\T  L^{-1} A L^{-\T} P D^{-1/2} \\
&= D^{-1/2} P^\T  G P D^{-1/2} = D^{-1/2} D D^{-1/2} = I_n,\\
U^\T  B U &= D^{-1/2} P^\T  L^{-1} B L^{-\T} P D^{-1/2} \\
&= D^{-1/2} P^\T  L^{-1} LL^\T  L^{-\T} P D^{-1/2} = D^{-1/2} P^\T  P  D^{-1/2} = D^{-1}.
\end{align*}
We set $\Lambda=D^{-1}$ which is an $n\times n$ diagonal matrix whose diagonal entries are all positive. This proves the claim.
\vspace{3mm}

Based on the claim, if we set $A=\theta_0 K_{\alpha_0,\nu}(S_n)$ and $B=\theta_0 K_{\alpha,\nu}(S_n)$, then we can find an invertible matrix $U$ such that \eqref{eq:UAU} holds. Because $\theta_0,\alpha_0,\nu$ are assumed to be fixed numbers, we can see that $U$ only changes with $\alpha$ and therefore we can write it as $U_{\alpha}$. Similarly, we can write $\Lambda_n(\alpha)$ to highlight its dependence on $\alpha$ and $n$. Correspondingly, we have $\theta_0  U_{\alpha}^\T  K_{\alpha_0,\nu}(S_n) U_{\alpha}= I_n$ and $\theta_0 U_{\alpha}^\T  K_{\alpha,\nu}(S_n) U_{\alpha}= \diag\{\lambda_{k,n}(\alpha):k=1,\ldots,n\} =\Lambda_n(\alpha)$.  This proves \eqref{diagonalize}. The existence of orthonormal basis functions $\psi_1,\ldots,\psi_n \in L_2(f_{\theta_0,\alpha_0,\nu})$ is proved in Section 4 of \citet{WangLoh11}, which does not involve the specific forms of covariance functions. This completes the proof of Lemma \ref{lem:URU}.
\end{proof}

\begin{lemma}\label{lem:specden_lambda}
Suppose that Assumption \ref{cond:spectral} holds. Then for $\{\lambda_{k,n}(\alpha):k=1,\ldots,n\}$ defined in Lemma \ref{lem:URU} and any $\alpha\in \RR_+$,
\begin{align}
& \sup_{|\alpha/\alpha_0 -1|\leq r_0} \lambda_{k,n}(\alpha) \leq 1 + L|\alpha/\alpha_0-1|^{\kappa}, \label{lambda.upper1} \\
& \inf_{|\alpha/\alpha_0 -1|\leq r_0}\lambda_{k,n}(\alpha) \geq 1 - L|\alpha/\alpha_0-1|^{\kappa}, \label{lambda.lower1}
\end{align}
for all $k=1,\ldots,n$. Furthermore, if $\alpha\leq \alpha_0$, then $\lambda_{k,n}(\alpha)\geq 1$ for all $k=1,\ldots,n$; if  $\alpha\geq \alpha_0$, then $0\leq \lambda_{k,n}(\alpha)\leq 1$ for all $k=1,\ldots,n$.
\end{lemma}

\begin{proof}[$\pof$ Lemma \ref{lem:specden_lambda}]
For \eqref{lambda.upper1} and \eqref{lambda.lower1}, we use the relation
$$\lambda_{k,n}(\alpha)= \int_{\RR^d} |\psi_k(w)|^2 f_{\theta_0,\alpha_0,\nu}(w) \cdot \frac{f_{\theta_0,\alpha,\nu}(w)}{f_{\theta_0,\alpha_0,\nu}(w)} \ud w$$
for $k=1,\ldots,n$ together with the bounds in Assumption \ref{cond:spectral} (i) to obtain that for all $\alpha$ satisfying $|\alpha/\alpha_0 -1|\leq r_0$,
\begin{align}\label{lambda30}
\lambda_{k,n}(\alpha)&\leq \sup_{w \in \RR^d} \frac{f_{\theta_0,\alpha,\nu}(w)}{f_{\theta_0,\alpha_0,\nu}(w)}  \cdot \int_{\RR^d} |\psi_k(w)|^2 f_{\theta_0,\alpha_0,\nu}(w) \ud w \leq 1 + L|\alpha/\alpha_0-1|^{\kappa}, \nonumber \\
\lambda_{k,n}(\alpha)&\geq \inf_{w \in \RR^d} \frac{f_{\theta_0,\alpha,\nu}(w)}{f_{\theta_0,\alpha_0,\nu}(w)}  \cdot \int_{\RR^d} |\psi_k(w)|^2 f_{\theta_0,\alpha_0,\nu}(w) \ud w \geq 1 - L|\alpha/\alpha_0-1|^{\kappa}.
\end{align}	
This proves \eqref{lambda.upper1} and \eqref{lambda.lower1}.

By the non-increasing property of $f_{\theta,\alpha,\nu}(w)$ in $\alpha$ in Assumption \ref{cond:spectral} (ii), we have that if $\alpha\leq \alpha_0$, $f_{\theta_0,\alpha,\nu}(w)\geq f_{\theta_0,\alpha_0,\nu}(w) $ for all $w\in \RR^d$, and for all $k=1,\ldots,n$,
\begin{align*}
\lambda_{k,n}(\alpha)& = \int_{\RR^d} |\psi_k(w)|^2 f_{\theta_0,\alpha,\nu}(w) \ud w \geq \int_{\RR^d} |\psi_k(w)|^2 f_{\theta_0,\alpha_0,\nu}(w) \ud w = 1.
\end{align*}
The conclusion for $\alpha\geq \alpha_0$ follows similarly.
\end{proof}

\begin{lemma} \label{lem:monotone}
Suppose that Assumption \ref{cond:spectral} (ii) holds. For any $\alpha,\alpha'\in \RR_+$ such that $\alpha\leq \alpha'$, we have that $K_{\alpha,\nu}(S_n)\succeq K_{\alpha',\nu}(S_n)$, for any set $S_n=\{s_1,\ldots,s_n\}$ of $n$ distinct locations in $[0,1]^d$.
\end{lemma}

\begin{proof}[$\pof$ Lemma \ref{lem:monotone}]
The proof of Lemma \ref{lem:monotone} is motivated by the proof of Lemma 1 in \citet{KauSha13}. We define the matrix $\Omega^{\dagger} = K_{\alpha,\nu}(S_n) - K_{\alpha',\nu}(S_n)$. Then the entries of $\Omega^{\dagger}$ can be expressed in terms of a function $\widetilde K_{\Omega^{\dagger}}: \RR^d\to \RR$, with its $(i,j)$-entry
\begin{align*}
\Omega^{\dagger}_{ij} &= \widetilde K_{\Omega^{\dagger}}(\bfs_i-\bfs_j) = K_{\alpha,\nu}(\bfs_i-\bfs_j) - K_{\alpha',\nu}(\bfs_i-\bfs_j),
\end{align*}
for $i,j\in \{1,\ldots,n\}$. The matrix $\Omega^{\dagger}$ is positive definite if $\widetilde K_{\Omega^{\dagger}}$ is a positive semidefinite function.

By the non-increasing property of $f_{\theta,\alpha,\nu}(w)$ in $\alpha$ for any given $\theta>0$ and $w\in \RR^d$ as in Assumption \ref{cond:spectral} (ii), we have that for $\alpha\leq \alpha'$,
\begin{align} \label{eq:f2.diff}
f_{\theta,\alpha,\nu}(w) & \geq f_{\theta,\alpha',\nu}(w).
\end{align}
Therefore, we can compute the spectral density of the function $\widetilde K_{\Omega^{\dagger}}$:
\begin{align} \label{eq:fOmega1}
f_{\Omega^{\dagger}}(w) &= \frac{1}{(2\pi)^d} \int_{\RR^d} \exp\left(-\imath w^\T x \right) \widetilde K_{\Omega^{\dagger}} (x) \ud x \nonumber \\
&= f_{1,\alpha,\nu}(w) - f_{1,\alpha',\nu}(w) \geq 0, \text{ for all } w \in \RR^d,
\end{align}
where the last step follows from \eqref{eq:f2.diff}. This has shown that $\widetilde K_{\Omega^{\dagger}}$ is indeed a positive semidefinite function. Therefore, $\Omega^{\dagger}$ is a positive semidefinite matrix and the conclusion follows.
\end{proof}

\begin{lemma}\label{lem:lambda.3bound}
Suppose that Assumption \ref{cond:spectral} holds. Let $\Acal_{n}^{\dagger}=\left\{\alpha\in \RR_+:|\alpha/\alpha_0-1|\leq n^{-2/\kappa} \right\}$ where $\kappa$ is as defined in Assumption \ref{cond:spectral}. There exists a large integer $N_2$ that only depends on $\nu,d,\alpha_0,L,r_0,\kappa$, such that for all $n>N_2$,
\begin{align}\label{eq:lambda.3bound}
& \sup_{\alpha \in \Acal_{n}^{\dagger}}  \left\{ \sum_{k=1}^n \left| \lambda_{k,n}(\alpha) - 1 \right| \right\} \leq L n^{-1}, \qquad \sup_{\alpha \in \Acal_{n}^{\dagger}}  \left\{ \sum_{k=1}^n \left| \lambda_{k,n}(\alpha) - 1 \right|^2 \right\} \leq L^2 n^{-3}, \nonumber \\
& \sup_{\alpha \in \Acal_{n}^{\dagger}}  \left\{ \sum_{k=1}^n \left| \lambda_{k,n}(\alpha)^{-1} - 1 \right| \right\} \leq 2L n^{-1}, \qquad \sup_{\alpha \in \Acal_{n}^{\dagger}}  \left\{ \sum_{k=1}^n \left| \lambda_{k,n}(\alpha)^{-1} - 1 \right|^2 \right\} \leq 4L^2 n^{-3}.
\end{align}
\end{lemma}

\begin{proof}[$\pof$ Lemma \ref{lem:lambda.3bound}]
For all sufficiently large $n$, $n^{-2/\kappa} < r_0$ for $r_0$ in Assumption \ref{cond:spectral}. By Lemma \ref{lem:specden_lambda}, we have that for any $k=1,\ldots,n$, for all $\alpha\in \Acal_{n}^{\dagger}$,
\begin{align}
& \quad~ \left|\lambda_{k,n}(\alpha) - 1 \right| \leq L|\alpha/\alpha_0-1|^{\kappa} \nonumber.
\end{align}
Therefore, uniformly over all $\alpha\in \Acal_{n}^{\dagger}$,
\begin{align} \label{eq:lambda.2sum}
& \sum_{k=1}^n \left| \lambda_{k,n}(\alpha) - 1 \right|  \leq n\cdot L|\alpha/\alpha_0-1|^{\kappa} \leq Ln^{-1}, \nonumber \\
& \sum_{k=1}^n \left| \lambda_{k,n}(\alpha) - 1 \right|^2  \leq n\cdot L^2 |\alpha/\alpha_0-1|^{2\kappa} \leq L^2 n^{-3},
\end{align}
which has proved the first two inequalities in \eqref{eq:lambda.3bound}.

For all $\alpha \in \Acal_{n}^{\dagger}$, for all sufficiently large $n$,
$$1-L|\alpha/\alpha_0-1|^{\kappa} \geq 1 - L n^{-2} \geq \frac{1}{2}. $$
Therefore, using \eqref{lambda.lower1} in Lemma \ref{lem:specden_lambda} and \eqref{eq:lambda.2sum}, we have that uniformly for all $\alpha \in \Acal_{n}^{\dagger}$, for all sufficiently large $n$,
\begin{align*}
& \min_{1\leq k\leq n} \lambda_{k,n}(\alpha) \geq 1/2, \quad \max_{1\leq k\leq n} \lambda_{k,n}(\alpha)^{-1}  \leq 2, \nonumber \\
& \sum_{k=1}^n \left| \lambda_{k,n}(\alpha)^{-1} - 1 \right| = \frac{\sum_{k=1}^n \left| \lambda_{k,n}(\alpha) - 1 \right|}{\min_{1\leq k\leq n} \lambda_{k,n}(\alpha)} \leq 2Ln^{-1}, \nonumber \\
& \sum_{k=1}^n \left| \lambda_{k,n}(\alpha)^{-1} - 1 \right|^2 = \frac{\sum_{k=1}^n \left| \lambda_{k,n}(\alpha) - 1 \right|}{\min_{1\leq k\leq n} \lambda_{k,n}(\alpha)^2} \leq 4L^2n^{-3}.
\end{align*}
which proves the last two inequalities in \eqref{eq:lambda.3bound}.
\end{proof}

\subsection{Technical Lemmas for Evidence Lower Bound} \label{subsec:ELBO}
Let $\widetilde Y_n = Y_n -  F_n  \beta_0$. Then $\widetilde Y_n \sim \Ncal(0,\theta_0 K_{\alpha_0,\nu}(S_n) + \tau_0 I_n)$, and $Y_n -  F_n  \beta = \widetilde Y_n -  F_n (\beta-\beta_0)$.

\begin{lemma} \label{lem:denom.term12}
Suppose that Assumptions \ref{cond:model} and \ref{cond:spectral} hold. Let $\Acal_{n} = \big\{(\theta,\alpha,\tau,\beta)\in \RR_+^{3} \times \RR^p:$ $|\alpha/\alpha_0-1|\leq n^{-2/\kappa} \big\}$. There exists a large integer $N_3$ that only depends on $\nu,d,\theta_0,\alpha_0,L,r_0,\kappa$, such that for any $\epsilon_1,\epsilon_2\in (0,1/2)$, for all $n>N_3$, with probability at least $1-2\exp(-\log^2 n)$,
\begin{align} \label{eq:denom.lower12}
& \inf_{\Bcal_0(\epsilon_1,\epsilon_2)\cap \Acal_{n}} \Bigg\{ -\frac{1}{2} \left(\tr\left[\left\{\theta K_{\alpha,\nu}(S_n) + \tau I_n\right\}^{-1} \widetilde Y_n \widetilde Y_n^\T  \right] - \tr\left[\left\{\theta_0 K_{\alpha_0,\nu}(S_n) + \tau_0 I_n\right\}^{-1} \widetilde Y_n \widetilde Y_n^\T  \right] \right) \nonumber \\
&\qquad - \frac{1}{2} \log \frac{\dett\left\{\theta K_{\alpha,\nu}(S_n) +\tau I_n\right\}}{\dett\left\{\theta_0 K_{\alpha_0,\nu}(S_n) +\tau_0 I_n\right\}} \Bigg\} \nonumber \\
& \geq - \left(\frac{5}{2}n\epsilon_{1 \vee 2}  + 5Ln^{-1} \right),
\end{align}
where $\epsilon_{1\vee 2} = \max(\epsilon_1,\epsilon_2)$.
\end{lemma}

\begin{proof}[$\pof$ Lemma \ref{lem:denom.term12}]
There are two terms inside the infimum in \eqref{eq:denom.lower12}. We provide lower bounds for each of them. For notational simplicity, we define
\begin{align} \label{eq:Omega}
\Omega_{\alpha} = \theta_0 K_{\alpha,\nu}(S_n) + \tau_0 I_n,\quad \Omega_0 = \theta_0 K_{\alpha_0,\nu}(S_n) +\tau_0 I_n .
\end{align}

\noindent \underline{\sc Lower bound for the first term in \eqref{eq:denom.lower12}:}
\vspace{2mm}

In the set $\Bcal_0(\epsilon_1,\epsilon_2)$, $\theta>(1-\epsilon_{1 \vee 2})\theta_0$ and $\tau>(1-\epsilon_{1 \vee 2})\tau_0$. Since $K_{\alpha,\nu}(S_n)$ is positive definite, we have that $\theta K_{\alpha,\nu}(S_n) + \tau I_n \succeq (1-\epsilon_{1 \vee 2})\left\{\theta_0 K_{\alpha,\nu}(S_n) + \tau_0 I_n \right\}$, and hence $\left\{\theta K_{\alpha,\nu}(S_n) + \tau I_n\right\}^{-1} \preceq (1-\epsilon_{1 \vee 2})^{-1}\left\{\theta_0 K_{\alpha,\nu}(S_n) + \tau_0 I_n \right\}^{-1}$. Therefore, for $\widetilde Y_n = Y_n -  F_n  \beta_0 \sim \Ncal(0,\theta_0 K_{\alpha_0,\nu}(S_n) + \tau_0 I_n)$, by Lemma \ref{lem:posdef} Part (ii), we have that
\begin{align} \label{eq:denom.expo1.1}
& \tr\left[\left\{\theta K_{\alpha,\nu}(S_n) + \tau I_n\right\}^{-1} \widetilde Y_n \widetilde Y_n^\T  \right] - \tr\left[\left\{\theta_0 K_{\alpha_0,\nu}(S_n) + \tau_0 I_n\right\}^{-1} \widetilde Y_n \widetilde Y_n^\T  \right] \nonumber \\
\leq{}& (1-\epsilon_{1 \vee 2})^{-1} \tr\left[\left\{\theta_0 K_{\alpha,\nu}(S_n) + \tau_0 I_n \right\}^{-1} \widetilde Y_n \widetilde Y_n^\T  \right] - \tr\left[\left\{\theta_0 K_{\alpha_0,\nu}(S_n) + \tau_0 I_n\right\}^{-1} \widetilde Y_n \widetilde Y_n^\T  \right] \nonumber \\
={}& (1-\epsilon_{1 \vee 2})^{-1} \left\{ \tr\left(\Omega_{\alpha}^{-1} \widetilde Y_n \widetilde Y_n^\T  \right) - \tr\left(\Omega_0^{-1} \widetilde Y_n \widetilde Y_n^\T  \right) \right\}  + \frac{\epsilon_{1 \vee 2}}{1-\epsilon_{1 \vee 2}} \tr\left(\Omega_0^{-1} \widetilde Y_n \widetilde Y_n^\T  \right).
\end{align}
We derive an upper bound for the first term inside the bracket in \eqref{eq:denom.expo1.1}. Using \eqref{diagonalize} in Lemma \ref{lem:URU}, we have that $\theta_0 K_{\alpha,\nu}(S_n) =U_{\alpha}^{-\T} \Lambda_n(\alpha) U_{\alpha}^{-1}$ and $\theta_0 K_{\alpha_0,\nu}(S_n) =U_{\alpha}^{-\T} U_{\alpha}^{-1}$. Therefore, we repeatedly apply the Sherman-Morrison-Woodbury formula for inverse matrices to obtain that
\begin{align} \label{eq:denom.expo1.2}
& ~~\Omega_{\alpha}^{-1} - \Omega_0^{-1} \nonumber \\
={}& \left(U_{\alpha}^{-\T} \Lambda_n(\alpha) U_{\alpha}^{-1} + \tau_0 I_n\right)^{-1} - \Omega_0^{-1} \nonumber \\
={}& \left\{U_{\alpha}^{-\T} \left(\Lambda_n(\alpha)-I_n\right) U_{\alpha}^{-1} + \Omega_0 \right\}^{-1} - \Omega_0^{-1} \nonumber \\
={}& \Omega_0^{-1} - \Omega_0^{-1}U_{\alpha}^{-\T} \left[\left\{\Lambda_n(\alpha) - I_n\right\}^{-1} + U_{\alpha}^{-1} \Omega_0^{-1} U_{\alpha}^{-\T} \right]^{-1} U_{\alpha}^{-1} \Omega_0^{-1}  - \Omega_0^{-1} \nonumber \\
={}& \Omega_0^{-1}U_{\alpha}^{-\T} \left[\left\{I_n - \Lambda_n(\alpha) \right\}^{-1} - U_{\alpha}^{-1} \Omega_0^{-1} U_{\alpha}^{-\T} \right]^{-1} U_{\alpha}^{-1} \Omega_0^{-1}.
\end{align}
Next we focus on the matrix $T_{\alpha}\equiv \left[\left\{I_n - \Lambda_n(\alpha) \right\}^{-1} - U_{\alpha}^{-1} \Omega_0^{-1} U_{\alpha}^{-\T} \right]^{-1}$. Obviously $U_{\alpha}^{-1} \Omega_0^{-1} U_{\alpha}^{-\T}$ is positive definite. Using the definition of $\Omega_0$ in \eqref{eq:Omega}, it follows from Lemma \ref{lem:posdef} Part (i) that
\begin{align}\label{eq:UOU}
0_{n\times n} \prec U_{\alpha}^{-1} \Omega_0^{-1} U_{\alpha}^{-\T} & = U_{\alpha}^{-1} \left(U_{\alpha}^{-\T}U_{\alpha}^{-1} +\tau_0 I_n\right)^{-1} U_{\alpha}^{-\T} \nonumber \\
& \preceq  U_{\alpha}^{-1} \left(U_{\alpha}^{-\T} U_{\alpha}^{-1}\right)^{-1} U_{\alpha}^{-\T} = I_n.
\end{align}
Therefore, all the eigenvalues of $U_{\alpha}^{-1} \Omega_0^{-1} U_{\alpha}^{-\T}$ are between 0 and 1. Meanwhile, if $\alpha>\alpha_0$, then by Lemma \ref{lem:monotone}, $K_{\alpha,\nu}(S_n)\preceq K_{\alpha_0,\nu}(S_n)$, and hence $\Omega_{\alpha}^{-1} \succeq \Omega_0^{-1}$. From \eqref{eq:denom.expo1.2}, this implies that $T_{\alpha}$ is positive definite when $\alpha> \alpha_0$. Similarly, $T_{\alpha}$ is negative definite when $\alpha< \alpha_0$. Let $T_{\alpha} = P_{\alpha}^{-1} \diag\{a_1,\ldots,a_n\} P_{\alpha}$ be the eigendecomposition of $T_{\alpha}$, where $P_{\alpha}$ is an $n\times n$ orthogonal matrix consisting of eigenvectors, and $a_1,\ldots,a_n$ are the eigenvalues of $T_{\alpha}$. Then $a_1^{-1},\ldots,a_n^{-1}$ are the eigenvalues of $T_{\alpha}^{-1}$. When $\alpha>\alpha_0$, since $T_{\alpha}^{-1}, \left\{I_n - \Lambda_n(\alpha) \right\}^{-1}, U_{\alpha}^{-1} \Omega_0^{-1} U_{\alpha}^{-\T}$ are all positive definite, by applying the Weyl's inequality to the relation $T_{\alpha}^{-1} = \left\{I_n - \Lambda_n(\alpha) \right\}^{-1}- U_{\alpha}^{-1} \Omega_0^{-1} U_{\alpha}^{-\T}$, we have that up to a permutation of $a_k$'s, $\{1-\lambda_{k,n}(\alpha)\}^{-1}-1 \leq a_k^{-1} \leq \{1-\lambda_{k,n}(\alpha)\}^{-1}$ for $k=1,\ldots,n$. Similarly, when $\alpha<\alpha_0$, $T_{\alpha}^{-1}$ is negative definite, and $\left\{I_n - \Lambda_n(\alpha) \right\}^{-1}$ is also negative definite since all $\lambda_{k,n}(\alpha)\geq 1$ by Lemma \ref{lem:specden_lambda}. We can apply the Weyl's inequality to the relation $(-T_{\alpha})^{-1} = \left\{\Lambda_n(\alpha) -I_n\right\}^{-1} + U_{\alpha}^{-1} \Omega_0^{-1} U_{\alpha}^{-\T}$ to obtain that $\{\lambda_{k,n}(\alpha)-1\}^{-1} \leq -a_k^{-1} \leq \{\lambda_{k,n}(\alpha)-1\}^{-1}+1$, or equivalently $\{1-\lambda_{k,n}(\alpha)\}^{-1}-1 \leq a_k^{-1} \leq \{1-\lambda_{k,n}(\alpha)\}^{-1} $ for $k=1,\ldots,n$.

On the other hand, from Lemma \ref{lem:specden_lambda}, we have that uniformly over all $\alpha\in \Acal_{n}$, for all sufficiently large $n$, $\max_{1\leq k\leq n}|\lambda_{k,n}(\alpha)-1|\leq Ln^{-2} < 1/2$, i.e., $\{1- \lambda_{k,n}(\alpha) \}^{-1} \in [-\infty,-2] \cup [2,+\infty]$ for all $k=1,\ldots,n$. Therefore, we have that for both the cases of  $\alpha>\alpha_0$ and $\alpha<\alpha_0$, the eigenvalues of $T_{\alpha}$ satisfy
\begin{align} \label{eq:lambda-1.bound}
& |a_k| \leq \max\left\{\left|\frac{1}{\{1-\lambda_{k,n}(\alpha)\}^{-1} }\right|,  \left|\frac{1}{\{1-\lambda_{k,n}(\alpha)\}^{-1}-1 }\right| \right\} \nonumber \\
&= \max\left\{|\lambda_{k,n}(\alpha)-1|, \left|\lambda_{k,n}(\alpha)^{-1} -1 \right| \right\} \leq |\lambda_{k,n}(\alpha)-1| + \left|\lambda_{k,n}(\alpha)^{-1} -1 \right|,
\end{align}
for $k=1,\ldots,n$. This also holds for the case $\alpha=\alpha_0$ since both sides of \eqref{eq:lambda-1.bound} are equal to zero.

We define the $n\times n$ diagonal matrix $Q_n(\alpha)$:
\begin{align} \label{eq:Qn}
Q_n(\alpha) & = \diag\left\{|\lambda_{1,n}(\alpha)-1| + \left|\lambda_{1,n}(\alpha)^{-1} -1 \right|, \ldots, |\lambda_{n,n}(\alpha)-1| + \left|\lambda_{n,n}(\alpha)^{-1} -1 \right| \right\}.
\end{align}
Then by Lemma \ref{lem:posdef} Part (i), \eqref{eq:lambda-1.bound} and the eigendecomposition $T_{\alpha} = P_{\alpha}^{-1} \diag\{a_1,\ldots,a_n\} P_{\alpha}$ imply that
\begin{align}
T_{\alpha} &= \left[\left\{I_n - \Lambda_n(\alpha) \right\}^{-1} - U_{\alpha}^{-1} \Omega_0^{-1} U_{\alpha}^{-\T} \right]^{-1} \preceq P_{\alpha}^{-1}Q_n(\alpha)P_{\alpha}. \nonumber
\end{align}
This together with \eqref{eq:denom.expo1.2} implies that for all $\alpha\in \Acal_{n}$,
\begin{align}
& \Omega_{\alpha}^{-1} - \Omega_0^{-1} \preceq \Omega_0^{-1}U_{\alpha}^{-\T}P_{\alpha}^{-1}Q_n(\alpha)P_{\alpha}U_{\alpha}^{-1} \Omega_0^{-1} .  \nonumber
\end{align}
Hence the first term in \eqref{eq:denom.expo1.1} can be upper bounded by
\begin{align} \label{eq:denom.expo1.5}
& (1-\epsilon_{1 \vee 2})^{-1} \left[ \tr\left\{\left(\theta_0 K_{\alpha,\nu}(S_n) + \tau_0 I_n \right)^{-1} \widetilde Y_n \widetilde Y_n^\T  \right\} - \tr\left\{\left(\theta_0 K_{\alpha_0,\nu}(S_n) + \tau_0 I_n\right)^{-1} \widetilde Y_n \widetilde Y_n^\T  \right\} \right] \nonumber \\
\leq{}& (1-\epsilon_{1 \vee 2})^{-1} \widetilde Y_n^\T \Omega_0^{-1}U_{\alpha}^{-\T} P_{\alpha}^{-1}Q_n(\alpha)P_{\alpha} U_{\alpha}^{-1} \Omega_0^{-1} \widetilde Y_n.
\end{align}
Since $\Omega_0$ is positive definite, there exists an invertible matrix $V_0$ such that $V_0^\T \Omega_0 V_0 = I_n$ ($V_0$ depends on $\alpha_0$ and $\tau_0$) and hence $\Omega_0^{-1}=V_0V_0^\T$. Define the random vector $Z_n = V_0^{\T} \widetilde Y_n$. Since $\widetilde Y_n\sim \mathcal{N}(0,\Omega_0)$, we have that $Z_n \sim \mathcal{N}(0,I_n)$. Therefore, we can rewrite the right-hand side of \eqref{eq:denom.expo1.5} as
\begin{align} \label{eq:denom.expo1.6}
& \widetilde Y_n^\T  \Omega_0^{-1}U_{\alpha}^{-\T} P_{\alpha}^{-1}Q_n(\alpha)P_{\alpha} \Omega_0^{-1} \widetilde Y_n
= Z_n^\T  V_0^\T  U_{\alpha}^{-\T} P_{\alpha}^{-1}Q_n(\alpha)P_{\alpha}  U_{\alpha}^{-1} V_0 Z_n.
\end{align}
Then we apply the Hanson-Wright inequality in Lemma \ref{lem:Hsuetal12} to the quadratic form in \eqref{eq:denom.expo1.6} to obtain that for any number $z>0$ and any $\alpha\in \Acal_{n}$,
\begin{align} \label{eq:denom.expo1.7}
& \pr \left\{Z_n^\T  \Sigma_{\alpha} Z_n \geq \tr(\Sigma_{\alpha}) + 2\left\{\tr(\Sigma_{\alpha}^2)z\right\}^{1/2} + 2\|\Sigma_{\alpha}\|_{\op}z \right\} \leq \ee^{-z},
\end{align}
where we define $\Sigma_{\alpha}=V_0^\T  U_{\alpha}^{-\T}  P_{\alpha}^{-1}Q_n(\alpha)P_{\alpha}  U_{\alpha}^{-1} V_0$.

We continue to find upper bounds for each term related to $\Sigma_{\alpha}$ in \eqref{eq:denom.expo1.7}, using the definition of $Q_n(\alpha)$ in \eqref{eq:Qn}:
\begin{align}
\tr(\Sigma_{\alpha}) & = \tr\left\{V_0^\T  U_{\alpha}^{-\T}  P_{\alpha}^{-1}Q_n(\alpha)P_{\alpha}  U_{\alpha}^{-1} V_0 \right\} =  \tr\left\{ P_{\alpha}^{-1}Q_n(\alpha)P_{\alpha}  U_{\alpha}^{-1} V_0 V_0^\T  U_{\alpha}^{-\T} \right\} \nonumber \\
&= \tr\left\{ P_{\alpha}^{-1}Q_n(\alpha)P_{\alpha}  U_{\alpha}^{-1}\Omega_0^{-1} U_{\alpha}^{-\T} \right\} \stackrel{(i)}{\leq} \tr\left\{ P_{\alpha}^{-1}Q_n(\alpha)P_{\alpha} \right\} = \tr\left\{ Q_n(\alpha) \right\} \nonumber \\
&=  \sum_{k=1}^n |\lambda_1(\alpha)-1| + \sum_{k=1}^n \left|\lambda_1(\alpha)^{-1} -1 \right|. \label{eq:denom.expo1.8} \\
\tr(\Sigma_{\alpha}^2) & = \tr\left[\left\{V_0^\T  U_{\alpha}^{-\T} P_{\alpha}^{-1}Q_n(\alpha)P_{\alpha} U_{\alpha}^{-1} V_0\right\}^2 \right] \nonumber \\
&= \tr\left\{ V_0^\T  U_{\alpha}^{-\T} P_{\alpha}^{-1}Q_n(\alpha)P_{\alpha} U_{\alpha}^{-1} V_0 V_0^\T  U_{\alpha}^{-\T} P_{\alpha}^{-1}Q_n(\alpha)P_{\alpha} U_{\alpha}^{-1} V_0 \right\} \nonumber \\
&= \tr\left\{P_{\alpha}^{-1}Q_n(\alpha)P_{\alpha} U_{\alpha}^{-1} \Omega_0^{-1}  U_{\alpha}^{-\T} P_{\alpha}^{-1}Q_n(\alpha)P_{\alpha} U_{\alpha}^{-1} \Omega_0^{-1}  U_{\alpha}^{-\T} \right\} \nonumber \\
&\stackrel{(ii)}{\leq}  \tr\left\{P_{\alpha}^{-1}Q_n(\alpha)P_{\alpha} U_{\alpha}^{-1} \Omega_0^{-1}  U_{\alpha}^{-\T} P_{\alpha}^{-1}Q_n(\alpha)P_{\alpha} \right\} \nonumber \\
&=\tr\left\{P_{\alpha}^{-1}Q_n(\alpha)^2P_{\alpha} U_{\alpha}^{-1} \Omega_0^{-1}  U_{\alpha}^{-\T}  \right\} \stackrel{(iii)}{\leq} \tr\left\{P_{\alpha}^{-1}Q_n(\alpha)^2P_{\alpha}\right\} = \tr\left\{Q_n(\alpha)^2\right\} \nonumber \\
&= \sum_{k=1}^n \left\{\left|\lambda_{k,n}(\alpha) -1\right| + \left|\lambda_{k,n}(\alpha)^{-1}-1 \right| \right\}^2 \nonumber \\
&\stackrel{(iv)}{\leq} 2\sum_{k=1}^n \left| \lambda_{k,n}(\alpha) -1 \right|^2 + 2 \sum_{k=1}^n \left|\lambda_{k,n}(\alpha)^{-1}-1\right|^2 , \label{eq:denom.expo1.9} \\
\|\Sigma_{\alpha}\|_{\op}^2 & = \mathsf{s}_{\max} \left[ \left\{ V_0^\T  U_{\alpha}^{-\T} P_{\alpha}^{-1}Q_n(\alpha)P_{\alpha} U_{\alpha}^{-1} V_0\right\}^2 \right] \leq \tr\left[\left\{V_0^\T  U_{\alpha}^{-\T} P_{\alpha}^{-1}Q_n(\alpha)P_{\alpha} U_{\alpha}^{-1} V_0\right\}^2 \right]   \nonumber \\
&\leq 2\sum_{k=1}^n \left| \lambda_{k,n}(\alpha) -1 \right|^2 + 2 \sum_{k=1}^n \left|\lambda_{k,n}(\alpha)^{-1}-1\right|^2 , \label{eq:denom.expo1.10}
\end{align}
where (i), (ii) and (iii) follow from Lemma \ref{lem:posdef} and \eqref{eq:UOU}, (iv) follows from the inequality $(a+b)^2\leq 2(a^2+b^2)$ for any $a,b\in \RR$.

If we take $z=(\log n)^2$, using the upper bounds in Lemma \ref{lem:lambda.3bound} (noticing that $\Acal_{n}^{\dagger}$ and $\Acal_{n}$ have the same condition on $\alpha$), we can obtain from \eqref{eq:denom.expo1.8}-\eqref{eq:denom.expo1.10} that for all sufficiently large $n$,
\begin{align} \label{eq:denom.expo1.11}
& \sup_{\alpha \in \Acal_{n}} \left\{ \tr(\Sigma_{\alpha}) + 2\left\{\tr(\Sigma_{\alpha}^2)\right\}^{1/2}\log n + 2\|\Sigma_{\alpha}\|_{\op} \log^2 n \right\}  \nonumber \\
& \leq Ln^{-1} + 2Ln^{-1} + 4 \left(L^2n^{-3} + 4L^2n^{-3} \right)^{1/2} \log n + 4 \left(L^2n^{-3} + 4L^2n^{-3} \right)^{1/2} \log^2 n \leq 4Ln^{-1}.
\end{align}
Because the matrix $K_{\alpha,\nu}(S_n)$ is non-increasing in $\alpha$ by Lemma \ref{lem:monotone}, we can see from the definition of $\Omega_{\alpha}$ in \eqref{eq:Omega} that the first term in \eqref{eq:denom.expo1.1} is an increasing function in $\alpha$. Let $\overline\alpha_n= \alpha_0 ( 1 + n^{-2/\kappa} ) $. Then the first term in \eqref{eq:denom.expo1.1} attains its supremum at $\overline\alpha_n$. We combine \eqref{eq:denom.expo1.5}, \eqref{eq:denom.expo1.6}, \eqref{eq:denom.expo1.7} and \eqref{eq:denom.expo1.11} to obtain that
\begin{align} \label{eq:denom.expo1.12}
& \pr \Bigg\{ \sup_{\Acal_{n} }(1-\epsilon_{1 \vee 2})^{-1} \Big[ \tr\left\{\left(\theta_0 K_{\alpha,\nu}(S_n) + \tau_0 I_n \right)^{-1} \widetilde Y_n \widetilde Y_n^\T  \right\} \nonumber \\
&\quad - \tr\left\{\left(\theta_0 K_{\alpha_0,\nu}(S_n) + \tau_0 I_n\right)^{-1} \widetilde Y_n \widetilde Y_n^\T  \right\} \Big] \geq (1-\epsilon_{1 \vee 2})^{-1} 4Ln^{-1} \Bigg\} \nonumber \\
={}& \pr \Bigg\{ (1-\epsilon_{1 \vee 2})^{-1} \Big( \tr\left[\left\{\theta_0 K_{\overline\alpha_n,\nu}(S_n) + \tau_0 I_n \right\}^{-1} \widetilde Y_n \widetilde Y_n^\T  \right] \nonumber \\
&\quad - \tr\left\{\left(\theta_0 K_{\alpha_0,\nu}(S_n) + \tau_0 I_n\right)^{-1} \widetilde Y_n \widetilde Y_n^\T  \right\} \Big) \geq (1-\epsilon_{1 \vee 2})^{-1} 4Ln^{-1}  \Bigg\} \nonumber \\
\leq{}& \pr \Bigg\{  \widetilde Y_n^\T   \Omega_0^{-1}U_{\overline\alpha_n}^{-\T} P_{\overline\alpha_n}^{-1}Q_n(\overline\alpha_n)P_{\overline\alpha_n} U_{\overline\alpha_n}^{-1} \Omega_0^{-1} \widetilde Y_n \geq 4Ln^{-1}  \Bigg\} \nonumber \\
={}& \pr \Bigg\{ Z_n^\T \Sigma_{\overline\alpha_n} Z_n \geq 4Ln^{-1} \Bigg\} \nonumber \\
\leq{}& \pr \Bigg\{ Z_n^\T  \Sigma_{\overline\alpha_n} Z_n \geq  \tr(\Sigma_{\overline\alpha_n}) + 2\left\{\tr(\Sigma_{\overline\alpha_n}^2)\right\}^{1/2}\log n + 2\|\Sigma_{\overline\alpha_n}\|_{\op} \log^2 n \Bigg\} \nonumber \\
\leq{} & \exp(-\log^2 n).
\end{align}
For the second term on the right-hand side of \eqref{eq:denom.expo1.1}, with the relation $\tr\left(\Omega_0^{-1} \widetilde Y_n \widetilde Y_n^\T  \right) = Z_n^\T  Z_n$, we apply the Hanson-Wright inequality in Lemma \ref{lem:Hsuetal12} with $\Sigma=I_n$ and $z=(\log n)^2$ to obtain that for all sufficiently large $n$,
\begin{align} \label{eq:denom.expo1.14}
&\quad~ \pr\left\{ \frac{\epsilon_{1 \vee 2}}{1-\epsilon_{1 \vee 2}}  \tr\left(\Omega_0^{-1} \widetilde Y_n \widetilde Y_n^\T  \right) \geq \frac{\epsilon_{1 \vee 2}}{1-\epsilon_{1 \vee 2}}  \cdot 2n \right\} \nonumber \\
&\leq \pr\left\{ Z_n^\T  Z_n \geq n + 2\sqrt{n} \log n + 2\log^2 n \right\} \leq \exp(-\log^2 n) .
\end{align}
\vspace{5mm}

\noindent \underline{\sc Lower bound for the second term in \eqref{eq:denom.lower12}:}
\vspace{2mm}

We first show a simple fact related to the Kullback-Leibler (KL) divergence. For two generic probability distributions $F_1,F_2$ with densities $f_1,f_2$, the KL divergence from $F_1$ to $F_2$ is defined as $\kl(F_1,F_2)=\int f_1\log (f_1/f_2)$. The KL divergence from $\mathcal{N}\left(0,\theta K_{\alpha,\nu}(S_n) +\tau I_n\right)$ to \\ $\mathcal{N}\left(0,\theta_0 K_{\alpha_0,\nu}(S_n) +\tau_0 I_n\right)$ is
\begin{align*}
&\kl\left(\mathcal{N}\left(0,\theta K_{\alpha,\nu}(S_n) +\tau I_n\right),\mathcal{N}\left(0,\theta_0 K_{\alpha_0,\nu}(S_n) +\tau_0 I_n\right)\right) \nonumber \\
={}& \frac{1}{2}\left[\log \frac{\dett\left(\theta_0 K_{\alpha_0,\nu}(S_n) +\tau_0 I_n\right)}{\dett\left(\theta K_{\alpha,\nu}(S_n) +\tau I_n\right)} + \tr\left\{\left(\theta_0 K_{\alpha_0,\nu}(S_n) +\tau_0 I_n\right)^{-1} \left(\theta K_{\alpha,\nu}(S_n) +\tau I_n\right) \right\} - n\right].
\end{align*}
Since the KL divergence is always nonnegative, this implies that
\begin{align} \label{eq:logdet.bound}
\log \frac{\dett\left\{\theta K_{\alpha,\nu}(S_n) +\tau I_n\right\}}{\dett\left\{\theta_0 K_{\alpha_0,\nu}(S_n) +\tau_0 I_n\right\}} \leq \tr\left[\left\{\theta_0 K_{\alpha_0,\nu}(S_n) +\tau_0 I_n\right\}^{-1} \left\{\theta K_{\alpha,\nu}(S_n) +\tau I_n\right\} \right] - n.
\end{align}
On the parameter set $\Bcal_0(\epsilon_1,\epsilon_2)$, $\theta\leq (1+\epsilon_1)\theta_0$ and $\tau \leq (1+\epsilon_2)\tau_0$. Therefore, we have $\theta K_{\alpha,\nu}(S_n) +\tau I_n\preceq (1+\epsilon_{1\vee 2}) \left(\theta_0 K_{\alpha,\nu}(S_n) +\tau_0 I_n\right)$.
By \eqref{eq:logdet.bound} and Lemma \ref{lem:posdef}, we have that
\begin{align} \label{eq:denom.expo2.1}
&\quad \log \frac{\dett\left\{\theta K_{\alpha,\nu}(S_n) +\tau I_n\right\}}{\dett\left\{\theta_0 K_{\alpha_0,\nu}(S_n) +\tau_0 I_n\right\}} \nonumber \\
& \leq \tr\left[\left\{\theta_0 K_{\alpha_0,\nu}(S_n) +\tau_0 I_n\right\}^{-1} \left\{\theta K_{\alpha,\nu}(S_n) +\tau I_n\right\} \right] - n \nonumber \\
&\leq \tr\left[\left\{\theta_0 K_{\alpha_0,\nu}(S_n) +\tau_0 I_n\right\}^{-1} \left\{(1+\epsilon_{1\vee 2})\theta_0 K_{\alpha,\nu}(S_n) + (1+\epsilon_{1\vee 2}) \tau_0 I_n\right\} \right] - n \nonumber \\
&= (1+\epsilon_{1\vee 2})\tr\left[\left\{\theta_0 K_{\alpha_0,\nu}(S_n) +\tau_0 I_n\right\}^{-1} \left\{\theta_0 K_{\alpha,\nu}(S_n) + \tau_0 I_n\right\} \right] - n \nonumber \\
&= (1+\epsilon_{1\vee 2})\theta_0 \tr\left[\left\{\theta_0 K_{\alpha_0,\nu}(S_n) +\tau_0 I_n\right\}^{-1} \left\{ K_{\alpha,\nu}(S_n) - K_{\alpha_0,\nu}(S_n)\right\} \right] \nonumber \\
&\quad +  (1+\epsilon_{1\vee 2})\tr\left[\left\{\theta_0 K_{\alpha_0,\nu}(S_n) +\tau_0 I_n\right\}^{-1} \left\{\theta_0 K_{\alpha_0,\nu}(S_n) + \tau_0 I_n\right\} \right]  - n \nonumber \\
&= (1+\epsilon_{1\vee 2})\theta_0 \tr\left[\left\{\theta_0 K_{\alpha_0,\nu}(S_n) +\tau_0 I_n\right\}^{-1} \left\{ K_{\alpha,\nu}(S_n) - K_{\alpha_0,\nu}(S_n)\right\} \right] + n \epsilon_{1\vee 2}.
\end{align}
We analyze the first term on the right-hand side of \eqref{eq:denom.expo2.1}. By Lemma \ref{lem:specden_lambda} and Lemma \ref{lem:posdef}, for all $|\alpha/\alpha_0-1|\leq n^{-2/\kappa}$, for all sufficiently large $n$,
\begin{align*}
& \theta_0\left\{K_{\alpha,\nu}(S_n) - K_{\alpha_0,\nu}(S_n)\right\} = U_{\alpha}^{-\T}\left\{ \Lambda_n(\alpha)- I_n \right\} U_{\alpha}^{-1} \\
& \preceq \sup_{\Acal_{n}} \max_{1\leq k\leq n} | \lambda_{k,n}(\alpha)-1 | \cdot U_{\alpha}^{-\T} U_{\alpha}^{-1}
\preceq  (Ln^{-2}) U_{\alpha}^{-\T} U_{\alpha}^{-1}.
\end{align*}
Therefore,
\begin{align} \label{eq:denom.expo2.2}
& \theta_0\tr\left\{\left(\theta_0 K_{\alpha_0,\nu}(S_n) +\tau_0 I_n\right)^{-1} \left( K_{\alpha,\nu}(S_n) - K_{\alpha_0,\nu}(S_n)\right) \right\} \nonumber \\
\leq {}& \tr\left\{\left(\theta_0 K_{\alpha_0,\nu}(S_n) \right)^{-1} \left( \theta_0K_{\alpha,\nu}(S_n) - \theta_0 K_{\alpha_0,\nu}(S_n)\right) \right\} \nonumber \\
\leq{}& \tr\left\{ U_{\alpha} U_{\alpha}^\T  (Ln^{-2}) U_{\alpha}^{-\T} U_{\alpha}^{-1} \right\} \nonumber \\
={}& \tr\left( Ln^{-2} I_n \right) = Ln^{-1}.
\end{align}
We combine \eqref{eq:denom.expo2.1} and \eqref{eq:denom.expo2.2} to obtain that for all sufficiently large $n$,
\begin{align} \label{eq:denom.expo2.4}
& \quad \sup_{\Acal_{n}} \log \frac{\dett\left(\theta K_{\alpha,\nu}(S_n) +\tau I_n\right)}{\dett\left(\theta_0 K_{\alpha_0,\nu}(S_n) +\tau_0 I_n\right)} \leq (1+\epsilon_{1\vee 2}) Ln^{-1} + n \epsilon_{1\vee 2}.
\end{align}
Finally, we combine \eqref{eq:denom.expo1.1}, \eqref{eq:denom.expo1.12}, \eqref{eq:denom.expo1.14} and \eqref{eq:denom.expo2.4} to obtain that with probability for all sufficiently large $n$, at least $1-2\exp(-\log^2 n)$,
\begin{align}
&\quad \inf_{\Bcal_0(\epsilon_1,\epsilon_2)\cap \Acal_{n}} \Bigg\{ -\frac{1}{2} \Big(\tr\left[\left\{\theta K_{\alpha,\nu}(S_n) + \tau I_n\right\}^{-1} \widetilde Y_n \widetilde Y_n^\T  \right]  \nonumber \\
&\qquad - \tr\left[\left\{\theta_0 K_{\alpha_0,\nu}(S_n) + \tau_0 I_n\right\}^{-1} \widetilde Y_n \widetilde Y_n^\T  \right] \Big) - \frac{1}{2} \log \frac{\dett\left\{\theta K_{\alpha,\nu}(S_n) +\tau I_n\right\}}{\dett\left\{\theta_0 K_{\alpha_0,\nu}(S_n) +\tau_0 I_n\right\}} \Bigg\} \nonumber \\
&\geq  -\frac{1}{2}\left( \frac{4Ln^{-1}}{1-\epsilon_{1 \vee 2}} +  \frac{2n\epsilon_{1 \vee 2}}{1-\epsilon_{1 \vee 2}} + (1+\epsilon_{1\vee 2} ) Ln^{-1} +  n \epsilon_{1\vee 2} \right) \nonumber \\
&\geq - \left(\frac{5}{2}n\epsilon_{1 \vee 2}  + 5Ln^{-1} \right),  \nonumber
\end{align}
where the last inequality follows since $\epsilon_1,\epsilon_2\in (0,1/2)$.
\end{proof}

\vspace{2mm}

\begin{lemma} \label{lem:denom.term34}
Suppose that Assumptions \ref{cond:model} and \ref{cond:spectral} hold. Let $\Dcal_{n} = \big\{(\theta,\alpha,\tau,\beta)\in \RR_+^{3} \times \RR^p:$ $\|\beta-\beta_0\|\leq n^{-3} \big\}$. There exists a large integer $N_4$ that $\nu,d,\theta_0,\alpha_0,L,r_0,\kappa$, such that for any $\epsilon_1,\epsilon_2\in (0,1/2)$, for all $n>N_4$, with probability at least $1-\exp(-\log^2 n)$,
\begin{align} \label{eq:denom.lower34}
& \inf_{\Bcal_0(\epsilon_1,\epsilon_2)\cap \Dcal_{n}} \Big\{ \widetilde Y_n^\T  \left[\left\{\theta K_{\alpha,\nu}(S_n) +\tau I_n\right\}^{-1}-\left\{\theta_0 K_{\alpha_0,\nu}(S_n) +\tau_0 I_n\right\}^{-1}\right] F_n (\beta-\beta_0) \nonumber \\
&\quad - \frac{1}{2} (\beta-\beta_0)^\T   F_n ^\T  \left[\left\{\theta K_{\alpha,\nu}(S_n) +\tau I_n\right\}^{-1}-\left\{\theta_0 K_{\alpha_0,\nu}(S_n) +\tau_0 I_n\right\}^{-1}\right] F_n (\beta-\beta_0) \Big\} \nonumber \\
& \geq - 8C_{\ff}p^{1/2} \left\{\theta_0 K_{\alpha_0,\nu}(0) + \tau_0 \right\}^{1/2} \tau_0^{-1} n^{-1}- 2C_{\ff}^2 p \tau_0^{-1} n^{-4} .
\end{align}
\end{lemma}

\begin{proof}[$\pof$ Lemma \ref{lem:denom.term34}]
We derive lower bounds for each of the two terms inside the infimum of \eqref{eq:denom.lower34}. On the event $\Dcal_{n}$, $\|\beta-\beta_0\|\leq n^{-3}$. By Assumption \ref{cond:model}, we have that
\begin{align}\label{eq:F.bound1}
& \| F_n (\beta-\beta_0)\|  \leq \|F_n\|_F \|\beta-\beta_0\| \leq C_{\ff} p^{1/2} n^{1/2} \|\beta-\beta_0\| \leq C_{\ff} p^{1/2} n^{-2}.
\end{align}
On the set $\Bcal_0(\epsilon_1,\epsilon_2)$, since both matrices $\theta K_{\alpha,\nu}(S_n) +\tau I_n$ and $\theta_0 K_{\alpha_0,\nu}(S_n) +\tau_0 I_n$ are positive definite, we apply the triangle inequality for the matrix operator norm and obtain that
\begin{align}\label{eq:omega.diff1}
&\quad ~ \left\|\{\theta K_{\alpha,\nu}(S_n) +\tau I_n\}^{-1} - \{\theta_0 K_{\alpha_0,\nu}(S_n) +\tau_0 I_n\}^{-1}\right\|_{\op} \nonumber \\
&\leq \left\|\{\theta K_{\alpha,\nu}(S_n) +\tau I_n\}^{-1}\right\|_{\op} + \left\|\{\theta_0 K_{\alpha_0,\nu}(S_n) +\tau_0 I_n\}^{-1}\right\|_{\op} \nonumber \\
&\leq \frac{1}{\tau} + \frac{1}{\tau_0} \leq \frac{2}{(1-\epsilon_2)\tau_0}.
\end{align}
Since $\widetilde Y_n \sim \Ncal(0,\Omega_0)$, similar to the proof of Lemma \ref{lem:denom.term12}, we let $V_0$ be an invertible matrix such that $V_0^\T  \Omega_0 V_0 = I_n$ and let $Z_n=V_0^\T  \widetilde Y_n$ such that $Z_n \sim \Ncal(0,I_n)$. Then by applying the Hanson-Wright inequality in Lemma \ref{lem:Hsuetal12}, we have that for all sufficiently large $n$,
\begin{align} \label{eq:Ynorm}
& \Pr \left(\|V_0^\T  \widetilde Y_n\| \geq \sqrt{2n} \right)  = \Pr \left( \|Z_n\|^2 \geq 2n \right) \nonumber \\
&\leq \Pr \left( Z_n^\T  Z_n \geq n+ 2\sqrt{n} \log n + 2\log^2 n\right) \leq \exp(-\log^2 n).
\end{align}
Notice that since $V_0^\T  \Omega_0 V_0 = I_n$, we have that $V_0^{-\T}V_0^{-1} = \Omega_0$ and hence
\begin{align} \label{eq:omega0.op}
\|V_0^{-1}\|_{\op}^2 & \leq \|\Omega_0\|_{\op} \leq \left\|\theta_0 K_{\alpha_0,\nu}(S_n) + \tau_0 I_n \right\|_{\op} \nonumber \\
&\leq \left\|\theta_0 K_{\alpha_0,\nu}(S_n) \right\|_{\op} + \|\tau_0 I_n\|_{\op} \leq \theta_0 \tr\left\{K_{\alpha_0,\nu}(S_n)\right\} + \tau_0  \nonumber \\
&\leq n\theta_0 K_{\alpha_0,\nu}(0)  +\tau_0,
\end{align}
where $K_{\alpha_0,\nu}(0) \in \RR_+$ is a constant since $K_{\alpha_0,\nu}(\bfs-\bft)$ is a covariance function for $\bfs,\bft\in [0,1]^d$.

We combine \eqref{eq:F.bound1}, \eqref{eq:omega.diff1}, \eqref{eq:Ynorm}, and \eqref{eq:omega0.op} to conclude that with probability at least $1-\exp(-\log^2 n)$, on the event $\Bcal_0(\epsilon_1,\epsilon_2)\cap \Dcal_{n}$,
\begin{align}\label{eq:term3.lower1}
&~ \quad \widetilde Y_n^\T  \left[\left\{\theta K_{\alpha,\nu}(S_n) +\tau I_n\right\}^{-1}-\left\{\theta_0 K_{\alpha_0,\nu}(S_n) +\tau_0 I_n\right\}^{-1}\right] F_n (\beta-\beta_0) \nonumber \\
&\geq - \|\widetilde Y_n^\T  V_0\|\cdot \|V_0^{-1}\|_{\op} \cdot \left\|\left\{\theta K_{\alpha,\nu}(S_n) +\tau I_n\right\}^{-1}-\left\{\theta_0 K_{\alpha_0,\nu}(S_n) +\tau_0 I_n\right\}^{-1}\right\|_{\op} \cdot \| F_n (\beta-\beta_0)\| \nonumber \\
&\geq - \left[2n \cdot \left\{n\theta_0 K_{\alpha_0,\nu}(0) +\tau_0\right\}\right]^{1/2} \cdot \frac{2}{(1-\epsilon_2)\tau_0} \cdot C_{\ff} p^{1/2} n^{-2} \nonumber \\
&\geq - 8C_{\ff}p^{1/2} \left\{\theta_0 K_{\alpha_0,\nu}(0) + \tau_0 \right\}^{1/2} \tau_0^{-1} n^{-1},
\end{align}
for all sufficiently large $n$, where in the last inequality we used $\epsilon_2<1/2$.

We can also derive from \eqref{eq:F.bound1} and \eqref{eq:omega.diff1} that on the event $\Bcal_0(\epsilon_1,\epsilon_2)\cap \Dcal_{n}$,
\begin{align}\label{eq:term4.lower1}
&~ \quad - \frac{1}{2} (\beta-\beta_0)^\T   F_n ^\T  \left[\left\{\theta K_{\alpha,\nu}(S_n) +\tau I_n\right\}^{-1}-\left\{\theta_0 K_{\alpha_0,\nu}(S_n) +\tau_0 I_n\right\}^{-1}\right] F_n (\beta-\beta_0) \nonumber \\
&\geq - \frac{1}{2} \left\|\left\{\theta K_{\alpha,\nu}(S_n) +\tau I_n\right\}^{-1}-\left\{\theta_0 K_{\alpha_0,\nu}(S_n) +\tau_0 I_n\right\}^{-1}\right\|_{\op} \cdot \| F_n (\beta-\beta_0)\|^2 \nonumber \\
&\geq - \frac{1}{2} \cdot \frac{2}{(1-\epsilon_2)\tau_0} \cdot \left(C_{\ff} p^{1/2} n^{-2}\right)^2 \nonumber \\
&\geq -\frac{2C_{\ff}^2 p}{\tau_0} n^{-4} ,
\end{align}
where in the last inequality we used $\epsilon_2<1/2$. The conclusion of Lemma \ref{lem:denom.term34} follows from \eqref{eq:term3.lower1} and \eqref{eq:term4.lower1}.
\end{proof}

\subsection{Other Technical Lemmas} \label{subsec:techlem}
\begin{lemma} \label{lem:posdef}
The following results hold for any $n\times n$ symmetric matrices $A_1,A_2$.
\begin{enumerate}
\item[(i)] If $A_1\preceq A_2$, then for any $n\times n$ matrix $B$, $B^\T  A_1 B  \preceq B^\T  A_2 B$;
\item[(ii)] If $A_1\preceq A_2$, then for any positive semidefinite matrix $B$, $\tr(BA_1) \leq \tr(BA_2)$ and $\tr(A_1B)\leq \tr(A_2B)$.
\end{enumerate}
\end{lemma}

\begin{proof}[$\pof$ Lemma \ref{lem:posdef}]
\noindent (i) If $A_1\preceq A_2$, then $A_2-A_1$ is positive semidefinite, which means that $B^\T  (A_2-A_1) B $ is positive semidefinite. Therefore, $B^\T  A_1 B  \preceq B^\T  A_1 B  + B^\T  (A_2-A_1) B = B^\T  A_2 B$.

\noindent (ii) If $A_1\preceq A_2$, then $A_2-A_1$ is positive semidefinite, and
\begin{align*}
\tr(BA_2)-\tr(BA_1) &= \tr\left\{B(A_2-A_1)\right\} = \tr\left\{B^{1/2}(A_2-A_1)B^{1/2}\right\} \geq 0,\\
\tr(A_2B)-\tr(A_1B) &= \tr\left\{(A_2-A_1)B\right\} = \tr\left\{B^{1/2}(A_2-A_1)B^{1/2}\right\} \geq 0,
\end{align*}
where $B^{1/2}$ denotes the symmetric square root matrix of a positive semidefinite matrix $B$.
\end{proof}

\begin{lemma}\label{lem:Hsuetal12}
(Hanson-Wright Inequality, Proposition 1.1 in \citealt{Hsuetal12}, Theorem 1.1 in \citealt{RudVer13}) Let $Z_1,\ldots,Z_n$ be i.i.d. $\mathcal{N}(0,1)$ random variables and $Z=(Z_1,\ldots,Z_n)^\T $. Let $\Sigma$ be a positive definite matrix. Then for any $z>0$,
\begin{align*}
&\pr \left\{Z^\T  \Sigma Z \geq \tr(\Sigma) + 2\big\{\tr(\Sigma^2)z\big\}^{1/2} + 2\|\Sigma\|_{\op} z \right\} \leq \ee^{-z}.
\end{align*}
Furthermore, there exists a universal absolute constant $C_{\HW}>0$ that does not depend on $\Sigma$, such that for any $\epsilon>0$,
\begin{align*}
&\pr \left\{ \left|Z^\T  \Sigma Z - \tr(\Sigma)\right| \geq \epsilon \right\} \leq 2 \exp\left(-C_{\HW}\min\left\{\frac{\epsilon^2}{\|\Sigma\|_{F}^2}, \frac{\epsilon}{\|\Sigma\|_{\op}}\right\}\right).
\end{align*}
\end{lemma}

\section{Posterior Inconsistency for $\alpha$ and $\beta$} \label{sec:inconsistency}
\begin{proposition}\label{thm:inconsistency}
Suppose that Assumptions \ref{cond:model} and \ref{cond:deriv} hold where the dimension of the domain $[0,1]^d$ is $d\in \{1,2,3\}$, and $\theta K_{\alpha,\nu}$ is the isotropic Mat\'ern covariance function in Example \ref{ex:Matern}. Suppose that we assign the uniform prior on the following set which covers the true parameter $(\theta_0,\alpha_0,\tau_0,\beta_0)$:
\begin{align*}
\Scal_{\Pi} & = \Big\{(\theta,\alpha,\tau,\beta): \theta \in [\theta_0/2,2\theta_0], \alpha \in [\alpha_0/2,2\alpha_0], \tau\in [\tau_0/2,2\tau_0],\\
& ~~~~ \beta_j\in [\beta_{0j}/2, 2\beta_{0j}],~ j=1,\ldots,p. \Big\}
\end{align*}
Then the posterior $\Pi(\cdot \mid Y_n, F_n )$ is inconsistent for $\alpha$ and $\beta$, i.e.,  there exist constants $\epsilon_0>0,\delta_0>0,\eta_0\in (0,1)$, such that for any $N\in \NN$, there exists $n>N$, such that
\begin{align*}
& \Pr\left(\Pi(|\alpha-\alpha_0|\geq \epsilon_0 |Y_n, F_n ) \geq \delta_0\right)\geq \eta_0, \\
\text{and } & \Pr\left(\Pi(\|\beta-\beta_0\|\geq \epsilon_0 |Y_n, F_n ) \geq \delta_0\right) \geq \eta_0.
\end{align*}
\end{proposition}

\begin{proof}[$\pof$ Proposition \ref{thm:inconsistency}]
We prove the posterior inconsistency for $\alpha$ and $\beta$ by contradiction. Suppose that the conclusion does not hold, i.e.,  for any $\epsilon>0,\delta>0,\eta\in (0,1)$, there exists a $n_0(\epsilon,\delta,\eta)\in \NN$, such that for all $n>n_0(\epsilon,\delta,\eta)$, $\Pi(|\alpha-\alpha_0|\geq \epsilon |Y_n, F_n ) < \delta$ and $\Pi(\|\beta-\beta_0\|\geq \epsilon |Y_n, F_n ) < \delta$ with $\PP_{(\theta_0,\alpha_0,\tau_0,\beta_0)}$-probability at least $1-\eta$. For a sufficiently small $\epsilon>0$, the uniform prior density $\pi(\theta,\alpha,\tau,\beta)$ is a constant in an open neighborhood of $(\theta_0,\alpha_0,\tau_0,\beta_0)$, denoted by
$$\Ucal(\epsilon)= \left \{(\theta,\alpha,\tau,\beta): |\theta-\theta_0| < \epsilon, |\alpha-\alpha_0|<\epsilon, |\tau-\tau_0|<\epsilon,\|\beta-\beta_0\|_{\infty} <\epsilon/\sqrt{p} \right\},$$
where $\|\beta\|_{\infty}=\max_{1\leq j\leq p}|\beta_j|$. Under Assumption \ref{cond:model}, the log-likelihood function $\Lcal_n(\theta,\alpha,\tau,\beta)$ is clearly finite for every $(\theta,\alpha,\tau,\beta)\in \Ucal(\epsilon)$. Therefore, the posterior distribution, provided exists, has a continuous density on $\Ucal(\epsilon)$. Now set $\delta=1/8$. Then with $\PP_{(\theta_0,\alpha_0,\tau_0,\beta_0)}$-probability at least $1-\eta$, the posterior median of $\alpha$, denoted by $\widetilde \alpha_n$, is in the interval $(\alpha_0-\epsilon,\alpha_0+\epsilon)$. It follows that for any $\epsilon>0,\eta>0$, for all $n>n_0(\epsilon,1/8,\eta)$, $|\widetilde\alpha_n-\alpha_0|<\epsilon$ with $\PP_{(\theta_0,\alpha_0,\tau_0,\beta_0)}$-probability at least $1-\eta$. Thus, $\widetilde\alpha_n$ is a consistent frequentist estimator for $\alpha_0$. Similarly, we can take the posterior median of each component $\beta_j$ ($j=1,\ldots,p$), denoted by $\widetilde \beta_{nj}$, and let $\widetilde\beta_n = (\widetilde \beta_{n1},\ldots,\widetilde \beta_{np})^\T $. Each $\widetilde \beta_{nj}$ is a consistent frequentist estimator for $\beta_{0j}$. As a result, $\widetilde\beta_n$ is a consistent frequentist estimator for $\beta_0$ given that $\big\|\widetilde\beta_n-\beta_0\big\|^2= \sum_{j=1}^p \big|\widetilde\beta_{nj}-\beta_{0j}\big|^2 < p\cdot (\epsilon^2/p) = \epsilon^2$. In this way, we have obtained two consistent frequentist estimators $\widetilde\alpha_n$ for $\alpha_0$ and $\widetilde\beta_n$ for $\beta_0$, both of which only depend on the data $(Y_n, F_n)$.

Now we show that such consistent frequentist estimators for $\alpha$ and $\beta$ should not exist. By \citet{Stein99a} and \citet{Zhang04}, the Gaussian measures induced by two Gaussian processes are either equivalent (absolutely continuous) or orthogonal to each other. For $\gp(m_0,\theta_0K_{\alpha_0,\nu})$ and $\gp(m_1,\theta_1K_{\alpha_1,\nu})$ with the same smoothness parameter $\nu$, the theory in Chapter 4 of \citet{Stein99a} (Section 4.2 and Corollary 5) and Theorem 2 of \citet{Zhang04} have shown that when $d=1,2,3$, their induced Gaussian measures are equivalent to each other if: (a) $\theta_0=\theta_1$; (b) Both mean functions $m_i$ lie in the reproducing kernel Hilbert space of $\theta_iK_{\alpha_i,\nu}$ for $i=0,1$. Let $L_2([0,1]^d)$ be the space of square integrable functions on $[0,1]^d$ with the norm $\|\cdot\|_2$). Then the Sobolev space of order $k$ for any $k>0$, denoted by $\Wcal_2^k([0,1]^d)$, is $\Wcal_2^k([0,1]^d) = \Big\{f\in L_2([0,1]^d): \|f\|^2_{\Wcal_2^k} = \sum_{\sj\in \NN^d: |\sj|\leq k} \left\|\sD^{\sj}f\right\|_2^2 <\infty\Big\}$. By Assumption \ref{cond:deriv}, we have that $m(\cdot) = \beta^\T \ff (\cdot)$ belongs to $\Wcal_2^{\nu+d/2}([0,1]^d)$ for any $\beta\in \RR^p$, and hence belongs to the reproducing kernel Hilbert space of $\theta_iK_{\alpha_i,\nu}$ for $i=0,1$. By Corollary 10.48 of \citet{Wen05}, the reproducing kernel Hilbert space of the isotropic Mat\'ern covariance function $\theta_iK_{\alpha_i,\nu}$ for $i=0,1$ are norm equivalent to the Sobolev space of order $\nu+d/2$.

We choose $\theta_0=\theta_1$, $\alpha_0\neq \alpha_1$, $\beta_0\neq \beta_1$ and let $m_i(\cdot ) = \beta_i^\T \ff(\cdot)$ for $i=0,1$. Then both the conditions (a) and (b) above hold true, such that the two Gaussian measures $\gp(m_0,\theta_0 K_{\alpha_0,\nu})$ and $\gp(m_1,\theta_1K_{\alpha_1,\nu})$ are equivalent. Using similar argument to Corollary 1 of \citet{Zhang04}, if $\widetilde\alpha_n$ and $\widetilde\beta_n$ are consistent for $\alpha_0$ and $\beta_0$, then we can always find almost-surely convergence subsequences $\{\widetilde\alpha_{n_k}\}$ and $\{\widetilde\beta_{n_j}\}$, such that $$\PP_{(\theta_0,\alpha_0,\tau_0,\beta_0)}\left(\lim_{k\to\infty}\widetilde\alpha_{n_k}= \alpha_0\right)=1, \text{ and } \PP_{(\theta_0,\alpha_0,\tau_0,\beta_0)}\left(\lim_{j\to\infty}\widetilde\beta_{n_j}= \beta_0\right)=1 .$$
But given the equivalence of the two Gaussian measures, this implies that $$\PP_{(\theta_1,\alpha_1,\tau_1,\beta_1)}\left(\lim_{k\to\infty}\widetilde\alpha_{n_k}= \alpha_0\right)=1 \text{ and } \PP_{(\theta_1,\alpha_1,\tau_1,\beta_1)}\left(\lim_{j\to\infty}\widetilde\beta_{n_j}= \beta_0\right)=1.$$
However, under $\PP_{(\theta_1,\alpha_1,\tau_1,\beta_1)}$, the limits of $\{\widetilde\alpha_{n_k}\}$ and $\{\widetilde\beta_{n_j}\}$ should be $\alpha_1$ and $\beta_1$, respectively, which are different from $\alpha_0$ and $\beta_0$. Therefore, this is a contradiction, and such consistent estimators $\widetilde\alpha_n $ and $\widetilde\beta_n$ cannot exist. This further implies that the posterior $\Pi(\cdot \mid Y_n, F_n )$ is inconsistent for $\alpha$ and $\beta$.
\end{proof}

\section{Proof of Proposition \ref{prop:4cov}} \label{sec:proof.prop1}
\begin{proof}[$\pof$ Proposition \ref{prop:4cov}]
We check (i) and (ii) in Assumption \ref{cond:spectral} for the four covariance functions in Examples \ref{ex:Matern}--\ref{ex:CH}, respectively.
\vspace{2mm}

\noindent \underline{\sc Checking (i) of Assumption \ref{cond:spectral}:}
\vspace{2mm}

Since for all $\alpha$ that satisfies $|\alpha/\alpha_0 -1|\leq r_0$,
\begin{align} \label{eq:assump1.deriv}
&\quad~ \sup_{w\in \RR^d} \left|f_{\theta_0,\alpha,\nu}(w)/f_{\theta_0,\alpha_0,\nu}(w)-1\right| \nonumber \\
& \leq  \sup_{w\in \RR^d} \frac{\left|f_{\theta_0,\alpha,\nu}(w) - f_{\theta_0,\alpha_0,\nu}(w)\right|}{ f_{\theta_0,\alpha_0,\nu}(w)}  \leq  \left\{\sup_{|\alpha/\alpha_0 -1|\leq r_0} \sup_{w\in \RR^d} \left\|\frac{\frac{\partial f_{\theta_0,\alpha,\nu}(w)}{\partial \alpha}}{f_{\theta_0,\alpha_0,\nu}(w)} \right\|\right\} \cdot \alpha_0\left|\frac{\alpha}{\alpha_0}-1\right|,
\end{align}
we can see that it suffices to show that the double supremum in the first term is a continuous function on the set $\{\alpha:|\alpha/\alpha_0 -1|\leq r_0\}$ for some $r_0\in (0,1/2)$ and is upper bounded by a constant that may depend on $\nu,d,\theta_0,\alpha_0$. We can simply set $r_0=1/4$ and $\kappa=1$.
\vspace{2mm}

\noindent (i) The isotropic Mat\'ern covariance function has the following spectral density (\citealt{Stein99a}) for any $(\theta,\alpha,\nu)\in \RR_+^{3}$:
\begin{align} \label{eq:Matern.spectral}
& f_{\theta,\alpha,\nu}(w) = \frac{\Gamma(\nu+d/2)}{\Gamma(\nu)\pi^{d/2}} \frac{\theta}{(\alpha^2+\|w\|^2)^{\nu+d/2}}.
\end{align}
Therefore,
\begin{align}
& \left|\frac{\frac{\partial f_{\theta_0,\alpha,\nu}(w)}{\partial \alpha}}{f_{\theta_0,\alpha_0,\nu}(w)}\right| = \left| \frac{\frac{\Gamma(\nu+d/2)}{\Gamma(\nu)\pi^{d/2}} \frac{- \theta\cdot 2(\nu+d/2) \alpha}{(\alpha^2+\|w\|^2)^{\nu+d/2+1}}} {\frac{\Gamma(\nu+d/2)}{\Gamma(\nu)\pi^{d/2}} \frac{\theta}{(\alpha^2+\|w\|^2)^{\nu+d/2}}} \right| = \frac{2(\nu+d/2)\alpha}{\alpha^2+\|w\|^2}. \nonumber
\end{align}
This implies that
\begin{align}
& \sup_{|\alpha/\alpha_0 -1|\leq r_0} \sup_{w\in \RR^d} \left\|\frac{\frac{\partial f_{\theta_0,\alpha,\nu}(w)}{\partial \alpha}}{f_{\theta_0,\alpha_0,\nu}(w)} \right\| \leq \sup_{|\alpha/\alpha_0 -1|\leq 1/4} \frac{2(\nu+d/2)}{\alpha} \leq \frac{4(2\nu+d)}{3\alpha_0}, \nonumber
\end{align}
which is a constant that depends only on $\alpha_0,\nu,d$. Thus the double supremum in \eqref{eq:assump1.deriv} is upper bounded by a constant for the isotropic Mat\'ern covariance function in Example \ref{ex:Matern}.
\vspace{2mm}

\noindent (ii) The tapered isotropic Mat\'ern covariance function (\citealt{Kauetal08}) has the following spectral density for any $(\theta,\alpha,\nu)\in \RR_+^{3}$:
\begin{align} \label{eq:taper.spectral}
& f_{\theta,\alpha,\nu}(w) = f_{\tap}(w)\cdot  \frac{\Gamma(\nu+d/2)}{\Gamma(\nu)\pi^{d/2}} \frac{\theta}{(\alpha^2+\|w\|^2)^{\nu+d/2}},
\end{align}
where $f_{\tap}$ satisfies $0< f_{\tap}(w)\leq C_{\tap} (1+\|w\|^2)^{-(\nu+d/2+\eta_{\tap})}$ for some constants $C_{\tap}>0,\eta_{\tap}>0$ and all $w\in \RR^d$. The function $f_{\tap}(w)$ only depends on $\nu,d$ but not $\alpha$ and $\theta$. Therefore,
\begin{align}
& \left|\frac{\frac{\partial f_{\theta_0,\alpha,\nu}(w)}{\partial \alpha}}{f_{\theta_0,\alpha_0,\nu}(w)}\right| = \frac{2(\nu+d/2)\alpha}{\alpha^2+\|w\|^2}, \nonumber
\end{align}
and hence \begin{align}
& \sup_{|\alpha/\alpha_0 -1|\leq r_0} \sup_{w\in \RR^d} \left\|\frac{\frac{\partial f_{\theta_0,\alpha,\nu}(w)}{\partial \alpha}}{f_{\theta_0,\alpha_0,\nu}(w)} \right\| \leq \sup_{|\alpha/\alpha_0 -1|\leq 1/4} \frac{2(\nu+d/2)}{\alpha} \leq \frac{4(2\nu+d)}{3\alpha_0}, \nonumber
\end{align}
which is a constant that depends only on $\alpha_0,\nu,d$. Thus the double supremum in \eqref{eq:assump1.deriv} is upper bounded by a constant for the tapered isotropic Mat\'ern covariance function in Example \ref{ex:taper}.

\vspace{2mm}

\noindent (iii) The isotropic generalized Wendland covariance function has the following spectral density for any $\mu > \nu + d$, $\nu \geq 1/2$, and $(\theta,\alpha)\in \RR_+^2$ (Theorem 1 of \citealt{Bevetal19}):
\begin{align} \label{eq:GW.spectral}
f_{\theta,\alpha,\nu}(w)& = \theta\cdot \frac{\Gamma(\mu+1)\Gamma(2\nu+d-1)}{2^{d+\nu-3/2}\pi^{d/2}\Gamma(\nu+(d-1)/2)\Gamma(\mu+2\nu+d) } \cdot \frac{\Gamma(\nu-1/2)}{2^{3/2-\nu}B(2\nu-1,\mu+1)} \nonumber \\
&\quad \cdot \alpha^{-(2\nu+d)} {}_1F_2\left(\nu+\frac{d}{2};\nu+\frac{d+\mu}{2},\nu+\frac{d+\mu+1}{2};-\frac{\|w\|^2}{4\alpha^2}\right) ,
\end{align}
where $\Gamma(\nu-1/2)/\big\{2^{3/2-\nu}B(2\nu-1,\mu+1)\big\}$ is defined to be 1 if $\nu=1/2$, and for any $a,b,c\in \RR_+$,
$${}_1F_2(a;b,c;z) = \sum_{k=0}^{\infty} \frac{(a)_k z^k}{(b)_k(c)_kk!}, \quad \text{for } z\in \RR, $$
with $(x)_k=\Gamma(x+k)/\Gamma(x)$ for any $x>0$ being the Pochhammer symbol. The derivative of ${}_1F_2(a;b,c;z)$ is
$$\frac{\ud}{\ud z}{}_1F_2(a;b,c;z) = \frac{a}{bc}{}_1F_2(a+1;b+1,c+1;z).$$
Therefore, we can calculate that
\begin{align*}
\frac{\frac{\partial f_{\theta_0,\alpha,\nu}(w)}{\partial \alpha}}{f_{\theta_0,\alpha_0,\nu}(w)} &=
\frac{ -(2\nu+d)\alpha^{-(2\nu+d+1)} {}_1F_2\left(\nu+\frac{d}{2};\nu+\frac{d+\mu}{2},\nu+\frac{d+\mu+1}{2};-\frac{\|w\|^2}{4\alpha^2}\right)} { \alpha^{-(2\nu+d)} {}_1F_2\left(\nu+\frac{d}{2};\nu+\frac{d+\mu}{2},\nu+\frac{d+\mu+1}{2};-\frac{\|w\|^2}{4\alpha^2}\right)} \\
&\quad+ \frac{\|w\|^2\alpha^{-(2\nu+d+3)} \frac{2\nu+d}{(2\nu+d+\mu)(2\nu+d+\mu+1)} {}_1F_2\left(\nu+\frac{d}{2}+1;\nu+\frac{d+\mu}{2}+1,\nu+\frac{d+\mu+3}{2};-\frac{\|w\|^2}{4\alpha^2}\right)}{\alpha^{-(2\nu+d)} {}_1F_2\left(\nu+\frac{d}{2};\nu+\frac{d+\mu}{2},\nu+\frac{d+\mu+1}{2};-\frac{\|w\|^2}{4\alpha^2}\right)} .
\end{align*}
Part (iii) in Theorems 1 and 2 of \citet{Bevetal19} have shown that as $\|w\|\to\infty$, $f_{\theta,\alpha,\nu}(w) \asymp \|w\|^{-(2\nu+d)}$ for a given set of $\theta,\alpha,\nu,\mu$. Furthermore, the function ${}_1F_2(a;b,c;z)$ is strictly positive and takes finite values for all $z\in [0,+\infty)$. Therefore, there exist $0< c_1(\mu,\nu,d,\alpha) < C_1(\mu,\nu,d,\alpha) <\infty$ and $0< c_2(\mu,\nu,d,\alpha) < C_2(\mu,\nu,d,\alpha) <\infty$ which are all continuous functions in $\alpha$, such that
\begin{align*}
& c_1(\mu,\nu,d,\alpha)\|w\|^{-(2\nu+d+2)} \leq {}_1F_2\left(\nu+\frac{d}{2}+1;\nu+\frac{d+\mu}{2}+1,\nu+\frac{d+\mu+3}{2};-\frac{\|w\|^2}{4\alpha^2}\right)\\
& \leq C_1(\mu,\nu,d,\alpha)\|w\|^{-(2\nu+d+2)}, \\
& c_2(\mu,\nu,d,\alpha)\|w\|^{-(2\nu+d)} \leq {}_1F_2\left(\nu+\frac{d}{2};\nu+\frac{d+\mu}{2},\nu+\frac{d+\mu+1}{2};-\frac{\|w\|^2}{4\alpha^2}\right)\\
& \leq C_2(\mu,\nu,d,\alpha)\|w\|^{-(2\nu+d)}.
\end{align*}
Therefore, for all $\alpha$ that satisfies $|\alpha/\alpha_0-1|\leq 1/4$, for all $w\in \RR^d$,
\begin{align*}
\left| \frac{\frac{\partial f_{\theta_0,\alpha,\nu}(w)}{\partial \alpha}}{f_{\theta_0,\alpha_0,\nu}(w)}  \right|
&\leq \frac{4(2\nu+d)}{3\alpha_0} + \frac{64\sup_{|\alpha/\alpha_0-1|\leq 1/4}C_1(\mu,\nu,d,\alpha)}{27\alpha_0^3 \inf_{|\alpha/\alpha_0-1|\leq 1/4}c_2(\mu,\nu,d,\alpha)},
\end{align*}
which is a finite number that depends on $\mu,\nu,d,\alpha_0$. Thus the double supremum in \eqref{eq:assump1.deriv} is upper bounded by a constant for the isotropic generalized Wendland covariance function in Example \ref{ex:GW}.
\vspace{2mm}

\noindent (iv) The isotropic confluent hypergeometric covariance function (\citealt{MaBha22}) has the following spectral density for any $(\theta,\alpha,\nu,\mu)\in \RR_+^4$:
\begin{align} \label{eq:CH.spectral}
f_{\theta,\alpha,\nu,\mu}(w)& = \frac{\theta\cdot 2^{\nu-\mu}\nu^{\nu}}{\Gamma(\nu+\mu)\pi^{d/2} \alpha^{2(\nu+\mu)}} \int_0^{\infty} \left(2\nu /t + \|w\|^2 \right)^{-(\nu+d/2)} t^{-(\nu+\mu+1)}\exp\left\{-1/(2\alpha^2 t)\right\} \ud t.
\end{align}
We highlight the dependence on $\mu$ in the spectral density for convenience of expressions below, though we have assumed that $\mu$ is known. The proof of Proposition 1 in \citet{MaBha22} has shown that
\begin{align} \label{eq:MaBha.prop1}
\lim_{\|w\|\to\infty} \frac{f_{\theta,\alpha,\nu,\mu}(w)}{\frac{2^{2\nu}\nu^{\nu}}{\pi^{d/2}} \theta \|w\|^{-(2\nu+d)} \left(\frac{2\nu\alpha^2\|w\|^2}{2\nu\alpha^2\|w\|^2+1}\right)^{\nu+d/2}} = 1.
\end{align}
For the derivative with respect to $\alpha$, we can calculate from \eqref{eq:CH.spectral} that
\begin{align} \label{eq:CH.deriv1}
\left| \frac{\frac{\partial f_{\theta_0,\alpha,\nu,\mu}(w)}{\partial \alpha}}{f_{\theta_0,\alpha_0,\nu,\mu}(w)}  \right|
& \leq \frac{2(\nu+\mu)}{\alpha} + \frac{2(\nu+\mu+1)}{\alpha}\cdot \frac{f_{\theta_0,\alpha,\nu,\mu+1}(w)}{f_{\theta_0,\alpha,\nu,\mu}(w)} .
\end{align}
From \eqref{eq:CH.spectral}, we clearly have that $0 < f_{\theta,\alpha,\nu,\mu}(w) <\infty $ for all $w\in \RR^d$ and $f_{\theta,\alpha,\nu,\mu}(w)$ is always a continuous function in $\|w\|$ and $\alpha$. Therefore, from \eqref{eq:MaBha.prop1}, there exists a function $0<C_1(\alpha,\theta_0,\nu,\mu)<\infty$ continuous in $\alpha$, such that $\sup_{w\in\RR^d }f_{\theta_0,\alpha,\nu,\mu+1}(w)/f_{\theta_0,\alpha,\nu,\mu}(w) \leq C_1(\alpha,\theta_0,\nu,\mu)$. Hence, we obtain from \eqref{eq:CH.deriv1} that for all $\alpha$ that satisfies $|\alpha/\alpha_0-1|\leq 1/4$, for all $w\in \RR^d$,
\begin{align}
\left| \frac{\frac{\partial f_{\theta_0,\alpha,\nu,\mu}(w)}{\partial \alpha}}{f_{\theta_0,\alpha_0,\nu,\mu}(w)}  \right|
& \leq \frac{8(\nu+\mu)}{3\alpha_0} + \frac{8(\nu+\mu+1)}{3\alpha_0}\cdot \sup_{|\alpha/\alpha_0-1|\leq 1/4}C_1(\alpha,\theta_0,\nu,\mu) , \nonumber
\end{align}
which is a finite number that depends on $\mu,\nu,d,\alpha_0$. Thus the double supremum in \eqref{eq:assump1.deriv} is upper bounded by a constant for the isotropic confluent hypergeometric covariance function in Example \ref{ex:CH}.
\vspace{3mm}

\noindent \underline{\sc Checking (ii) of Assumption \ref{cond:spectral}:}
\vspace{2mm}

From \eqref{eq:Matern.spectral}, the spectral density of isotropic Mat\'ern covariance function is clearly a decreasing function in $\alpha$. From \eqref{eq:taper.spectral}, the spectral density of the tapered isotropic Mat\'ern is also decreasing in $\alpha$ since $f_{\tap}$ does not depend on $\alpha$. The proof of Lemma 1 in \citet{Bevetal19} has shown that under our condition $\mu>\nu + d$, the spectral density of generalized Wendland in \eqref{eq:GW.spectral} is a decreasing function in $\alpha$. The proof of Lemma 4 in \citet{MaBha22} has shown that for fixed $\theta,\nu,\mu$, the spectral density of confluent hypergeometric covariance function in \eqref{eq:CH.spectral} is a decreasing function in $\alpha$. This completes the proof of Proposition \ref{prop:4cov}.
\end{proof}

\section{Proof of Theorem \ref{thm:post.rate2}} \label{sec:proof.post.rate2}
\begin{proof}[$\pof$ Theorem \ref{thm:post.rate2}]
We first recall from Theorem \ref{thm:E1n.rate} that for the higher-order quadratic variation estimators $\widehat\theta_n$ and $\widehat\tau_n$, they satisfy Assumption \ref{cond:exptest1} with $b_1$ and $b_2$ defined by
\begin{align} \label{eq:b12.recall}
b_1& = \min\Bigg\{\frac{1}{2}  - \frac{2(1-\gamma)\nu}{d}- \rho_{22} , ~ \frac{1-\gamma}{2} - \varsigma , ~ \frac{1}{4} + \frac{(1-\gamma)(\ell_{\star}-2\nu)}{d} - \frac{\rho_{1} + \rho_{22}}{2}, \nonumber \\
& \qquad \qquad \frac{1-\gamma}{4} + \frac{(1-\gamma)(\ell_{\star}-\nu)}{d} - \frac{\rho_{1}}{2} - \varsigma \Bigg\}  ,  \nonumber \\
b_2& =  \min\Bigg\{\frac{1}{2} , ~ \frac{(1-\gamma)(4\nu+d)}{2d}  - \rho_{21} - \varsigma , ~ \frac{1}{4} +\frac{ (1-\gamma)\ell_{\star}}{d} - \frac{\rho_{1} + \rho_{21}}{2}, ~ \nonumber \\
& \qquad \qquad \frac{(1-\gamma)(4\nu+d +4\ell_{\star})}{4d} - \frac{\rho_{1} + 2\rho_{21}}{2} - \varsigma  \Bigg\} .
\end{align}
And recall that $\gamma$ needs to satisfy $\max\big\{1-d/(4\nu),0\big\} < \gamma < 1$.

In the expression of $b_1$ above, we first consider the third and fourth terms inside the minimum. Given that $\ell_{\star} = \lceil \nu + d/2 \rceil \geq \nu+d/2$, the fourth term satisfies
\begin{align*}
& \frac{1-\gamma}{4} + \frac{(1-\gamma)(\ell_{\star}-\nu)}{d} - \frac{\rho_{1}}{2} - \varsigma \geq \frac{1-\gamma}{4} + \frac{1-\gamma}{2} - \frac{\rho_{1}}{2} - \varsigma.
\end{align*}
Therefore, when $\rho_1$ and $\varsigma$ are sufficiently small, the fourth term is lower bounded by the second term inside the minimum in the expression of $b_1$.

For the third term inside the minimum in the expression of $b_1$, we have
\begin{align*}
& \frac{1}{4} + \frac{(1-\gamma)(\ell_{\star}-2\nu)}{d} - \frac{\rho_{1} + \rho_{22}}{2} \geq \frac{1}{4} + \frac{(1-\gamma)(\nu+d/2-2\nu)}{d} - \frac{\rho_{1} + \rho_{22}}{2} \\
&= \frac{1}{2} \left[\frac{1}{2}  - \frac{2(1-\gamma)\nu}{d} \right] + \frac{1-\gamma}{2} - \frac{\rho_{1} + \rho_{22}}{2} .
\end{align*}
In other words, when $\rho_1,\rho_{22},\varsigma$ are sufficiently small, the third term is approximately a linear combination of the first and second terms inside the minimum in the expression of $b_1$. As a result, it is sufficient to balance the first and second terms inside the minimum in the expression of $b_1$ in \eqref{eq:b12.recall}. We set them equal and obtain that $\gamma = 4\nu/(4\nu+d)$. Then we have from \eqref{eq:b12.recall} that when $\gamma = 4\nu/(4\nu+d)$,
\begin{align} \label{eq:b1.lower}
b_1 & \geq \min\Bigg\{\frac{1}{2}  - \frac{2(1-\gamma)\nu}{d} - \rho_{22} , ~ \frac{1-\gamma}{2}- \varsigma  , ~ \frac{1}{2} \left[\frac{1}{2}  - \frac{2(1-\gamma)\nu}{d} \right] + \frac{1-\gamma}{2} - \frac{\rho_{1} + \rho_{22}}{2}, \nonumber \\
& \qquad \qquad \frac{3(1-\gamma)}{4} - \frac{\rho_{1}}{2} - \varsigma   \Bigg\}   \nonumber  \\
& = \min\Bigg\{\frac{d}{2(4\nu+d)}- \rho_{22} , ~~ \frac{d}{2(4\nu+d)} - \varsigma , ~~ \frac{3d}{4(4\nu+d)} - \frac{\rho_{1} + \rho_{22}}{2}, \nonumber \\
& \qquad \qquad \frac{3d}{4(4\nu+d)} - \frac{\rho_{1}}{2} - \varsigma  \Bigg\}  , \nonumber  \\
&\geq \frac{1}{2(4\nu/d+1)} - \varrho,
\end{align}
where $\varrho  =  \max(\rho_{22}, \rho_1/2)$ since $\varsigma>0$ can be arbitrarily small. As such, the condition in \eqref{eq:rho.relation1} becomes
\begin{align}
& 0< \rho_1 < \frac{d}{4\nu+d},  \quad 0<\rho_{21} < \frac{2\nu}{4\nu+d}, \quad 0<\rho_{22} < \frac{d}{2(4\nu+d)}, \quad \rho_{31}>0, \quad 0<\rho_{32} < \frac{1}{4\nu+d}, \nonumber
\end{align}
which are satisfied by \eqref{eq:rho.relation2} of Theorem \ref{thm:post.rate2}.

In the expression of $b_2$ above, since $0<1-\gamma < \min\{1,d/(4\nu)\} $, we can choose $1-\gamma$ sufficiently close to its upper bound, say $1-\gamma = \min\{1,d/(4\nu)\}-\eta$ for some small $\eta \in (0,\min\{1,d/(4\nu)\})$, such that
\begin{align}
&\quad~\frac{(1-\gamma)(4\nu+d)}{2d}  - \rho_{21}  - \varsigma \nonumber \\
&= \frac{\min\{1,d/(4\nu)\}(4\nu+d)}{2d} - \frac{\eta(4\nu+d)}{2d} - \rho_{21}  - \varsigma  \nonumber  \\
&= \frac{1}{2} + \min\left\{\frac{2\nu}{d}, \frac{d}{8\nu} \right\} -  \frac{\eta(4\nu+d)}{2d}  - \rho_{21} - \varsigma , \label{eq:b2.rho1} \\
&\quad~ \frac{1}{4} +\frac{ (1-\gamma)\ell_{\star}}{d} - \frac{\rho_{1} + \rho_{21}}{2}  \nonumber \\
&\geq \frac{1}{4} +\frac{(\nu+d/2) \min\{1,d/(4\nu)\}}{d} - \frac{\eta (\nu+d/2)}{d} - \frac{\rho_{1} + \rho_{21}}{2}  \nonumber \\
&\geq \frac{1}{2} + \min\left\{\frac{1}{4}+\frac{\nu}{d}, \frac{d}{8\nu}\right\}  - \frac{\eta (\nu+d/2)}{d} - \frac{\rho_{1} + \rho_{21}}{2} , \label{eq:b2.rho2} \\
&\quad~ \frac{(1-\gamma)(4\nu+d +4\ell_{\star})}{4d} - \frac{\rho_{1} + 2\rho_{21} }{2}  - \varsigma  \nonumber \\
&\geq \frac{(8\nu+3d)\min\{1,d/(4\nu)\}}{4d} - \frac{\eta(8\nu+3d) }{4d} - \frac{\rho_{1} + 2\rho_{21} }{2}  - \varsigma  \nonumber \\
&= \frac{1}{2} + \min\left\{\frac{1}{4}+\frac{2\nu}{d}, \frac{3d}{16\nu}\right\}  - \frac{\eta(8\nu+3d) }{4d}  - \frac{\rho_{1} + 2\rho_{21} }{2} - \varsigma . \label{eq:b2.rho3}
\end{align}
In order to make the right-hand sides of \eqref{eq:b2.rho1}, \eqref{eq:b2.rho2}, and \eqref{eq:b2.rho3} strictly larger than $1/2$, it is sufficient to have
\begin{align} \label{eq:rho.b2.1}
&  \rho_{21} < \min\left\{\frac{2\nu}{d}, \frac{d}{8\nu} \right\} , ~~ \frac{\rho_{1} + \rho_{21}}{2} < \min\left\{\frac{1}{4}+\frac{\nu}{d}, \frac{d}{8\nu}\right\} , \nonumber \\
& \frac{\rho_{1} + 2\rho_{21}}{2} < \min\left\{\frac{1}{4}+\frac{2\nu}{d}, \frac{3d}{16\nu}\right\},
\end{align}
since both $\eta>0$ and $\varsigma>0$ can be chosen as arbitrarily small. We notice that the third relations in \eqref{eq:rho.b2.1} is implied by adding up the first two relations in \eqref{eq:rho.b2.1}, which are included in \eqref{eq:rho.relation2} of Theorem \ref{thm:post.rate2}. Therefore, with such choice of $\gamma$, we can set $b_2=1/2$ in \eqref{eq:b12.recall}. Finally, the conclusion of Theorem \ref{thm:post.rate2} follows by combining the lower bound of $b_1$ in \eqref{eq:b1.lower} and $b_2=1/2$ together with the posterior convergence in Theorem \ref{thm:convergence}.
\end{proof}

\section{Proof of Proposition \ref{prop:prior}} \label{sec:proof.prop:prior}
\begin{proof}[$\pof$ Proposition \ref{prop:prior}]

First, we verify Assumption \ref{cond:prior1}. Clearly the joint prior density $\pi(\theta,\alpha,\tau,\beta)$ is continuous everywhere and satisfies $\pi(\theta_0,\alpha_0,\tau_0,\beta_0)>0$. Furthermore, since $\tau\sim \textup{IG}(a_2,b_2)$, we have that for any $n\in \ZZ_+$,
\begin{align*}
\int_0^{\infty} \tau^{-n/2}\pi(\tau) \ud \tau & = \frac{b_2^{a_2}}{\Gamma(a_2)} \int_0^{\infty} \tau^{-n/2} \cdot \tau^{-(a_2+1)} \exp(-b_2/\tau) \ud \tau = \frac{\Gamma(a_2+n/2)}{b_2^{n/2} \Gamma(a_2)} <\infty.
\end{align*}
Therefore Assumption \ref{cond:prior1} is satisfied.

Next, we verify Assumption \ref{cond:prior2}. For the isotropic Mat\'ern covariance function, Assumption \ref{cond:spectral} is satisfied with $\kappa=1$. We first decompose the set $\Ecal_{n}^c$ based on Equation \ref{eq:E1n} into several sets:
\begin{align} \label{eq:Ecal.complement}
\Ecal_{n}^c & \subseteq \{\|\beta\|^2/\theta > n^{\rho_1}\} \cup \{ \tau/\theta < n^{-\rho_{21}}\} \cup \{\tau/\theta > n^{\rho_{22}} \}  \cup \{\alpha < n^{-\rho_{31}}\} \cup \{ \alpha > n^{\rho_{32}}\}.
\end{align}
We show each of the set on the right-hand side of \eqref{eq:Ecal.complement} satisfies Assumption \ref{cond:prior2} with the prior specified in Proposition \ref{prop:prior}. Recall that for the isotropic Mat\'ern in Example \ref{ex:Matern}, we have $\kappa=1$ from the proof of Proposition \ref{prop:4cov}.

First, given $\beta\sim \Ncal(0,a_0I_p)$, $\theta\sim \text{IG}(a_1,b_1)$, we have that $\|\beta\|^2/a_0 \sim \chi^2_{p}$, and hence for sufficiently large $n$,
\begin{align} \label{eq:Ecal.c.1}
\Pi(\|\beta\|^2/\theta > n^{\rho_1}) &\leq \int_0^{\infty} \frac{b_1^{a_1}}{\Gamma(a_1)}\theta^{-(a_1+1)} \ee^{-b_1/\theta} \int_{n^{\rho_1} \theta/a_0}^{\infty} \frac{1}{2^p\Gamma(p/2)} z^{p/2-1} \ee^{-z/2} \ud z \ud \theta \nonumber \\
&\leq \int_0^{n^{-\rho_1/2}} \frac{b_1^{a_1}}{\Gamma(a_1)}\theta^{-(a_1+1)} \ee^{-b_1/\theta} \int_0^{\infty} \frac{1}{2^p\Gamma(p/2)} z^{p/2-1} \ee^{-z/2} \ud z \ud \theta  \nonumber \\
&\quad + \int_{n^{-\rho_1/2}}^{\infty} \frac{b_1^{a_1}}{\Gamma(a_1)}\theta^{-(a_1+1)} \ee^{-b_1/\theta} \int_{n^{\rho_1/2}/a_0}^{\infty} \frac{1}{2^p\Gamma(p/2)} z^{p/2-1} \ee^{-z/2} \ud z \ud \theta  \nonumber \\
&\stackrel{(i)}{\leq} \int_{n^{\rho_1/2}}^{\infty}  \frac{b_1^{a_1}}{\Gamma(a_1)} t^{a_1-1} \ee^{-b_1t} \ud t  + \int_{n^{\rho_1/2}/a_0}^{\infty} \frac{1}{2^p\Gamma(p/2)} z^{p/2-1} \ee^{-z/2} \ud z  \nonumber \\
&\stackrel{(ii)}{<} \int_{n^{\rho_1/2}}^{\infty}  \frac{b_1^{a_1}}{\Gamma(a_1)} \ee^{-b_1t/2 } \ud t  + \int_{n^{\rho_1/2}/a_0}^{\infty} \frac{1}{2^p\Gamma(p/2)}  \ee^{-z/4} \ud z  \nonumber \\
&= \frac{2b_1^{a_1-1}}{\Gamma(a_1)} \exp\big(-b_1n^{\rho_1/2}/2\big) + \frac{1}{2^{p-2}\Gamma(p/2)} \exp\big\{-n^{\rho_1/2}/(4a_0)\big\}  \nonumber \\
&\prec n^{-(3p+6)},
\end{align}
where for (i) we use a change of variable $t=1/\theta$, and (ii) follows when $n$ is sufficiently large such that $z^{p/2-1}<\ee^{z/4}$ for all $z\geq n^{\rho_1/2}/a_0$.

Second, given $\tau\sim \text{IG}(a_2,b_2)$, $\theta\sim \text{IG}(a_1,b_1)$, we have that
\begin{align} \label{eq:Ecal.c.2}
\Pi(\tau/\theta < n^{-\rho_{21}}) &\leq \int_0^{\infty} \frac{b_1^{a_1}}{\Gamma(a_1)}\theta^{-(a_1+1)} \ee^{-b_1/\theta} \int_0^{n^{-\rho_{21}} \theta} \frac{b_2^{a_2}}{\Gamma(a_2)} \tau^{-(a_2+1)} \ee^{-b_2/\tau} \ud \tau \ud \theta \nonumber \\
&\leq \int_0^{n^{\rho_{21}/2}} \frac{b_1^{a_1}}{\Gamma(a_1)}\theta^{-(a_1+1)} \ee^{-b_1/\theta} \int_0^{n^{-\rho_{21}/2}} \frac{b_2^{a_2}}{\Gamma(a_2)} \tau^{-(a_2+1)} \ee^{-b_2/\tau} \ud \tau \ud \theta   \nonumber \\
&\quad + \int_{n^{\rho_{21}/2}}^{\infty} \frac{b_1^{a_1}}{\Gamma(a_1)}\theta^{-(a_1+1)} \ee^{-b_1/\theta} \int_0^{n^{-\rho_{21}} \theta} \frac{b_2^{a_2}}{\Gamma(a_2)} \tau^{-(a_2+1)} \ee^{-b_2/\tau} \ud \tau \ud \theta   \nonumber \\
&\leq \int_0^{n^{-\rho_{21}/2}} \frac{b_2^{a_2}}{\Gamma(a_2)} \tau^{-(a_2+1)} \ee^{-b_2/\tau} \ud \tau  + \int_{n^{\rho_{21}/2}}^{\infty} \frac{b_1^{a_1}}{\Gamma(a_1)}\theta^{-(a_1+1)} \ee^{-b_1/\theta} \ud \theta   \nonumber \\
&\stackrel{(i)}{=}  \int_{n^{\rho_{21}/2}}^{\infty} \frac{b_2^{a_2}}{\Gamma(a_2)} t^{a_2-1} \ee^{-b_2 t} \ud t  + \int_{n^{\rho_{21}/2}}^{\infty} \frac{b_1^{a_1}}{\Gamma(a_1)}\theta^{-(a_1+1)} \ee^{-b_1/\theta} \ud \theta   \nonumber \\
&\leq \int_{n^{\rho_{21}/2}}^{\infty} \frac{b_2^{a_2}}{\Gamma(a_2)} \ee^{-b_2 t/2} \ud t  + \int_{n^{\rho_{21}/2}}^{\infty} \frac{b_1^{a_1}}{\Gamma(a_1)} \theta^{-(a_1+1)}  \ud \theta   \nonumber \\
&= \frac{2b_2^{a_2-1}}{\Gamma(a_2)} \exp\big(-b_2 n^{\rho_{21}/2}/2 \big) + \frac{b_1^{a_1}}{a_1\Gamma(a_1)}n^{ -a_1\rho_{21}/2}  \nonumber \\
&\prec n^{-(3p+6)},
\end{align}
if $a_1\rho_{21}/2> 3p+6$, or $a_1 > 2(3p+6)/\rho_{21}$, where in (i) we use the change of variable $t=1/\tau$.

Third, for the right tail of $\tau/\theta$, we have that
\begin{align} \label{eq:Ecal.c.3}
\Pi(\tau/\theta > n^{\rho_{22}}) &\leq \int_0^{\infty} \frac{b_1^{a_1}}{\Gamma(a_1)}\theta^{-(a_1+1)} \ee^{-b_1/\theta} \int_{n^{\rho_{22}}\theta}^{\infty} \frac{b_2^{a_2}}{\Gamma(a_2)} \tau^{-(a_2+1)} \ee^{-b_2/\tau} \ud \tau \ud \theta \nonumber \\
&\leq \int_0^{n^{-\rho_{22}/2}} \frac{b_1^{a_1}}{\Gamma(a_1)}\theta^{-(a_1+1)} \ee^{-b_1/\theta} \int_{0}^{\infty} \frac{b_2^{a_2}}{\Gamma(a_2)} \tau^{-(a_2+1)} \ee^{-b_2/\tau} \ud \tau \ud \theta   \nonumber \\
&\quad + \int_{n^{-\rho_{22}/2}}^{\infty} \frac{b_1^{a_1}}{\Gamma(a_1)}\theta^{-(a_1+1)} \ee^{-b_1/\theta} \int_{n^{\rho_{22}/2}}^{\infty} \frac{b_2^{a_2}}{\Gamma(a_2)} \tau^{-(a_2+1)} \ee^{-b_2/\tau} \ud \tau \ud \theta   \nonumber \\
&\leq \int_0^{n^{-\rho_{22}/2}} \frac{b_1^{a_1}}{\Gamma(a_1)}\theta^{-(a_1+1)} \ee^{-b_1/\theta} \ud \theta  + \int_{n^{\rho_{22}/2}}^{\infty} \frac{b_2^{a_2}}{\Gamma(a_2)} \tau^{-(a_2+1)} \ee^{-b_2/\tau} \ud \tau   \nonumber \\
&\stackrel{(i)}{=}  \int_{n^{\rho_{22}/2}}^{\infty} \frac{b_1^{a_1}}{\Gamma(a_1)} t^{a_1-1} \ee^{-b_1 t} \ud t  + \int_{n^{\rho_{22}/2}}^{\infty} \frac{b_2^{a_2}}{\Gamma(a_2)}\tau^{-(a_2+1)} \ee^{-b_2/\tau} \ud \tau   \nonumber \\
&\leq \int_{n^{\rho_{22}/2}}^{\infty} \frac{b_1^{a_1}}{\Gamma(a_1)} \ee^{-b_1 t/2} \ud t  + \int_{n^{\rho_{22}/2}}^{\infty} \frac{b_2^{a_2}}{\Gamma(a_2)} \tau^{-(a_2+1)}  \ud \tau   \nonumber \\
&= \frac{2b_1^{a_1-1}}{\Gamma(a_1)} \exp\big(-b_1 n^{\rho_{22}/2}/2 \big) + \frac{b_2^{a_2}}{a_2\Gamma(a_2)}n^{ -a_2\rho_{22}/2}  \nonumber \\
&\prec n^{-(3p+6)},
\end{align}
if $a_2\rho_{22}/2 > 3p+6$, or $a_2 > 2(3p+6)/\rho_{22}$, where in (i) we use the change of variable $t=1/\theta$.

Fourth, the inverse Gaussian prior on $\alpha$ has the density for $\alpha\in \RR_+$,
\[
\pi(\alpha) = \left(\frac{\lambda}{2\pi\alpha^3}\right)^{1/2} \exp\left\{-\frac{\lambda(\alpha-\mu)^2}{2\mu^2\alpha} \right\}.
\]
For the left tail of $\alpha$, we have that
\begin{align} \label{eq:Ecal.c.4}
\Pi(\alpha < n^{-\rho_{31}}) &= \int_0^{n^{-\rho_{31}}} \left(\frac{\lambda}{2\pi\alpha^3}\right)^{1/2} \exp\left(-\frac{\lambda\alpha}{2\mu^2} + \frac{\lambda}{\mu} - \frac{\lambda}{2\alpha} \right) \ud \alpha\nonumber  \\
&\leq \{\lambda/(2\pi)\}^{1/2}\ee^{\lambda/\mu} \int_0^{n^{-\rho_{31}}}  \alpha^{-3/2} \exp \big\{-\lambda/(2\alpha) \big\} \ud \alpha   \nonumber \\
&\stackrel{(i)}{=}  \{\lambda/(2\pi)\}^{1/2}\ee^{\lambda/\mu} \int_{n^{\rho_{31}}}^{\infty}  t^{-1/2} \exp \big\{-\lambda t/2 \big\} \ud t   \nonumber \\
&\leq \{2/(\lambda\pi)\}^{1/2}\ee^{\lambda/\mu} \exp \big(-\lambda n^{\rho_{31}}/2\big) \nonumber \\
&\prec n^{-(3p+6)},
\end{align}
where in (i) we use the change of variable $t=1/\alpha$.

Similarly, for the right tail of $\alpha$, we have that
\begin{align} \label{eq:Ecal.c.5}
\Pi(\alpha > n^{\rho_{32}}) &= \int_{n^{\rho_{32}}}^{\infty} \left(\frac{\lambda}{2\pi\alpha^3}\right)^{1/2} \exp\left(-\frac{\lambda\alpha}{2\mu^2} + \frac{\lambda}{\mu} - \frac{\lambda}{2\alpha} \right) \ud \alpha   \nonumber  \\
&\leq \{\lambda/(2\pi)\}^{1/2} \ee^{\lambda/\mu} \int_{n^{\rho_{32}}}^{\infty} \exp \big\{-\lambda\alpha/(2\mu^2) \big\} \ud \alpha   \nonumber \\
&\leq \{2/(\lambda\pi)\}^{1/2} \mu^2 \ee^{\lambda/\mu} \exp \big\{ -\lambda n^{\rho_{32}}/(2\mu^2) \big\} \nonumber \\
&\prec n^{-(3p+6)}.
\end{align}

Finally, we combine \eqref{eq:Ecal.c.1}, \eqref{eq:Ecal.c.2}, \eqref{eq:Ecal.c.3}, \eqref{eq:Ecal.c.4} and \eqref{eq:Ecal.c.5} with \eqref{eq:Ecal.complement} to conclude that $\Pi(\Ecal_{n}^c)\leq n^{-(3p+6)}$ for all sufficiently large $n$. Hence, Assumption \ref{cond:prior2} is satisfied.
\end{proof}

\section{Proof of Theorem \ref{thm:E1n.rate}} \label{sec:proof.E1n.rate}
The proof of Theorem \ref{thm:E1n.rate} is long and we proceed in several steps. We first present the Taylor series expansion for Mat\'ern covariance function. The Mat\'ern covariance function $\theta K_{\alpha,\nu}$ defined in \eqref{eq:MaternCov} can be expressed as the sum of an infinite series, whose formula depends whether or not $\nu$ is an integer. The following expansion of $\theta K_{\alpha,\nu}(\bfs-\bft)$ for $\bfs,\bft \in [0,1]^d$ can be found on page 2772 of \citet{Loh15}.
\begin{align}\label{eq:K expansion}
& \theta K_{\alpha,\nu}(\bfs-\bft) = \textstyle \sum_{j=0}^{\infty}\big\{ \zeta_j\|\bfs-\bft\|^{2j} + \zeta^*_{\nu+j}G_{\nu+j}(\|\bfs-\bft\|) \big\},
\end{align}
where $\ZZ_+$ is the set of all positive integers, and
\begin{align} \label{eq:G.nu}
G_s (t) & = \left\{ \begin{array}{ll} t^{2 s} \log(t), & \text{ if } s \in \ZZ_+,
\\
t^{2 s}, & \text{ if } s \notin \ZZ_+.
\end{array}
\right.
\end{align}
The terms of $\zeta_j$ and $\zeta^*_{\nu+j}$ for $j=0,1,2,\ldots$ are defined as follows:

If $\nu\not\in\mathbb{Z}$, then
\begin{align}\label{eq:zeta.xi1}
\zeta_j &= \theta\alpha^{2j-2\nu} \xi_j, \qquad \xi_j = 1/ \big\{2^{2j}j!\textstyle \prod_{i=1}^{j}(i-\nu) \big\}, \nonumber \\
\zeta^*_{\nu+j} &= \theta \alpha^{2j} \xi^*_{\nu+j}, \qquad  \xi^*_{\nu+j} = -\pi/ \left\{2^{2j+2\nu}j!\Gamma(j+1+\nu) \Gamma(\nu)\sin(\nu\pi)\right\} .
\end{align}
If $\nu\in\mathbb{Z}_+$, then
\begin{align}\label{eq:zeta.xi2}
&\zeta_j=  \theta \alpha^{2j-2\nu} (-1)^j \frac{(\nu-j-1)!}{2^{2j}(\nu-1)!j!}, \quad \text{if ~} j<\nu, \nonumber \\
&\zeta_j=  \theta \alpha^{2j-2\nu}\left\{\xi_{1,j}+\xi_{2,j}\log(\alpha)\right\}, ~{\rm if ~} j\ge \nu, ~ \text{ where } \nonumber \\
&\xi_{1,j}  = \frac{(-1)^{\nu}\left[\psi(j-\nu+1)+\psi(j+1)+2\log 2\right]}{2^{2j}(j-\nu)!j!(\nu-1)!}, \quad \xi_{2,j}  = \frac{(-1)^{\nu+1}}{2^{2j-1}(j-\nu)!j!(\nu-1)!}, \nonumber \\
&\zeta^*_{\nu+j} = \theta \alpha^{2j} \xi^*_{\nu+j}, \quad \xi^*_{\nu+j} = \frac{(-1)^{\nu+1}}{2^{2\nu+2j-1}(\nu-1)!j!(\nu+j)!} ,
\end{align}
where $\psi(\cdot)$ is the digamma function. For both cases, the coefficients $\xi^*_{\nu+j},\xi_{1,j},\xi_{2,j}$ for all $j=0,1,\ldots$ are all upper bounded by constant.
\vspace{2mm}

We then cite an important lemma about the series of $m$-dependent constants $c_{\bfi,d,\ell}^{(k_1,\ldots,k_d)}$ and $c_{d,\ell}^{(k_1,\ldots,k_d)}$ as defined in Lemma \ref{lem:Cor1.Loh21} of the main text. As $n=m^d\to\infty$ (and so $m\to\infty$), the $m$-dependent constants $c_{\bfi,d,\ell}^{(k_1,\ldots,k_d)}$ are uniformly close to a deterministic sequence of $m$-independent constants $c_{d,\ell}^{(k_1,\ldots,k_d)}$ which satisfy certain relations, as shown by Lemma 2 and Corollary 2 of \citet{Lohetal21}.
\begin{lemma}\label{lem:Cor2.Loh21}
(Lemma 2 and Corollary 2 of \citet{Lohetal21}) Let $d,\ell\in \ZZ_+$. Let $\bfi=(i_1,\ldots,i_d)^\T $ where $1\leq i_1,\ldots,i_d\leq m - \ell \omega_m$. Then there exists a deterministic sequence of $m$-independent constants $c_{d,\ell}^{(k_1,\ldots,k_d)}$, such that $c_{\bfi,d,\ell}^{(k_1,\ldots,k_d)}$ in Lemma \ref{lem:Cor1.Loh21} of the main text satisfies
\begin{align*}
& c_{\bfi,d,\ell}^{(k_1,\ldots,k_d)} = c_{d,\ell}^{(k_1,\ldots,k_d)} + O(\omega_m^{-1}), \quad \text{for all } 0\leq k_1,\ldots,k_d\leq \ell,
\end{align*}
as $m\to\infty$, where the term $O(\omega_m^{-1})$ is uniform over all $1\leq i_1,\ldots,i_d\leq m- \ell\omega_m$ and any $\delta_{\bfi;k}\in [0,1)$. Furthermore,
\begin{align*}
& \sum_{0\leq k_1,\ldots,k_d\leq \ell} c_{d,\ell}^{(k_1,\ldots,k_d)} k_d^{\ell} = \ell!, \\
\text{and } & \sum_{0\leq k_1,\ldots,k_d\leq \ell} c_{d,\ell}^{(k_1,\ldots,k_d)} \cdot \prod_{j=1}^d k_j^{l_j} = 0,
\end{align*}
for all integers $l_1,\ldots,l_d$ satisfying $0\leq l_1,\ldots,l_{d-1}\leq \ell$, $0\leq l_d\leq \ell-1$, and $0\leq l_1+\ldots + l_d\leq \ell$.
\end{lemma}

In the rest of the proof, we use $C$ to denote a generic positive constant that can take different values at different places, and $C$ only depends on $\nu,d,p,\theta_0,\alpha_0,\tau_0,\beta_0$. The order notation $O(\cdot)$ and $o(\cdot)$ is uniform over all $m$. For abbreviation, we write $\PP_0 = \PP_{(\theta_0,\alpha_0,\tau_0,\beta_0)}$, use $\EE_0$ to denote the expectation under $\PP_0$, and use $\EE$ to denote the expectation under $\PP_{(\theta,\alpha,\tau,\beta)}$ for a generic parameter vector $(\theta,\alpha,\tau,\beta)$. We will assume that Assumptions \ref{cond:model}, \ref{cond:sampling} and \ref{cond:deriv} hold throughout this section.

\begin{proof}[$\pof$ Theorem \ref{thm:E1n.rate}]
\vspace*{1mm}

\noindent \underline{Proof of Part (i):}
\vspace{2mm}

By definition, we have that $(\theta_0,\alpha_0,\tau_0,\beta_0)\in \Ecal_{n}$ as $n\to\infty$. From \eqref{eq:EV.case1.ratio1} and \eqref{eq:EV.case2.ratio1} in Section \ref{subsec:EV}, for any given $\epsilon_1>0$, we have that uniformly for all $(\theta,\alpha,\tau,\beta)\in \Ecal_{n}$, for all sufficiently large $n$,
\begin{align} \label{eq:V1.e4}
\left|\EE\left(V_{1,d,\ell_{\star}}\right)/(\theta g_{\ell_{\star},\nu}) - 1\right| < \epsilon_1/4,
\end{align}
with the expectation taken under $\PP_{(\theta,\alpha,\tau,\beta)}$.

Under the true measure $\PP_0$, for any $\epsilon_1\in (0,1)$,
\begin{align} \label{eq:V1.ineq1}
&\quad~ \PP_0 \left(\left|\frac{\widehat\theta_n}{\theta_0}-1\right|\geq \frac{\epsilon_1}{2} \right) \nonumber \\
&\leq \PP_0\left(\left|\frac{\EE_0\left(V_{1,d,\ell_{\star}}\right)}{\theta_0 g_{\ell_{\star},\nu}}\right| \cdot \left|\frac{V_{1,d,\ell_{\star}}}{\EE_0\left(V_{1,d,\ell_{\star}}\right)} -1 \right|  +  \left|\frac{\EE_0\left(V_{1,d,\ell_{\star}}\right)}{\theta_0 g_{\ell_{\star},\nu}} -1\right|\geq \frac{\epsilon_1}{2} \right) \nonumber \\
&\leq  \PP_0\left( \left|\frac{V_{1,d,\ell_{\star}}}{\EE_0\left(V_{1,d,\ell_{\star}}\right)} -1\right| > \frac{\epsilon_1}{5} \right).
\end{align}

Similarly, on the set $\Bcal^c(\epsilon_1,\epsilon_2)\cap \Ecal_{n}$, since $|\theta/\theta_0-1|\geq \epsilon_1$, for all sufficiently large $n$ such that \eqref{eq:V1.e4} holds for all $(\theta,\alpha,\tau,\beta)\in \Ecal_{n}$, we have that for any $\epsilon_1\in (0,1/2)$,
\begin{align} \label{eq:V1.ineq2}
&\quad~ \sup_{\Bcal^c(\epsilon_1,\epsilon_2)\cap \Ecal_{n}} \PP_{(\theta,\alpha,\tau,\beta)} \left(\left|\frac{\widehat\theta_n}{\theta_0}-1\right|\leq \frac{\epsilon_1}{2} \right) \nonumber \\
&\leq \sup_{\Bcal^c(\epsilon_1,\epsilon_2)\cap \Ecal_{n}} \PP_{(\theta,\alpha,\tau,\beta)} \left(\left|\frac{\theta}{\theta_0}-1\right| - \frac{\theta}{\theta_0}\left|\frac{\widehat\theta_n}{\theta}-1\right| \leq \frac{\epsilon_1}{2} \right) \nonumber \\
&\leq \sup_{\Bcal^c(\epsilon_1,\epsilon_2)\cap \Ecal_{n}} \PP_{(\theta,\alpha,\tau,\beta)} \left(\left|\frac{\widehat\theta_n}{\theta}-1\right| \geq \frac{\epsilon_1}{2(1+\epsilon_1)} \right) \nonumber \\
&\leq \sup_{\Bcal^c(\epsilon_1,\epsilon_2)\cap \Ecal_{n}} \PP_{(\theta,\alpha,\tau,\beta)} \left( \left|\frac{\EE\left(V_{1,d,\ell_{\star}}\right)}{\theta g_{\ell_{\star},\nu}}\right|  \cdot \left|\frac{V_{1,d,\ell_{\star}}}{\EE\left(V_{1,d,\ell_{\star}}\right)} -1\right|  +  \left|\frac{\EE(V_{1,d,\ell_{\star}})}{\theta g_{\ell_{\star},\nu}} -1\right|\geq \frac{\epsilon_1}{2} \right) \nonumber \\
&\leq  \sup_{\Bcal^c(\epsilon_1,\epsilon_2)\cap \Ecal_{n}} \PP_{(\theta,\alpha,\tau,\beta)}  \left( \left|\frac{V_{1,d,\ell_{\star}}}{\EE\left(V_{1,d,\ell_{\star}}\right)} -1\right| > \frac{\epsilon_1}{15} \right).
\end{align}

From \eqref{eq:V1.concentration} in Section \ref{subsec:V.EV}, we have that uniformly for all $(\theta,\alpha,\tau,\beta)\in \Ecal_{n}$, there exists $c_1>0$, such that
\begin{align} \label{eq:theta.E1n.concentration}
&\quad~  \PP_{(\theta,\alpha,\tau,\beta)}  \left( \left|\frac{V_{1,d,\ell_{\star}}}{\EE\left(V_{1,d,\ell_{\star}}\right)} -1\right| > \frac{\epsilon_1}{15} \right) \leq \exp \left\{- c_1 \varphi\left(n^{b_1'}\epsilon_1\right) \right\}, \nonumber \\
\text{where } b_1' &= \min\Bigg\{\frac{1}{2}  - \frac{2(1-\gamma)\nu}{d}- \rho_{22} , ~ \frac{1-\gamma}{2} - \frac{\log\log n}{\log n}, ~ \frac{1}{4} + \frac{(1-\gamma)(\ell_{\star}-2\nu)}{d}- \frac{\rho_{1} +\rho_{22}}{2}, \nonumber \\
& \qquad \qquad \frac{1-\gamma}{4} + \frac{(1-\gamma)(\ell_{\star}-\nu)}{d} - \frac{\rho_{1}}{2} - \frac{\log\log n}{2\log n} \Bigg\} \nonumber \\
&\geq \min\Bigg\{\frac{1}{2}  - \frac{2(1-\gamma)\nu}{d}- \rho_{22} , ~ \frac{1-\gamma}{2} -\varsigma, ~  \frac{1}{4} + \frac{(1-\gamma)(\ell_{\star}-2\nu)}{d}- \frac{\rho_{1} +\rho_{22}}{2} , \nonumber \\
& \qquad \qquad \frac{1-\gamma}{4} + \frac{(1-\gamma)(\ell_{\star}-\nu)}{d} - \frac{\rho_{1}}{2} - \varsigma \Bigg\} = b_1,
\end{align}
for an arbitrarily small $\varsigma>0$ as $n\to\infty$.
Therefore, \eqref{eq:V1.ineq1}, \eqref{eq:V1.ineq2} and \eqref{eq:theta.E1n.concentration} implies that \eqref{eq:test.theta1} and \eqref{eq:test.theta2} in Assumption \ref{cond:exptest1} are satisfied with $b_1$ given in \eqref{eq:E1n.b1}.
\vspace{2mm}

\noindent \underline{Proof of Part (ii):}
\vspace{2mm}

Similar to the proof of Part (i), for $\widehat\tau_n$, by using \eqref{eq:EV.case1.ratio2} and \eqref{eq:EV.case2.ratio2} from Section \ref{subsec:EV}, we obtain that for any $\epsilon_2\in (0,1)$,
\begin{align} \label{eq:V0.ineq1}
&\PP_0 \left(\left|\frac{\widehat\tau_n}{\tau_0}-1\right|\geq \frac{\epsilon_2}{2} \right) \leq  \PP_0\left( \left|\frac{V_{0,d,\ell_{\star}}}{\EE_0\left(V_{0,d,\ell_{\star}}\right)} -1\right| > \frac{\epsilon_2}{5} \right), \nonumber \\
&\sup_{\Bcal^c(\epsilon_1,\epsilon_2)\cap \Ecal_{n}} \PP_{(\theta,\alpha,\tau,\beta)} \left(\left|\frac{\widehat\tau_n}{\tau_0}-1\right|\leq \frac{\epsilon_2}{2} \right) \leq  \sup_{\Bcal^c(\epsilon_1,\epsilon_2)\cap \Ecal_{n}} \PP_{(\theta,\alpha,\tau,\beta)}  \left( \left|\frac{V_{0,d,\ell_{\star}}}{\EE\left(V_{0,d,\ell_{\star}}\right)} -1\right| > \frac{\epsilon_2}{15} \right).
\end{align}
From \eqref{eq:V0.concentration} in Section \ref{subsec:V.EV}, we have that for any $\epsilon_2\in (0,1/2)$, uniformly for all $(\theta,\alpha,\tau,\beta)\in \Ecal_{n}$, there exists $c_2>0$, such that
\begin{align} \label{eq:tau.E1n.concentration}
&\quad~  \PP_{(\theta,\alpha,\tau,\beta)}  \left( \left|\frac{V_{0,d,\ell_{\star}}}{\EE\left(V_{0,d,\ell_{\star}}\right)} -1\right| > \frac{\epsilon_2}{15} \right) \leq \exp \left\{- c_2 \varphi\left(n^{b_2'}\epsilon_2\right) \right\}, \nonumber \\
\text{where } b_2' &= \min\Bigg\{\frac{1}{2} , ~ \frac{(1-\gamma)(4\nu+d)}{2d}  - \rho_{21} - \frac{\log\log n}{\log n}, \nonumber \\
& \qquad~~ \frac{1}{4} +\frac{ (1-\gamma)\ell_{\star}}{d} - \frac{\rho_{1} + \rho_{21}}{2}, ~ \frac{(1-\gamma)(4\nu+d +4\ell_{\star})}{4d} - \frac{\rho_{1} + 2\rho_{21} }{2} - \frac{\log\log n}{2\log n} \Bigg\} \nonumber \\
&\geq \min\Bigg\{\frac{1}{2} , ~ \frac{(1-\gamma)(4\nu+d)}{2d}  - \rho_{21} - \varsigma, \nonumber \\
& \qquad~~ \frac{1}{4} +\frac{ (1-\gamma)\ell_{\star}}{d} - \frac{\rho_{1} + \rho_{21}}{2}, ~ \frac{(1-\gamma)(4\nu+d +4\ell_{\star})}{4d} - \frac{\rho_{1} + 2\rho_{21} }{2} - \varsigma \Bigg\} = b_2,
\end{align}
for an arbitrarily small $\varsigma>0$ as $n\to\infty$.
Therefore, \eqref{eq:V0.ineq1} and \eqref{eq:tau.E1n.concentration} implies that \eqref{eq:test.tau1} and \eqref{eq:test.tau2} in Assumption \ref{cond:exptest1} are satisfied with $b_2$ given in \eqref{eq:E1n.b2}.
\end{proof}

The rest of this section includes several auxiliary technical results. Section \ref{subsec:EV} presents the derivation for the uniform upper bounds of $\EE(V_{u,d,\ell_{\star}})$ for $u\in\{0,1\}$ and all $(\theta,\alpha,\tau,\beta)\in \Ecal_{n}$. Section \ref{subsec:V.EV} presents the uniform error bounds for $V_{u,d,\ell_{\star}}/\EE\left(V_{u,d,\ell_{\star}}\right)-1$ for $u\in\{0,1\}$ and all $(\theta,\alpha,\tau,\beta)\in \Ecal_{n}$, which has been used in the proof of Theorem \ref{thm:E1n.rate} above. Section \ref{subsec:Sigma.Fnorm} includes the derivation for the Frobenius norms of some matrices that are used in Section \ref{subsec:V.EV}. Section \ref{subsec:Matern.deriv} presents a technical lemma on the derivative of Mat\'ern covariance function that is used in Section \ref{subsec:V.EV}.

\subsection{Uniform Bounds for $\EE(V_{u,d,\ell_{\star}})$ on $\Ecal_{n}$}  \label{subsec:EV}

We first derive some useful bounds for $\EE(V_{u,d,\ell_{\star}})$ for $u\in \{0,1\}$ and all $(\theta,\alpha,\tau,\beta)\in \Ecal_{n}$. Recall that $\ell_{\star}=\lceil \nu+d/2 \rceil$. For short, we write $\mm(\bfs)=\ff(\bfs)^\T  \beta$ for $\bfs\in [0,1]^d$ as the mean function of $Y(\cdot)$. In the following derivation, we always write $\bfk_1=(k_1,\ldots,k_d)^\T $ and $\bfk_2=(k_{d+1},\ldots,k_{2d})^\T $ for $k_1,\ldots,k_{2d}\in \NN$.

From Lemma \ref{lem:Cor1.Loh21} of the main text, we have that for any integer $0\leq \ell \leq \ell_{\star}-1$, for any $\bfi\in \Xi_{u,m}$ for $u\in\{0,1\}$,
\begin{align}
& \sum_{0\leq k_1, \ldots, k_{2d}\leq \ell_{\star}} c_{\bfi, d, \ell_{\star}}^{(k_1, \ldots, k_d)}c_{\bfi +u \bfe_1,  d, \ell_{\star}}^{(k_{d+1}, \ldots, k_{2d})} \left\|\bfs(\bfi+\bfk_1\omega_m)-\bfs(\bfi+u\bfe_1+\bfk_1\omega_m)\right\|^{2\ell} \nonumber \\
={}& \sum_{0\leq k_1, \ldots, k_{2d}\leq \ell_{\star}} c_{\bfi, d, \ell_{\star}}^{(k_1, \ldots, k_d)}c_{\bfi +u \bfe_1,  d, \ell_{\star}}^{(k_{d+1}, \ldots, k_{2d})} \left\{\sum_{j=1}^d \left[ s_j(\bfi+\bfk_1\omega_m)-s_j(\bfi+u\bfe_1+\bfk_1\omega_m) \right]^2 \right\}^{\ell} \nonumber \\
={}& 0, \label{eq:2equiv.x1} \\
& \sum_{0\leq k_1, \ldots, k_{2d}\leq \ell_{\star}} c_{\bfi, d, \ell_{\star}}^{(k_1, \ldots, k_d)}c_{\bfi +u \bfe_1,  d, \ell_{\star}}^{(k_{d+1}, \ldots, k_{2d})} \left\|\bfs(\bfi+\bfk_1\omega_m)-\bfs(\bfi+u\bfe_1+\bfk_1\omega_m)\right\|^{2\ell_{\star}} \nonumber \\
={}& \frac{(-1)^{\ell_{\star}}(2\ell_{\star})!}{(\ell_{\star}!)^2} \sum_{0\leq k_1, \ldots, k_{2d}\leq \ell_{\star}} c_{\bfi, d, \ell_{\star}}^{(k_1, \ldots, k_d)}c_{\bfi +u \bfe_1,  d, \ell_{\star}}^{(k_{d+1}, \ldots, k_{2d})} s_j(\bfi+\bfk_1\omega_m)^{\ell_{\star}} s_j(\bfi+u\bfe_1+\bfk_1\omega_m)^{\ell_{\star}} \nonumber \\
={}& (-1)^{\ell_{\star}}(2\ell_{\star})! \left(\frac{\omega_m}{m}\right)^{2\ell_{\star}}. \label{eq:2equiv.x2}
\end{align}
These two formulas will be useful in the following derivation.

We observe that
\begin{align} \label{eq:E1}
& \quad ~ \EE (V_{u,d,\ell_{\star}}) \nonumber \\
&= \sum_{\bfi\in \Xi_{u,m}} \sum_{0\leq k_1, \ldots, k_{2d}\leq \ell_{\star}} c_{\bfi, d, \ell_{\star}}^{(k_1, \ldots, k_d)}c_{\bfi +u \bfe_1,  d, \ell_{\star}}^{(k_{d+1}, \ldots, k_{2d})}   \EE \left\{  Y(\bfs(\bfi + \bfk_1 \omega_m)) Y(\bfs(\bfi + u\bfe_1+ \bfk_2\omega_m))  \right\} \nonumber \\
&= \sum_{\bfi\in \Xi_{u,m}} \sum_{0\leq k_1, \ldots, k_{2d}\leq \ell_{\star}} c_{\bfi, d, \ell_{\star}}^{(k_1, \ldots, k_d)}c_{\bfi+u \bfe_1,  d, \ell_{\star}}^{(k_{d+1}, \ldots, k_{2d})} \mm \big(\bfs(\bfi + \bfk_1\omega_m) \big) \mm\big(\bfs(\bfi + u \bfe_1+ \bfk_2 \omega_m) \big) \nonumber \\
&\quad + \tau(1-u) \sum_{\bfi\in \Xi_{u,m}} \sum_{0\leq k_1, \ldots, k_d\leq \ell_{\star}} \left(c_{\bfi, d, \ell_{\star}}^{(k_1, \ldots, k_d)}\right)^2 \nonumber \\
&\quad +\sum_{\bfi\in \Xi_{u,m}} \sum_{0\leq k_1, \ldots, k_{2d}\leq \ell_{\star}}
	c_{\bfi, d, \ell_{\star}}^{(k_1, \ldots, k_d)}c_{\bfi +u \bfe_1,  d, \ell_{\star}}^{(k_{d+1}, \ldots, k_{2d})} \theta K_{\alpha,\nu} \Big(  \bfs(\bfi+ \bfk_1 \omega_m )-	\bfs( \bfi+u\bfe_1+ \bfk_2 \omega_m) \Big).
\end{align}
	
For the first term in \eqref{eq:E1}, since the partial derivatives of $\ff_1,\ldots,\ff_p$ up to the order $\lceil \nu+d/2 \rceil \geq \ell_{\star}$ are all upper bounded by $C_{\ff}'$ by Assumption \ref{cond:deriv}, we can apply the Taylor series expansion to the mean function $\mm(\cdot)$ to obtain that
\begin{align}\label{eq:mm.bound}
&\quad~\left|\sum_{\bfi\in \Xi_{u,m}} \sum_{0\leq k_1, \ldots, k_d\leq \ell_{\star}}  c_{\bfi, d, \ell_{\star}}^{(k_1, \ldots, k_d)} \mm\big(\bfs(\bfi+\bfk_1\omega_m) \big)\right|\nonumber\\
&= \Bigg|\sum_{\bfi\in \Xi_{u,m}} \sum_{0\leq k_1, \ldots, k_d\leq \ell_{\star}}
c_{\bfi, d, \ell_{\star}}^{(k_1, \ldots, k_d)}  \nonumber\\
&\quad \times \Bigg\{\sum_{a_1+\ldots + a_d \le \ell_{\star}-1}\frac{\mm^{(a_1,\ldots,a_d)}(\bfs(\bfi) ) }{a_1!\ldots a_d!}\prod_{j=1}^{d}[s_j(\bfi+\bfk_1\omega_m)-s_j(\bfi) ]^{a_j} \nonumber\\
&\quad +\sum_{a_1+\ldots + a_d =\ell_{\star}}\frac{\ell_{\star}}{a_1!\ldots a_d!} \prod_{j=1}^{d}[s_j(\bfi+\bfk_1\omega_m)-s_j(\bfi) ]^{a_j} \int_0^1 (1-t)^{\ell_{\star}-1}   \nonumber\\
&\quad\times \mm^{(a_1,\ldots,a_d)}\Big(\bfs(\bfi)+t \big( \bfs(\bfi+\bfk_1\omega_m)-\bfs(\bfi) \big)  \Big) \ud t \Bigg\} \Bigg| \nonumber \\
&\stackrel{(i)}{=} \Bigg|\sum_{\bfi\in \Xi_{u,m}} \sum_{0\leq k_1, \ldots, k_d\leq \ell_{\star}}
c_{\bfi, d, \ell_{\star}}^{(k_1, \ldots, k_d)}  \nonumber\\
&\quad \times \Bigg\{\sum_{a_1+\ldots + a_d =\ell_{\star}}\frac{\ell_{\star}}{a_1!\ldots a_d!} \prod_{j=1}^{d}[s_j(\bfi+\bfk_1\omega_m)-s_j(\bfi) ]^{a_j} \int_0^1 (1-t)^{\ell_{\star}-1}   \nonumber\\
&\quad\times \mm^{(a_1,\ldots,a_d)}\Big(\bfs(\bfi)+t \big( \bfs(\bfi+\bfk_1\omega_m)-\bfs(\bfi) \big)  \Big) \ud t \Bigg\}\nonumber\\
&\quad \times \Bigg\{\sum_{a_1+\ldots + a_d =\ell_{\star}}\frac{\ell_{\star}}{a_1!\ldots a_d!} \prod_{j=1}^{d}[s_j(\bfi+u \bfe_1+\bfk_2\omega_m)-s_j(\bfi+u \bfe_1)]^{a_j} \int_0^1 (1-t)^{\ell_{\star}-1}   \nonumber\\
&\quad\times \mm^{(a_1,\ldots,a_d)}\Big(\bfs(\bfi+u \bfe_1)+t \big( \bfs(\bfi+u \bfe_1+\bfk_2\omega_m)-\bfs(\bfi+u \bfe_1) \big)  \Big)\ud t  \Bigg\} \Bigg| \nonumber \\
&\stackrel{(ii)}{\leq} C C_{\ff}' \|\beta\|_{1} \left(\frac{\omega_m}{m}\right)^{\ell_{\star}},
\end{align}
as $n\rightarrow\infty$ for some $C>0$ dependent on $\ell_{\star},\nu,d$, where $\|\beta\|_1=\sum_{j=1}^{p}|\beta_j|$; the equation (i) follows from Lemma \ref{lem:Cor1.Loh21} of the main text; the inequality (ii) follows from that $\big|\mm^{(a_1,\ldots,a_d)}(\bfs)\big|\leq C_{\ff}' \|\beta\|_1$ by Assumption \ref{cond:model}, $\int_0^1 (1-t)^{\ell_{\star}-1}\ud t\leq 1$, and Lemma \ref{lem:Cor1.Loh21} of the main text. Therefore,
\begin{align} \label{eq:E2}
&\quad~\left|\sum_{\bfi\in \Xi_{u,m}} \sum_{0\leq k_1, \ldots, k_{2d} \leq \ell_{\star}}  c_{\bfi, d, \ell_{\star}}^{(k_1, \ldots, k_d)} c_{\bfi, d, \ell_{\star}}^{(k_{d+1}, \ldots, k_{2d})} \mm\big(\bfs(\bfi+\bfk_1\omega_m) \big)\mm\big(\bfs(\bfi+u\bfe_1+\bfk_1\omega_m) \big)  \right|\nonumber\\
&\leq C C_{\ff}^{'2}\|\beta\|_{1}^2 |\Xi_{u,m}|\left(\frac{\omega_m}{m}\right)^{2\ell_{\star}} \leq C C_{\ff}^{'2}\|\beta\|_{1}^2 m^d \left(\frac{\omega_m}{m}\right)^{2\ell_{\star}} .
\end{align}

Now we turn to the second and third terms in \eqref{eq:E1}. We consider two cases, either $\nu\notin \ZZ$ or $\nu\in\ZZ_+$. Define
\begin{align}
\mathcal{F}_{u,d,\ell} (\nu)
&= \zeta_\nu^*\sum_{\bfi  \in \Xi_{u,m}} \sum_{0\leq k_1, \ldots, k_{2d}\leq \ell}  	c_{\bfi, d, \ell}^{(k_1, \ldots, k_d)}c_{\bfi + u\bfe_1, d, \ell}^{(k_{d+1}, \ldots, k_{2d})}
\nonumber \\
& \quad \times
G_\nu\Big(\|\bfs( \bfi+ \bfk_1 \omega_m ) - \bfs(\bfi+u\bfe_1+ \bfk_2 \omega_m ) \|\Big),
\end{align}
where $G_\nu$ is as defined in \eqref{eq:G.nu}.

\noindent {\sc Case A1.} If $\nu\notin\ZZ$, then using \eqref{eq:2equiv.x2}, we observe that
\begin{align}\label{eq:EV_u 1}
&\quad~ \tau(1-u) \sum_{\bfi\in \Xi_{0,m}} \sum_{0\leq k_1, \ldots, k_{d} \leq \ell_{\star}} (c_{\bfi,d,\ell_{\star}}^{k_1,\ldots,k_d})^2  +  \sum_{\bfi \in \Xi_{u,m}} \sum_{0\leq k_1, \ldots, k_{2d}\leq \ell_{\star}} c_{\bfi, d, \ell_{\star}}^{(k_1, \ldots, k_d)}c_{\bfi+ u\bfe_1,  d, \ell_{\star}}^{(k_{d+1}, \ldots, k_{2d})} \nonumber \\
&\quad \times \theta K_{\alpha,\nu} \left(  \bfs(\bfi+ \bfk_1 \omega_m )-\bfs( \bfi+u\bfe_1+ \bfk_2 \omega_m) \right) \nonumber\\
&= \tau(1-u) \sum_{\bfi\in \Xi_{0,m}} \sum_{0\leq k_1, \ldots, k_{d} \leq \ell_{\star}} (c_{\bfi,d,\ell_{\star}}^{k_1,\ldots,k_d})^2     + \sum_{\bfi\in \Xi_{u,m}} \sum_{0\leq k_1, \ldots, k_{2d}\leq \ell_{\star}} c_{\bfi, d, \ell_{\star}}^{(k_1, \ldots, k_d)}c_{\bfi+ u\bfe_1,  d, \ell_{\star}}^{(k_{d+1}, \ldots, k_{2d})} \nonumber \\
&\quad \times \sum_{j=0}^\infty \Big\{ \zeta_j \|\bfs( \bfi+ \bfk_1 \omega_m ) -  \bfs( \bfi+u\bfe_1+ \bfk_2 \omega_m ) \|^{2j}
	\nonumber \\
& \quad+ \zeta^*_{\nu+j} G_{\nu+j} \Big(\|\bfs( \bfi+ \bfk_1 \omega_m ) -  \bfs( \bfi+u\bfe_1+ \bfk_2 \omega_m ) \|\Big) \Big\}\nonumber \\
&= \tau(1-u) \sum_{\bfi\in \Xi_{0,m}} \sum_{0\leq k_1, \ldots, k_{d} \leq \ell_{\star}} (c_{\bfi,d,\ell_{\star}}^{k_1,\ldots,k_d})^2 + \zeta_{\ell_{\star}} |\Xi_{u,m}| (-1)^{\ell_{\star}}(2\ell_{\star})!\left(\frac{\omega_m}{m}\right)^{2\ell_{\star}}  \nonumber\\
&\quad +\sum_{\bfi \in \Xi_{u,m}}  \sum_{0\leq k_1, \ldots, k_{2d}\leq \ell_{\star}} c_{\bfi, d, \ell_{\star}}^{(k_1, \ldots, k_d)}c_{\bfi+u \bfe_1, d, \ell_{\star}}^{(k_{d+1}, \ldots, k_{2d})} \Big[ \zeta_\nu^* \|\bfs( \bfi+ \bfk_1 \omega_m ) -  \bfs( \bfi+u\bfe_1+ \bfk_2 \omega_m ) \|^{2\nu} \nonumber \\
&\quad +  \sum_{j=\ell_{\star}+1}^{\infty}\zeta_j \|\bfs( \bfi+ \bfk_1 \omega_m ) -  \bfs( \bfi+u\bfe_1+ \bfk_2 \omega_m ) \|^{2j} \nonumber\\
&\quad + \sum_{j=1}^{\infty}\zeta^*_{\nu+j} \|\bfs( \bfi+ \bfk_1 \omega_m ) -  \bfs( \bfi+u\bfe_1+ \bfk_2 \omega_m ) \|^{2\nu+2j}  \Big] \nonumber \\
&= \tau(1-u) \sum_{\bfi\in \Xi_{0,m}} \sum_{0\leq k_1, \ldots, k_{d} \leq \ell_{\star}} (c_{\bfi,d,\ell_{\star}}^{k_1,\ldots,k_d})^2  + \theta \alpha^{2\ell_{\star}-2\nu}\xi_{\ell_{\star}} |\Xi_{u,m}| (-1)^{\ell_{\star}}(2\ell_{\star})!\left(\frac{\omega_m}{m}\right)^{2\ell_{\star}} \nonumber\\
&\quad +  \mathcal{F}_{u,d,\ell_{\star}} (\nu) +\theta\sum_{\bfi \in \Xi_{u,m}}  \sum_{0\leq k_1, \ldots, k_{2d}\leq \ell_{\star}} c_{\bfi, d, \ell_{\star}}^{(k_1, \ldots, k_d)}c_{\bfi+u \bfe_1,  d, \ell_{\star}}^{(k_{d+1}, \ldots, k_{2d})} \nonumber \\
&\quad \times \Bigg\{ \sum_{j=\ell_{\star}+1}^{\infty}\alpha^{2j-2\nu}\xi_j \|\bfs( \bfi+ \bfk_1 \omega_m )-  \bfs( \bfi+u\bfe_1+ \bfk_2 \omega_m ) \|^{2j}\nonumber\\
&\qquad + \sum_{j=1}^{\infty}\alpha^{2j}\xi^*_{\nu+j} \|\bfs( \bfi+ \bfk_1 \omega_m )-  \bfs( \bfi+u\bfe_1+ \bfk_2 \omega_m ) \|^{2\nu+2j}  \Bigg\}.
\end{align}
For $\bfk_1 \neq \bfk_2$, we use Lemma \ref{lem:Cor2.Loh21} to obtain that
\begin{align*}
&\quad~ \left\|\bfs( \bfi+ \bfk_1 \omega_m ) -  \bfs( \bfi+u \bfe_1+ \bfk_2 \omega_m ) \right\|^{2\nu} = \left(\frac{\omega_m }{m}\right)^{2\nu}  \left\{ \sum_{j=1}^d \left[ k_j - k_{d+j} + O(\omega_m^{-1}) \right]^2 \right\}^\nu  \nonumber \\
&= \left(\frac{\omega_m }{m}\right)^{2\nu}  \left\{ \sum_{j=1}^d ( k_j - k_{d+j} )^2 + O(\omega_m^{-1})  \right\}^\nu  = \left(\frac{\omega_m }{m}\right)^{2\nu}  \left\{ \left[\sum_{j=1}^d ( k_j - k_{d+j} )^2 \right]^\nu + O(\omega_m^{-1})  \right\},
\end{align*}
as $m\rightarrow \infty$ uniformly over $\bfi\in \Xi_{u,m}$.

For $\bfk_1 = \bfk_2$, we have that
\begin{align*}
&\quad~ \left\|\bfs( \bfi+ \bfk_1 \omega_m ) -  \bfs( \bfi+u\bfe_1+ \bfk_2 \omega_m) \right\|^{2\nu} \\
&= \left\{\sum_{j=1}^d \left[\frac{\delta_{\bfi+ \bfk_1 \omega_m;j}-ue_j-\delta_{\bfi+u\bfe_1+ \bfk_2 \omega_m;j}}{m}\right]\right\}^{\nu} = u O \left(\frac{1}{m^{2\nu}}\right),
\end{align*}
as $m\rightarrow \infty$ uniformly over $\bfi\in \Xi_{u,m}$. Therefore, by Lemma \ref{lem:Cor2.Loh21},
\begin{align}\label{eq:F.udl}
\mathcal{F}_{u,d,\ell_{\star}} (\nu) &= \zeta_\nu^* \sum_{\bfi\in \Xi_{u,m}} \sum_{0\leq k_1, \ldots, k_{2d}\leq \ell_{\star}}  c_{\bfi, d, \ell_{\star}}^{(k_1, \ldots, k_d)}c_{\bfi+u \bfe_1, d, \ell_{\star}}^{(k_{d+1}, \ldots, k_{2d})}  \nonumber \\
&\times  \left\|\bfs( \bfi+ \bfk_1 \omega_m ) -  \bfs( \bfi+u\bfe_1+ \bfk_2 \omega_m) \right\|^{2\nu} \nonumber \\
&= \zeta_\nu^* \sum_{\bfi\in \Xi_{u,m}}  \sum_{0\leq k_1, \ldots, k_{2d}\leq \ell_{\star}: \bfk_1 \neq \bfk_2}  \left\{ c_{d, \ell_{\star}}^{(k_1, \ldots, k_d)} + O(\omega_m^{-1}) \right\}  \nonumber \\
&\quad \times \left\{ c_{d, \ell_{\star}}^{(k_{d+1}, \ldots, k_{2d})} + O(\omega_m^{-1}) \right\}\left(\frac{\omega_m }{m}\right)^{2\nu}  \left\{ \left[\sum_{j=1}^d ( k_j - k_{d+j} )^2 \right]^\nu + O(\omega_m^{-1})  \right\}  \nonumber \\
&\quad +\zeta_\nu^* \sum_{\bfi\in \Xi_{u,m}} \sum_{0\leq k_1, \ldots, k_{2d}\leq \ell_{\star}: \bfk_1 = \bfk_2}
\left\{ c_{d, \ell_{\star}}^{(k_1, \ldots, k_d)} + O(\omega_m^{-1}) \right\} \nonumber \\
&\quad \times \left\{ c_{d, \ell_{\star}}^{(k_{d+1}, \ldots, k_{2d})} + O(\omega_m^{-1}) \right\} uO\left(\frac{1}{m^{2\nu}}\right)  \nonumber \\
&= \zeta_\nu^*\left(\frac{\omega_m}{m}\right)^{2\nu} |\Xi_{u,m}| \sum_{0\leq k_1, \ldots, k_{2d}\leq \ell_{\star}}c_{d, \ell_{\star}}^{(k_1, \ldots, k_d)} c_{d, \ell_{\star}}^{(k_{d+1}, \ldots, k_{2d})} \left\| \bfk_1- \bfk_2\right\|^{2\nu}	\nonumber \\
&\quad    +\zeta^*_\nu |\Xi_{u,m}| O \left\{ \left(\frac{\omega_m}{m}\right)^{2\nu} \omega_m^{-1} + u\left(\frac{ \omega_m}{m}\right)^{2 \nu}\omega_m^{-2 \nu} \right\}  \nonumber \\
&= \theta \xi_\nu^* \left(\frac{\omega_m }{m}\right)^{2\nu} |\Xi_{u,m}| H_{\ell_{\star},\nu} + \theta \xi_\nu^* |\Xi_{u,m}|\left(\frac{\omega_m}{m}\right)^{2\nu} O\left\{\omega_m^{-1}+u\omega_m^{-2\nu} \right\} \\
&\asymp  \theta m^d \left(\frac{\omega_m}{m}\right)^{2\nu},\nonumber
\end{align}
as $m\rightarrow \infty$, where $H_{\ell_{\star},\nu}$ defined in \eqref{eq:CgH} satisfies $H_{\ell_{\star},\nu}>0$, and $O(\cdot)$ in the last equality only depends on $\nu$. Combining \eqref{eq:F.udl} with \eqref{eq:E1}, \eqref{eq:E2} and \eqref{eq:EV_u 1}, we conclude that
\begin{align}\label{eq:EV.case1}
&\quad~ \EE \left(V_{u,d,\ell_{\star}}\right) \nonumber \\
&= \tau(1-u) \sum_{\bfi\in \Xi_{0,m}} \sum_{0\leq k_1, \ldots, k_{d} \leq \ell_{\star}} \left(c_{\bfi,d,\ell_{\star}}^{(k_1,\ldots,k_d)}\right)^2 + \theta \alpha^{2\ell_{\star}-2\nu}\xi_{\ell_{\star}} |\Xi_{u,m}| (-1)^{\ell_{\star}}(2\ell_{\star})!\left(\frac{\omega_m}{m}\right)^{2\ell_{\star}} \nonumber \\
&\quad +\zeta_\nu^* \left(\frac{\omega_m }{m}\right)^{2\nu} |\Xi_{u,m}| H_{\ell_{\star},\nu} +\zeta^*_\nu |\Xi_{u,m}|\left(\frac{\omega_m}{m}\right)^{2\nu} O\{\omega_m^{-1}+u\omega_m^{-2\nu} \}  \nonumber\\
&\quad +\theta\sum_{\bfi \in \Xi_{u,m}}  \sum_{0\leq k_1, \ldots, k_{2d}\leq \ell_{\star}} c_{\bfi, d, \ell_{\star}}^{(k_1, \ldots, k_d)}c_{\bfi+u \bfe_1,  d, \ell_{\star}}^{(k_{d+1}, \ldots, k_{2d})}  \nonumber \\
&\qquad \times \Bigg[ \sum_{j=\ell_{\star}+1}^{\infty} \alpha^{2j-2\nu}\xi_j \left\|\bfs( \bfi+ \bfk_1 \omega_m ) -  \bfs( \bfi+u\bfe_1+ \bfk_2 \omega_m ) \right\|^{2j} \nonumber\\
&\qquad + \sum_{j=1}^{\infty} \alpha^{2j} \xi^*_{\nu+j} \left\|\bfs( \bfi+ \bfk_1 \omega_m ) -  \bfs( \bfi+u\bfe_1+ \bfk_2 \omega_m ) \right\|^{2\nu+2j}  \Bigg] + O(1)  \|\beta\|_1^2 |\Xi_{u,m}|\left(\frac{\omega_m}{m}\right)^{2\ell_{\star}}  \nonumber \\
&= \tau(1-u) \sum_{\bfi\in \Xi_{0,m}} \sum_{0\leq k_1, \ldots, k_{d} \leq \ell_{\star}} \left(c_{\bfi,d,\ell_{\star}}^{(k_1,\ldots,k_d)}\right)^2  + \theta \alpha^{2\ell_{\star}-2\nu}\xi_{\ell_{\star}} |\Xi_{u,m}| (-1)^{\ell_{\star}}(2\ell_{\star})!\left(\frac{\omega_m}{m}\right)^{2\ell_{\star}}  \nonumber \\
&\quad + \theta\xi_\nu^* \left(\frac{\omega_m }{m}\right)^{2\nu} |\Xi_{u,m}| H_{\ell_{\star},\nu} +\theta\xi^*_\nu |\Xi_{u,m}|\left(\frac{\omega_m}{m}\right)^{2\nu} O\left(\omega_m^{-1}+u\omega_m^{-2\nu} \right) \nonumber\\
&\quad +\theta\left(\frac{\omega_m}{m}\right)^{2\nu}|\Xi_{u,m}|O(1)\left\{\sum_{j=\ell_{\star}+1}^{\infty} \xi_j \left(\alpha \frac{\omega_m}{m}\right)^{2j-2\nu}+\sum_{j=1}^{\infty}\xi^*_{\nu+j} \left(\alpha \frac{\omega_m}{m}\right)^{2j} \right\} \nonumber \\
&\quad + O(1) \|\beta\|_1^2 |\Xi_{u,m}|\left(\frac{\omega_m}{m}\right)^{2\ell_{\star}},
\end{align}
as $m\rightarrow \infty$, and $O(\cdot)$ only depends on $\ell_{\star},\nu,d,C_{\ff}'$ and not on the parameters $\beta,\theta,\tau,\alpha$.

Hence, using the definitions of $C_{V,0},g_{\ell_{\star},\nu},H_{\ell_{\star},\nu}$ in \eqref{eq:CgH}, we can derive from \eqref{eq:EV.case1} that when $u=1$,
\begin{align}
\frac{\EE \left(V_{1,d,\ell_{\star}}\right)}{\theta g_{\ell_{\star},\nu}} &= \frac{\EE\left(V_{1,d,\ell_{\star}}\right)}{\zeta_\nu^* (\frac{\omega_m }{m})^{2\nu} |\Xi_{1,m}| H_{\ell_{\star},\nu}}  \nonumber \\
&= 1+ O\left(\omega_m^{-1}+\omega_m^{-2\nu} \right) + O(1)\frac{\|\beta\|_1^2}{\theta} |\Xi_{u,m}|\left(\frac{\omega_m}{m}\right)^{2\ell_{\star}-2\nu} \nonumber \\
&\quad +  O(1) \left[ \sum_{j=\ell_{\star}}^{\infty}\frac{\xi_j}{\xi^*_{\nu}}\left(\alpha\frac{\omega_m }{m}\right)^{2j-2\nu}
+ \sum_{j=1}^{\infty}\frac{\xi^*_{\nu+j}}{\xi^*_{\nu}}\left(\alpha\frac{\omega_m }{m}\right)^{2j} \right] \label{eq:EV.case1.ratio1} \\
&\stackrel{(i)}{=} 1+o(1), \nonumber
\end{align}
as $m\to\infty$, where the relation (i) holds uniformly on the set $\Ecal_{n}$, because according to \eqref{eq:rho.relation1}, for all $(\theta,\alpha,\tau,\beta)\in \Ecal_{n}$, as $n\to\infty$,
\begin{align*}
& \frac{\|\beta\|_1^2}{\theta}\left(\frac{\omega_m}{m}\right)^{2\ell_{\star}-2\nu} \leq \frac{p\|\beta\|^2}{\theta}\left(\frac{\omega_m}{m}\right)^{2\ell_{\star}-2\nu} \leq pn^{\rho_{1} - (2\ell_{\star}-2\nu)(1-\gamma)/d} = o(1),  \\
& \alpha \frac{\omega_m}{m} \leq n^{\rho_{32} - (1 - \gamma)/d} =o(1),
\end{align*}
and that the two summations in \eqref{eq:EV.case1.ratio1} are all finite given the definition of $\xi_j,\xi^*_{\nu+j}$ for $j\in\NN$ in \eqref{eq:zeta.xi1}. The $O(\cdot)$ and $o(\cdot)$ only depend on $\ell_{\star},\nu,d,C_{\ff}'$ and not on $\beta,\theta,\tau,\alpha$, and the same applies to the derivation below as well.

For $u=0$, we have that
\begin{align}
\frac{\EE\left(V_{0,d,\ell_{\star}}\right)-\tau C_{V,0}}{\zeta_\nu^* (\frac{\omega_m }{m})^{2\nu} |\Xi_{0,m}| H_{\ell_{\star},\nu}}
&= 1+ O\left(\omega_m^{-1}+\omega_m^{-2\nu} \right)  +O(1)\frac{\|\beta\|_1^2}{\theta} |\Xi_{u,m}|\left(\frac{\omega_m}{m}\right)^{2\ell_{\star}-2\nu} \nonumber \\
&\quad + O(1) \left[ \sum_{j=\ell_{\star}}^{\infty}\frac{\xi_j}{\xi^*_{\nu}}\left(\alpha\frac{\omega_m }{m}\right)^{2j-2\nu}
+ \sum_{j=1}^{\infty}\frac{\xi^*_{\nu+j}}{\xi^*_{\nu}}\left(\alpha\frac{\omega_m }{m}\right)^{2j} \right] . \label{eq:EV.case1.ratio2.1}
\end{align}	
The summations in \eqref{eq:EV.case1.ratio2.1} are all finite given the definition of $\xi_j,\xi^*_{\nu+j}$ for $j\in\NN$ in \eqref{eq:zeta.xi1}.

Furthermore, since as $m\to\infty$,
\begin{align} \label{eq:tau.CV0}
\tau C_{V,0} &= \tau \sum_{\bfi \in \Xi_{0,m}} \sum_{0\leq k_1, \ldots, k_{d}\leq \ell} \left(c_{\bfi, d, \ell_{\star}}^{(k_1, \ldots, k_d)} \right)^2 \nonumber \\
&\stackrel{(i)}{=} \tau |\Xi_{0,m}| \left\{\sum_{0\leq k_1, \ldots, k_{d}\leq \ell_{\star}} \left(c_{d, \ell_{\star}}^{(k_1, \ldots, k_d)} \right)^2 + O(\omega_m^{-1}) \right\}  \\
&\stackrel{(ii)}{\asymp} \tau m^d, \nonumber
\end{align}
where (i) follows from Lemma \ref{lem:Cor2.Loh21} and the $O(\cdot)$ is uniform over all $\bfi \in \Xi_{0,m}$; (ii) follows because the constants $c_{d, \ell_{\star}}^{(k_1, \ldots, k_d)}$ do not depend on $\bfi$ and $m$, and that $\ell_{\star}$ is a constant.

From \eqref{eq:EV.case1.ratio2.1} and \eqref{eq:tau.CV0}, it follows that
\begin{align}\label{eq:EV.case1.ratio2}
\frac{\EE\left(V_{0,d,\ell_{\star}}\right)}{\tau C_{V,0}}&= \frac{\EE\left(V_{0,d,\ell_{\star}}\right)-\tau C_{V,0}}{\zeta_\nu^* (\frac{\omega_m}{m})^{2\nu} |\Xi_{0,m}| H_{\ell_{\star},\nu}} \frac{\zeta_\nu^* \left(\frac{\omega_m }{m}\right)^{2\nu} |\Xi_{0,m}| H_{\ell_{\star},\nu}}{\tau C_{V,0}} + 1  \nonumber  \\
&= 1+ O(1)\frac{\theta}{\tau} \left(\frac{\omega_m}{m}\right)^{2\nu} \Bigg\{1+ O\left(\omega_m^{-1}+\omega_m^{-2\nu} \right)  +\frac{\|\beta\|_1^2}{\theta} |\Xi_{u,m}|\left(\frac{\omega_m}{m}\right)^{2\ell-2\nu}   \nonumber \\
&\quad + O(1) \left[ \sum_{j=\ell}^{\infty}\frac{\xi_j}{\xi^*_{\nu}}\left(\alpha\frac{\omega_m}{m}\right)^{2j-2\nu} + \sum_{j=1}^{\infty}\frac{\xi^*_{\nu+j}}{\xi^*_{\nu}}\left(\alpha\frac{\omega_m}{m}\right)^{2j} \right]  \Bigg\} \\
&\stackrel{(i)}{=} 1+o(1), \nonumber
\end{align}
as $m\rightarrow\infty$, where the relation (i) holds uniformly on the set $\Ecal_{n}$, because according to \eqref{eq:rho.relation1}, for all $(\theta,\alpha,\tau,\beta)\in \Ecal_{n}$, as $n\to\infty$,
\begin{align*}
& \frac{\theta}{\tau} \left(\frac{\omega_m}{m}\right)^{2\nu} \leq n^{\rho_{21} - 2\nu(1-\gamma)/d} = o(1), \\
& \frac{\|\beta\|_1^2}{\theta}\left(\frac{\omega_m}{m}\right)^{2\ell_{\star}-2\nu} \leq \frac{p\|\beta\|^2}{\theta}\left(\frac{\omega_m}{m}\right)^{2\ell_{\star}-2\nu} \leq pn^{\rho_{1} - (2\ell_{\star}-2\nu)(1-\gamma)/d} = o(1),  \\
& \alpha \frac{\omega_m}{m} \leq n^{\rho_{32} - (1 - \gamma)/d} =o(1),
\end{align*}
and that the two summations in \eqref{eq:EV.case1.ratio2} are all finite given the definition of $\xi_j,\xi^*_{\nu+j}$ for $j\in\NN$ in \eqref{eq:zeta.xi1}.

\vspace{3mm}

\noindent {\sc Case A2.} If $\nu\in\ZZ_+$, then the bound in \eqref{eq:mm.bound} still applies, and we only need to derive the order for the second and the third terms in \eqref{eq:E1}.

We observe that	
\begin{align}\label{eq:EV_u 2}
&\quad~ \tau(1-u) \sum_{\bfi\in \Xi_{0,m}} \sum_{0\leq k_1, \ldots, k_{d} \leq \ell_{\star}} \left(c_{\bfi,d,\ell_{\star}}^{(k_1,\ldots,k_d)}\right)^2  + \sum_{\bfi\in \Xi_{u,m}} \sum_{0\leq k_1, \ldots, k_{2d}\leq \ell_{\star}}  c_{\bfi, d, \ell_{\star}}^{(k_1, \ldots, k_d)}c_{\bfi+u \bfe_1, d, \ell_{\star}}^{(k_{d+1}, \ldots, k_{2d})}
\nonumber \\
&\quad \times \theta K_{\alpha,\nu} \left(\bfs(\bfi+ \bfk_1 \omega_m )- \bfs( \bfi+u\bfe_1+ \bfk_2 \omega_m) \right)
\nonumber\\
&= \tau(1-u) \sum_{\bfi\in \Xi_{0,m}} \sum_{0\leq k_1, \ldots, k_{d} \leq \ell_{\star}}\left(c_{\bfi,d,\ell_{\star}}^{(k_1,\ldots,k_d)}\right)^2   + \sum_{\bfi\in \Xi_{u,m}} \sum_{0\leq k_1, \ldots, k_{2d}\leq \ell_{\star}}  c_{\bfi, d, \ell_{\star}}^{(k_1, \ldots, k_d)}c_{\bfi+u \bfe_1,  d, \ell_{\star}}^{(k_{d+1}, \ldots, k_{2d})} \nonumber \\
& \quad \times \sum_{j=0}^\infty \Big\{ \zeta_j \|\bfs(\bfi+ \bfk_1 \omega_m ) -  \bfs( \bfi+u\bfe_1+ \bfk_2 \omega_m) \|^{2j} \nonumber \\
& \quad+ \zeta^*_{\nu+j} \|\bfs( \bfi+ \bfk_1 \omega_m ) -  \bfs( \bfi+u\bfe_1+ \bfk_2 \omega_m) \|^{2\nu+2j}  \log \left(\|\bfs(\bfi+ \bfk_1 \omega_m ) -  \bfs( \bfi+u\bfe_1+ \bfk_2 \omega_m) \| \right) \Big\} \nonumber \\
&=  \tau(1-u) \sum_{\bfi\in \Xi_{0,m}} \sum_{0\leq k_1, \ldots, k_{d} \leq \ell_{\star}} \left(c_{\bfi,d,\ell_{\star}}^{(k_1,\ldots,k_d)}\right)^2   + \zeta_{\ell_{\star}} |\Xi_{u,m}| (-1)^{\ell_{\star}}(2\ell_{\star})!\left(\frac{\omega_m}{m}\right)^{2\ell_{\star}} \nonumber\\
&\quad +\sum_{\bfi\in \Xi_{u,m}}\sum_{0\leq k_1, \ldots, k_{2d}\leq \ell_{\star}} c_{\bfi,  d, \ell_{\star}}^{(k_1, \ldots, k_d)}c_{\bfi+ u\bfe_1, d, \ell_{\star}}^{(k_{d+1}, \ldots, k_{2d})} \nonumber \\
&\quad \times \Big[ \zeta_\nu^* \|\bfs(\bfi+ \bfk_1 \omega_m ) -  \bfs( \bfi+u\bfe_1+ \bfk_2 \omega_m) \|^{2\nu}
\log \left(\frac{m}{\omega_m } \|\bfs(\bfi+ \bfk_1 \omega_m ) -  \bfs( \bfi+u\bfe_1+ \bfk_2 \omega_m) \| \right)	\nonumber \\
&\quad +\zeta^*_{\nu+1} \|\bfs(\bfi+ \bfk_1 \omega_m ) -  \bfs( \bfi+u\bfe_1+ \bfk_2 \omega_m) \|^{2\nu+2} \nonumber \\
&\quad \times  \log \left(\frac{m}{\omega_m } \|\bfs( \bfi+ \bfk_1 \omega_m ) -  \bfs( \bfi+u\bfe_1+ \bfk_2 \omega_m) \| \right) \nonumber \\
&\quad -\zeta^*_{\nu+1} \|\bfs(\bfi+ \bfk_1 \omega_m ) -  \bfs( \bfi+u\bfe_1+ \bfk_2 \omega_m) \|^{2\nu+2} \log \left(\frac{m}{\omega_m} \right)\nonumber \\
&\quad  +  \sum_{j=\ell_{\star}+1}^{\infty}\zeta_j \|\bfs( \bfi+ \bfk_1 \omega_m ) -  \bfs( \bfi+u\bfe_1+ \bfk_2 \omega_m ) \|^{2j}  \nonumber\\
&\quad + \sum_{j=2}^{\infty}\zeta^*_{\nu+j} \|\bfs( \bfi+ \bfk_1 \omega_m ) -  \bfs( \bfi+u\bfe_1+ \bfk_2 \omega_m ) \|^{2\nu+2j} \nonumber\\
&\qquad \times\log (\|\bfs( \bfi+ \bfk_1 \omega_m ) -  \bfs( \bfi+u\bfe_1+ \bfk_2 \omega_m ) \|) \Big] \nonumber \\
&=  \tau(1-u) \sum_{\bfi\in \Xi_{0,m}} \sum_{0\leq k_1, \ldots, k_{d} \leq \ell_{\star}} \left(c_{\bfi,d,\ell_{\star}}^{(k_1,\ldots,k_d)}\right)^2 \nonumber\\
&\quad + \theta \alpha^{2\ell_{\star}-2\nu}\{\xi_{1,\ell_{\star}}+\xi_{2,\ell_{\star}}\log(\alpha)\} |\Xi_{u,m}| (-1)^{\ell_{\star}}(2\ell_{\star})!\left(\frac{\omega_m}{m}\right)^{2\ell_{\star}} \nonumber\\
&\quad + \mathcal{F}_{u,d,\ell_{\star}}(\nu)-\theta \alpha^2\xi^*_{\nu+1}|\Xi_{u,m}| (-1)^{\ell_{\star}}(2\ell_{\star})!\left(\frac{\omega_m}{m}\right)^{2\ell_{\star}}\log \left(\frac{m}{\omega_m } \right) \Ical\{\ell_{\star}=\nu+1 \}\nonumber\\
&\quad +\theta \sum_{\bfi\in \Xi_{u,m}}\sum_{0\leq k_1, \ldots, k_{2d}\leq \ell_{\star}} c_{\bfi,  d, \ell_{\star}}^{(k_1, \ldots, k_d)}c_{\bfi+ u\bfe_1,  d, \ell_{\star}}^{(k_{d+1}, \ldots, k_{2d})} \nonumber\\
&\quad \times \Big[ \alpha^2\xi^*_{\nu+1} \|\bfs(\bfi+ \bfk_1 \omega_m ) -  \bfs( \bfi+u\bfe_1+ \bfk_2 \omega_m) \|^{2\nu+2} \nonumber \\
&\qquad \times  \log \left(\frac{m}{\omega_m} \|\bfs( \bfi+ \bfk_1 \omega_m ) -  \bfs( \bfi+u\bfe_1+ \bfk_2 \omega_m) \| \right) \nonumber \\
&\qquad +  \sum_{j=\ell_{\star}+1}^{\infty}\alpha^{2j-2\nu}\xi_{1,j} \|\bfs( \bfi+ \bfk_1 \omega_m ) -  \bfs( \bfi+u\bfe_1+ \bfk_2 \omega_m ) \|^{2j}  \nonumber\\
&\qquad  +  \sum_{j=\ell_{\star}+1}^{\infty}\alpha^{2j-2\nu}\log(\alpha)\xi_{2,j} \|\bfs( \bfi+ \bfk_1 \omega_m ) -  \bfs( \bfi+u\bfe_1+ \bfk_2 \omega_m ) \|^{2j}   \nonumber\\
&\qquad + \sum_{j=2}^{\infty}\alpha^{2j}\xi^*_{\nu+j} \|\bfs( \bfi+ \bfk_1 \omega_m ) -  \bfs( \bfi+u\bfe_1+ \bfk_2 \omega_m ) \|^{2\nu+2j}  \nonumber\\
&\qquad \times\log \left(\|\bfs( \bfi+ \bfk_1 \omega_m ) -  \bfs( \bfi+u\bfe_1+ \bfk_2 \omega_m ) \|\right) \Big].
\end{align}

Similar to Case A1, we can use Lemma \ref{lem:Cor2.Loh21} to derive the order for $\mathcal{F}_{u,d,\ell_{\star}} (\nu)$:
\begin{align}\label{eq:F.udl2}
& \quad~ \mathcal{F}_{u,d,\ell_{\star}} (\nu) \nonumber \\
&= \zeta_\nu^* \sum_{\bfi\in \Xi_{u,m}} \sum_{0\leq k_1, \ldots, k_{2d}\leq \ell_{\star}} c_{\bfi,  d, \ell_{\star}}^{(k_1, \ldots, k_d)}c_{\bfi+u\bfe_1,  d, \ell_{\star}}^{(k_{d+1}, \ldots, k_{2d})}  \|\bfs( \bfi+ \bfk_1 \omega_m ) -  \bfs( \bfi+u\bfe_1+ \bfk_2 \omega_m) \|^{2\nu} \nonumber \\
&\times \log \left( \|\bfs( \bfi+ \bfk_1 \omega_m ) -  \bfs( \bfi+u\bfe_1+ \bfk_2 \omega_m) \| \right) \nonumber \\
&=\zeta_\nu^*  \sum_{\bfi\in \Xi_{u,m}} \sum_{0\leq k_1, \ldots, k_{2d}\leq \ell_{\star}: \bfk_1\neq \bfk_2} \left\{ c_{d, \ell_{\star}}^{(k_1, \ldots, k_d)} + O(\omega_m^{-1}) \right\} \left \{ c_{d, \ell_{\star}}^{(k_{d+1}, \ldots, k_{2d})} + O(\omega_m^{-1}) \right\} \left(\frac{\omega_m}{m}\right)^{2\nu}   \nonumber \\
& \qquad\times \left\{ \left[\sum_{j=1}^d ( k_j - k_{d+j} )^2 \right]^\nu + O(\omega_m^{-1})  \right\}  \log \left\{ \left[\sum_{j=1}^d ( k_j - k_{d+j} )^2 \right]^{1/2} + O(\omega_m^{-1})  \right\} \nonumber \\
&+ \zeta_\nu^* \sum_{\bfi\in \Xi_{u,m}} \sum_{0\leq k_1, \ldots, k_{2d}\leq \ell_{\star}: \bfk_1= \bfk_2} \left\{ c_{d, \ell_{\star}}^{(k_1, \ldots, k_d)} + O(\omega_m^{-1}) \right\} \nonumber \\
& \qquad \times \left\{ c_{d, \ell_{\star}}^{(k_{d+1}, \ldots, k_{2d})} + O(\omega_m^{-1}) \right\} uO \left\{  \left(\frac{1}{m}\right)^{2\nu} \log m   \right\} \nonumber \\
&= \zeta_\nu^* \left(\frac{\omega_m}{m}\right)^{2\nu}|\Xi_{u,m}| H_{\ell_{\star},\nu} + \zeta_\nu^*\left(\frac{\omega_m}{m}\right)^{2\nu} |\Xi_{u,m}| O\left\{ \omega_m^{-1}  + u\omega_m^{-2\nu} \log m \right\} \nonumber \\
&= \theta\xi_\nu^* \left(\frac{\omega_m}{m}\right)^{2\nu}|\Xi_{u,m}| H_{\ell_{\star},\nu} + \theta\xi_\nu^*\left(\frac{\omega_m}{m}\right)^{2\nu} |\Xi_{u,m}| O( \omega_m^{-1} + u\omega_m^{-2\nu} \log m )  \\
&\asymp  \theta m^d \left(\frac{\omega_m}{m}\right)^{2\nu},  \nonumber
\end{align}
as $m\rightarrow \infty$. Combining \eqref{eq:F.udl2} with \eqref{eq:E1}, \eqref{eq:E2} and \eqref{eq:EV_u 2}, we conclude that
\begin{align}\label{eq:EV.case2}
&\quad ~ \EE \left(V_{u,d,\ell_{\star}}\right)  \nonumber \\
&= \tau(1-u) \sum_{\bfi\in \Xi_{0,m}} \sum_{0\leq k_1, \ldots, k_{d} \leq \ell_{\star}} \left(c_{\bfi,d,\ell_{\star}}^{(k_1,\ldots,k_d)}\right)^2  + \theta\xi_\nu^* \left(\frac{\omega_m}{m}\right)^{2\nu} |\Xi_{u,m}| H_{\ell_{\star},\nu}  \nonumber \\
&\quad + \theta\xi^*_\nu |\Xi_{u,m}|\left(\frac{\omega_m}{m}\right)^{2\nu} O(\omega_m^{-1}) + \theta \alpha^{2\ell_{\star}-2\nu}\{\xi_{1,\ell_{\star}}+\xi_{2,\ell_{\star}}\log(\alpha)\} |\Xi_{u,m}| (-1)^{\ell_{\star}}(2\ell_{\star})!\left(\frac{\omega_m}{m}\right)^{2\ell_{\star}}  \nonumber\\
&\quad -\theta \alpha^2\xi^*_{\nu+1}|\Xi_{u,m}| (-1)^{\ell_{\star}}(2\ell_{\star})!\left(\frac{\omega_m}{m}\right)^{2\ell_{\star}}\log \left(\frac{m}{\omega_m}\right) \Ical\{\ell_{\star}=\nu+1 \}  \nonumber\\
&\quad +\theta\left(\frac{\omega_m}{m}\right)^{2\nu}|\Xi_{u,m}|O(1)\Bigg\{\xi^*_{\nu+1}(\alpha \frac{\omega_m}{m})^2+ \sum_{j=\ell_{\star}+1}^{\infty} \xi_{1,j}\left(\alpha \frac{\omega_m}{m}\right)^{2j-2\nu}  \nonumber \\
&\quad+\sum_{j=\ell_{\star}+1}^{\infty} \log(\alpha)\xi_{2,j}\left(\alpha \frac{\omega_m}{m}\right)^{2j-2\nu} +\sum_{j=1}^{\infty}\xi^*_{\nu+j}(\alpha \frac{\omega_m}{m})^{2j}\log\left(\frac{m}{\omega_m}\right)  \Bigg\}  \nonumber \\
&\quad+O(1) \|\beta\|_1^2 |\Xi_{u,m}|\left(\frac{\omega_m}{m}\right)^{2\ell_{\star}},
\end{align}
as $m\rightarrow \infty$.

Hence, using the definitions of $C_{V,0},g_{\ell_{\star},\nu},H_{\ell_{\star},\nu}$ in \eqref{eq:CgH}, we can derive from \eqref{eq:EV.case2} that for $u=1$,
\begin{align}
\frac{\EE V_{1,d,\ell_{\star}}}{\theta g_{\ell_{\star},\nu}}
&= \frac{\EE V_{1,d,\ell_{\star}}}{\zeta_\nu^* \left(\frac{\omega_m}{m}\right)^{2\nu} |\Xi_{1,m}| H_{\ell_{\star},\nu}}  \nonumber \\
&= 1+ O(\omega_m^{-1})+O(1)\frac{\|\beta\|_1^2}{\theta}\left(\frac{\omega_m}{m}\right)^{2\ell_{\star}-2\nu} \nonumber \\
&\quad + O(1) \frac{\xi^*_{\nu+1}}{\xi^*_{\nu}}\left(\alpha \frac{\omega_m}{m}\right)^{2}\log\left(\frac{m}{\omega_m}\right) \Ical\{ \ell_{\star}=\nu+1\} \nonumber  \\
&\quad +  O(1) \Bigg[ \frac{\xi^*_{\nu+1}}{\xi^*_{\nu}}\left(\alpha \frac{\omega_m}{m}\right)^{2}+\sum_{j=\ell_{\star}}^{\infty}\frac{\xi_{1,j}}{\xi^*_{\nu}}
\left(\alpha\frac{\omega_m}{m}\right)^{2j-2\nu}  \nonumber \\	
&\quad +\sum_{j=\ell_{\star}}^{\infty}\frac{\xi_{2,j}}{\xi^*_{\nu}}\log(\alpha)\left(\alpha\frac{\omega_m}{m}\right)^{2j-2\nu}
+ \sum_{j=2}^{\infty}\frac{\xi^*_{\nu+j}}{\xi^*_{\nu}}\left(\alpha\frac{\omega_m}{m}\right)^{2j}\log\left(\frac{m}{\omega_m}\right) \Bigg] \label{eq:EV.case2.ratio1} \\
&\stackrel{(i)}{=} 1+ o(1), \nonumber
\end{align}
where the last relation (i) holds uniformly on the set $\Ecal_{n}$, because according to \eqref{eq:rho.relation1}, for all $(\theta,\alpha,\tau,\beta)\in \Ecal_{n}$, as $n\to\infty$,
\begin{align} \label{eq:case2.rho.order1}
& \frac{\|\beta\|_1^2}{\theta}\left(\frac{\omega_m}{m}\right)^{2\ell_{\star}-2\nu} \leq \frac{p\|\beta\|^2}{\theta}\left(\frac{\omega_m}{m}\right)^{2\ell_{\star}-2\nu} \leq pn^{\rho_{1} - (2\ell_{\star}-2\nu)(1-\gamma)/d} = o(1),  \nonumber \\
& \alpha \frac{\omega_m}{m} \leq n^{\rho_{32} - (1 - \gamma)/d} =o(1), \nonumber \\
& \left(\alpha \frac{\omega_m}{m} \right)^2 |\log \alpha| \leq n^{2\rho_{32} - 2(1 - \gamma)/d} \cdot \frac{\rho_{31}+\rho_{32}}{d}\log n = o(1),  \nonumber \\
& \left(\alpha \frac{\omega_m}{m} \right)^{2(\ell_{\star}-\nu)} \log \left(\frac{m}{\omega_m}\right) \leq   n^{2(\ell_{\star}-\nu) [2\rho_{32} - 2(1 - \gamma)/d]} \cdot \frac{1-\gamma}{d}\log n = o(1),
\end{align}
and that the three summations in \eqref{eq:EV.case2.ratio1} are all finite given the definition of $\xi_{1,j},\xi_{2,j},\xi^*_{\nu+j}$ for $j\in\NN$ in \eqref{eq:zeta.xi2}.

For $u=0$, we have that
\begin{align}
\frac{\EE \left(V_{0,d,\ell_{\star}}\right)-\tau C_{V,0}}{\zeta_\nu^* \left(\frac{\omega_m}{m}\right)^{2\nu} |\Xi_{0,m}| H_{\ell_{\star},\nu}}
&= 1+ O(\omega_m^{-1})+O(1)\frac{\|\beta\|_1^2}{\theta}\left(\frac{\omega_m}{m}\right)^{2\ell_{\star}-2\nu}  \nonumber  \\
&\quad + O(1) \frac{\xi^*_{\nu+1}}{\xi^*_{\nu}}\left(\alpha \frac{\omega_m}{m}\right)^{2}\log\left(\frac{m}{\omega_m}\right)\Ical\{ \ell_{\star}=\nu+1\}  \nonumber  \\
&\quad +  O(1) \Bigg[ \frac{\xi^*_{\nu+1}}{\xi^*_{\nu}}\left(\alpha \frac{\omega_m}{m}\right)^{2}+\sum_{j=\ell_{\star}}^{\infty}\frac{\xi_{1,j}}{\xi^*_{\nu}}
\left(\alpha\frac{\omega_m}{m}\right)^{2j-2\nu} \nonumber \\	
&\quad +\sum_{j=\ell_{\star}}^{\infty}\frac{\xi_{2,j}}{\xi^*_{\nu}}\log(\alpha)\left(\alpha\frac{\omega_m}{m}\right)^{2j-2\nu}
 + \sum_{j=2}^{\infty}\frac{\xi^*_{\nu+j}}{\xi^*_{\nu}}\left(\alpha\frac{\omega_m}{m}\right)^{2j}
 \log\left(\frac{m}{\omega_m}\right) \Bigg]. \nonumber
\end{align}
Therefore,
\begin{align}
\frac{\EE \left(V_{0,d,\ell_{\star}}\right)}{\tau C_{V,0}} &= \frac{\EE \left(V_{0,d,\ell_{\star}}\right)-\tau C_{V,0}}{\zeta_\nu^* \left(\frac{\omega_m}{m}\right)^{2\nu} |\Xi_{0,m}| H_{\ell_{\star},\nu}} \frac{\zeta_\nu^* \left(\frac{\omega_m}{m}\right)^{2\nu} |\Xi_{0,m}| H_{\ell_{\star},\nu}}{\tau C_{V,0}} + 1 \nonumber \\
&= 1+ O(1)\frac{\theta}{\tau} \left(\frac{\omega_m}{m}\right)^{2\nu} \Bigg\{1+ O(\omega_m^{-1})+O(1)\frac{\|\beta\|_1^2}{\theta}\left(\frac{\omega_m}{m}\right)^{2\ell_{\star}-2\nu}  \nonumber \\
&\quad + O(1) \frac{\xi^*_{\nu+1}}{\xi^*_{\nu}}\left(\alpha \frac{\omega_m}{m}\right)^{2}\log\left(\frac{m}{\omega_m}\right)\Ical\{ \ell_{\star}=\nu+1\} \nonumber \\
&\quad +  O(1) \Bigg[ \frac{\xi^*_{\nu+1}}{\xi^*_{\nu}}\left(\alpha \frac{\omega_m}{m}\right)^{2}+\sum_{j=\ell_{\star}}^{\infty}\frac{\xi_{1,j}}{\xi^*_{\nu}}
\left(\alpha\frac{\omega_m}{m}\right)^{2j-2\nu} \nonumber \\ &\quad+\sum_{j=\ell_{\star}}^{\infty}\frac{\xi_{2,j}}{\xi^*_{\nu}}\log(\alpha)\left(\alpha\frac{\omega_m}{m}\right)^{2j-2\nu}
+ \sum_{j=2}^{\infty}\frac{\xi^*_{\nu+j}}{\xi^*_{\nu}}\left(\alpha\frac{\omega_m}{m}\right)^{2j}
\log\left(\frac{m}{\omega_m}\right) \Bigg]  \Bigg\} \label{eq:EV.case2.ratio2} \\
&\stackrel{(i)}{=} 1+ o(1), \nonumber
\end{align}	
as $m\to\infty$, where the relation (i) holds uniformly on the set $\Ecal_{n}$ because of the relations in \eqref{eq:case2.rho.order1} on $\Ecal_{n}$,
\begin{align*}
& \frac{\theta}{\tau} \left(\frac{\omega_m}{m}\right)^{2\nu} \leq n^{\rho_{21} - 2\nu(1-\gamma)/d} = o(1)
\end{align*}
on $\Ecal_{n}$, and that the three summations in \eqref{eq:EV.case2.ratio2} are all finite given the definition of $\xi_{1,j},\xi_{2,j},\xi^*_{\nu+j}$ for $j\in\NN$ in \eqref{eq:zeta.xi2}.

\subsection{Uniform Error bounds for $V_{u,d,\ell_{\star}}/\EE\left(V_{u,d,\ell_{\star}}\right)-1$ on $\Ecal_{n}$} \label{subsec:V.EV}
We consider the two cases $u=0$ and $u=1$ separately. The $u=0$ case is used for showing the convergence of $\widehat\tau_n$ in \eqref{eq:tau.theta.es} and the $u=1$ case is used for showing the convergence of $\widehat\theta_n$ in \eqref{eq:tau.theta.es}.
\vspace{2mm}

\noindent {\sc Case B1.} If $u=0$, we write $\widetilde{W}=\big(W_1,\ldots,W_{|\Xi_{0,m}|}\big)^\T  $ where
\begin{align*}
\left\{W_1,\ldots,  W_{|\Xi_{0,m}|}\right\} &= \left\{ \frac{\nabla_{d,\ell_{\star}} Y\left(\bfs(\bfi)\right)}{\sqrt{ \EE  \left(V_{0,d,\ell_{\star}}\right)}} :
\bfi=(i_1,\ldots, i_d)^\T   \in \Xi_{0,m} \right\}.
\end{align*}
Define $\mu_W=\EE \big(\widetilde{W}\big)$, $\Sigma_W= \EE \Big[\big(\widetilde{W}-\mu_W\big)\big(\widetilde{W}-\mu_W\big)^\T   \Big]$. We can write
\begin{align} \label{eq:V0.decomp}
\frac{V_{0,d,\ell_{\star}}}{\EE \left(V_{0,d,\ell_{\star}}\right)}= \widetilde{Z}^\T  \Sigma_W \widetilde{Z}+2\mu^\T  _W\Sigma_W^{1/2}\widetilde{Z}+\mu_W^\T  \mu_W,
\end{align}
where $\widetilde{Z} = (Z_1,\ldots, Z_{|\Xi_{0,m}|} )^\T   \sim \Ncal \left({0}_{|\Xi_{0,m}|}, I_{|\Xi_{0,m}|}\right)$.

Therefore, using the upper bounds of $\mu_W^\T  \mu_W$ and $\|\Sigma_W\|_F$ proved in the later Section \ref{subsec:Sigma.Fnorm}, \eqref{eq:E2}, \eqref{eq:EV.case2.ratio2} and \eqref{eq:tau.CV0} imply that on the set $\Ecal_{n}$, as $m\to\infty$ (equivalently $n=m^d\to\infty$),
\begin{align}\label{eq:mu2.SigmaW.o1}
\mu_W^\T  \mu_W &=  \frac{1}{\EE  (V_{0,d,\ell_{\star}})}\sum_{\bfi \in \Xi_{0,m}}\sum_{0\leq k_1, \ldots, k_{2d}\leq \ell_{\star}} c_{\bfi, d, \ell_{\star}}^{(k_1, \ldots, k_d)}c_{\bfi,  d, \ell_{\star}}^{(k_{d+1}, \ldots, k_{2d})} \mm \big(\bfs(\bfi+\bfk_1\omega_m) \big) \mm \big(\bfs(\bfi+ \bfk_2 \omega_m) \big) \nonumber \\
&=   O(1) \frac{\|\beta\|_1^2}{\tau}\left(\frac{\omega_m}{m}\right)^{2\ell_{\star}}  \nonumber \\
&\leq O(1)\frac{p\|\beta\|^2}{\theta} \frac{\theta}{\tau} \left(\frac{\omega_m}{m}\right)^{2\ell_{\star}}  \nonumber \\
&\leq C n^{\rho_{1} + \rho_{21} - 2(1-\gamma)\ell_{\star}/d} = o(1) , \nonumber \\
\|\Sigma_W\|_F^2  &\leq  Cm^{-d} + C\frac{\theta^2}{\tau^2}\left(\frac{\omega_m}{m}\right)^{4\nu+d} \left(1 + [1+|\log(\alpha)|]^2 \right) \nonumber \\
&\leq C \left\{n^{-1} + n^{2\rho_{21} - (1-\gamma)(4\nu+d)/d} \log^2 n \right\} =o(1) ,
\end{align}
where the two $o(1)$'s are based on the order of $\omega_m=\lfloor m^{\gamma} \rfloor$, the upper bounds of $\|\beta_1\|^2/\theta$, $\tau/\theta$ and $\alpha$ in the set $\Ecal_{n}$, and the condition \eqref{eq:rho.relation1}.
Hence $\mu_W^\T  \Sigma_W\mu_W\le \mu_W^\T  \mu_W\big\|\Sigma_W\big\|_F=o(1)$ as well.

Notice that taking expectation on both sides of \eqref{eq:V0.decomp} implies that $1= \EE\left(\widetilde{Z}^\T  \Sigma_W \widetilde{Z}\right) + \mu_W^\T  \mu_W$, and hence $\EE\left(\widetilde{Z}^\T  \Sigma_W \widetilde{Z}\right) =1- \mu_W^\T  \mu_W$. We conclude that for any $\epsilon>0$, for all sufficiently large $n$,
\begin{align} \label{eq:V0.concentration}
&\quad ~ \PP\left(\left|\frac{V_{0,d,\ell_{\star}}}{\EE V_{0,d,\ell_{\star}}}-1\right|>\epsilon\right) \nonumber \\
&\leq  \PP\left(\left|\widetilde{Z}^\T  \Sigma_W\widetilde{Z} + \mu_W^\T  \mu_W - 1\right| > \frac{\epsilon}{2}\right)  +  \PP\left(\left|\mu_W^\T  \Sigma_W^{1/2}\widetilde{Z}\right| > \frac{\epsilon}{4}\right) \nonumber \\
&=\PP\left(\left|\widetilde{Z}^\T  \Sigma_W\widetilde{Z}- \EE\left(\widetilde{Z}^\T  \Sigma_W\widetilde{Z}\right)\right| > \frac{\epsilon}{4}\right) + \PP\left(\left|\mu_W^\T  \Sigma_W^{1/2}\widetilde{Z}\right| > \frac{\epsilon}{4} \right) \nonumber \\
&\stackrel{(i)}{\leq}  2\exp\left\{-\frac{C_{\HW}}{16}\min\left( \frac{\epsilon^2}{\|\Sigma_W\|_F^2},\frac{\epsilon}{\|\Sigma_W\|_F} \right) \right\}
+ 2 \exp\left\{-\frac{\epsilon^2}{16\mu_W^\T  \Sigma_W\mu_W} \right\} \nonumber \\
&\leq 2\exp\left\{-\frac{C_{\HW}}{16}\min\left( \frac{\epsilon^2}{\|\Sigma_W\|_F^2},\frac{\epsilon}{\|\Sigma_W\|_F} \right) \right\}  +  2\exp\left\{-\frac{\epsilon^2}{16\mu_W^\T  \mu_W\big\|\Sigma_W\big\|_F} \right\} \nonumber \\
&\stackrel{(ii)}{\leq} 2\exp\left\{-C \varphi\left(\min\left\{n^{1/2}, n^{(1-\gamma)(4\nu+d)/(2d) - \rho_{21}} /\log n\right\} \epsilon \right) \right\} \nonumber \\
&\quad +  2\exp\left\{- Cn^{2(1-\gamma)\ell_{\star}/d - \rho_{1} -\rho_{21}} \min\left\{n^{1/2} , n^{(1-\gamma)(4\nu+d)/(2d) - \rho_{21}} /\log n \right\}  \epsilon^2 \right\} \nonumber \\
&\leq 2\exp\Big\{ -C \varphi\Big( \min\Big\{ n^{1/2}, ~~ n^{(1-\gamma)(4\nu+d)/(2d) - \rho_{21}} /\log n,  \nonumber \\
&\qquad  n^{1/4 + (1-\gamma)\ell_{\star}/d - (\rho_1+\rho_{21})/2} , ~~ n^{(1-\gamma)(4\nu+d+4\ell_{\star})/(4d) - (\rho_1+2\rho_{21})/2} /\log^{1/2} n \Big\}  \epsilon \Big)\Big\},
\end{align}
for some constant $C>0$, where in the inequality (i), the first term follows from the Hanson-Wright inequality in Lemma \ref{lem:Hsuetal12} and $\|\Sigma_W\|_{\op}\leq \|\Sigma_W\|_F$, and the second term follows from the sub-Gaussian concentration inequality for the Gaussian random variable $\mu_W^\T  \Sigma_W^{1/2}\widetilde{Z}$; the inequality (ii) follows from \eqref{eq:mu2.SigmaW.o1} and \eqref{eq:Sigma.W.5} in Section \ref{subsec:Sigma.Fnorm}.
\vspace{3mm}
	
{\sc Case B2.} If $u=1$, we write $\widetilde{U} = \big(U_1,\ldots, U_{2|\Xi_{1,m}|}\big)^\T  $,  where
\begin{align*}
\left\{ U_1,\ldots, U_{|\Xi_{1,m}| } \right\} &=   \left\{ \frac{\nabla_{d,\ell_{\star}} Y(\bfs(\bfi))}{\sqrt{ \EE  \left(V_{1,d,\ell_{\star}}\right)}} :  \bfi=(i_1,\ldots, i_d)^\T   \in \Xi_{1,m} \right\}, \nonumber \\
\left\{ U_{|\Xi_{1,m}| +1},\ldots, U_{2 |\Xi_{1,m}| } \right\} &=   \left\{ \frac{\nabla_{d,\ell_{\star}} Y(\bfs(\bfi+\bfe_1))}{\sqrt{\EE \left(V_{1,d,\ell_{\star}}\right)}} : \bfi=(i_1,\ldots, i_d)^\T   \in \Xi_{1,m} \right\}.
\end{align*}
Define the  $2|\Xi_{1,m}| \times 2 |\Xi_{1,m}|$ symmetric matrix $A = (a_{i,k})_{1\leq i, k\leq 2|\Xi_{1,m}|}$ by
\begin{align*}
a_{j, |\Xi_{1,m}| +j} &=   a_{|\Xi_{1,m}| + j, j} = 1/2, \qquad \text{for } j= 1,\ldots, |\Xi_{1,m}|,\nonumber \\
a_{j,k} &=   0, \qquad \mbox{otherwise.}
\end{align*}
Then
\begin{align}
\frac{ V_{1,d,\ell_{\star}} }{\EE  ( V_{1,d,\ell_{\star}} )} = \widetilde{U}^\T   A \widetilde{U}. \nonumber
\end{align}
Define $\mu_U=\EE \left(\widetilde{U}\right)$ and $\Sigma_U = \EE \left[(\widetilde{U}-\mu_U) (\widetilde{U}-\mu_U)^\T  \right]$.
We observe that
\begin{align}\label{eq:V1.decomp}
\frac{ V_{1,d,\ell_{\star}} }{\EE  \left(V_{1,d,\ell_{\star}} \right)}  = \widetilde{Z}^\T   \Sigma_U^{1/2} A \Sigma_U^{1/2} \widetilde{Z}+ 2\mu_U^\T  A \Sigma_U^{1/2}\widetilde{Z} + \mu_U^\T  A A^\T  \mu_U,
\end{align}
where $\widetilde{Z} = \big(Z_1,\ldots, Z_{2|\Xi_{1,m}|} \big)^\T   \sim \Ncal \left(0_{|\Xi_{0,m}|}, I_{2 |\Xi_{1,m}|}\right)$.
	
Therefore, using the upper bounds of $\mu_U^\T  AA^\T  \mu_U$, $\left\| \Sigma^{1/2}_U A \Sigma^{1/2}_U \right\|_F$, and $\left\|\Sigma_U \right\|_F$ proved in the later Section \ref{subsec:Sigma.Fnorm}, and from \eqref{eq:E2}, \eqref{eq:EV.case1.ratio1} and \eqref{eq:EV.case2.ratio1}, we have that on the event $\Ecal_{n}$, as $m\to\infty$ (equivalently $n=m^d\to\infty$),
\begin{align}\label{eq:mu2.SigmaU.o1}
\mu_U^\T  AA^\T  \mu_U &=  \frac{1}{4\EE \left(V_{1,d,\ell_{\star}}\right)}\sum_{\bfi \in \Xi_{1,m}}\Bigg\{\sum_{0\leq k_1, \ldots, k_{2d}\leq \ell_{\star}}   c_{\bfi, d, \ell_{\star}}^{(k_1, \ldots, k_d)}c_{\bfi,  d, \ell_{\star}}^{(k_{d+1}, \ldots, k_{2d})} \nonumber\\
&\quad\times \mm \big(\bfs(\bfi+\bfk_1\omega_m) \big) \mm \big(\bfs(\bfi+ \bfk_2 \omega_m) \big)  \nonumber \\
& \quad +\sum_{0\leq k_1, \ldots, k_{2d}\leq \ell_{\star}}c_{\bfi+\bfe_1, d, \ell_{\star}}^{(k_1, \ldots, k_d)}c_{\bfi+ \bfe_1,  d, \ell_{\star}}^{(k_{d+1}, \ldots, k_{2d})} \mm \big(\bfs(\bfi+\bfe_1+\bfk_1\omega_m) \big) \mm \big(\bfs(\bfi+\bfe_1+ \bfk_2 \omega_m) \big)\Bigg\} \nonumber \\
&=  O(1)\frac{\|\beta\|_1^2}{\theta}\left(\frac{\omega_m}{m}\right)^{2\ell_{\star}-2\nu}   \nonumber \\
&\leq O(1)\frac{p\|\beta\|^2}{\theta}\left(\frac{\omega_m}{m}\right)^{2\ell_{\star}-2\nu}  \nonumber \\
&\leq C n^{\rho_{1} - (1-\gamma)(2\ell_{\star}-2\nu)/d} = o(1) , \nonumber \\
\left\| \Sigma^{1/2}_U A \Sigma^{1/2}_U \right\|_F^2  &\leq C \Bigg\{\frac{\tau^2}{\theta^2} m^{-d}\left(\frac{\omega_m}{m}\right)^{-4\nu} + \frac{\tau}{\theta} m^{-d} \left(\frac{\omega_m}{m}\right)^{-2\nu}  \nonumber \\
&\qquad + \left(1+ [1+|\log(\alpha)|]^2\right)\left(\frac{\omega_m}{m}\right)^{d} \Bigg\} \nonumber \\
&\leq C\left\{n^{2\rho_{22} - 1 + 4(1-\gamma)\nu/d}  + n^{- (1-\gamma)} \log^2 n\right\} = o(1), \nonumber \\
\left\|\Sigma_U \right\|_F^2 &\leq  C \Bigg\{\frac{\tau^2}{\theta^2} m^{-d}\left(\frac{\omega_m}{m}\right)^{-4\nu} + \frac{\tau}{\theta} m^{-d} \left(\frac{\omega_m}{m}\right)^{-2\nu}  \nonumber \\
&\qquad + \left(1+[1+|\log(\alpha)|]^2\right) \left(\frac{\omega_m}{m}\right)^{d} \Bigg\} \nonumber \\
&\leq C\left\{ n^{2\rho_{22} - 1 + 4(1-\gamma)\nu/d}  + n^{- (1-\gamma)} \log^2 n \right\}  = o(1),
\end{align}
where the three $o(1)$'s are based on the order of $\omega_m=\lfloor m^{\gamma} \rfloor$, the upper bounds of $\|\beta_1\|^2/\theta$, $\tau/\theta$ and $\alpha$ in the set $\Ecal_{n}$, and the condition \eqref{eq:rho.relation1}.
Hence $\mu_U^\T  A \Sigma_U A^\T  \mu_U\leq \mu_U^\T  AA^\T  \mu_U \big\|\Sigma_U\big\|_F=o(1)$ as well.
Notice that taking expectation on both sides of \eqref{eq:V1.decomp} implies that $1= \EE\left(\widetilde{Z}^\T  \Sigma_U^{1/2}A \Sigma_U^{1/2} \widetilde{Z}\right) + \mu_U^\T  AA^\T  \mu_U$, and hence $\EE\left(\widetilde{Z}^\T  \Sigma_U^{1/2}A \Sigma_U^{1/2} \widetilde{Z}\right) = 1- \mu_U^\T  AA^\T  \mu_U$. Then similar to the derivation of \eqref{eq:V0.concentration}, we conclude from \eqref{eq:mu2.SigmaU.o1}, \eqref{eq:Sigma.U.3} and \eqref{eq:Sigma.U.4} in Section \ref{subsec:Sigma.Fnorm} that for any $\epsilon>0$, for all sufficiently large $n$,
\begin{align} \label{eq:V1.concentration}
&\quad~ \PP \left(\left|\frac{V_{1,d,\ell_{\star}}}{\EE \left(V_{1,d,\ell_{\star}}\right)}-1\right|>\epsilon\right)  \nonumber \\
&\leq  \PP\left(\left|\widetilde{Z}^\T  \Sigma_U^{1/2}A\Sigma_U\widetilde{Z}-\mu_U^\T  AA^\T  \mu_U-1\right| >\frac{\epsilon}{2}\right) +  \PP\left(\left|\mu^\T _U A\Sigma_U^{1/2}\widetilde{Z}\right|>\frac{\epsilon}{4}\right)  \nonumber \\
&= \PP\left(\left|\widetilde{Z}^\T  \Sigma_U^{1/2}A\Sigma_U\widetilde{Z}- \EE\left(\widetilde{Z}^\T  \Sigma_U^{1/2}A\Sigma_U\widetilde{Z}\right)\right|>\frac{\epsilon}{2}\right) + \PP\left(\left|\mu^\T _U A\Sigma_U^{1/2}\widetilde{Z}\right|>\frac{\epsilon}{4}\right) \nonumber \\
&\leq  2\exp\left\{- \frac{C_{\HW}}{16} \min\left( \frac{\epsilon^2}{\big\| \Sigma^{1/2}_U A \Sigma^{1/2}_U \big\|_F^2}, \frac{\epsilon}{\big\| \Sigma^{1/2}_U A \Sigma^{1/2}_U \big\|_F} \right) \right\}  \nonumber \\
&\quad + 2 \exp\left\{-\frac{\epsilon^2}{16\mu_U^\T  A\Sigma_U A^\T  \mu_U} \right\}  \nonumber \\
&\leq  2\exp\left\{- \frac{C_{\HW}}{16}\min\left( \frac{\epsilon^2}{\big\| \Sigma^{1/2}_U A \Sigma^{1/2}_U \big\|_F^2}, \frac{\epsilon}{\big\| \Sigma^{1/2}_U A \Sigma^{1/2}_U \big\|_F} \right) \right\}  \nonumber \\
&\quad + 2 \exp\left\{-\frac{\epsilon^2}{16\mu_U^\T  A A^\T  \mu_U \left\|\Sigma_U\right\|_F} \right\}  \nonumber \\
&\stackrel{(i)}{\leq}  2\exp\left\{- C \varphi\left(\min\left\{n^{1/2 - \rho_{22} - 2(1-\gamma)\nu/d } , n^{(1-\gamma)/2 } /\log n \right\} \epsilon\right) \right\}  \nonumber \\
&\quad + 2 \exp\left\{-C n^{(1-\gamma)(2\ell_{\star}-2\nu)/d - \rho_{1}} \min\left\{ n^{ 1/2 - 2(1-\gamma)\nu/d - \rho_{22}} , n^{(1-\gamma)/2} /\log n \right\}\epsilon^2 \right\}  \nonumber \\
&\leq  2\exp\Big\{- C \varphi\Big(\min\Big\{ n^{1/2 - \rho_{22} - 2(1-\gamma)\nu/d } , ~~ n^{(1-\gamma)/2} /\log n ,  \nonumber \\
&\qquad   n^{ 1/4 + (1-\gamma)(\ell_{\star}-2\nu)/d - (\rho_{1} + \rho_{22})/2} , ~~
n^{(1-\gamma)/4 + (1-\gamma)(\ell_{\star}-\nu)/d - \rho_{1}/2 }  /\log^{1/2} n \Big\} \epsilon \Big) \Big\} ,
\end{align}
for some constant $C>0$, where the inequality (i) follows from \eqref{eq:Sigma.U.3} and \eqref{eq:Sigma.U.4} in Section \ref{subsec:Sigma.Fnorm}.

\subsection{Bounds for the Frobenius Norms of $\big\|\Sigma_W\big\|_F^2$, $\big\| \Sigma^{1/2}_U A \Sigma^{1/2}_W \big\|_F^2$ and $\big\|\Sigma_U\big\|_F^2$} \label{subsec:Sigma.Fnorm}
This section provides the detailed derivation of upper bounds for $\big\|\Sigma_W\big\|_F^2$, $\big\| \Sigma^{1/2}_U A \Sigma^{1/2}_W \big\|_F^2$, and $\big\|\Sigma_U\big\|_F^2$.
\vspace{3mm}

\noindent \underline{(1) Upper bound for $\|\Sigma_W\|_F^2$.}
\vspace{2mm}

We first use $(a+b)^2\leq 2(a^2+b^2)$ for any $a,b\in \RR$ to obtain that
\begin{align} \label{eq:Sigma.W.1}
\|\Sigma_W\|_F^2&=  \frac{1}{\left\{\EE \left(V_{0,d,\ell_{\star}}\right)\right\}^2} \sum_{\bfi, \bfj\in \Xi_{0,m}}
\Bigg\{ \sum_{0\leq k_1, \ldots, k_{2d} \leq \ell_{\star}} c_{\bfi,d,\ell_{\star}}^{(k_1,\ldots,k_d)}
c_{\bfj,d,\ell_{\star}}^{(k_{d+1},\ldots,k_{2d})} \nonumber \\
& \quad \times \Big[ \tau \Ical \left\{ \bfs(\bfi+ \bfk_1 \omega_m  )=  \bfs (\bfj+ \bfk_2 \omega_m  ) \right\}  + \theta K_{\alpha,\nu}  \left( \bfs(\bfi+ \bfk_1 \omega_m  )-  \bfs (\bfj+ \bfk_2 \omega_m  ) \right) \Big] \Bigg\}^2 \nonumber \\
&\leq  \frac{2\tau^2}{\left\{\EE \left(V_{0,d,\ell_{\star}}\right)\right\}^2} \sum_{\bfi, \bfj\in \Xi_{0,m}}
\Bigg\{ \sum_{0\leq k_1, \ldots, k_{2d} \leq \ell_{\star}} c_{\bfi,  d,\ell_{\star}}^{(k_1,\ldots,k_d)}
c_{\bfj,  d,\ell_{\star}}^{(k_{d+1},\ldots,k_{2d})} \nonumber \\
&\qquad \times \Ical \{ \bfs(\bfi+ \bfk_1 \omega_m )=  \bfs (\bfj+ \bfk_2 \omega_m  ) \} \Bigg\}^2 \nonumber \\
&\quad  + \frac{2}{\left\{\EE \left(V_{0,d,\ell_{\star}}\right)\right\}^2}  \sum_{\bfi, \bfj\in \Xi_{0,m}} \Bigg\{ \sum_{0\leq k_1, \ldots, k_{2d} \leq \ell_{\star}} c_{\bfi,  d,\ell_{\star}}^{(k_1,\ldots,k_d)} c_{\bfj,  d,\ell_{\star}}^{(k_{d+1},\ldots,k_{2d})}  \nonumber \\
&  \qquad\times \theta K_{\alpha,\nu}  \left( \bfs(\bfi+ \bfk_1 \omega_m )-  \bfs (\bfj+ \bfk_2 \omega_m  ) \right) \Bigg\}^2.
\end{align}

For the first term in \eqref{eq:Sigma.W.1}, from Lemma \ref{lem:Cor2.Loh21}, we observe that as $m\to\infty$,
\begin{align} \label{eq:Sigma.W.2}
&\quad~ \sum_{\bfi, \bfj\in \Xi_{0,m}} \left\{ \sum_{0\leq k_1, \ldots, k_{2d} \leq \ell_{\star}} c_{\bfi, d,\ell_{\star}}^{(k_1,\ldots,k_d)} c_{\bfj,  d,\ell_{\star}}^{(k_{d+1},\ldots,k_{2d})} \Ical \{ \bfs(\bfi+ \bfk_1 \omega_m  )=  \bfs (\bfj+ \bfk_2 \omega_m  ) \} \right\}^2  \nonumber \\
&= \sum_{\bfi, \bfj\in \Xi_{0,m}} \Bigg\{ \sum_{0\leq k_1, \ldots, k_{2d} \leq \ell_{\star}} \left[c_{d,\ell_{\star}}^{(k_1,\ldots,k_d)}+O(\omega_m^{-1}) \right] \left[c_{d,\ell_{\star}}^{(k_{d+1},\ldots,k_{2d})} +O(\omega_m^{-1}) \right] \nonumber \\
&  \qquad\times \Ical \{ \bfs(\bfi+ \bfk_1 \omega_m  )=  \bfs (\bfj+ \bfk_2 \omega_m  ) \} \Bigg\}^2 \nonumber \\
&\leq  C m^d,
\end{align}
for some $C>0$ that only depends on $\ell_{\star},\nu,d$.

For the second term in \eqref{eq:Sigma.W.1}, we consider two cases, $\|\bfi-\bfj\|\leq 5\ell_{\star}\omega_m\sqrt{d}$ and $\|\bfi-\bfj\|>5\ell_{\star}\omega_m\sqrt{d}$. For the first case, the number of such $(\bfi,\bfj)$ pairs is at most of order $m^d \omega_m^d$, which implies that
\begin{align} \label{eq:Sigma.W.3}
&\quad~ \sum_{\bfi, \bfj\in \Xi_{0,m}: \| \bfi - \bfj \| \leq 5\ell_{\star} \omega_m \sqrt{d}} \Bigg\{ \sum_{0\leq k_1, \ldots, k_{2d} \leq \ell_{\star}} c_{\bfi,d,\ell_{\star}}^{(k_1,\ldots,k_d)} c_{\bfj, d,\ell_{\star}}^{(k_{d+1},\ldots,k_{2d})} \nonumber \\
&  \qquad\times \theta K_{\alpha,\nu}  \left( \bfs( \bfi+ \bfk_1 \omega_m  )- \bfs (\bfj+ \bfk_2 \omega_m ) \right) \Bigg\}^2 \nonumber \\
&= \sum_{\bfi, \bfj\in \Xi_{0,m}: \| \bfi - \bfj \| \leq 5\ell_{\star} \omega_m \sqrt{d}} \Big\{ \zeta_{\ell_{\star}}(-1)^{\ell_{\star}}(2\ell_{\star})!\left(\frac{\omega_m}{m}\right)^{2\ell_{\star}} \nonumber \\
&\qquad  + \sum_{0\leq k_1, \ldots, k_{2d} \leq \ell_{\star}} c_{\bfi, d,\ell_{\star}}^{(k_1,\ldots,k_d)}c_{\bfj,d,\ell_{\star}}^{(k_{d+1},\ldots,k_{2d})} \nonumber \\
&\qquad \times \Bigg[\sum_{j=\ell_{\star}+1}^{\infty}\zeta_{j} \|\bfs( \bfi+ \bfk_1 \omega_m  )-  \bfs (\bfj+ \bfk_2 \omega_m ) \|^{2j} \nonumber \\
& \qquad+ \sum_{j=0}^{\infty} \zeta^*_{\nu+j} G_{\nu+j}(\|\bfs( \bfi+ \bfk_1 \omega_m  )- \bfs (\bfj+ \bfk_2 \omega_m ) \|) \Bigg] \Bigg\}^2 \nonumber \\
&\leq  C \theta^2 m^d \omega_m^d\left(\frac{\omega_m}{m}\right)^{4\nu},
\end{align}
as $m\to\infty$ for some $C>0$ dependent only on $\ell_{\star},\nu,d$.

For the second case of $\|\bfi-\bfj\|>5\ell_{\star}\omega_m\sqrt{d}$, if we let $\sa=(a_1,\ldots,a_d)\in \NN^d$, then we have that
\begin{align}\label{eq:Sigma.W.4}
&\quad~  \sum_{\bfi, \bfj\in \Xi_{0,m}: \| \bfi-\bfj\| > 5 \ell_{\star} \omega_m \sqrt{d}}
\Bigg\{  \sum_{0\leq k_1,\ldots, k_{2d}\leq \ell_{\star}}  c_{\bfi, d, \ell_{\star}}^{(k_1,\ldots, k_d)}  c_{\bfj,  d, \ell_{\star}}^{(k_{d+1},\ldots, k_{2d})} \nonumber \\
&\qquad \times \theta K_{\alpha,\nu} \left(\bfs( \bfi + \bfk_1 \omega_m ) - \bfs(\bfj + \bfk_2 \omega_m ) \right) \Bigg\}^2 \nonumber \\
&= \sum_{\bfi, \bfj\in \Xi_{0,m}: \| \bfi-\bfj\| > 5 \ell_{\star} \omega_m \sqrt{d}} \Bigg\{\sum_{0\leq k_1,\ldots, k_{2d}\leq \ell_{\star}}  c_{\bfi,d,\ell_{\star}}^{(k_1,\ldots, k_d)}  c_{\bfj,d,\ell_{\star}}^{(k_{d+1},\ldots, k_{2d})} \nonumber \\
&\quad \times \Bigg[ \sum_{0\leq a_1+\ldots + a_d \leq 2\ell_{\star} -1} \frac{\sD^{\sa} \theta K_{\alpha,\nu}  (\bfs( \bfi ) - \bfs(\bfj ) ) }{a_1! \ldots a_d!} \nonumber \\
&\qquad \times \prod_{q=1}^d \left[ s_q(\bfi +\bfk_1 \omega_m ) - s_q(\bfj +\bfk_2 \omega_m ) - s_q (\bfi) + s_q(\bfj) \right]^{a_q} \nonumber \\
&\qquad  + \sum_{a_1+\ldots + a_d = 2\ell_{\star}} \frac{2 \ell_{\star}}{a_1! \ldots a_d!}  \prod_{q=1}^d \left[s_q(\bfi +\bfk_1 \omega_m ) - s_q(\bfj +\bfk_2 \omega_m ) - s_q (\bfi) + s_q(\bfj)\right]^{a_q} \nonumber \\
&\qquad \times \int_0^1 (1-t)^{2 \ell_{\star}-1} \sD^{\sa} \theta K_{\alpha,\nu}  \Big( \bfs (\bfi) - \bfs(\bfj) \nonumber \\
&\qquad + t \{ \bfs (\bfi +\bfk_1 \omega_m ) - \bfs (\bfj +\bfk_2 \omega_m) -\bfs(\bfi) + \bfs(\bfj) \} \Big) \ud t \Bigg] \Bigg\}^2 \nonumber \\
&= \sum_{\bfi, \bfj\in \Xi_{0,m}: \| \bfi-\bfj\| > 5 \ell_{\star} \omega_m \sqrt{d}} \Bigg\{
   \sum_{0\leq k_1,\ldots, k_{2d}\leq \ell_{\star}} c_{\bfi, d, \ell_{\star}}^{(k_1,\ldots, k_d)}  c_{\bfj,  d, \ell_{\star}}^{(k_{d+1},\ldots, k_{2d})} \nonumber \\
&\quad  \times \sum_{a_1+\ldots + a_d = 2\ell_{\star}} \frac{2 \ell_{\star}}{a_1! \ldots a_d!}  \prod_{q=1}^d \left[s_q(\bfi +\bfk_1 \omega_m ) - s_q(\bfj +\bfk_2 \omega_m ) -s_q (\bfi) + s_q(\bfj) \right]^{a_q}  \nonumber \\
&\qquad  \times \int_0^1 (1-t)^{2 \ell_{\star}-1} \sD^{\sa}  \theta K_{\alpha,\nu} \Big( \bfs (\bfi) - \bfs(\bfj) \nonumber \\
&  \qquad + t \{ \bfs (\bfi +\bfk_1 \omega_m) - \bfs (\bfj +\bfk_2 \omega_m )
	-\bfs(\bfi) + \bfs(\bfj) \} \Big) \ud t \Bigg\}^2 \nonumber \\
&\stackrel{(i)}{\leq} C \left(\frac{\omega_m}{m}\right)^{4 \ell_{\star}}  \sum_{\bfi, \bfj\in \Xi_{0,m}: \| \bfi-\bfj\| > 5 \ell_{\star} \omega_m \sqrt{d}} \Bigg\{\sum_{0\leq k_1,\ldots, k_{2d}\leq \ell_{\star}} \sup_{\substack{0\leq t\leq 1 \\ a_1+\ldots a_d=2\ell_{\star}}} \sD^{(a_1,\ldots,a_d)} \theta K_{\alpha,\nu}  \Big( \bfs (\bfi) - \bfs(\bfj) \nonumber \\
&  \qquad  + t \{ \bfs (\bfi +\bfk_1 \omega_m ) - \bfs (\bfj +\bfk_2 \omega_m ) -\bfs(\bfi) + \bfs(\bfj) \} \Big) \Bigg\}^2 \nonumber \\
&\stackrel{(ii)}{\leq}  C \theta^2 (1+|\log(\alpha)|)^2 \left(\frac{\omega_m}{m}\right)^{4 \ell_{\star}}  \nonumber\\
&\quad \times \sum_{\bfi, \bfj\in \Xi_{0,m}: \| \bfi-\bfj\| > 5 \ell_{\star} \omega_m \sqrt{d}}  \left\{ \|\bfs (\bfi) - \bfs(\bfj) \| -\frac{ (4\ell_{\star} \omega_m +2) \sqrt{d} }{m} \right\}^{4\nu -4\ell_{\star}}   \nonumber \\
&\leq  C \theta^2 (1+|\log(\alpha)|)^2 \left(\frac{\omega_m}{m}\right)^{4 \ell_{\star}}\sum_{\bfi, \bfj\in \Xi_{0,m}: \| \bfi-\bfj\| > 5 \ell_{\star} \omega_m \sqrt{d}} \|\bfs (\bfi) - \bfs(\bfj) \|^{4\nu -4\ell_{\star}}  \nonumber \\
&\leq C\theta^2 (1+|\log(\alpha)|)^2 \left(\frac{\omega_m}{m}\right)^{4 \ell_{\star}} \sum_{\bfi, \bfj\in \Xi_{0,m}: \| \bfi-\bfj\| > 5 \ell_{\star} \omega_m \sqrt{d}} \left(\frac{ \|\bfi - \bfj \| }{m}\right)^{4\nu -4\ell_{\star}} \nonumber \\
&\leq C \theta^2 (1+|\log(\alpha)|)^2 \left(\frac{\omega_m}{m}\right)^{4\ell_{\star}} \sum_{\bfi\in \Xi_{0,m}} m^d \int_{\omega_m/m}^{\sqrt{d}} s^{4 \nu -4\ell_{\star} +d-1} \ud s  \nonumber \\
&\stackrel{(iii)}{\leq} C \theta^2  (1+|\log(\alpha)|)^2 m^{2d} \left(\frac{\omega_m}{m}\right)^{4\nu+d},
\end{align}
as $m\to\infty$, where the inequality (i) follows because all $s_q(\bfi +\bfk_1 \omega_m ) , s_q(\bfj +\bfk_2 \omega_m ) , s_q (\bfi) ,  s_q(\bfj) \in [0,1]$ and the power $2\ell_{\star}-1 =2 \lceil \nu+d/2 \rceil - 1 \geq 2 -1 = 1$,
the inequality (ii) follows from Lemma \ref{lem:Macov.deriv}, and (iii) follows from $|\Xi_{0,m}|\leq m^d$ and the definition that $\ell_{\star}=\lceil \nu+d/2 \rceil \geq \nu+d/2 > \nu+d/4$ so $4\nu-4\ell_{\star} +d <0$.

Therefore, we combine \eqref{eq:EV.case1.ratio2}, \eqref{eq:EV.case2.ratio2}, \eqref{eq:Sigma.W.1}, \eqref{eq:Sigma.W.2}, \eqref{eq:Sigma.W.3} and \eqref{eq:Sigma.W.4} to conclude that for all sufficiently large $m$,
\begin{align}\label{eq:Sigma.W.5}
\big\|\Sigma_W\big\|_F^2 &\leq \frac{(\tau C_{V,0})^2}{\left\{\EE \left(V_{0,d,\ell_{\star}}\right)\right\}^2} \cdot \frac{2}{(\tau C_{V,0})^2} \Bigg\{C\tau^2 m^d + C\theta^2 m^d \omega_m^d \left(\frac{\omega_m}{m}\right)^{4\nu} \nonumber \\
&\quad +  C \theta^2 (1+|\log(\alpha)|)^2 m^{2d} \left(\frac{\omega_m}{m}\right)^{4\nu+d}\Bigg\}  \nonumber \\
&\stackrel{(i)}{\leq} \frac{C}{\tau^2 m^{2d}} \Bigg\{\tau^2 m^d + \theta^2 m^d \omega_m^d \left(\frac{\omega_m}{m}\right)^{4\nu} \nonumber \\
&\quad +  \theta^2 (1+|\log(\alpha)|)^2 m^{2d} \left(\frac{\omega_m}{m}\right)^{4\nu+d}\Bigg\} \nonumber \\
&= C \left\{m^{-d} + \frac{\theta^2}{\tau^2}\left(\frac{\omega_m}{m}\right)^{4\nu+d}  \left(1+(1+|\log(\alpha)|)^2\right)  \right\} ,
\end{align}
where the inequality (i) follows from \eqref{eq:tau.CV0}, \eqref{eq:EV.case1.ratio2}, and \eqref{eq:EV.case2.ratio2}.
\vspace{3mm}

\noindent \underline{(2) Upper bound for $\|\Sigma^{1/2}_U A \Sigma^{1/2}_U\|_F^2$ and $\|\Sigma_U\|_F^2$.}
\vspace{2mm}

Based on the definition of $\Sigma_U$ and $A$, we partition them into $2\times 2$ blocks of size $|\Xi_{1,m}|\times |\Xi_{1,m}|$ as follows:
\begin{align*}
& \Sigma_U = \begin{pmatrix}
\Sigma_U^{1,1} & \Sigma_U^{1,2} \\
\Sigma_U^{2,1} & \Sigma_U^{2,2}
\end{pmatrix},
\qquad
A = \begin{pmatrix}
A^{1,1} & A^{1,2} \\
A^{2,1} & A^{2,2}
\end{pmatrix},
\end{align*}
where $\Sigma_U^{j,k}$ and $A^{j,k}$ are $|\Xi_{1,m}|\times |\Xi_{1,m}|$ matrices for $j,k\in \{1,2\}$. From the definition of $A$, we have $A^{1,1}=A^{2,2}=0\times I_{|\Xi_{1,m}|}$ and $A^{1,2}=A^{2,1}=I_{|\Xi_{1,m}|}$. Consequently, we observe that
\begin{align}\label{eq:SAS.1}
\left\|\Sigma_U^{1/2} A \Sigma_U^{1/2} \right\|_F^2 & = \tr\left(A \Sigma_U A \Sigma_U\right) \nonumber \\
&=\frac{1}{4} \tr\left\{
\begin{pmatrix}
\Sigma_U^{2,1} & \Sigma_U^{2,2} \\
\Sigma_U^{1,1} & \Sigma_U^{1,2}
\end{pmatrix}
\begin{pmatrix}
\Sigma_U^{2,1} & \Sigma_U^{2,2} \\
\Sigma_U^{1,1} & \Sigma_U^{1,2}
\end{pmatrix}
\right\} \nonumber \\
&=\frac{1}{4} \tr \left\{
\begin{pmatrix}
\Sigma_U^{2,1}\Sigma_U^{2,1} + \Sigma_U^{1,1}\Sigma_U^{2,2} & \Sigma_U^{2,1}\Sigma_U^{2,2}+\Sigma_U^{2,2}\Sigma_U^{1,2} \\
\Sigma_U^{1,1}\Sigma_U^{2,1} + \Sigma_U^{1,2}\Sigma_U^{1,1} & \Sigma_U^{1,2}\Sigma_U^{1,2}+\Sigma_U^{1,1}\Sigma_U^{2,2}
\end{pmatrix}
\right\} \nonumber \\
&= \frac{1}{2}\left\{ \tr\left(\Sigma_U^{1,1}\Sigma_U^{2,2}\right) + \tr\left(\Sigma_U^{1,2}\Sigma_U^{1,2}\right) \right\}.
\end{align}

We further observe that
\begin{align} \label{eq:Sig11.Sig22}
&\quad ~ \tr\left( \Sigma_U^{1,1}\Sigma_U^{2,2}\right)  =   \sum_{1\leq i,j\leq |\Xi_{1,m}|} \left(\Sigma_U^{1,1}\right)_{ij}\left(\Sigma_U^{2,2}\right)_{ij}  \nonumber \\
&=   \frac{1}{ \left\{ \EE \left( V_{1, d, \ell_{\star}} \right) \right\}^2}   \sum_{ \bfi, \bfj \in \Xi_{1,m}}
\EE \left\{ \left[\nabla_{d,\ell_{\star}} Y(\bfs(\bfi))\right] \left[\nabla_{d,\ell_{\star}} Y(\bfs(\bfj)) \right]\right\} \nonumber\\
&\qquad\times \EE \left\{ \left[\nabla_{d,\ell_{\star}} Y(\bfs(\bfi+\bfe_1))\right] \left[\nabla_{d,\ell_{\star}} Y(\bfs(\bfj+\bfe_1))\right]\right\}  \nonumber \\
&=  \frac{1}{ \left\{ \EE \left( V_{1,  d, \ell_{\star}} \right) \right\}^2}  \sum_{\bfi, \bfj \in \Xi_{1,m}}
\sum_{0\leq k_1,\ldots, k_{4d}\leq \ell_{\star}} c_{\bfi,d,\ell_{\star}}^{(k_1,\ldots, k_d)}
c_{\bfj,  d, \ell_{\star}}^{(k_{d+1},\ldots, k_{2d})} c_{\bfi+\bfe_1,  d, \ell_{\star}}^{(k_{2d+1},\ldots, k_{3d})}
c_{\bfj+\bfe_1, d, \ell_{\star}}^{(k_{3d+1},\ldots, k_{4d})}    \nonumber \\
& \quad\times  \left[\tau \Ical \left\{ \bfs(\bfi+\bfk_1 \omega_m )= \bfs(\bfj+ \bfk_2 \omega_m ) \right\}
+ \theta K_{\alpha,\nu}  \left( \bfs(\bfi+\bfk_1 \omega_m )- \bfs(\bfj+ \bfk_2 \omega_m )  \right) \right]  \nonumber \\
& \quad\times  \Big[\tau \Ical \left\{ \bfs(\bfi+\bfe_1 +\bfk_3 \omega_m )= \bfs(\bfj+\bfe_1 + \bfk_4 \omega_m ) \right\} \nonumber \\
&\qquad + \theta K_{\alpha,\nu} \left( \bfs(\bfi+\bfe_1 +\bfk_3 \omega_m )- \bfs(\bfj+\bfe_1 + \bfk_4 \omega_m ) \right) \Big] ,
\end{align}
and similarly
\begin{align} \label{eq:Sig12.Sig12}
&\quad~ \tr\left( \Sigma_U^{1,2}\Sigma_U^{1,2}\right) =   \sum_{1\leq i,j\leq |\Xi_{1,m}|} \left(\Sigma_U^{1,2}\right)_{ij} \left(\Sigma_U^{1,2}\right)_{ji} =  \sum_{1\leq i,j\leq |\Xi_{1,m}|} \left(\Sigma_U^{1,2}\right)_{ij} \left(\Sigma_U^{2,1}\right)_{ji} \nonumber \\
&=   \frac{1}{ \left\{ \EE \left(V_{1,  d, \ell_{\star}} \right) \right\}^2} \sum_{\bfi, \bfj \in \Xi_{1,m}}
\EE \left\{ \left[\nabla_{d,\ell_{\star}} Y(\bfs(\bfi)) \right] \left[\nabla_{d,\ell_{\star}} Y(\bfs(\bfj+\bfe_1)) \right] \right\} \nonumber \\
&\qquad\times \EE \left\{ \left[\nabla_{d,\ell_{\star}} Y(\bfs(\bfi+\bfe_1))\right] \left[\nabla_{d,\ell_{\star}} Y(\bfs(\bfj))\right] \right\} \nonumber \\
&=  \frac{1}{\left\{ \EE \left( V_{1, d, \ell_{\star}} \right) \right\}^2}   \sum_{\bfi, \bfj \in \Xi_{1,m}}  \sum_{0\leq k_1,\ldots, k_{4d}\leq \ell_{\star}} c_{\bfi,d,\ell_{\star}}^{(k_1,\ldots, k_d)}
c_{\bfj+\bfe_1,  d, \ell_{\star}}^{(k_{d+1},\ldots, k_{2d})} c_{\bfi+\bfe_1,  d, \ell_{\star}}^{(k_{2d+1},\ldots, k_{3d})}
c_{\bfj, d, \ell_{\star}}^{(k_{3d+1},\ldots, k_{4d})}    \nonumber \\
&\quad \times  \left[\tau \Ical \left\{ \bfs(\bfi+\bfk_1 \omega_m )= \bfs(\bfj+\bfe_1+ \bfk_2 \omega_m ) \right\}
+ \theta K_{\alpha,\nu}  \left( \bfs(\bfi+\bfk_1 \omega_m )- \bfs(\bfj+\bfe_1+ \bfk_2 \omega_m )  \right) \right]  \nonumber \\
&\quad \times  \Big[\tau \Ical \left\{ \bfs(\bfi+\bfe_1 +\bfk_3 \omega_m )= \bfs(\bfj + \bfk_4 \omega_m ) \right\} \nonumber \\
&\qquad + \theta K_{\alpha,\nu}  \left( \bfs(\bfi+\bfe_1 +\bfk_3 \omega_m )- \bfs(\bfj + \bfk_4 \omega_m )  \right) \Big].
\end{align}
Let $\{\bfb_1, \bfb_2 \} = \{ 0_{|\Xi_{1,m}|}, \bfe _1 \}$. Then by \eqref{eq:Sigma.W.2}, we have that
\begin{align} \label{eq:6.4}
& \tau^2 \sum_{\bfi, \bfj \in \Xi_{1,m}}  \sum_{0\leq k_1,\ldots, k_{4d}\leq \ell_{\star}} c_{\bfi, d, \ell_{\star}}^{(k_1,\ldots, k_d)}  c_{\bfj+\bfb_1,  d, \ell_{\star}}^{(k_{d+1},\ldots, k_{2d})} c_{\bfi+\bfe_1,  d, \ell_{\star}}^{(k_{2d+1},\ldots, k_{3d})} c_{\bfj+\bfb_2,  d, \ell_{\star}}^{(k_{3d+1},\ldots, k_{4d})}  \nonumber \\
&\times \Ical\{ \bfs (\bfi + \bfk_1 \omega_m)  = \bfs(\bfj + \bfb_1 + \bfk_2 \omega_m ) \} \Ical\{ \bfs(\bfi+\bfe_1+ \bfk_3  \omega_m)  = \bfs(\bfj+\bfb_2 + \bfk_4 \omega_m ) \}   \nonumber \\
\leq{}& C \tau^2 m^d,   \nonumber \\
&  \tau \sum_{\bfi, \bfj \in \Xi_{1,m}}  \sum_{0\leq k_1,\ldots, k_{4d}\leq \ell_{\star}} c_{\bfi, d, \ell_{\star}}^{(k_1,\ldots, k_d)} c_{\bfj+\bfb_1,  d, \ell_{\star}}^{(k_{d+1},\ldots, k_{2d})} c_{\bfi+\bfe_1,  d, \ell_{\star}}^{(k_{2d+1},\ldots, k_{3d})}  c_{\bfj+\bfb_2,  d, \ell_{\star}}^{(k_{3d+1},\ldots, k_{4d})} \nonumber \\
&\times  \Ical \{ \bfs(\bfi+ \bfk_1 \omega_m) = \bfs(\bfj +\bfb_1 + \bfk_2\omega_m ) \}
\theta K_{\alpha,\nu}  \left( \bfs(\bfi+\bfe_1 + \bfk_3 \omega_m )- \bfs(\bfj+\bfb_2 + \bfk_4 \omega_m ) \right)  \nonumber \\
\leq{}& C \tau \theta m^d \left(\frac{\omega_m}{m}\right)^{2\nu},  \nonumber \\
&  \tau \sum_{\bfi, \bfj \in \Xi_{1,m}}  \sum_{0\leq k_1,\ldots, k_{4d}\leq \ell_{\star}} c_{\bfi,  d, \ell_{\star}}^{(k_1,\ldots, k_d)} c_{\bfj+\bfb_1,  d, \ell_{\star}}^{(k_{d+1},\ldots, k_{2d})} c_{\bfi+\bfe_1,  d, \ell_{\star}}^{(k_{2d+1},\ldots, k_{3d})} c_{\bfj+\bfb_2,  d, \ell_{\star}}^{(k_{3d+1},\ldots, k_{4d})} \nonumber \\
& \times \theta K_{\alpha,\nu}  \left( \bfs(\bfi+ \bfk_1 \omega_m )- \bfs(\bfj + \bfb_1 + \bfk_2 \omega_m ) \right)  \Ical \{ \bfs(\bfi+\bfe_1+ \bfk_3  \omega_m) =\bfs( \bfj+\bfb_2 + \bfk_4 \omega_m ) \} \nonumber \\
\leq{}& C\tau \theta m^d\left(\frac{\omega_m}{m}\right)^{2\nu},
\end{align}
By the Cauchy-Schwarz inequality, we can obtain from \eqref{eq:6.4}, \eqref{eq:Sigma.W.3} and \eqref{eq:Sigma.W.4} that for all sufficiently large $m$,
\begin{align}
&  \sum_{\bfi, \bfj \in \Xi_{1,m}: \|\bfi-\bfj\| \leq 5 \ell \omega_m \sqrt{d}} \sum_{0\leq k_1,\ldots, k_{4d}\leq \ell_{\star}} c_{\bfi, d, \ell}^{(k_1,\ldots, k_d)} c_{\bfj+\bfb_1,  d, \ell}^{(k_{d+1},\ldots, k_{2d})}  c_{\bfi+\bfe _1,  d, \ell}^{(k_{2d+1},\ldots, k_{3d})} c_{\bfj+\bfb_2,  d, \ell}^{(k_{3d+1},\ldots, k_{4d})} \nonumber \\
&\quad \times  \theta K_{\alpha,\nu}  \left( \bfs(\bfi + \bfk_1 \omega_m )- \bfs(\bfj + \bfb_1 + \bfk_2 \omega_m ) \right) \nonumber\\
&\quad \times \theta K_{\alpha,\nu}  \left( \bfs( \bfi+ \bfe_1 + \bfk_3 \omega_m )- \bfs(\bfj+\bfb_2 + \bfk_4 \omega_m) \right)  \nonumber \\
&\leq   C\theta^2  m^d \omega_m^d \left(\frac{\omega_m}{m}\right)^{4\nu}, \label{eq:Sigma.U.1} \\
& \sum_{\bfi, \bfj \in \Xi_{1,m}: \|\bfi-\bfj\| > 5 \ell \omega_m \sqrt{d}} \sum_{0\leq k_1,\ldots, k_{4d}\leq \ell_{\star}} c_{\bfi, d, \ell}^{(k_1,\ldots, k_d)} c_{\bfj+\bfb_1,  d, \ell}^{(k_{d+1},\ldots, k_{2d})}  c_{\bfi+\bfe _1,  d, \ell}^{(k_{2d+1},\ldots, k_{3d})} c_{\bfj+\bfb_2,  d, \ell}^{(k_{3d+1},\ldots, k_{4d})} \nonumber \\
&\quad \times  \theta K_{\alpha,\nu}  \left( \bfs(\bfi + \bfk_1 \omega_m )- \bfs(\bfj + \bfb_1 + \bfk_2 \omega_m ) \right) \nonumber\\
&\quad \times  \theta K_{\alpha,\nu}  \left( \bfs( \bfi+ \bfe_1 + \bfk_3 \omega_m )- \bfs(\bfj+\bfb_2 + \bfk_4 \omega_m) \right)  \nonumber \\
&\leq C \theta^2 [1+|\log(\alpha)|]^2  m^{2d} \left(\frac{\omega_m}{m}\right)^{4\nu+d} . \label{eq:Sigma.U.2}
\end{align}
We combine \eqref{eq:EV.case1.ratio1}, \eqref{eq:EV.case2.ratio1}, \eqref{eq:SAS.1}, \eqref{eq:Sig11.Sig22}, \eqref{eq:Sig12.Sig12}, \eqref{eq:6.4}, \eqref{eq:Sigma.U.1}, \eqref{eq:Sigma.U.2} to conclude that for all sufficiently large $m$,
\begin{align} \label{eq:Sigma.U.3}
\left\|\Sigma_U^{1/2} A \Sigma_U^{1/2} \right\|_F^2 & \leq   \frac{\left[\theta g_{\ell_{\star},\nu}\right]^2}{\EE \left(V_{1,d,\ell_{\star}}^2\right)} \cdot \frac{1}{\left[\theta g_{\ell_{\star},\nu}\right]^2} \times \Bigg\{ C\tau^2 m^d + C\tau\theta m^d \left(\frac{\omega_m}{m}\right)^{2\nu} \nonumber \\
&\qquad + C \theta^2  m^d \omega_m^d \left(\frac{\omega_m}{m}\right)^{4\nu}  + C \theta^2 [1+|\log(\alpha)|]^2  m^{2d} \left(\frac{\omega_m}{m}\right)^{4\nu+d} \Bigg\} \nonumber \\
&\leq \frac{C(1+o(1))}{m^{2d} \left(\frac{\omega_m}{m}\right)^{4\nu}\left[\theta g_{\ell_{\star},\nu}\right]^2} \Bigg\{\tau^2 m^d + \tau\theta m^d \left(\frac{\omega_m}{m}\right)^{2\nu} + \theta^2  m^d \omega_m^d \left(\frac{\omega_m}{m}\right)^{4\nu}  \nonumber \\
&\qquad +\theta^2 [1+|\log(\alpha)|]^2 m^{2d} \left(\frac{\omega_m}{m}\right)^{4\nu+d} \Bigg\} \nonumber \\
&\leq C \Bigg\{\frac{\tau^2}{\theta^2} m^{-d}\left(\frac{\omega_m}{m}\right)^{-4\nu} + \frac{\tau}{\theta} m^{-d} \left(\frac{\omega_m}{m}\right)^{-2\nu}  \nonumber \\
&\qquad +  \left( 1+ [1+|\log(\alpha)|]^2 \right) \left(\frac{\omega_m}{m}\right)^{d} \Bigg\} .
\end{align}
For $\|\Sigma_U\|_F^2 = \tr\big[(\Sigma_U^{1,1})^2\big]+\tr\big[(\Sigma_U^{2,2})^2\big]+\tr\big[(\Sigma_U^{1,2})^2\big]+\tr\big[(\Sigma_U^{2,1})^2\big]$, by the similar argument to above, we have that for all sufficiently large $m$,
\begin{align} \label{eq:Sigma.U.4}
\left\|\Sigma_U \right\|_F^2 &\leq C \Bigg\{\frac{\tau^2}{\theta^2} m^{-d}\left(\frac{\omega_m}{m}\right)^{-4\nu} + \frac{\tau}{\theta} m^{-d} \left(\frac{\omega_m}{m}\right)^{-2\nu}  \nonumber \\
&\qquad + \left( 1+ [1+|\log(\alpha)|]^2 \right) \left(\frac{\omega_m}{m}\right)^{d} \Bigg\} .
\end{align}

\subsection{A Lemma on the Derivatives of Mat\'ern Covariance Function} \label{subsec:Matern.deriv}
\begin{lemma}\label{lem:Macov.deriv}
For any $\sa=(a_1,\ldots,a_d)\in \NN^d$, the partial derivative $\sD^{\sa} \theta K_{\alpha,\nu} (\bfx)$ is a continuous function. For any $\ell\in \NN$, there exists a positive constant $C_{\sD,\ell}$ that only depends on $\ell,\nu,d$, such that
\begin{align*}
& \left|\sD^{\sa} \theta K_{\alpha,\nu} (\bfs-\bft)\right| \leq C_{\sD,\ell} \theta \left(1+|\log \alpha| \right) \|\bfs-\bft \|^{2\nu - 2\ell},
\end{align*}
for all $\sa=(a_1,\ldots,a_d)\in \NN^d$ satisfying $a_1+\ldots + a_d = 2\ell$ and all $\bfs,\bft \in [0,1]^d, \bfs\neq \bft$.
\end{lemma}

\begin{proof}[$\pof$ Lemma \ref{lem:Macov.deriv}]
We consider two different cases below. For short, we write $r=\|\bfs-\bft\|$ and $\overline \sK(r;\theta,\alpha,\nu) = \theta K_{\alpha,\nu} (\bfs-\bft)$. As such, $\sD^{\sa} \theta K_{\alpha,\nu} (\bfs-\bft)$ is the derivative of the composite function $\overline \sK\circ r$.

We first cite an important result about the higher-order derivatives of composite functions. Proposition 1 of \citet{Har06} has generalized this formula to the functions with multivariate arguments: for the composite function $f(g(\bfx))$ with $\bfx=(x_1,\ldots,x_d)$ and any $k\in \NN$, its $k$th partial derivative satisfies
\begin{align} \label{eq:hardy}
& \frac{\partial^k f(g(\bfx))}{\partial \tilde x_1\ldots \partial \tilde x_k} = \sum_{\pi} \left(\frac{\ud^{|\pi|} f}{\ud y^{|\pi|}}\right) \Bigg|_{y=g(\bfx)}  \prod_{B\in \pi} \frac{\partial^{|B|}g(\bfx)}{\prod_{j\in B} \partial \tilde x_j},
\end{align}
where each $\tilde x_j$'s represents some component in $\bfx=(x_1,\ldots,x_d)$, the sum is over all partition $\pi$ of $\{1,\ldots,k\}$, the product is over all blocks $B$ in the partition $\pi$, and $|\pi|$ denotes the number of blocks in $\pi$. Proposition 2 of \citet{Har06} has shown that even if some $\tilde x_j$'s correspond to the same component in $\bfx=(x_1,\ldots,x_d)$, \eqref{eq:hardy} still holds true.

In this case, we take $f(\cdot)$ to be $\overline \sK(\cdot;\theta,\alpha,\nu)$ and $g(\cdot)$ to be the radius function $r(\cdot)$. By induction, it is straightforward to verify that for any $d\in \ZZ_+$, for any $k\in \NN$, for all $\sa=(a_1,\ldots,a_d)\in \NN^d$ satisfying $a_1+\ldots + a_d = k$, the partial derivatives of $r(\cdot)$ satisfies
\begin{align} \label{eq:Rad.deriv}
& \left| \sD^{\sa} r \right| \leq C_{k}'/r^{k-1},
\end{align}
for some absolute constant $C_{k}'>0$ that only depends on $k,d$. Therefore, we conclude from \eqref{eq:hardy} and \eqref{eq:Rad.deriv} that for any $\ell\in \NN$, for all $\sa=(a_1,\ldots,a_d)\in \NN^d$ satisfying $a_1+\ldots + a_d = 2\ell$ and all $\bfs,\bft \in [0,1]^d, \bfs\neq \bft$,
\begin{align} \label{eq:DK.composite}
& \left|\sD^{\sa} \theta K_{\alpha,\nu} (\bfs-\bft)\right| \leq \sum_{k=1}^{2\ell} C_{1k} \left| \frac{\ud^{k}}{\ud r^k} \overline \sK(r;\theta,\alpha,\nu)\right| \cdot r^{-(2\ell-k)} \nonumber \\
&=   \sum_{k=1}^{2\ell} C_{1k} r^{k-2\ell} \left| \frac{\ud^{k}}{\ud r^k} \overline \sK(r;\theta,\alpha,\nu)\right| ,
\end{align}
where $C_{1k}$'s are some positive constants that only depend on $\ell$, $0\leq k\leq \ell$, and $d$.

Next we bound the derivative of $\overline \sK(\cdot)$ in \eqref{eq:DK.composite}. Since the series expansion in \eqref{eq:K expansion} is convergent for all radius $r\in \RR_+$, we can differentiate term-by-term. We proceed in two different cases.
\vspace{2mm}

\noindent \underline{{\sc Case 1.}} If $\alpha\leq 2$, then $\alpha\|\bfs-\bft\|\le 2\sqrt{d}$. First consider the case $\nu\notin\ZZ_+$. Then for any $r>0$, any $k\in \NN$,
\begin{align}\label{eq:case1.alpha<2.notZ}
&\quad~  \left| \frac{\ud^{k}}{\ud r^{k}} \overline \sK(r;\theta,\alpha,\nu) \right|    \nonumber \\
&\leq \theta \left|\sum_{j=\lceil k/2\rceil}^{\infty}\xi_j\frac{(2j)!}{(2j-k)!} \alpha^{2j-2\nu} r^{2j-k}\right| +  \theta \left|\sum_{j=0}^{\infty} \xi_{\nu+j}^* \alpha^{2j} r^{2\nu-k+2j} \prod_{l=1}^{k}(2\nu+2j+1-l)\right| \nonumber \\
&= \theta r^{2\nu-k} \left\{\left|\sum_{j=\lceil k/2\rceil}^{\infty}\xi_j\frac{(2j)!}{(2j-k)!} (\alpha r)^{2j-2\nu} \right| +   \left|\sum_{j=0}^{\infty} \xi_{\nu+j}^* (\alpha r)^{2j}\prod_{l=1}^{k}(2\nu+2j+1-l)\right|\right\} \nonumber  \\
&\leq \theta r^{2\nu-k} \Bigg\{ \sum_{j=\lceil k/2\rceil}^{\infty} |\xi_j| \frac{(2j)!}{(2j-k)!} \left(2\sqrt{d}\right)^{2j-2\nu} +  \sum_{j=0}^{\infty} |\xi_{\nu+j}^*| \left(2\sqrt{d}\right)^{2j} \prod_{l=1}^{k}|2\nu+2j+1-l| \Bigg\} \nonumber \\
&\leq C_{2k} \theta r^{2\nu-k},
\end{align}
for some large constant $C_{2k}$ that depends only on $k,\nu,d$, where the last inequality follows because the two sequences are both convergent given the definition of $\xi_j$ in \eqref{eq:zeta.xi1} and $\xi_{\nu+j}^*$ and \eqref{eq:zeta.xi2}.
\vspace{1mm}

If $\nu\in\ZZ_+$, then we can similarly obtain that for a large absolute constant $C_{2k}$, for any $r>0$,
\begin{align}\label{eq:case1.alpha<2.Z}
& \left| \frac{\ud^{k}}{\ud r^{k}} \overline \sK(r;\theta,\alpha,\nu) \right| \leq C_{2k} \theta (1+|\log(\alpha)|) r^{2\nu-k}.
\end{align}

We combine \eqref{eq:DK.composite}, \eqref{eq:case1.alpha<2.notZ}, and \eqref{eq:case1.alpha<2.Z} to conclude that when $\alpha\leq 2$, for any $\ell\in \NN$, for all $\sa=(a_1,\ldots,a_d)\in \NN^d$ satisfying $a_1+\ldots + a_d = 2\ell$ and all $\bfs,\bft \in [0,1]^d, \bfs\neq \bft$,
\begin{align} \label{eq:DK.deriv.case1}
&\quad~ \left|\sD^{\sa} \theta K_{\alpha,\nu} (\bfs-\bft)\right| \nonumber \\
&\leq \sum_{k=1}^{2\ell} C_{1k} r^{k-2\ell} \cdot C_{2k} \theta (1+|\log(\alpha)|) r^{2\nu-k} \nonumber \\
&= \theta (1+|\log(\alpha)|) \sum_{k=1}^{2\ell} C_{1k} C_{2k} r^{2\nu-2\ell}  = C_{3\ell}  \theta (1+|\log(\alpha)|) r^{2\nu-2\ell},
\end{align}
where $C_{3\ell}=\sum_{k=1}^{2\ell} C_{1k} C_{2k}$.
\vspace{2mm}

\noindent {\sc Case 2.} If $\alpha\geq 2$, then we consider two different subcases below.
\vspace{1mm}

\noindent {(i)}  If $\alpha\|\bfs-\bft\|\leq 2\sqrt{d}$, then similarly to Case 1, \eqref{eq:case1.alpha<2.notZ} holds for any $r>0$. Hence \eqref{eq:DK.deriv.case1} remains true for this subcase.
\vspace{1mm}

\noindent {(ii)} For $\alpha\|\bfs-\bft\|> 2\sqrt{d}$, \citet{Paris84} has shown that for the modified Bessel function of the second kind $\Kcal_{\nu}$, it satisfies
\begin{align*}
&\frac{\Kcal_{\nu}(x) }{\Kcal_{\nu}(y)} > \ee^{y-x}(x/y)^{\nu},\quad \text{for } \nu>-\frac{1}{2}, ~0<x<y.
\end{align*}
Since $\alpha r=\alpha\|\bfs-\bft\|> 2\sqrt{d} \geq 2$, we set $x=2$ and $y=\alpha r$ to obtain that
\begin{align}\label{eq:Bessel.1}
& \Kcal_{\nu}(\alpha r) < (\alpha r/2)^{\nu} \ee^{2-\alpha r} \Kcal_{\nu}(2) = 2^{-\nu}\ee^2 \Kcal_{\nu}(2) (\alpha r)^{\nu} \ee^{-\alpha r}.
\end{align}

On the other hand, \citet{Watson1944} has proved that (on page 79)
\begin{align*}
& [x^{\nu}\Kcal_{\nu}(x)]^{(1)}=-x^{\nu}\Kcal_{\nu-1}(x), \qquad\text{for any }~\nu\in\RR, ~ x\in \RR_+.
\end{align*}
where the superscript $(k)$ denotes the $k$th derivative for $k\in \ZZ_+$. Let $h_{\nu}(r)=r^{\nu}\Kcal_{\nu}(\alpha r)$ for $r\in \RR_+$. By induction, we can see that for any integer $s\in \NN$ and any $r\in \RR_+$,
\begin{align}\label{eq:Bessel.h.deriv}
h^{(2s)}_{\nu}(r)&=  \sum_{k=0}^{s}c_{1,k,s} \alpha^{s+k} r^{\nu-s+k} \Kcal_{\nu-s-k}(\alpha r), \nonumber \\
h^{(2s+1)}_{\nu}(r)&= \sum_{k=0}^{s}c_{2,k,s}\alpha^{s+k+1} r^{\nu-s+k}\Kcal_{\nu-s-k-1}(\alpha r),
\end{align}
where $c_{1,k,s}$'s and $c_{2,k,s}$'s are finite constants dependent on $k,s,\nu$ only.

From \eqref{eq:Bessel.h.deriv}, for $\alpha r>2$ and any $s\in \NN$, we can find a finite constant $C_{4s}=\sum_{k=0}^s |c_{1,k,s}|>0$ that depends only on $s,\nu,d$, such that
\begin{align}\label{eq:Bessel.2}
\left|h_{\nu}^{(2s)}(r)\right| &\leq  \sum_{k=0}^{s} |c_{1,k,s}| \alpha^{s+k} r^{\nu-s+k} \Kcal_{\nu-s-k}(\alpha r) \stackrel{(i)}{=} \sum_{k=0}^{s} |c_{1,k,s}| \alpha^{s+k} r^{\nu-s+k} \Kcal_{|s+k-\nu|}(\alpha r) \nonumber \\
&\stackrel{(ii)}{\leq} 2^{-\nu}\ee^2  \sum_{k=0}^{s} |c_{1,k,s}| \alpha^{s+k} r^{\nu-s+k} \cdot \Kcal_{|s+k-\nu|}(2) (\alpha r)^{|s+k-\nu|} \ee^{-\alpha r} \nonumber \\
&=  2^{-\nu} \ee^2 \alpha^{\nu} r^{2\nu-2s} \sum_{k=0}^{s} |c_{1,k,s}| \Kcal_{|s+k-\nu|}(2) (\alpha r)^{|s+k-\nu|+(s+k-\nu)} \ee^{-\alpha r}  \nonumber \\
&\stackrel{(iii)}{\leq} C_{4s} 2^{-\nu} \ee^2 \alpha^{\nu} r^{2\nu-2s} (s+1)  \left\{\max_{t\in \{0,1,\ldots, |s-\nu| + |2s-\nu|\}}\Kcal_{t}(2)\right\} (\alpha r)^{2(|s-\nu|+|2s-\nu|)} \ee^{-\alpha r}  \nonumber \\
&\stackrel{(iv)}{\leq} C_{4s} 2^{-\nu} \ee^2 (s+1)\lceil 2(|s-\nu| +|2s-\nu|) \rceil! \alpha^{\nu} r^{2\nu-2s} \left\{\max_{t\in \{0,1,\ldots,\lceil|s-\nu| +|2s-\nu|\rceil\}}\Kcal_{t}(2)\right\} \nonumber \\
&\stackrel{(v)}{\leq} C_{5s} \alpha^{\nu} r^{2\nu-2s} ,
\end{align}
where (i) follows from the relation $\Kcal_{-s}(r)=\Kcal_s(r)$ for any $s\in \RR$ and $r\in \RR_+$; (ii) follows from the upper bound in \eqref{eq:Bessel.1}; (iii) follows from $\alpha r > 2$ and $|s+k-\nu|+(s+k-\nu)\leq 2|s+k-\nu| \leq 2(|s-\nu| +|2s-\nu|)$; (iv) follows because $\exp(x)>x^s/s!$ for all $x>0$ and $s\in \ZZ_+$, we can set $x=\lceil 2(|s-\nu| +|2s-\nu|) \rceil$, such that
\begin{align*}
& \exp(\alpha r)> (\alpha r)^{\lceil 2(|s-\nu| +|2s-\nu|) \rceil} / \lceil 2(|s-\nu| +|2s-\nu|) \rceil! \\
\text{and hence } & (\alpha r)^{2(|s-\nu| +|2s-\nu|)} \ee^{-\alpha r} \leq \lceil 2(|s-\nu| +|2s-\nu|) \rceil! ;
\end{align*}
and the constant $C_{5s}$ in (v) only depends on $s,\nu,d$ but not $\alpha$ and $r$.

Similarly, using the second relation in \eqref{eq:Bessel.h.deriv}, we can show that for $\alpha r>2$ and any $s\in \NN$, we can find a finite constant $C_{6s}>0$ that depends on $c_{2,k,s},s,\nu,d$, such that
\begin{align}\label{eq:Bessel.3}
\left|h_{\nu}^{(2s+1)}(r)\right| &\leq C_{6s} \alpha^{\nu} r^{2\nu-2s-1} .
\end{align}
Therefore, using the definition $\overline \sK(r;\theta,\alpha,\nu)= 2^{1-\nu}\Gamma(\nu)^{-1}\theta \alpha^{-\nu} r^{\nu} \Kcal_{\nu}(\alpha r) =  2^{1-\nu}\Gamma(\nu)^{-1}\theta \alpha^{-\nu} h_{\nu}(r)$, we combine \eqref{eq:DK.composite}, \eqref{eq:Bessel.2}, and \eqref{eq:Bessel.3} to conclude that when $\alpha\|\bfs-\bft\|> 2\sqrt{d}$, for any $\ell\in \NN$, for all $\sa=(a_1,\ldots,a_d)\in \NN^d$ satisfying $a_1+\ldots + a_d = 2\ell$ and all $\bfs,\bft \in [0,1]^d, \bfs\neq \bft$,
\begin{align} \label{eq:DK.deriv.case2}
&\quad~ \left|\sD^{\sa} \theta K_{\alpha,\nu} (\bfs-\bft)\right| \nonumber \\
&\leq \sum_{s=1}^{\ell} C_{7s} r^{2s-2\ell} \cdot \frac{2^{1-\nu}\theta \alpha^{-\nu}}{\Gamma(\nu)} \left|h_{\nu}^{(2s)}(r)\right| + \sum_{s=0}^{\ell-1} C_{8s} r^{2s+1-2\ell} \cdot \frac{2^{1-\nu}\theta \alpha^{-\nu}}{\Gamma(\nu)} \left|h_{\nu}^{(2s+1)}(r)\right| \nonumber \\
&\leq \sum_{s=1}^{\ell}\frac{2^{1-\nu} C_{7s}}{\Gamma(\nu)} \theta  \alpha^{-\nu}  r^{2s-2\ell} \cdot \alpha^{\nu} r^{2\nu-2s}  + \sum_{s=0}^{\ell-1}\frac{2^{1-\nu} C_{8s}}{\Gamma(\nu)} \theta  \alpha^{-\nu} r^{2s+1-2\ell} \cdot \alpha^{\nu} r^{2\nu-2s-1} \nonumber \\
&\leq C_{9\ell} \theta r^{2\nu-2\ell},
\end{align}
where $C_{7s},C_{8s},C_{9\ell}$ are constants that depends only on $0\leq s\leq \ell$ and $\ell,\nu,d$.

The conclusion is proved by combining Case 1 and Case 2.
\end{proof}

\section{Additional Numerical Results} \label{sec:add.simu}
Before we present the simulation results for $d=1$, we first explain the formulas for calculating the mean squared error for Bayesian posterior prediction and the oracle prediction. In the matrix format, our model can be written as $Y_n = F_n\beta + X_n + e_n$, where $X_n=(X(\bfs_1),\ldots,X(\bfs_n))^\T$ and $e_n=(\varepsilon(\bfs_1),\ldots,\varepsilon(\bfs_n))^\T$. Since in the simulations, we assign the normal prior $\beta \sim \Ncal(0,a_0I_p)$ as in Proposition \ref{prop:prior}, we can obtain that the conditional posterior distribution of $\beta$ is
\begin{align} \label{eq:beta.cond.post}
& \beta \mid \theta,\alpha,\tau, Y_n,F_n \sim \Ncal\left(\tilde \beta, ~\big[ F_n^\T \big\{\theta K_{\alpha}(S_n)+\tau I_n\big\}^{-1} F_n + a_0 I_p \big]^{-1} \right),  \nonumber \\
\text{where } & \tilde \beta  =  \big[ F_n^\T \big\{\theta K_{\alpha}(S_n)+\tau I_n\big\}^{-1} F_n + a_0 I_p \big]^{-1} F_n^\T \big\{\theta K_{\alpha}(S_n)+\tau I_n\big\}^{-1} Y_n.
\end{align}
Under the model $Y\sim \gp (\beta^\T \ff, \theta K_{\alpha} + \tau \delta_0)$, by Section 1.5 of \citet{Stein99a}, the best linear unbiased predictor for $Y(\bfs^*)$ at a new location $\bfs^*\in [0,1]^d$ based on the data $(Y_n,F_n)$ is
\begin{align}\label{eq:BLUP}
\widehat Y(\bfs^*;\theta,\alpha,\tau,\beta) & =
\beta^\T \ff(\bfs^*) + \theta K_{\alpha}(S_n, \bfs^*)^\T \big\{\theta K_{\alpha}(S_n)+\tau I_n\big\}^{-1} \left(Y_n - F_n \beta\right) ,
\end{align}
where $K(S_n,\bfs^*)=\{K(\bfs_1,\bfs^*),\ldots,K(\bfs_n,\bfs^*)\}^\T$ is an $n \times 1$ vector.

Under the Bayesian setup, we randomly draw $(\theta,\alpha,\tau,\beta)$ from the posterior $\Pi(\cdot\mid Y_n,F_n)$ to predict $Y(\bfs^*)$. We denote the predicted variable as $\tilde Y(\bfs^*)$. The Gaussian process predictive distribution implies that
\begin{align} \label{eq:GP.predict}
&\quad~ \tilde Y(\bfs^*) \mid Y_n,F_n, \theta,\alpha,\tau,\beta \nonumber \\
&\sim \Ncal \left(\widehat Y(\bfs^*;\theta,\alpha,\tau,\beta), ~ \theta \left\{K_{\alpha}(\bfs^*,\bfs^*)- \theta K_{\alpha}(S_n, \bfs^*)^\T  \big\{\theta K_{\alpha}(S_n)+\tau I_n\big\}^{-1}  K_{\alpha}(S_n, \bfs^*) \right\} \right).
\end{align}
From \eqref{eq:beta.cond.post}, \eqref{eq:BLUP} and \eqref{eq:GP.predict}, with some algebra, we can integrate out $\beta$ from the posterior distribution and obtain that for any $\bfs^*\in [0,1]^d$,
\begin{align}\label{eq:BLUP.pred}
&\tilde Y(\bfs^*)\mid Y_n,F_n,\theta,\alpha,\tau \sim \Ncal \Big( Y^{\dagger}(\bfs^*;\theta,\alpha,\tau), \vv_n(\bfs^*;\theta,\alpha,\tau) \Big), \quad \text{where } \nonumber \\
& Y^{\dagger}(\bfs^*;\theta,\alpha,\tau) = \theta K_{\alpha}(S_n,\bfs^*)^\T \big\{\theta K_{\alpha}(S_n)+\tau I_n\big\}^{-1} Y_n \nonumber \\
&\quad\quad + \bb(\bfs^*;\theta,\alpha,\tau)^\T \left[ F_n^\T \big\{\theta K_{\alpha}(S_n)+\tau I_n\big\}^{-1} F_n  + a_0I_p \right]^{-1} F_n^\T \big\{\theta K_{\alpha}(S_n)+\tau I_n\big\}^{-1} Y_n, \nonumber \\
&{\vv}_n(\bfs^*;\theta,\alpha,\tau) = \theta \left\{K_{\alpha}(\bfs^*,\bfs^*)- \theta K_{\alpha}(S_n, \bfs^*)^\T  \big\{\theta K_{\alpha}(S_n)+\tau I_n\big\}^{-1}  K_{\alpha}(S_n, \bfs^*) \right\} \nonumber \\
&\quad\quad +  \bb(\bfs^*;\theta,\alpha,\tau)^\T \left[ F_n^\T \big\{\theta K_{\alpha}(S_n)+\tau I_n\big\}^{-1} F_n  + a_0I_p \right]^{-1} \bb(\bfs^*;\theta,\alpha,\tau), \text{ and } \nonumber \\
& \bb(\bfs^*;\theta,\alpha,\tau) = \ff(\bfs^*) - F_n^\T \big\{\theta K_{\alpha}(S_n)+\tau I_n\big\}^{-1} \theta K_{\alpha}(S_n,\bfs^*) .
\end{align}
Let $Y_0(\bfs^*) = \beta_0^\T \ff(\bfs^*) + X(\bfs^*)$ be the true mean function at $\bfs^*$ without measurement error. Then, for the Bayesian posterior prediction $\tilde Y(\bfs^*)$ given a random draw $(\theta,\alpha,\tau)$ from the posterior density $\pi(\theta,\alpha,\tau\mid Y_n, F_n)$, the mean squared error has the formula
\begin{align} \label{eq:mse.bayes}
M_{\text{post}}(\bfs^*) &= \left[{\EE}_{\tilde Y(\bfs^*) \mid \theta,\alpha,\tau} \left\{\tilde Y(\bfs^*) \right\} - Y_0(\bfs^*)\right]^2 + {\Var}_{\tilde Y(\bfs^*) \mid \theta,\alpha,\tau} \left\{ \tilde Y(\bfs^*) \right\} \nonumber \\
&= \left\{ Y^{\dagger}(\bfs^*;\theta,\alpha,\tau) - \beta_0^\T \ff(\bfs^*) - X(\bfs^*) \right\}^2 + {\vv}_n(\bfs^*;\theta,\alpha,\tau),
\end{align}
where $Y^{\dagger}(\bfs^*;\theta,\alpha,\tau)$ and ${\vv}_n(\bfs^*;\theta,\alpha,\tau)$ are defined in \eqref{eq:BLUP.pred}.

For the oracle prediction mean squared error of the predicted value $\tilde Y_0(\bfs^*)$, we simply replace all $(\theta,\alpha,\tau)$ in \eqref{eq:mse.bayes} by the true parameters $(\theta_0,\alpha_0,\tau_0)$ to obtain that
\begin{align} \label{eq:mse.oracle}
M_{0}(\bfs^*) &= \left[{\EE}_{\tilde Y_0(\bfs^*) \mid \theta_0,\alpha_0,\tau_0} \left\{\tilde Y_0(\bfs^*) \right\} - Y_0(\bfs^*)\right]^2 + {\Var}_{\tilde Y_0(\bfs^*) \mid \theta_0,\alpha_0,\tau_0} \left\{ \tilde Y_0(\bfs^*) \right\} \nonumber \\
&= \left\{ Y^{\dagger}(\bfs^*;\theta_0,\alpha_0,\tau_0) - \beta_0^\T \ff(\bfs^*) - X(\bfs^*) \right\}^2 + {\vv}_n(\bfs^*;\theta_0,\alpha_0,\tau_0),
\end{align}
where $Y^{\dagger}(\bfs^*;\theta_0,\alpha_0,\tau_0)$ and ${\vv}_n(\bfs^*;\theta_0,\alpha_0,\tau_0)$ are defined in \eqref{eq:BLUP.pred} with $(\theta,\alpha,\tau)$ replaced by $(\theta_0,\alpha_0,\tau_0)$.
\vspace{8mm}

Now we present additional simulation results for the spatial Gaussian process model \eqref{eq:obs.model} with domain dimension $d=1$ and the isotropic Mat\'ern covariance function in \eqref{eq:MaternCov}. We still consider two values of the smoothness parameter $\nu=1/2$ and $\nu=1/4$. The true covariance parameters are set to be $\theta_0=5,\alpha_0=1,\tau_0=0.5$ for both $\nu=1/2$ and $\nu=1/4$. For $d=1$ and the domain $[0,1]$, we choose the sampling points $S_n$ to be the regular grid $\bfs_i=(2i-1)/(2n)$ for $i=1,\ldots,n$, where we choose $n=400, 500, 625, 781, 976, 1220, 1525, 1907, 2384, 2980$, such that the sample size roughly follows the geometric sequence $n\approx 400\times 1.25^{k-1}$ for $k=1,\ldots,10$. For the regression functions, we let $\ff(\bfs)=\left(1,\bfs,\bfs^2,\bfs^3\right)^\T$ for $\bfs \in [0,1]$ and $\beta_0=(1,0.66,-1.5,1)^\T$. Other aspects of the Bayesian simulation setup, including the prior specification and the posterior sampling procedures are all the same as the $d=2$ case described in the main text.

The results are summarized in Figure \ref{fig:para.0.5.dim1} for $\nu=1/2$ and Figure \ref{fig:para.0.25.dim1} for $\nu=1/4$. In the left panels, the boxplots of both the marginal posterior distributions of $\theta$ and $\tau$ contract to the true parameter values $\theta_0=5$ and $\tau_0=0.5$ as $n$ increases. The right panels show that the means of absolute differences from all posterior draws of $(\theta,\tau)$ to the true parameters $(\theta_0,\tau_0)$ decreases approximately linearly in the sample size on the logarithm scale, indicating polynomial posterior contraction rates for $\theta$ and $\alpha$.

\begin{figure}[H]
\center
\includegraphics[width=0.76\textwidth]{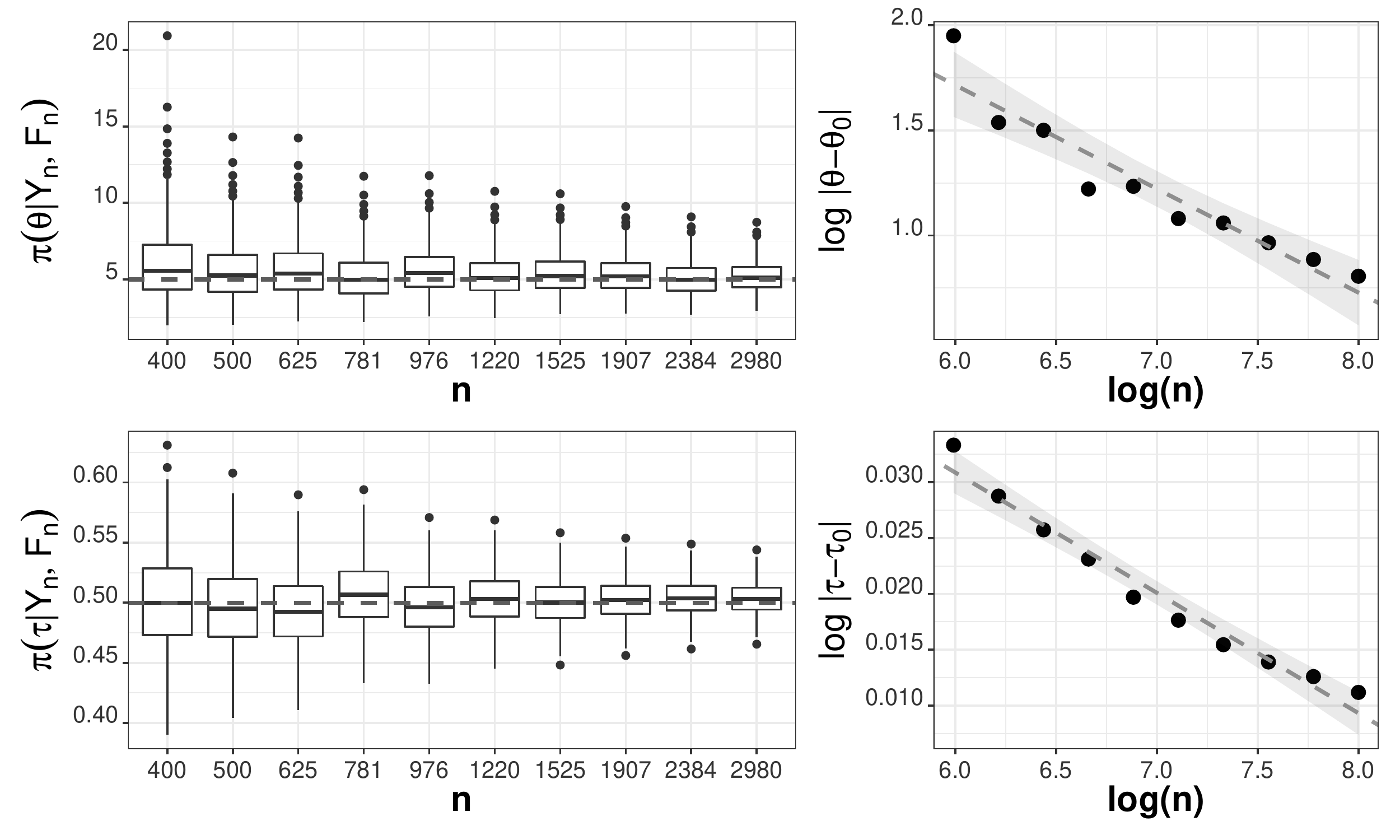}
\caption{\small Posterior contraction for $d=1$ and $\nu=1/2$. Left column: Boxplots for the marginal posterior densities of $\theta$ and $\tau$ versus the increasing sample size $n$. The grey dashed lines are the true parameters $\theta_0=5$ and $\tau_0=0.5$. Right column: Posterior means of $|\theta-\theta_0|$ and $|\tau-\tau_0|$ versus the increasing sample size $n$, on the logarithm scale. The dashed lines are the linear regression fits, and the grey shaded areas are the 95\% confidence bands. All posterior summaries are averaged over 40 macro Monte Carlo replications.}
\label{fig:para.0.5.dim1}
\end{figure}
\begin{figure}[H]
\center
\includegraphics[width=0.76\textwidth]{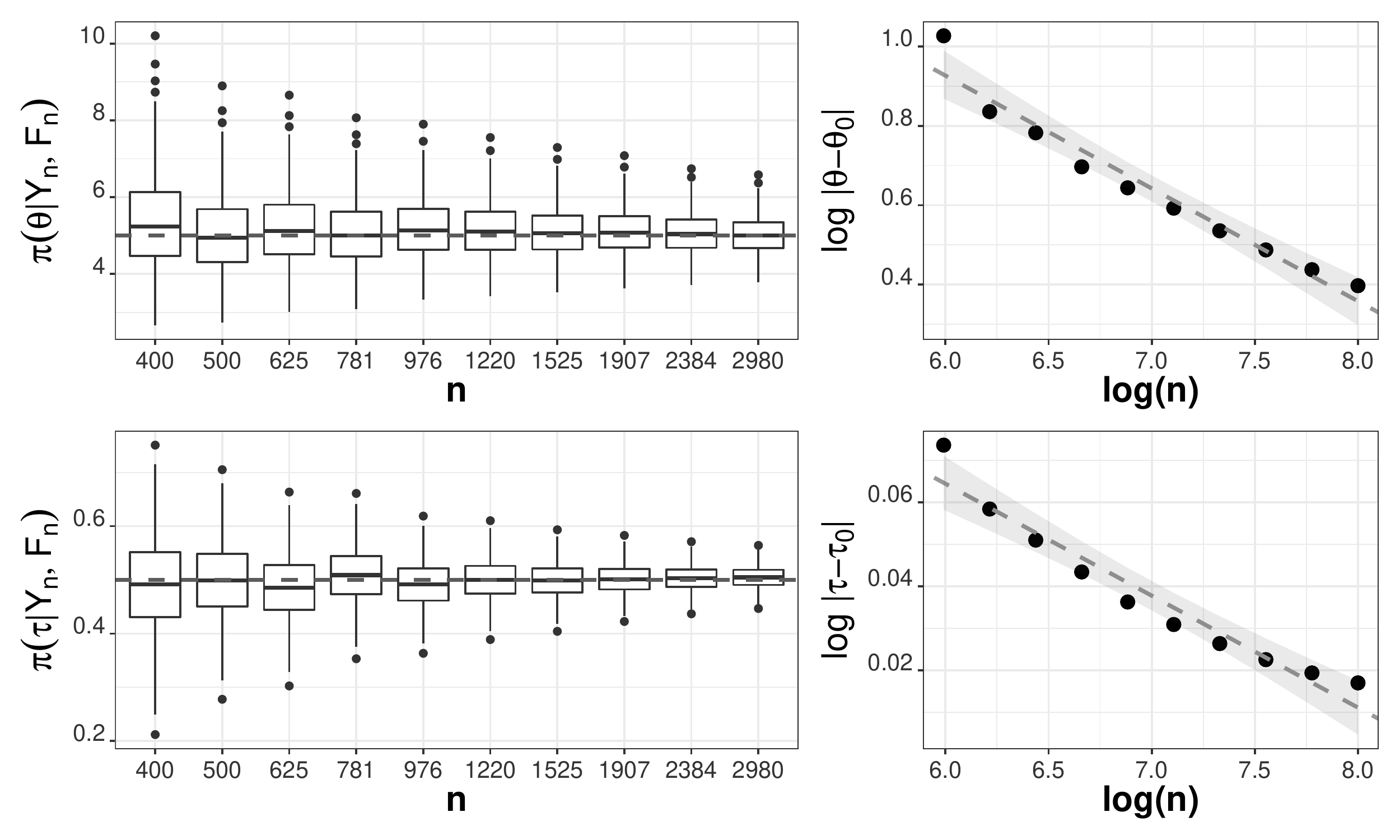}
\caption{\small Posterior contraction for $d=1$ and $\nu=1/4$. Left column: Boxplots for the marginal posterior densities of $\theta$ and $\tau$ versus the increasing sample size $n$. The grey dashed lines are the true parameters $\theta_0=5$ and $\tau_0=0.5$. Right column: Posterior means of $|\theta-\theta_0|$ and $|\tau-\tau_0|$ versus the increasing sample size $n$, on the logarithm scale. The dashed lines are the linear regression fits, and the grey shaded areas are the 95\% confidence bands. All posterior summaries are averaged over 40 macro Monte Carlo replications.}
\label{fig:para.0.25.dim1}
\end{figure}
\begin{figure}[H]
\center
\includegraphics[width=0.82\textwidth]{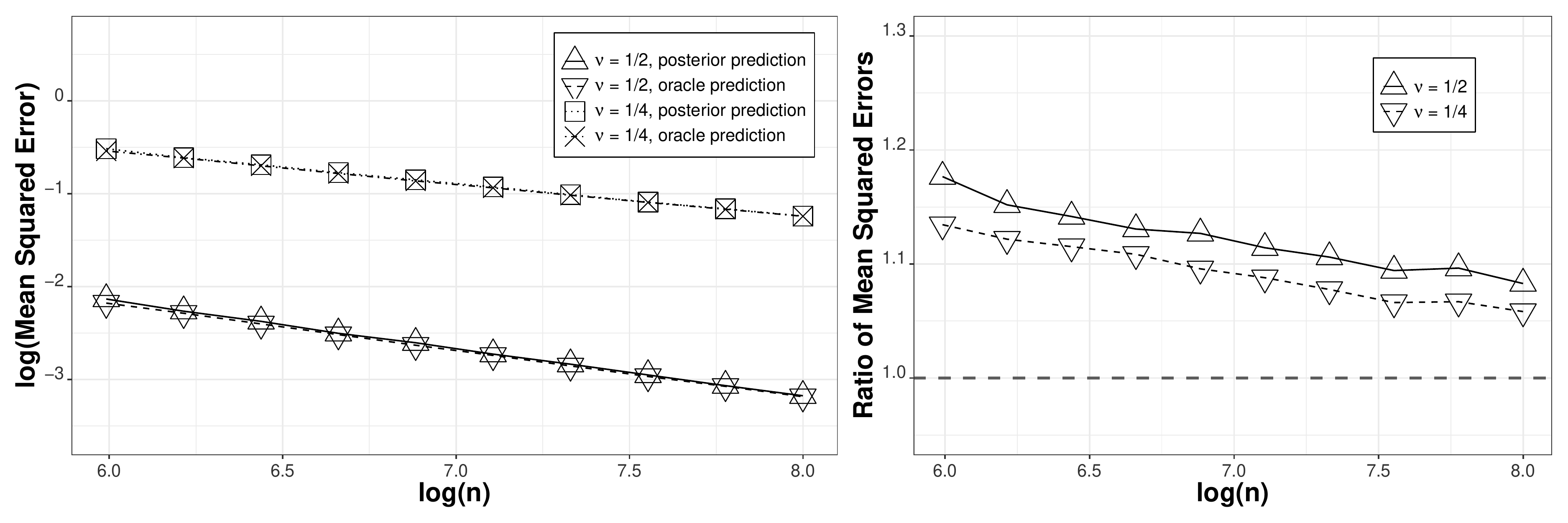}
\caption{\small Prediction mean squared errors for $\nu=1/2$ and $\nu=1/4$ when $d=1$. Left column: The prediction mean squared errors under both the Bayesian posterior prediction and the oracle prediction based on the true parameters. Right column: Ratios of the Bayesian prediction mean squared error and the oracle prediction mean squared error. All posterior summaries are averaged over 2000 testing locations in $[0,1]$ and 40 macro Monte Carlo replications.}
\label{fig:mse.1}
\end{figure}

We also investigate the Bayesian posterior prediction performance for the $d=1$ case and compare with the best possible prediction. We draw another $N=2000$ points $\bfs_1^*,\ldots,\bfs_N^*$ uniformly from the domain $[0,1]$ as the testing locations, and compute the prediction mean squared error $\EE_{Y_n}  \big\{N^{-1}\sum_{l=1}^N M_{\text{post}}(\bfs_l^*)\big\} $ versus the oracle prediction mean squared error $\EE_{Y_n}  \big\{N^{-1}\sum_{l=1}^N M_{0}(\bfs_l^*)\big\}$. The results are shown in Figure \ref{fig:mse.1}. The left panel shows that the Bayesian predictive posterior has almost the same mean square errors as the oracle prediction for both $\nu=1/2$ and $\nu=1/4$. The right panel shows that their ratio averaged over $N=2000$ testing locations decreases towards 1 as the sample size $n$ increases, indicating the asymptotic efficiency of Bayesian posterior prediction.

\bibliographystyle{Chicago}
\bibliography{bvm}

\end{document}